\newcommand{\pointsize}{11pt}

\documentclass[oneside, \pointsize]{amsbook}

\usepackage[
   includehead,
   includefoot,
     left = 1.0in, 
      top = 1.0in, 
    right = 1.0in,
   bottom = 1.0in
]{geometry}
\usepackage{fancyhdr}
\usepackage{setspace}
\usepackage{calc}

\fancyheadoffset[R]{0.5in} 

\fancypagestyle{prelim}{%    
   \renewcommand{\headrulewidth}{0pt} 
   \fancyhf{}           
   \pagenumbering{roman}    
   \cfoot{-\thepage-}       
}

\fancypagestyle{maintext}{%
   \renewcommand{\headrulewidth}{0.4pt}
   \pagenumbering{arabic}
   \fancyhf{}
   \fancyhead[L]{\rightmark}
   \cfoot{\thepage}
}

\fancypagestyle{abstract1}{%
   \renewcommand{\headrulewidth}{0.4pt}
	 \fancyheadoffset[R]{0in}
   \pagenumbering{arabic}
   \fancyhf{}
}

\fancypagestyle{abstract2}{%
   \renewcommand{\headrulewidth}{0.4pt}
   \fancyheadoffset[R]{0in}
	 \pagenumbering{arabic}
   \fancyhf{}
   \rhead{-\thepage-}
}

\numberwithin{figure}{chapter} 
\numberwithin{table}{chapter}
\numberwithin{equation}{chapter}
\numberwithin{section}{chapter}

\usepackage{amsmath,amsfonts}
\usepackage{amsthm}
\usepackage{amssymb}
\usepackage{geometry}
\usepackage{pgf}
\usepackage{xspace}
\usepackage{hyperref}
\usepackage{framed}
\usepackage{graphicx}
\usepackage{caption}
\usepackage{subcaption}
\usepackage{booktabs} %for nice tables
\usepackage{multirow} %multi rows for tables
\usepackage{setspace} % for double space
\usepackage{ragged2e} % for justify algorithm block
\usepackage{tabularx}
\usepackage{enumitem} %for space of item bullets
\usepackage{xcolor}
\usepackage{csquotes} % for block quotes
\usepackage{float}
\usepackage{tikz}
\usetikzlibrary{arrows}
\usetikzlibrary{calc}
\usepackage{siunitx}
\usepackage{listings} % for matlab code

\usepackage{algorithm}
\usepackage{algorithmic}

\newtheorem{theorem}{Theorem}[section]
\newtheorem{lemma}[theorem]{Lemma}

\newtheorem{corollary}[theorem]{Corollary}
\newtheorem{cor}[theorem]{Corollary}
\newtheorem{proposition}[theorem]{Proposition}

\theoremstyle{definition}

\newtheorem{definition}[theorem]{Definition}

\newtheorem{remark}[theorem]{Remark}
\newtheorem{example}[theorem]{Example}

\newcommand{\fmax}{f_{\text{\rm max}}}
\newcommand{\fmin}{f_{\text{\rm min}}}
\newcommand{\fhan}{f_{\text{\rm han}}}

\newcommand{\C}{{\mathbb C}}
\newcommand{\R}{{\mathbb R}}
\newcommand{\Q}{{\mathbb Q}}	
\newcommand{\Z}{{\mathbb Z}}
\newcommand{\N}{{\mathbb N}}

\let\ve=\mathbf
\newcommand\fractional[1]{\{#1\}}
\newcommand\ceil[1]{\lceil{#1}\rceil}
\newcommand\floor[1]{\lfloor{#1}\rfloor}
\renewcommand{\ll}{{\langle}}
\newcommand{\rr}{{\rangle}}

\newcommand{\CG}{{\mathcal{G}}}
\newcommand{\CE}{{\mathcal E}}
\newcommand{\CF}{{\mathcal F}}

\newcommand{\PP}{{ P}}

\newcommand{\Res}{\operatorname{Res}}
\newcommand{\res}{\operatorname{res}}

\newcommand{\Cone}{\operatorname{Cone}}
\newcommand{\Conv}{\operatorname{Conv}}
\newcommand{\TCone}{\operatorname{TCone}}

\newcommand{\vol}{{\mathrm{\rm vol}}}
\renewcommand{\ll}{{\langle}}

\newcommand\coneC{\mathfrak{c}}

\newcommand\coneU{\mathfrak{u}}

\DeclareMathOperator{\lcm}{lcm}
\renewcommand\d{\mathrm d}
\newcommand\e{\mathrm e}
\newcommand\T{\top}

\newcommand\vexi{\boldsymbol{\xi}}
\newcommand\vebeta{\boldsymbol{\beta}}
\newcommand\veeta{\boldsymbol{\eta}}

\newcommand{\norm}[1]{\left\lVert#1\right\rVert}
\renewcommand\d{\,\mathrm{d}}

\newcommand\latteintegrale{{\tt LattE integrale}\xspace} 
\newcommand\maple{{\tt Maple}\xspace}
\newcommand\mapleKnapsack{\emph{M-Knapsack}\xspace}
\newcommand\latteKnapsack{\emph{LattE Knapsack}\xspace}
\newcommand\coneApx{\emph{LattE Top-Ehrhart}\xspace}
\newcommand{\latte}{{\tt LattE}\xspace}
\newcommand{\latteInt}{{\tt LattE integrale}\xspace}

\newcommand{\cubpack}{{\tt CUBPACK}\xspace}
\newcommand{\matlab}{{\tt MATLAB}\xspace}

\newcommand\maplecode[1]{\textsf{#1}}
\newcommand\shellcode[1]{\texttt{#1}}

\newcommand\vealpha{\boldsymbol{\alpha}}
\renewcommand{\a}{{\vealpha}}

\newcommand{\timeFaster}[1]{{\bf #1}}
\newcommand{\timeSlower}[1]{#1}
\newcommand{\timeTie}[1]{#1}
\newcommand{\timeNotComputed}[1]{--}

\newcommand\smallstep[1]{\fractional{-#1}}  % careful when applying this macro
\newcommand\inlinefrac[2]{{#1}/{#2}}

\begin{document}
   \frontmatter

   \pagestyle{prelim}
   
   % Redefine plain page style so that the first pages of chapters
   % have desired page style.
   %
   \fancypagestyle{plain}{%
      \fancyhf{}
      \cfoot{-\thepage-}
   }%
\begin{center}
   \null\vfill
   \textbf{%
     Decomposition Methods for Nonlinear Optimization and Data Mining
   }%
   \\
   \bigskip
   By \\
   \bigskip
   BRANDON EMMANUEL DUTRA \\
   \bigskip
   B.S. (University of California, Davis) 2012 \\
   M.S. (University of California, Davis) 2015 \\
   \bigskip
   DISSERTATION \\
   \bigskip
   Submitted in partial satisfaction of the requirements for the
   degree of \\
   \bigskip
   DOCTOR OF PHILOSOPHY \\
   \bigskip
   in \\
   \bigskip
   Applied Mathematics \\
   \bigskip
   in the \\
   \bigskip
   OFFICE OF GRADUATE STUDIES \\
   \bigskip        
   of the \\
   \bigskip
   UNIVERSITY OF CALIFORNIA \\
   \bigskip
   DAVIS \\
   \bigskip
   Approved: \\
   \bigskip
   \bigskip
   \makebox[3in]{\hrulefill} \\
   Jesus De Loera \\
   \bigskip
   \bigskip
   \makebox[3in]{\hrulefill} \\
   Matthias K{\"o}ppe \\
   \bigskip
   \bigskip
   \makebox[3in]{\hrulefill} \\
   David Woodruff \\
   \bigskip
   Committee in Charge \\
   \bigskip
   2016 \\
   \vfill
\end{center}

   \newpage

   %%% (optional) copyright page <== this page is not numbered!
   \thispagestyle{empty}
   \begin{titlepage}
   \vspace*{50em}
   \begin{center} 
      \copyright \ Brandon E.\ Dutra, 2016.  All rights reserved.  
   \end{center}
   \end{titlepage}
   \newpage
   \stepcounter{page}
   
   % Begin Double Spacing
   %
   \doublespacing
   
   \setcounter{tocdepth}{2} %http://tex.stackexchange.com/questions/17877/how-to-show-subsections-and-subsubsections-in-toc
   \tableofcontents
   
   % Begin Double Spacing
   %
   %\doublespacing   
   
   \newpage
   
\textbf{Abstract}

We focus on two central themes in this dissertation. The first one is on decomposing polytopes and polynomials in ways that allow us to perform nonlinear optimization. We start off by explaining important results on decomposing a polytope into special polyhedra. We use these decompositions and develop methods for computing a special class of integrals exactly. Namely, we are interested in computing the exact value of integrals  of polynomial functions over convex polyhedra. We present prior work and new extensions of the integration algorithms. Every integration method we present requires that the polynomial has a special form. We explore two special polynomial decomposition algorithms that are useful for integrating polynomial functions. Both polynomial decompositions have strengths and weaknesses, and we experiment with how to practically use them. 

After developing practical algorithms and efficient software tools for integrating a polynomial over a polytope, we focus on the problem of maximizing a polynomial function over the continuous domain of a polytope. This maximization problem is NP-hard, but we develop approximation methods that run in polynomial time when the dimension is fixed. Moreover, our algorithm for approximating the maximum of a polynomial over a polytope is related to integrating the polynomial over the polytope. We show how the integration methods can be used for optimization. 

We then change topics slightly and consider a problem in combinatorics. Specifically, we seek to compute the function $E(t)$ that counts the number of nonnegative integer solutions to the equation $\alpha_1 x_1 + \cdots + \alpha_n x_n =t$ where the $\alpha_i$ are given positive integers. It is known that this function is a quasi-polynomial function, and computing every term is $\#P$-hard. Instead of computing every term, we compute the top $k$ terms of this function in polynomial time in varying dimension when $k$ is fixed. We review some applications and places where this counting function appears in mathematics. Our new algorithm for computing the largest order terms of $E(t)$ is based on the polyhedral decomposition methods we used in integration and optimization. We also use an additional polyhedral decomposition: Barvinok's fast decomposition of a polyhedral cone into unimodular cones.

The second central topic in this dissertation is on problems in data science. We first consider a heuristic for mixed-integer linear optimization. We show how many practical  mixed-integer linear have a special substructure containing set partition constraints. We then describe a nice data structure for finding feasible zero-one integer solutions to systems of set partition constraints.

Finally, we end with an applied project using data science methods in medical research. The focus is on identifying how T-cells and nervous-system cells interact in the spleen during inflammation. To study this problem, we apply topics in data science and computational geometry to clean data and model the problem. We then use clustering algorithms and develop models for identifying when a spleen sample is responding to inflammation. This project's lifetime surpasses the author's involvement in it. Nevertheless, we focus on the author's contributions, and on the future steps.

   \newpage

\section*{Acknowledgments}

\emph{My advisers: Jesus De Loera and Matthias K{\"o}ppe.} These two people have played a big role during my time at UC Davis. I appreciated working with two experts with diverse backgrounds. They greatly enriched my experience in graduate school. My advisers have also been supportive of my career development. As I went through graduate school, my career and research goals changed a few times, but my advisers have been agile and supportive. Some of the most important career development experiences I had were from my three summer internships, and I'm grateful for their support in getting and participating in these opportunities. 

Through undergraduate and graduate school at UC Davis, I have spent about 7 years working on the \latte project. I am very grateful for this enriching experience. I have enjoyed every aspect of working within the intersection of mathematical theory and mathematical software. \latte was my first experience with the software life cycle and with real software development tools like GNU Autotools, version control, unit testing, et  cetera. I have grown to love software development through this work, and how it has enriched my graduate experience. I also want to acknowledge the \latte users. I am thankful they find our work interesting, useful, and care enough to tell us how they use it and how to improve it. I hope that future graduate students get the experience of working on this great project.

\emph{Professors.} I wish to thank David Woodruff for being in my dissertation and qualifying exam committees. I also thank Dan Gusfield and Angela Cheer for also taking part in my qualifying exam committee. I also owe a thank you to Janko Gravner and Ben Morris for their support and letters of recommendation.

\emph{Co-authors.} I wish to thank Velleda Baldoni (University of Rome Tor Vergata, Rome, Italy), Nicole Berline (\'{E}cole Polytechnique, Palaiseau, France), and Mich\`{e}le Vergne (The Mathematical Institute of Jussieu, Paris Rive Gauche, Paris, France) for their work in Chapter 4. I also thank my collaborators from the UC Davis School of Veterinary Medicine who contributed to Chapter 6: Colin Reardon and Ingrid Brust-Mascher.

\emph{Life.}
I am also very lucky to have found my wife in graduate school. I am very happy, and I look forward to our life together. She is an amazing person, and my best friend.

\emph{Money.} I am thankful for the funding I received from my advisers and the mathematics department. I am especially thankful for the funding I received over my first year and summer, which allowed me to focus on passing the preliminary examinations. I also owe a big thank you to my friend Swati Patel who played a monumental role in editing my NSF Graduate Research Fellowship Program application, which resulted in me obtaining the award! The financial support I received over the years greatly reduced stress and made the experience great. Money makes all the difference. I also wish to thank some organizations for their financial support for conferences: American Institute of Mathematics, Institute for Mathematics and Its Applications, and the Rocky Mountain Mathematics Consortium at the University of Wyoming. I was partially supported by NSF award number 0914107, and a significant amount of this dissertation was supported by NSF grant number DGE-1148897.

\emph{Internships.} I am grateful for three great summer internships: two at SAS Institute (in 2013 and 2014 under Manoj Chari and Yan Xu, respectively), and one at Google, Inc. (in 2015 under Nicolas Mayoraz). All three showed me how diverse the software industry can be.

\emph{People.} All family members alive and dead. 
Angel Castro.
Lorenzo Medina.
Andy Tan.
Greg Webb.
Anne Carey.
Julia Mack.
Tom Brounstein.
Travis Scrimshaw.
Gordon Freeman.
Jim Raynor.
Sarah Kerrigan.
Dan Gusfield.
Sean Davis.
Mohamed Omar.
Yvonne Kemper.
Robert Hildebrand.
Mark Junod.

I would like to end with a quote that perfectly captures why I like mathematics, 

\begin{displayquote}
Despite some instances where physical application may not exist, mathematics has historically been the primary tool of the social, life, and physical sciences. It is remarkable that a study, so potentially pure, can be so applicable to everyday life. Albert Einstein questions, ``How can it be that mathematics, being after all a product of human thought which is independent of experience, is so admirably appropriate to the objects of reality?" This striking duality gives mathematics both power and charm. \cite[p. 171]{wapner2005pea}
\end{displayquote}

   \mainmatter
   
   \pagestyle{maintext}
   
   % Redefine plain page style so that the first pages of 
   % chapters have desired page style.
   %
   \fancypagestyle{plain}{%
      \renewcommand{\headrulewidth}{0pt}
      \fancyhf{}
      %\rhead{\thepage}
      \cfoot{\thepage}
   }%

   %%%%%%%%%%%%%%%%%%%%%%
\chapter{Introduction}
\label{ch:background} 

The first three chapters of this thesis are focused on the optimization of a polynomial function where the domain is a polytope. That is, we focus on the continuous optimization problem
\begin{equation*}
\begin{split}
\max & \; f(x) \\
 & x \in \PP,
\end{split}
\end{equation*}
where $P$ is a polytope and $f(x)$ is a polynomial. As we review in Section \ref{ch:bg:sh:poly-complexity}, exactly computing the maximum of a polynomial over a polytopal domain is hard, and even approximating the maximum is still hard. However, this has not damped research in this area, and many of the popular methods for approximating the optimum depend on decomposing the polynomial function, approximating the polynomial function with similar functions, or decomposing the domain. References are numerous in the literature \cite{anjos2011handbook, deKlerk2015, lasserre2009momentsBook, Lasserre2000929, Lasserre01globaloptimization, lasserre2002semidefinite, lasserre2011NewLook, marshall2008positive, parrilo2003SDP}. A common characteristic between these methods is their reliance on ideas in real semialgebraic geometry and semidefinite programming. A key contribution of this thesis is another algorithm for approximating the maximum of a polynomial function over $\PP$. Unlike previous methods, our method is based on combinatorial results. When convenient, to help develop our tools for the continuous optimization problem, we also state analogous results for the discrete optimization problem
\begin{equation*}
\begin{split}
\max & \; f(x) \\
 & x \in \PP \cap \Z^d.
\end{split}
\end{equation*}

One key step of our method for approximating the polynomial optimization problem requires computing the integral $\int_P (f(x))^k \d x$ where $k$ is some integer power. Chapter \ref{ch:Integration} is devoted to this step. Then Chapter \ref{ch:polynomialOptimization} connects the pieces together and culminates in an efficient algorithm for the continuous polynomial optimization problem. Some of the tools developed, namely the way we apply polyhedral decompositions and generating functions, can also be applied to a different type of problem: computing the Ehrhart polynomial of a knapsack polytope. Chapter \ref{ch:knapsack} addresses this idea. 

The remaining chapters cover the second part of this thesis: topics in data science. In particular, Chapter \ref{ch:dancingLinks} develops a useful heuristic for finding solutions to set partition constraints, which are a common constraint type in linear integer programming. Then Chapter \ref{ch:spleen} applies tools from distance geometry and cluster analysis to identify disease in spleens.

In this chapter, we review the background material used in all the other chapters. In the figure below, we suggest possible reading orders and identify which chapters builds upon topics in other chapters. 

\begin{center}
\begin{tikzpicture}[scale=0.80,transform shape,->,>=stealth',shorten >=1pt,auto,node distance=3cm,
	  thick,main node/.style={rectangle,draw,font=\sffamily\Large\bfseries}]
	  \node[main node] (1) {Chapter 1} ;
	  \node[main node] (2) [below left of =1] {Chapter 2};
	  \node[main node] (3) [below  of =2] {Chapter 3};
	  \node[main node] (4) [below  of =3] {Chapter 4};	  
	  \node[main node] (5) [right  of =2] {Chapter 5};	  
	  \node[main node] (6) [right  of =5] {Chapter 6};	  
	
	  \path[every node/.style={font=\sffamily\small}]
		 (1) edge node  {} (2)
		(2) edge node {} (3)
		(3) edge node {} (4)
		(1) edge node {} (5)
		(1) edge node {} (6);
\end{tikzpicture}
\end{center}

\section{Polyhedra and their representation}

Polytopes and polyhedra appear as a central object in this thesis. We state just the basic definitions and results that we need. For a complete review, see \cite{barvinokzurichbook, de2013algebraic, schrijver, zieglerpolybook}.

\begin{definition}
	Let $x_1, \dots x_k \in \R^d$, then the combination $a_1x_1 + \dots + a_kx_k$  with $a_i \in \R$ is called
	\begin{itemize}
		\item \emph{linear} with no restrictions on the $a_i$
		\item \emph{affine} if $\sum_{i=1}^k a_i = 1$
		\item \emph{conical} if $a_i \geq 0$
		\item \emph{convex} if it is affine and conical.
	\end{itemize}
\end{definition}

We can define a polytope as a special kind of convex set.

\begin{definition}
A set $C \in \R^d$ is  \emph{convex} if $\forall x, y \in C \Rightarrow tx + (1-t)y \in C, \:\forall \:0 \leq t \leq 1$. This means, the line segment between $x$ and $y$ is in $C$. 
\end{definition}

\begin{definition}
Let $C \in \R^d$, the \emph{convex hull of $C$} is \[\Conv(C) = \{ t_1x_1 + \dots + t_kx_k \; | \; x_i \in C \:\: and \:\: t_i \geq 0 \:\: and \:\:\sum_{i=1}^k t_i = 1 \}\] 
\end{definition}

\begin{definition}
Let $V \in \R^d$, be a finite point set, then a \emph{polytope $P$} is $P = \Conv(V)$. 
\end{definition}

Polytopes are the convex hull of finite point sets. But there are other ways to represent a polytope. Instead of looking at convex combinations, we can look at halfspaces: 

\begin{definition}
Let $a \in \R^d$, then $H = \{x \in R^d \; | \; a^Tx \leq b \}$ is a \emph{halfspace}. A halfspace is ``one side'' of a linear function.
\end{definition}

\begin{definition}
Let $P = \{x \in R^d \; | \; Ax \leq b \}$ be a finite intersection of halfspaces, then $P$ is called a \emph{polyhedron}.
\end{definition}

\begin{example}
Take the unit square with vertices $(0,0),$ $(1,0),$ $(0,1),$ and $(1,1)$ in $\R^2$. The interior of the square is given by all convex combinations of the vertices. It is also given by all $x, y \in \R$ such that

\begin{displaymath}
\begin{matrix}
	0 \leq x \leq 1 \\
	0 \leq y \leq 1 \\	
\end{matrix}	
\end{displaymath}

but this can be rewritten as 

\begin{displaymath}
\begin{matrix}
	x &\leq 1 \\
	-x &\leq 0 \\
	y &\leq 1 \\	
	-y &\leq 0 \\
\end{matrix}	
\end{displaymath}

or in matrix form $Ax \leq b$ as in

\begin{displaymath}
\left(\begin{matrix}
	1 & 0 \\
	-1 & 0 \\
	0 & 1 \\
	0 & -1 \\
\end{matrix}\right) 
\left( \begin{matrix}
	x \\
	y \\
\end{matrix}\right) \leq
\left(\begin{matrix}
	1 \\
	0 \\
	1 \\
	0 \\
\end{matrix}\right). 
\end{displaymath}
\end{example}

The unit square can be described by two different objects: as convex combinations of a point set, and the bounded intersection of finitely many half spaces. By the next theorem, these descriptions are equivalent as every polytope has these two representations.

\begin{theorem}[Finite basis theorem for polytopes, Minkowski-Steinitz-Weyl, see Corollary 7.1c in  \cite{schrijver}]
\label{thm:nonpointed-polyhedra}
	A set $P$ is a polytope if and only if it is a bounded polyhedron.
\end{theorem}

Because both polytope representations are very important, there are many ways or algorithms to convert one to the other. Instead of describing convex hull algorithms and others, we will consider them a technology and seek an appropriate software tool when needed. For more details about transferring from one representation to another, see \cite{4ti2, avis2000revised, cddlib-094a, polymake-software}.

\subsection{Polyhedral cones}

A special type of unbounded polyhedra that will appear often is a polyhedral cone. Generally, a \emph{cone} is a set that is closed under taking nonnegative scalar multiplication, and a \emph{convex cone} is also closed under addition. For example the set $D = \{ (x, 0): x \geq 0\} \cup \{ (0,y) : y \geq 0\}$ is closed under nonnegative scalar multiplication because if $z \in C$ then $az \in C$ for any $a \in \R_{\geq 0}$, but $D$ is not closed under addition. But if $K = \Conv(C)$, then $K$ is $\R^d_{\geq 0}$ and $K$ is a convex cone. We will always want cones to be convex, and we will use cone to mean convex cone.

\begin{definition}
A \emph{finitely generated cone $C$} has the form \[C = \left\{\sum_{j=1}^m a_jx_j \; | \; a_j \geq 0, j = 1, \dots, m \right\},\] for some finite collections of points $x_j \in \R^d$.
\end{definition}

\begin{definition}
 A \emph{polyhedral cone} is a cone of the form $\{x \in \R^d \; | \; Ax \leq 0\}$. Therefore, a polyhedral cone is a finite set of homogeneous linear inequalities.
\end{definition}

Just as bounded polyhedra and polytopes are the same object, a polyhedral cone and a finitely generated cone are two descriptions of the same object. 

\begin{theorem}[Farkas-Minkowski-Weyl, see Corollary 7.1a in  \cite{schrijver}]
\label{thm:nonpointed-polyhedra}
	A convex cone is polyhedral if and only if it is finitely generated.
\end{theorem}

\begin{definition}
Let $K$ be a convex set, then the \emph{polar cone $K$} is $K^o = \{ y \in \R^d \; | \; y^Tx \leq 0, \forall x \in K \}$. 
\end{definition}

The polar of a finitely generated cone is easy to compute. 

\begin{theorem}
\emph{(Polar of a finitely generated cone)}
\label{thm:polar-of-finitely-generated-cone}
Let $K = cone(\{c_1, \dots, c_m\})$, then the polar cone is the interception of a finite number of halfspaces: $K^o = \{ y \in \R^d : c_j^Ty \leq 0, \forall j=1, \dots, m \}$. Likewise, if $K$ is given by $Cy \leq 0$, then $K^o$ is generated by the rows of $C$.
\end{theorem}

\section{Generating functions for integration and summation}
\label{ch:bg:sec:integration-stuff}
   
Chapters \ref{ch:Integration},  \ref{ch:polynomialOptimization}, and \ref{ch:knapsack} make great use of encoding values in generating functions. This section gives a general introduction to how they are used in the later chapters. For a more complete description of the topics in this section, see \cite{barvinokzurichbook, BarviPom}.
   
\subsection{Working with generating functions: an example}
 
 Let us start with an easy example. Consider the one dimensional polyhedra in $\R$ given by $\PP = [0, n]$. We encode the lattice points of $\PP \cap \Z$ by placing each integer point as the power of a monomial, thereby obtaining the polynomial $S(\PP; z) := z^0 + z + z^2 + z^3 + \cdots + z^n$. The polynomial $S(\PP; z)$ is called the \emph{generating function of $\PP$.} Notice that counting $\PP \cap \Z$ is equivalent to evaluating $S(\PP, 1)$.
 
 In terms of the computational complexity, listing each monomial in the polynomial $S(\PP, z)$ results in a polynomial with exponential length in the bit length of $n$. However, we can rewrite the summation with one term:
 
 \[ S(\PP, z) = 1 + z^1 + \cdots + z^n = \frac{1-z^{n+1}}{1-z}.\]

Counting the number of points in $|\PP \cap \Z|$ is no longer as simple as evaluating $\frac{1-z^{n+1}}{1-z}$ at $z=1$ because this is a singularity. However, this singularity is removable. One could perform long-polynomial division, but this would result in a exponentially long polynomial in the bit length of $n$. Another option that yields a polynomial time algorithm would be to apply L'Hospital's rule:

\[\lim_{z \rightarrow 1} S(\PP, z) = \lim_{z \rightarrow 1} \frac{-(n+1)z^{n}}{1} = n+1. \]

Notice that $S(\PP, z)$ can be written in two ways:

\[ S(\PP, z) = \frac{1}{1-z} - \frac{z^{n+1}}{1-z} = \frac{1}{1-z} + \frac{z^n}{1-z^{-1}}.\]

The first two rational expressions have a nice description in terms of their series expansion:

\[ 1+z + \cdots + z^n = (1 + z^1 + \cdots ) -  (z^{n+1} + z^{n+2} + \cdots).\]

For the second two rational functions, we have to be careful about the domain of convergence when computing the series expansion. Notice that in the series expansion,

\begin{align*}
\frac{1}{1-z} =& \begin{cases} 
      1 + z^1 + z^2 \cdots & \;\;\;\;\; \text{ if } |z| < 1 \\
      -z^{-1} -z^{-2} - z^{-3} - \cdots & \;\;\;\;\; \text{ if }  |z| > 1 
   \end{cases} \\
\frac{z^n}{1-z^{-1}} =& \begin{cases} 
		-z^{n+1} - z^{n+2} - z^{n+3} -\cdots & \text{ if } |z| < 1 \\
      z^{n} + z^{n-1} + z^{n-2} + \cdots & \text{ if }  |z| > 1 
   \end{cases} \\
\end{align*}

adding the terms when $|z| <1$  or $|z| > 1$ results in the desired polynomial: $1 + z^1 + \cdots + z^n$. But we can also get the correct polynomial by adding the series that correspond to different domains of convergence. However, to do this we must now add the series $\cdots + z^{-2} + z^{-1} + 1 + z + z^2 + \cdots$ which corresponds to the polyhedra that is the entire real line:
\begin{align*}
1 + z + \cdots z^n &= (1 + z + z^2 + \cdots)\\
& \quad  + (z^n + z^{n-1} + \cdots)  \\
& \quad - (\cdots + z^{-2} + z^{-1} + 1 + z + z^2 + \cdots)
\end{align*}
and
\begin{align*}
1 + z + \cdots z^n &= (-z^{-1} -z^{-2} - z^{-3} - \cdots)\\
& \quad  + (-z^{n+1} - z^{n+2} - z^{n+3} -\cdots)  \\
& \quad + (\cdots + z^{-2} + z^{-1} + 1 + z + z^2 + \cdots)
\end{align*}

Hence by including the series $\cdots + z^{-2} + z^{-1} + 1 + z + z^2 + \cdots$, we can perform the series expansion of $\frac{1}{1-z} + \frac{z^n}{1-z^{-1}}$ by computing the series expansion of each term on potentially different domains of convergence. 

In the next sections, we will develop a rigorous justification for adding the series  $\cdots + z^{-2} + z^{-1} + 1 + z + z^2 + \cdots$.
 
\subsection{Indicator functions}

\begin{definition} 
The indicator function, $[A]: \R^d \rightarrow \R$, of a set $A \subseteq \R^d$ takes two values: $[A](x) =1$ if $x \in A$ and $[A](x)=0$ otherwise. 
\end{definition} 

The set of indicator functions on $\R^d$ spans a vector space with point-wise additions and scalar multiplication. The set also has an algebra structure where $[A]\cdot [B] = [A \cap B]$, and $[A]+[B] = [A \cup B] + [A \cap B]$. 

Recall the \emph{cone} of a set $A \subseteq \R^d$ is all conic combinations of the points from $A$:
\[\Cone(A) := \left\{ \sum_i \alpha_i a_i \mid a_i \in A, \alpha_i \in \R_{\geq 0} \right\}. \]

\begin{definition} 
Let $\PP$ be a polyhedron and $x \in \PP$. Then the \emph{tangent cone}, of $\PP$ at $x$ is the polyhedral cone 
\[\TCone(\PP, x) := x + \Cone(\PP - x)\]  
\end{definition} 

\begin{definition} 
Let $\PP$ be a polyhedron and $x \in \PP$. Then the \emph{cone of feasible directions}, of $\PP$ at $x$ is the polyhedral cone $\Cone(\PP - x)$.
\end{definition} 
 
Note that if $x$ is a vertex of $\PP$, and $\PP$ is given by an inequality description, then the tangent cone $\TCone(\PP, x)$ is the intersection of inequalities that are tight at $x$. Also, $\TCone(\PP, x)$ includes the affine hull of the smallest face that $x$ is in, so the tangent cone is pointed only if $x$ is a vertex. The difference between a tangent cone and a cone of feasible directions is that the latter is shifted to the origin. 

When $F$ is a face of $\PP$, we will also use the notation $\TCone(\PP, F)$ to denote $\TCone(\PP, x)$ where $x$ is any interior point of $F$.

\begin{theorem}[\cite{brianchon1837}, \cite{gram1871}]
Let $\PP$ be a polyhedron, then 
\[[\PP] = \sum_{F} (-1)^{\dim(F)}[\TCone(\PP, F)] \]
where the sum ranges over all faces $F$ of $\PP$ including $F = \PP$ but excluding $F = \emptyset$
\end{theorem} 

This theorem is saying that if the generating function of a polytope is desired, it is sufficient to just find the generating function for every face of $\PP$. The next
 corollary takes this a step further and says it is sufficient to just construct the generating functions associated at each vertex. This is because, as we will see, the generating functions for non-pointed polyhedra can be ignored.
 
\begin{cor}
\label{cor:tcone-mod-lines}
Let $\PP$ be a polyhedron, then 
\[ [\PP] \equiv \sum_{v \in V} [\TCone(\PP, v)] \pmod{\text{indicator functions of non-pointed polyhedra}},\]
where $V$ is the vertex set of $\PP$.
\end{cor}

\subsection{Generating functions of simple cones}

In this section, we quickly review the generating function for summation and integration when the polyhedron is a cone.

And there's still confusion
regarding multiplication: To make a vector space, you need addition of
two elements and multiplication of an element by a scalar (field
element). The multiplication of two indicator functions is NOT a
multiplication by a scalar. Instead, multiplication by a scalar is
really just scaling a function: Take indicator function of positive
real numbers: f(x) = 1 if x>=0; 0 if x < 0. Take a real number, say 7.
Then (7 . f)(x) = 7 if x >= 0; 0 if x < 0.
This makes the "algebra of polyhedra" a real vector space.
But the algebra of polyhedra is also an "algebra". For that you need
another multiplication, namely the multiplication of two elements; and
that is the multiplication that you describe (actually it's the
bilinear extension of what you describe -- because the multiplication
needs to be defined not only for two indicator functions, but for two
R-linear combinations of indicator functions). 
 
\begin{definition}
Let $V$ and $W$ be vector spaces. Let $\mathcal{P}$ be the real vector space spanned by the indicator functions of all polyhedra in $V$ where scalar multiplication is with a real number and an indicator function, and the addition operator is addition of indicator functions.  When $\mathcal{P}$ is equipped with the additional binary operation from $\mathcal{P} \times \mathcal{P}$ to $\mathcal{P}$ representing multiplication of indicator functions, then $\mathcal{P}$ is called the \emph{algebra of polyhedra}. A \emph{valuation} $T$ is a linear transformation  $T: \mathcal{P} \rightarrow W$.
\end{definition}
 
The next Proposition serves as a basis for all the summation algorithms we will discus. Its history can be traced to Lawrence in \cite{lawrence91-2}, and Khovanskii and Pukhlikov in \cite{pukhlikov1992riemann}. It is well described as Theorem 13.8 in \cite{barvinokzurichbook}.

\begin{proposition}\label{valuationI}
There exists a unique valuation  $S(\;\cdot\;,\ell)$ which  associates  to every rational polyhedron
$\PP\subset \R^d$ a meromorphic function in $\ell$ so that the following properties hold 

\begin{itemize}
\item If $\ell \in \R^d$ such that $e^{\langle \ell, x\rangle}$ is summable over the lattice points of $\PP$, then
$$
S(\PP,\ell)= \sum_{\PP \cap \Z^d} e^{\langle \ell,x\rangle}.
$$

\item For every point $s \in \Z^d$, one has
$$
S(s+\PP,\ell) = e^{\langle \ell,s\rangle}S(\PP,\ell).
$$
\item If $\PP$ contains a straight line, then $S(\PP,\ell)=0$.
\end{itemize}
\end{proposition}

A consequence of the valuation property is the following fundamental theorem. 
It follows from the Brion--Lasserre--Lawrence--Varchenko decomposition theory of a
polyhedron into the supporting cones at its vertices \cite{barvinokzurichbook, beck-haase-sottile:theorema, Brion88, lasserre-volume1983}.

\begin{lemma} \label{brion-exp} Let $\PP$ be a polyhedron with set of vertices $V(\PP)$. For each
vertex~$s$, let $C_s(\PP)$ be the cone of feasible directions at vertex $s$. Then
\begin{equation*}
S(\PP,\ell)=\sum_{s\in V(\PP)}S(s+C_s(\PP),\ell).
\end{equation*}
\end{lemma}

This last lemma can be identified as the natural result of combining Corollary \ref{cor:tcone-mod-lines} and Proposition \ref{valuationI} part (3). A non-pointed polyhedron is another characterization of  a polyhedron that contains a line.

Note that the cone $C_s(\PP)$ in Lemma~\ref{brion-exp} may not be simplicial, but for simplicial cones there are explicit rational function formulas. As we will see in Proposition~\ref{prop:integral-exp-simplicial}, one can derive an explicit formula for 
the rational function $S(s+C_s(\PP),\ell)$ in terms of the geometry of the cones.

\begin{proposition} 
  \label{prop:summation-exp-simplicial}
  For a simplicial full-dimensional pointed cone $C$ generated by rays $u_1,u_2,\dots u_d$ (with vertex $0$) where $u_i \in \Z^d$ and for any point $s$
\begin{equation*}
S(s+C,\ell)
=\sum_{a \in (s+\Pi_C) \cap \Z^d} e^{\langle \ell, a
  \rangle} \prod_{i=1}^d \frac1{1- e^{\langle \ell, u_i \rangle}}
\end{equation*}

where $\Pi_c := \{ \sum_{i=1}^d \alpha_i u_i \mid 0 \leq \alpha_i < 1\}$
This identity holds as a meromorphic function of~$\ell$ 
and pointwise for every $\ell$ such that $\langle \ell, u_i \rangle \neq 0$ for
all $u_i$.
\end{proposition}

The set $\Pi_C$ is often called the \emph{half-open fundamental parallelepiped} of $C$. It is also common to force each ray $u_i$ to be \emph{primitive}, meaning that the greatest common divisor of the elements in $u_i$ is one, and this can be accomplished by scaling each ray. 

The continuous generating function for $\PP$ almost mirrors the discrete case. It can again be attributed to Lawrence, Khovanskii, and Pukhlikov, and appears as  Theorem 8.4 in \cite{barvinokzurichbook}.

\begin{proposition}\label{valuationII}
There exists a unique valuation  $I(\;\cdot\;, \ell)$ which  associates  to every polyhedron
$\PP\subset \R^d$ a meromorphic function so that the following properties hold 

\begin{enumerate}
\item If $\ell$ is a linear form such that $e^{\ll \ell, x \rr}$ is integrable over $\PP$ with the standard Lebesgue measure on $\R^d$ , then 
\[ I(\PP,\ell) = \int_\PP e^{\ll \ell, x \rr} \d x\]
\item For every point $s \in \R^d$, one has
\[I(s+\PP, \ell) = e^{\ll \ell, s \rr}I(\PP,\ell).\]
\item If $\PP$ contains a line, then $I(\PP, \ell) = 0$.
\end{enumerate}
\end{proposition}

\begin{lemma}
\label{ch:bg:lemma-integration-brion}  Let $\PP$ be a polyhedron with set of vertices $V(\PP)$. For each
vertex~$s$, let $C_s(\PP)$ be the cone of feasible directions at vertex $s$. Then
\begin{equation*}
I(\PP,\ell)=\sum_{s\in V(\PP)}I(s+C_s(\PP),\ell).
\end{equation*}
\end{lemma}

Again, this last lemma can be identified as the natural result of combining Corollary \ref{cor:tcone-mod-lines} and Proposition \ref{valuationII} part (3).

\begin{proposition} 
  \label{prop:integral-exp-simplicial}
  For a simplicial full-dimensional pointed cone $C$ generated by rays $u_1,u_2,\dots u_d$ (with vertex $0$) and for any point $s$
\begin{equation*}
I(s+C,\ell) = \text{vol} (\Pi_C)e^{\ll \ell, s \rr} \prod_{i=1}^d \frac{1}{-\ll \ell, u_i \rr}.
\end{equation*}
These identities holds as a meromorphic function of~$\ell$ 
and pointwise for every $\ell$ such that $\langle \ell, u_i \rangle \neq 0$ for
all $u_i$.
\end{proposition}

\subsubsection{Integration example}

Let $a < b < c < d$, then it is a well known fact from calculus that 

\[
\int_a^c e^{\ell_1 x_1} \d x_1 = \int_a^b e^{\ell_1 x_1} \d x_1 + \int_b^c e^{\ell_1 x_1} \d x_1 = \int_a^d e^{\ell_1 x_1} \d x_1 - \int_c^d e^{\ell_1 x_1} \d x_1.
\]

%\begin{align*}
%\int_a^c e^{\ell_1 x_1} \d x_1 &= \int_a^b e^{\ell_1 x_1} \d x_1 + \int_b^c e^{\ell_1 x_1} \d x_1 \\
%&= \int_a^d e^{\ell_1 x_1} \d x_1 - \int_c^d e^{\ell_1 x_1} \d x_1.
%\end{align*}

However, the domain $[a, c]$ cannot be decomposed in every way. For example \[\int_a^c e^{\ell_1 x_1} \d x_1 \neq \int_a^\infty e^{\ell_1 x_1} \d x_1 + \int_{-\infty}^c e^{\ell_1 x_1} \d x_1 - \int_{-\infty}^\infty e^{\ell_1 x_1} \d x_1.\] Notice that not only is the expression on the right hand side of the equation not equal to the left hand side, but there is no value for $\ell_1$ that makes the three integrals finite. However, results in this section allow us to assign numbers (or meromorphic functions) to the integrals that do not converge! 

We now consider an example in dimension two. Consider the triangle below with coordinates at $(1,1)$, $(1, 3)$, and $(4,1)$. This domain can we written as the sum of 7 polyhedrons: addition of three tangent cones, subtraction of three halfspaces, and the addition of one copy of $\R^2$. 

\begin{tabular}{ccccc}
\begin{tikzpicture}[scale=.4]
	\definecolor{greeo}{RGB}{91,170,210}

	\coordinate (A1) at (1,1);
	\coordinate (A2) at (1,3);
	\coordinate (A3) at (4,1);
    
    \draw[fill=greeo,opacity=0.6,line width=0.05cm] (A1) -- (A2) -- (A3) -- (A1);
    
    % Axes:
    \draw [->] (0,0) -- (4.5,0) node [above left]  {};
    \draw [->] (0,0) -- (0,3.5) node [below right] {};    
    %\foreach \n in {1,2,...,3}{%
    %    \draw (-3pt,\n) -- (3pt,\n)   node [right] {$\n $};
    %}    
    %\foreach \n in {1,2,...,4}{%
    %    \draw (\n,-3pt) -- (\n,3pt)   node [above] {$\n$};  
    %}     
\end{tikzpicture}
&
\raisebox{1.5em}{$=$}
&
\begin{tikzpicture}[scale=.4]
	\definecolor{greeo}{RGB}{91,170,210}

	\coordinate (A1) at (1,1);
	\coordinate (A2) at (1,3);
	\coordinate (A3) at (4,1);
	
	\coordinate (center) at (4,1);
	\coordinate (r1) at ($(center) +(146.5:4)$);
	\coordinate (r2) at ($(center) +(180:4)$);

    % Axes:
    \draw [->] (0,0) -- (4.5,0) node [above left]  {};
    \draw [->] (0,0) -- (0,3.5) node [below right] {};

%crazy math to draw the circle section
\filldraw[draw=black, fill=greeo, fill opacity=0.3,dotted]
   let \p1 = ($(r1) - (center)$),
       \p2 = ($(r2) - (center)$),
       \n0 = {veclen(\x1,\y1)},            % Radius
       \n1 = {atan(\y1/\x1)+180*(\x1<0)},  % initial angle
       \n2 = {atan(\y2/\x2)+180*(\x2<0)}   % Final angle
    in
    (center) --  (r1) arc(\n1:\n2:\n0) -- cycle;
%Draw two bold rays    
\draw[fill=greeo,opacity=0.6,line width=0.05cm,->] (A3) -- (r1);    
\draw[fill=greeo,opacity=0.6,line width=0.05cm,->] (A3) -- (r2);
\end{tikzpicture}
 &
\begin{tikzpicture}[scale=.4]
	\definecolor{greeo}{RGB}{91,170,210}

	\coordinate (A1) at (1,1);
	\coordinate (A2) at (1,3);
	\coordinate (A3) at (4,1);
	
	\coordinate (center) at (1,1);
	\coordinate (r1) at ($(center) +(0:3)$);
	\coordinate (r2) at ($(center) +(90:3)$);	
    
    % Axes:
    \draw [->] (0,0) -- (4.5,0) node [above left]  {};
    \draw [->] (0,0) -- (0,3.5) node [below right] {};    

    \filldraw[draw=black, fill=greeo, fill opacity=0.3,dotted]
    (center) --  (r1) arc(0:90:3) -- cycle;
%Draw two bold rays    
\draw[fill=greeo,opacity=0.6,line width=0.05cm,->] (A1) -- (r1);    
\draw[fill=greeo,opacity=0.6,line width=0.05cm,->] (A1) -- (r2);
\end{tikzpicture}
&
\begin{tikzpicture}[scale=.4]
	\definecolor{greeo}{RGB}{91,170,210}

	\coordinate (A1) at (1,1);
	\coordinate (A2) at (1,3);
	\coordinate (A3) at (4,1);
	
	\coordinate (center) at (1,3);
	\coordinate (r1) at ($(center) +(270:4)$);
	\coordinate (r2) at ($(center) +(300:4)$);

    % Axes:
    \draw [->] (0,0) -- (4.5,0) node [above left]  {};
    \draw [->] (0,0) -- (0,3.5) node [below right] {};    

    \filldraw[draw=black, fill=greeo, fill opacity=0.3,dotted]
    (center) --  (r1) arc(270:300:4) -- cycle;
%Draw two bold rays    
\draw[fill=greeo,opacity=0.6,line width=0.05cm,->] (A2) -- (r1);    
\draw[fill=greeo,opacity=0.6,line width=0.05cm,->] (A2) -- (r2);
\end{tikzpicture}
%-----------------------------------------------------------
\\
&&
\begin{tikzpicture}[scale=.4]
	\definecolor{greeo}{RGB}{255, 152, 33}

    % Axes:
    \draw [->] (0,0) -- (4.5,0) node [above left]  {};
    \draw [->] (0,0) -- (0,3.5) node [below right] {};    

    \filldraw[draw=black, fill=greeo, fill opacity=0.3,dotted]
    (-1,1) --  (4,1) -- (4,3) -- (-1,3)-- cycle;
%Draw two bold rays    
\draw[fill=greeo,opacity=0.6,line width=0.05cm,<->] (-1,1) -- (4,1);    
\end{tikzpicture}
 &
\begin{tikzpicture}[scale=.4]
	\definecolor{greeo}{RGB}{255, 152, 33}

    % Axes:
    \draw [->] (0,0) -- (4.5,0) node [above left]  {};
    \draw [->] (0,0) -- (0,3.5) node [below right] {};    

    \filldraw[draw=black, fill=greeo, fill opacity=0.3,dotted]
    (1,3) --  (1,-1) -- (4,-1) -- (4,3)-- cycle;
%Draw two bold rays    
\draw[fill=greeo,opacity=0.6,line width=0.05cm,<->] (1,3) -- (1,-1);    
\end{tikzpicture}
&
\begin{tikzpicture}[scale=.4]
	\definecolor{greeo}{RGB}{255, 152, 33}

    % Axes:
    \draw [->] (0,0) -- (4.5,0) node [above left]  {};
    \draw [->] (0,0) -- (0,3.5) node [below right] {};    

    \filldraw[draw=black, fill=greeo, fill opacity=0.3,dotted]
    (-.5,4) --  (-1,4) -- (-1,-1) -- (4,-1) -- (4,1) -- cycle;
%Draw two bold rays    
\draw[fill=greeo,opacity=0.6,line width=0.05cm,<->] (-.5,4) -- (4,1);    
\end{tikzpicture}
%-----------------------------------------------------------
\\
&&
\begin{tikzpicture}[scale=.4]
	\definecolor{greeo}{RGB}{91,170,210}

	\coordinate (A1) at (1,1);
	\coordinate (A2) at (1,3);
	\coordinate (A3) at (4,1);
	
	\coordinate (center) at (1,3);
	\coordinate (r1) at ($(center) +(270:4)$);
	\coordinate (r2) at ($(center) +(300:4)$);

    % Axes:
    \draw [->] (0,0) -- (4.5,0) node [above left]  {};
    \draw [->] (0,0) -- (0,3.5) node [below right] {};    

	\fill[draw=black, fill=greeo, fill opacity=0.3,dotted] (2,2) circle(2);
\end{tikzpicture}
 &

&
\end{tabular}

For example, the point $(1,1)$ is part of the triangle, so it is counted once. In the decomposition, the point $(1,1)$ is counted positively four times (once in each tangent cone and once in $\R^2$), and is counted negatively three times (once in each halfspace), resulting in being counted exactly once. A similar calculation shows that $(0,0)$ is counted negatively in one of the halfspaces, and positively in $\R^2$, resulting in a total count of zero, meaning $(0,0)$ is not part of the triangle. 

The integral of $e^{\ell_1 x + \ell_2 y}$ over the triangle clearly exist because the function is continuous and the domain is compact. As the triangle can be written as the sum of 7 other polyhedrons, we want to integrate $e^{\ell_1 x + \ell_2 y}$ over each of the 7 polyhedrons. However, the integral of $e^{\ell_1 x + \ell_2 y}$ over some of them does not converge! Instead, we map each domain to a meromorphic function using Propositions \ref{valuationII} and \ref{prop:integral-exp-simplicial}. Because $I(\cdot, \ell)$ is a valuation, we can apply $I(\cdot, \ell)$ to each domain. The fact that $I(P, \ell)=0$ if $P$ contains a line, simplifies the calculation to just the three tangent cones: $s_1 +C_1$, $s_2 +C_2$, and $s_3 +C_3$.

\begin{tabular}{ccccc}
\begin{tikzpicture}[scale=.4]
	\definecolor{greeo}{RGB}{91,170,210}

	\coordinate (A1) at (1,1);
	\coordinate (A2) at (1,3);
	\coordinate (A3) at (4,1);
    
    \draw[fill=greeo,opacity=0.6,line width=0.05cm] (A1) -- (A2) -- (A3) -- (A1);
    
    % Axes:
    \draw [->] (0,0) -- (4.5,0) node [above left]  {};
    \draw [->] (0,0) -- (0,3.5) node [below right] {};    
    %\foreach \n in {1,2,...,3}{%
    %    \draw (-3pt,\n) -- (3pt,\n)   node [right] {$\n $};
    %}    
    %\foreach \n in {1,2,...,4}{%
    %    \draw (\n,-3pt) -- (\n,3pt)   node [above] {$\n$};  
    %}     
\end{tikzpicture}
&
\raisebox{1.5em}{$=$}
&
\begin{tikzpicture}[scale=.4]
	\definecolor{greeo}{RGB}{91,170,210}

	\coordinate (A1) at (1,1);
	\coordinate (A2) at (1,3);
	\coordinate (A3) at (4,1);
	
	\coordinate (center) at (4,1);
	\coordinate (r1) at ($(center) +(146.5:4)$);
	\coordinate (r2) at ($(center) +(180:4)$);

    % Axes:
    \draw [->] (0,0) -- (4.5,0) node [above left]  {};
    \draw [->] (0,0) -- (0,3.5) node [below right] {};

%crazy math to draw the circle section
\filldraw[draw=black, fill=greeo, fill opacity=0.3,dotted]
   let \p1 = ($(r1) - (center)$),
       \p2 = ($(r2) - (center)$),
       \n0 = {veclen(\x1,\y1)},            % Radius
       \n1 = {atan(\y1/\x1)+180*(\x1<0)},  % initial angle
       \n2 = {atan(\y2/\x2)+180*(\x2<0)}   % Final angle
    in
    (center) --  (r1) arc(\n1:\n2:\n0) -- cycle;
%Draw two bold rays    
\draw[fill=greeo,opacity=0.6,line width=0.05cm,->] (A3) -- (r1);    
\draw[fill=greeo,opacity=0.6,line width=0.05cm,->] (A3) -- (r2);
\end{tikzpicture}
 &
\begin{tikzpicture}[scale=.4]
	\definecolor{greeo}{RGB}{91,170,210}

	\coordinate (A1) at (1,1);
	\coordinate (A2) at (1,3);
	\coordinate (A3) at (4,1);
	
	\coordinate (center) at (1,1);
	\coordinate (r1) at ($(center) +(0:3)$);
	\coordinate (r2) at ($(center) +(90:3)$);	
    
    % Axes:
    \draw [->] (0,0) -- (4.5,0) node [above left]  {};
    \draw [->] (0,0) -- (0,3.5) node [below right] {};    

    \filldraw[draw=black, fill=greeo, fill opacity=0.3,dotted]
    (center) --  (r1) arc(0:90:3) -- cycle;
%Draw two bold rays    
\draw[fill=greeo,opacity=0.6,line width=0.05cm,->] (A1) -- (r1);    
\draw[fill=greeo,opacity=0.6,line width=0.05cm,->] (A1) -- (r2);
\end{tikzpicture}
&
\begin{tikzpicture}[scale=.4]
	\definecolor{greeo}{RGB}{91,170,210}

	\coordinate (A1) at (1,1);
	\coordinate (A2) at (1,3);
	\coordinate (A3) at (4,1);
	
	\coordinate (center) at (1,3);
	\coordinate (r1) at ($(center) +(270:4)$);
	\coordinate (r2) at ($(center) +(300:4)$);

    % Axes:
    \draw [->] (0,0) -- (4.5,0) node [above left]  {};
    \draw [->] (0,0) -- (0,3.5) node [below right] {};    

    \filldraw[draw=black, fill=greeo, fill opacity=0.3,dotted]
    (center) --  (r1) arc(270:300:4) -- cycle;
%Draw two bold rays    
\draw[fill=greeo,opacity=0.6,line width=0.05cm,->] (A2) -- (r1);    
\draw[fill=greeo,opacity=0.6,line width=0.05cm,->] (A2) -- (r2);
\end{tikzpicture}
%-----------------------------------------------------------
\\
$I(\Delta, \ell)$& $=$ & $I(s_1 +C_1, \ell)$& $+ I(s_2 +C_2, \ell)$ & $+ I(s_3 +C_3, \ell)$
\end{tabular}

\begin{align*}
\int_1^4 \int_1^{\frac{-2}{3}x + \frac{11}{3}} e^{\ell_1 x + \ell_2 y}\d y \d x = & 2e^{4\ell_1 + \ell_2} \frac{1}{(-3\ell_1 + 2\ell_2)} \frac{1}{(-\ell_1 )} \\
& + 1e^{\ell_1 + \ell_2} \frac{1}{\ell_1} \frac{1}{\ell_2} \\
& + 3e^{1\ell_1 + 3\ell_2} \frac{1}{(-\ell_2)} \frac{1}{(3\ell_1 -2\ell_2)}
\end{align*}

Propositions \ref{valuationII} and \ref{prop:integral-exp-simplicial} say that the meromorphic function associated with the triangle is equal to the sum of the three meromorphic functions associated at each tangent cone. Moreover, because the integral over the triangle exist, the meromorphic function associated with the triangle gives the integral. For example, evaluating the meromorphic function at $\ell_1 = 2, \ell_2 = 1$ results in $1924.503881$, which is the integral of $e^{2x + y}$ over the triangle. 

There is one problem. The integral of $e^{\ell_1 x + \ell_2 y}$ over the triangle is a holomorphic function, and we have written it as the sum of three meromorphic functions, so this means the poles of the meromorphic functions must cancel in the sum. Consider evaluating at $\ell_1 = 2, \ell_2 = 3$. This would produce division by zero, and so $\ell_1 = 2, \ell_2 = 3$ is among the poles. A common approach is to instead evaluate $\ell$ at $\ell_1 = 2 + \epsilon$, $\ell_2 = 3 + \epsilon$ and take the limit as $\epsilon \rightarrow 0$. Hence 
\[\int_1^4 \int_1^{\frac{-2}{3}x + \frac{11}{3}} e^{2 x + 3 y}\d y \d x = \lim_{\epsilon\rightarrow 0} \sum_{i=1}^3 I(s_i +C_i, (2, 3) + \epsilon (1, 1)).\] 
Notice that for each $i$, $I(s_i +C_i, (2, 3) + \epsilon (1, 1))$ is a meromorphic in $\epsilon$, but $\sum_{i=1}^3 I(s_i +C_i, (2, 3) + \epsilon (1, 1))$ is a holomorphic function (as it is the integral of $e^{(2 + \epsilon)x + (3 + \epsilon)y}$ over the triangle). This means that in the Laurent series expansion of $I(s_i +C_i, (2, 3) + \epsilon (1, 1))$, any terms where $\epsilon$ has a negative exponent will cancel out in the sum. Thus the limit can be computed by finding the Laurent series expansion at $\epsilon=0$ for each $I(s_i +C_i, (2, 3) + \epsilon (1, 1))$ and summing the coefficient of $\epsilon^0$ in each Laurent series. Chapter \ref{ch:Integration} will show that computing the Laurent series is easy in this case. 

This is a common technique, and we will see it used many times in this manuscript.

\subsection{Generating function for non-simple cones}

Lemma \ref{brion-exp} and Proposition \ref{prop:summation-exp-simplicial} (or Lemma \ref{ch:bg:lemma-integration-brion} and Proposition \ref{prop:integral-exp-simplicial}) can be used for computing the summation (or integral) over a polytope only if the polytope is a \emph{simple} polytope. Meaning, for a $d$-dimensional polytope, every vertex of the polytope is adjacent to exactly $d$ edges.

In this section, we review the generating function of $\sum_{x \in \PP \cap \Z^d } e^{\ll \ell, x \rr }$ and $\int_\PP e^{\ll \ell, x \rr} \d x$ for a general polytope $\PP$. When $\PP$ is not simple, the solution is to triangulate it the tangent cones.

\begin{definition}
A triangulation of a cone $C$ is the set $\Gamma$ of simplicial cones $C_i$ of the same dimension as the affine hull of $C$ such that
\begin{enumerate}
\item the union of all the simplicial cones in $\Gamma$ is $C$,
\item the intersection of any pair of simplicial cones in $\Gamma$ is a common face of both simplicial cones,
\item and every ray of every simplicial cone is also a ray of $C$.
\end{enumerate}
\end{definition}
  
There are many references and software tools for computing a triangulation of a polytope or polyhedral cone, see \cite{4ti2, avis2000revised, DRStriangbook, cddlib-094a,  polymake-software, leeRegularTriangulations, rambau2002topcom}.  
  
 Let $C$ be a full-dimensional pointed polyhedral cone, and $\Gamma_1 = \{ C_i \mid i \in I_1\}$ be a triangulation into simplicial cones $C_i$ where $I_1$ is a finite index set. It is true that $C =  \bigcup_{i \in I_1} C_i$, but $[C] = \sum_{i \in I_1} [C_i]$ is false as points on the boundary of two adjacent simplicial cones are counted multiple times. The correct approach is to use the inclusion-exclusion formula:
 
 \[ [C] = \sum_{\emptyset \neq J \subseteq I_1} (-1)^{|J|-1} [\cap_{j \in J} C_j]\]

Also, note that this still holds true when $C$ (and the $C_i$) is shifted by a point $s$. When $|J| \geq 2$, $I(\cap_{j \in J} C_j, \ell) = 0$ as $\bigcap_{j \in J} C_j$ is not full-dimensional, and the integration is done with respect to the Lebesgue measure on $\R^d$. This leads us to the next lemma.

\begin{lemma}
\label{lemma:integration-triangulation}
For any triangulation $\Gamma_s$ of the feasible cone at each of the vertices $s$ of the polytope $\PP$ we have $I(\PP,\ell) = \sum_{s \in V(\PP)} \sum_{C \in \Gamma_s} I(s+C,\ell)$
\end{lemma}

Lemma \ref{lemma:integration-triangulation} states that we can triangulate a polytope's feasible cones and apply the integration formulas on each simplicial cone without worrying about shared boundaries among the cones. Note that there is no restriction on how the triangulation is performed. 

More care is needed for the discrete case as $S(\cap_{j \in J} C_j, \ell) \neq 0$ when $|J| \geq 2$. We want to avoid using the inclusion-exclusion formula as it contains exponentially many terms (in size of $|I_1|$). 

The discrete case has another complication. Looking at Proposition \ref{prop:summation-exp-simplicial}, we see that the sum 

\[ \sum_{a \in (s+\Pi_C) \cap \Z^d} e^{\ll \ell, a \rr}\]

has to be enumerated. However, there could be an exponential number of points in $(s+\Pi_C) \cap \Z^d$ in terms of the bit length of the simplicial cone $C$. 

We will illustrate one method for solving these problems called the \emph{Dual Barvinok Algorithm}.

\subsubsection{Triangulation}

Recall that the \emph{polar} of a set $A \subset \R^d$ is the set $A^\circ = \{ x \in \R^d \mid  \ll x, a \rr \leq 1 \text{ for every } a \in A \}$. Cones enjoy many properties under the polar operation. If $C$ is a finitely generated cone in $\R^d$, then

\begin{enumerate}
\item $C^\circ = \{ x \in \R^d \mid \ll x, c \rr \leq 0, \forall c \in C \} \}$, 
\item $C^\circ$ is also a cone,
\item $(C^\circ)^\circ = C$, and
\item if $C = \{ x \mid A^Tx \leq 0 \}$, then $C^\circ$ is generated by the columns of $A$.
\end{enumerate}
The next lemma is core to  Brion's ``polarization trick'' \cite{Brion88} for dealing with the inclusion-exclusion terms.

\begin{lemma}[Theorem 5.3 in \cite{barvinokzurichbook}]
Let $\mathcal{C}$ be the vector space spanned by the indicator functions of all closed convex sets in $\R^d$. Then there is a unique linear transformation  $\mathcal{D}$ from $\mathcal{C}$ to itself such that $\mathcal{D}([A]) = [A^\circ]$ for all non-empty closed convex sets $A$.
\end{lemma}

Instead of taking the non-simplicial cone $C$ and triangulating it, we first compute $C^\circ$ and triangulate it to $\Gamma' = \{C_i^\circ \mid i \in I_2 \}$. Then
 \[ [C^\circ] = \sum_{i \in I_2} [C_i^\circ] + \sum_{\emptyset \neq J \subseteq I_2, |J| > 1} (-1)^{|J|-1} [\cap_{j \in J} C_j^\circ].\]
Applying the fact that $(C^\circ)^\circ = C$ we get
  \[ [C] = \sum_{i \in I_2} [C_i] + \sum_{\emptyset \neq J \subseteq I_2, |J| > 1} (-1)^{|J|-1} [(\cap_{j \in J} C_j^\circ)^\circ].\]
Notice that the polar of a full-dimensional pointed cone is another full-dimensional pointed cone. For each $J$ with $|J| \geq 2$, $\cap_{j \in J} C_j^\circ$ is not a full-dimensional cone. The polar of a cone that is not full dimensional is a cone that contains a line. Hence $S((\cap_{j \in J} C_j^\circ)^\circ, \ell) = 0$. By polarizing a cone, triangulating in the dual space, and polarizing back, the boundary terms from the triangulation can be ignored. 

\subsubsection{Unimodular cones}
Next, we address the issue that for a simplicial cone $C$, the set $\Pi_C \cap \Z^d$ contains too many terms for an enumeration to be efficient. The approach then is to decompose $C$ into cones that only have one lattice point in the fundamental parallelepiped. Such cones are called \emph{unimodular cones}. Barvinok in \cite{bar} first developed such a decomposition and showed that it can be done in polynomial time when the dimension is fixed. We next give an outline of Barvinok's decomposition algorithm. 

Given a pointed simplicial full-dimensional cone $C$, Barvinok's decomposition method will produce new simplicial cones $C_i$ such that $|\Pi_{C_i} \cap \Z^d| \leq |\Pi_{C} \cap \Z^d|$ and values $\epsilon_i \in \{1, -1\}$ such that

\[ [C] \equiv \sum_{i \in I} \epsilon_i [C_i] \pmod{\text{indicator functions of lower dimensional cones}}.\]

Let $C$ be generated by the rays $u_1, \dots, u_d$. The algorithm first constructs a vector $w$ such that
 \[w = \alpha_1 u_1 + \cdots + \alpha_d u_d \text{ and } |\alpha_i| \leq |\det(U)|^{-1/d} \leq 1,\]
 
 where the columns of $U$ are the $u_i$. This is done with integer programming or using a lattice basis reduction method \cite{latte1}. Let $K_i = \Cone(u_1, \dots, u_{i-1}, w, u_{i+1}, \dots, u_d)$, then it can be shown that $|\Pi_{K_i} \cap \Z^d| \leq |\Pi_{C} \cap \Z^d|^{\frac{d-1}{d}}$, meaning that these new cones have less integer points in the fundamental parallelepiped than the old cone. This process can be recursively repeated until unimodular cones are obtained. 
 
\begin{theorem}[Barvinok \cite{bar}]
Let $C$ be a simplicial full-dimensional cone generated by rays $u_1, \dots, u_d$. Collect the rays into the columns of a matrix $U \in \Z^{d \times d}$. Then the depth of the recursive decomposition tree is at most 
\[ \left\lfloor 1 + \frac{\log_2 \log_2 |\det(U)|}{\log_2 \frac{n}{n-1}} \right\rfloor. \]
\end{theorem}

Because at each node in the recursion tree has at most $n$ children, and and the depth of the tree is doubly logarithmic in $\det(U)$, only polynomial many unimodular cones are constructed. 

In \cite{bar}, the inclusion-exclusion formula was applied to boundaries between the unimodular cones in the primal space. However, like in triangulation, the decomposition can be applied in the dual space where the lower dimensional cones can be ignored. For the full details of Barvinok's decomposition algorithm, see \cite{latte1}, especially Algorithm 5 therein. This variant of Barvinok’s algorithm has
efficient implementations in \texttt{LattE}  \cite{latte-1.2} and the library \texttt{barvinok} \cite{barvinok-noversion}.

\subsection{Generating functions for full-dimensional polytopes}

In this section, we explicitly combine the results from the last two sections and write down the polynomial time algorithms for computing the discrete and continuous generating function for a polytope $\PP$.

\begin{algorithm}
\caption{Barvinok's Dual Algorithm}
\label{alg:barvinok-dual}
\begin{justify}
Output: the rational generating function for $S(\PP, \ell) = \sum\limits_{x \in \PP \cap \Z^d} e^{\langle \ell, x \rangle}$ in the form
\[S(\PP,\ell) = \sum_{i \in I} \epsilon_i \frac{e^{\langle \ell, v_i \rangle}}{\prod_{j=1}^d ( 1 - e^{\langle \ell, u_{ij}\rangle})}\] 
where $\epsilon_i \in \{-1,1\}, v_i \in \Z^d$, $u_{ij} \in \Z^d$, and $|I|$ is polynomially bounded in the input size of $\PP$ when $d$ is fixed. Each $i \in I$ corresponds to a simplicial unimodular cone $v_i + C_i$ where $u_{i1}, \dots, u_{id}$ are the $d$ rays of the cone $C_i$.
\end{justify}
\begin{enumerate}
\item  Compute all vertices $v_i$ and corresponding supporting cones $C_i$ of $\PP$
\item Polarize the supporting cones $C_i$ to obtain $C_i^\circ$
\item Triangulate $C_i^\circ$ into simplicial cones $C_{ij}^\circ$, discarding lower-dimensional cones
\item Apply Barvinok’s signed decomposition (see \cite{latte1}) to the cones $C_{ij}^\circ$ to obtain cones $C_{ijk}^\circ$, which results in the identity

\[ [C^\circ_{ij}] \equiv \sum_{k} \epsilon_{ijk} [C^\circ_{ijk}] \pmod{\text{indicator functions of lower dimensional cones}}.\]

Stop the recursive decomposition when unimodular cones are obtained. Discard all lower-dimensional cones
\item Polarize back $C_{ijk}^\circ$ to obtain cones $C_{ijk}$
\item $v_i$ is the unique integer point in the fundamental parallelepiped of every resulting cone $v_i+C_{ijk}$
\item Write down the above formula
\end{enumerate}
\end{algorithm}

The key part of this variant of Barvinok's algorithm is that computations with rational generating are simplified when non-pointed cones are used. The reason is that the rational generating
function of every non-pointed cone is zero. By operating in the dual space, when computing the polar cones, lower-dimensional cones can be safely discarded because this is equivalent to discarding non-pointed cones in the primal space. 

Triangulating a non-simplicial cone in the dual space was done to avoid the many lower-dimensional terms that arise from using the inclusion-exclusion formula for the indicator function of the cone. Other ways to get around this exist. In \cite{beck-haase-sottile:theorema, beck-sottile:irrational, koeppe:irrational-barvinok}, irrational decompositions were developed which are decompositions of polyhedra whose proper faces do not contain any lattice points. Counting formulas for lattice points that are constructed with irrational decompositions do not need the inclusion-exclusion principle. The implementation of this idea \cite{latte-macchiato} was the first practically efficient variant of Barvinok’s algorithm that works in the primal space.

For an extremely well written discussion on other practical algorithms to solve these problems using slightly different decompositions, see \cite{koeppe:irrational-barvinok}. For completeness, we end with the algorithmic description for the continuous generating function. 

\begin{algorithm}
\caption{Continuous generating function}

\begin{justify}
Output: the rational generating function for $I(\PP, \ell) = \int\limits_{\PP} e^{\langle \ell, x \rangle}\d x$ in the form

\[ I(\PP,\ell) = \sum_{s \in V(\PP)} \sum_{C \in \Gamma_s} \text{vol} (\Pi_C)e^{\ll \ell, s \rr} \prod_{i=1}^d \frac{1}{-\ll \ell, u_i \rr}. \]

where $u_1, \dots, u_d$ are the rays of cone $C$.
\end{justify}
\begin{enumerate}
\item  Compute all vertices $V(\PP)$ and corresponding supporting cones $\Cone(\PP - s)$
\item Triangulate $\Cone(\PP - s)$ into a collection of simplicial cones $\Gamma_s$ using any method
\item Write down the above
\end{enumerate}
\end{algorithm}

Note that in fixed dimension, the above algorithms compute the generating functions in polynomial time. We will repeatedly use the next lemma to multiply series in polynomial time in fixed dimension. The idea is to multiply each factor, one at a time, truncating after total degree $M$. 

\subsection{A power of a linear form}

Above, we developed an expression for $\sum_{x \in \PP \cap \Z^d} e^{\langle \ell, x \rangle}$, and $I(\PP, \ell) = \int_{\PP} e^{\langle \ell, x \rangle}\d x$. Later in Chapters \ref{ch:Integration} and \ref{ch:polynomialOptimization}, we will compute similar sums and integrals where instead of an exponential function, the summand or integrand is a power of a linear form, or more generally, a product of affine functions. The common trick will be to introduce a new variable $t$ and compute $S(\PP, \ell \cdot t)$ or $I(\PP, \ell \cdot t)$. If the series expansion in $t$ about $t=0$ is computed, we get a series in $t$ where the coefficient of $t^m$ is $\sum_{x \in \PP \cap \Z^d} \langle \ell, x \rangle^m$ or $ \int_{\PP} \langle \ell, x \rangle^m \d x$. To compute these series expansions, many polynomials will be multiplied together while deleting monomials whose degree is larger than some value $M$. The next lemma shows that this process is efficient when the number of variables is fixed, and we  repeatedly apply it in Chapters \ref{ch:Integration} and \ref{ch:polynomialOptimization}.

\begin{lemma}[Lemma 4 in \cite{baldoni-berline-deloera-koeppe-vergne:integration}]
\label{lemma:poly-mult}
For $k$ polynomials $h_1,\dots,h_k \in \Q[x]$ in $d$ variables, the product $h_1\cdots h_k$ can be truncated at total degree $M$ by performing $O(kM^{2d})$ elementary rational operations. 
\end{lemma}

\section{Handelman's Theorem and polynomial optimization}
In this section, we comment on the problem
\begin{equation*}
\begin{split}
\max & \; f(x) \\
 & x \in \PP,
\end{split}
\end{equation*}
where $f(x)$ is a polynomial and $\PP$ is a polytope. Handelman's Theorem is used in Chapters \ref{ch:Integration} and \ref{ch:polynomialOptimization} as a tool for rewriting a polynomial in a special way. This section introduces Handelman's theorem along with how it can directly be used for polynomial optimization. Section \ref{ch:bg:sec:handelman:moments} briefly illustrates how Handelman's theorem can be used instead of sum of squares polynomials. Then  finally in Section \ref{ch:bg:sh:poly-complexity}, we review the computational complexity of the polynomial optimization problem. 

In Chapter \ref{ch:polynomialOptimization}, Handelman's theorem is \emph{not} used to perform optimization. It is used as a tool to decompose the polynomial $f(x)$ into a form that makes integrating $f(x)^k$ more practical. These integrals are then used to produce bounds on the optimal value. With this view, we are using Handelman's theorem in a novel way. 

%\subsection{Handelman's theorem for polynomial optimization}
%\label{bg:handel:direct}
\begin{theorem}[Handelman \cite{handelman1988}]
Assume that $g_1,\dots,g_n\in\R[x_1, \dots, x_d]$ are linear polynomials and that the semialgebraic set
\begin{equation}\label{eqK}
S=\{x\in\R^d \mid g_1(x)\geq 0,\dots,g_n(x)\geq 0\}
\end{equation}
is compact and has a non-empty interior. Then any polynomial $f\in\R[x_1, \dots, x_d]$ strictly positive on $S$ can be written as
$f(x) = \sum_{\alpha\in\N^n} c_\alpha g_1^{\alpha_1}\cdots g_n^{\alpha_n}$
for some nonnegative scalars $ c_{\alpha}$.
\end{theorem}

We define the \emph{degree} of a Handelman decomposition be $\max |\alpha|$, where the maximum is taken over all the exponent vectors $\alpha$ of $g_i(x)$ that appear in a decomposition. 

Note that this theorem is true when $S$ is a polytope $\PP$, and the polynomials $g_i(x)$ correspond to the rows in the constraint matrix $b - Ax \geq 0$. See \cite{Castle20091285, deKlerk2015, monique2014} for a nice introduction to the Handelman decomposition. The Handelman decomposition is only guaranteed to exist if the polynomial is strictly greater than zero on $\PP$, and the required degree of the Handelman decomposition can grow as the minimum of the polynomial approaches zero  \cite{sankaranarayanan2013lyapunov}. The next three examples are taken from Section 3.1 of \cite{sankaranarayanan2013lyapunov}.

\begin{example}
Consider the polynomial $f(x) \in \Q[x]$ given by $f(x) = x^2 -3x+2 = (x-2)(x-1)$ on $[-1,1]$. Because $f(1)=0$, Handelman's theorem does not say that $f(x)$ must have a Handelman decomposition. However, 
\[ f(x) = (1-x)^2 + (1-x) = 1\cdot g_1^2g_2^0 + 1\cdot g_2^1 g_2^0\]
where $g_1 :=1-x $ and $g_2:=x+1 $, so $f(x)$ has a Handelman decomposition.
\end{example}

To apply Handelman's theorem, we must have that $f(x) > 0$ on $\PP$. 

\begin{example}
Consider the polynomial $f(x) \in \Q[x]$ given by $f(x) = x^2$ on $[-1,1]$. Because $f(0)=0$, Handelman's theorem does not say that $f(x)$ must have a Handelman decomposition. If $f(x)$ had a decomposition, then there would be numbers $c_\alpha > 0$ and integers $\alpha_i \in \Z_{\geq 0}$ such that \[x^2 = \sum_{\alpha \in J} c_\alpha (1-x)^{\alpha_1} (x+1)^{\alpha_2},\]
with $J$ being a finite subset of $\Z_{\geq 0}^2$. Evaluating both sides at zero produces the contradiction $0 = \sum_{\alpha \in J} c_\alpha > 0$. Hence $f(x)=x^2$ does not have a Handelman decomposition on $[-1,1]$.
\end{example}

\begin{example}
For every fixed $\epsilon > 0$, $p(x) = x^2+\epsilon$ must have a Handelman decomposition on $[-1,1]$. Let $t = \max\{\alpha_1 + \alpha_2 \mid \alpha \in J\}$ be the total degree of a Handelman representation. Then the next table lists what is the smallest $\epsilon$ value for which $x^2 + \epsilon$ has a degree $t$ decomposition.

\begin{center}
\begin{tabular}{ l | r r r r r r }
  \hline                       
  $t$ & 2 & 3 & 4 & 5 & 6 & 7 \\
  \hline
  $\epsilon$ & 1 & 1/3 & 1/3 & 1/5 & 1/5 & 1/7 \\
  \hline  
\end{tabular}

\end{center}
\end{example}

There are many questions relating to Handelman's theorem. For instance, answers to these questions are not well known or completely unknown.

\begin{itemize}
\item Given a nonnegative polynomial $f(x)$ on a polytope $P$, how large does the Handelman degree have to be? 
\item By adding a positive shift $s$ to $f(x)$, how can the  Handelman degree change for $f(x)+s$?
\item Fix $t$. Can a large enough shift be added to $f(x)$ so that $f(x)+s$ has a degree $t$ Handelman decomposition?
\item How can redundant inequalities in $\PP$'s description lower the Handelman degree or reduce the number of Handelman terms?
\end{itemize}

However, these questions do not prevent us from using Handelman's theorem as an effective tool for polynomial optimization. We now present a hierarchy of linear relaxations as described in \cite{monique2014} for maximizing a polynomial over a polytope. This is the most traditional way Handelman's theorem can directly be applied for optimization. Let $g$ denote the set of polynomials $g_1,\ldots,g_n$.
For an integer
$t\ge 1$, define the {\em Handelman set of order t} as
\begin{equation*}
\mathcal{H}_t(g):=\left\{\sum_{\alpha\in \Z_{\geq 0}^n \;:\; |\alpha| \leq t} c_{\alpha} g^\alpha: c_\alpha \geq 0\right\}
\end{equation*}
 and the
corresponding {\em Handelman bound of order $t$} as
\begin{equation*}\label{eqphan}
\fhan^{(t)}:=\inf \{\lambda: \lambda -f(x) \in \mathcal{H}_t(g)\}.
\end{equation*}
Clearly, any polynomial in $\mathcal{H}_t(g)$ is nonnegative on $\PP$ and one has the following chain of inclusions:
 $$\mathcal{H}_1(g)\subseteq \ldots \subseteq \mathcal{H}_t(g)\subseteq\mathcal{H}_{t+1}(g) \subseteq \ldots $$
giving the chain of inequalities:
 $\fmax\le \fhan^{(t+1)}\le  \fhan^{(t)}\le\dots\le \fhan^{(1)}$ for $t\geq 1$. When $\PP$ is a polytope with non-empty interior, the asymptotic convergence of the bounds $\fhan^{(t)}$ to $\fmax$ as the order $t$ increases is guaranteed by  Handelman's theorem.

Some results have been proved on the convergence rate in the case
when $\PP$ is the standard simplex or the hypercube $[0,1]^d$.

\begin{theorem}\cite{deKlerk2006210}
Let $\PP$ be the standard simplex $\PP = \{ x \in \R^d \mid x_i \geq 0, \sum_{i=1}^d x_i = 1 \}$, and let $f$ be a polynomial of total degree $D$. Let $\fmax$ and $\fmin$ be the maximum and minimum value $f$ takes on $\PP$, respectively. Then
\begin{eqnarray*}
\fhan^{(t)} -\fmax
\le  D^D {2D-1\choose D} \frac{{D\choose 2}}{t-{D\choose 2}} (\fmax-\fmin).
\end{eqnarray*}
\end{theorem}

\begin{theorem}\cite{klerk2010}
Let $K=[0,1]^d$ be the hypercube.
If $f$ is a polynomial of degree $D$ and $r\ge 1$ is an integer then the Handelman bound of order $t=rk$ satisfies:
$$\fhan^{(rk)}-\fmax \le {L(f) \over r} {D+1\choose 3} d^D,$$
with $L(f) := \max_{\alpha}  {\alpha!\over |\alpha|!}|f_\alpha|$.
\end{theorem}

\subsection{Handelman and moment methods}
\label{ch:bg:sec:handelman:moments}

Lasserre \cite{Lasserre01globaloptimization} introduced the observation that optimizing a  polynomial over a polytope can be done by integrating over all probability measures that are  absolutely continuous with respect to the  Lebesgue measure: 
\[\fmax = \sup_\sigma \int_\PP f(x) \sigma(x) \d x,\]
where the supremum is taken over all absolutely continuous functions $\sigma(x)$ such that $\sigma(x) \geq 0$ on $\PP$ and $\int_\PP \sigma(x) \d x = 1.$ 

A common relaxation is to just integrate over some subset of measures. Such methods are generally called \emph{moment methods}. One such method, the \emph{Lasserre hierarchy} \cite{Lasserre01globaloptimization}, is given by 

\[ f_t^{sos} := \max_{\sigma \in SOS_t} \left\{ \int_\PP f(x) \sigma(x) \d x \mid \int_\PP \sigma(x) \d x = 1\right\}, \]
where $SOS_t$ is the set of all polynomials which are sums of squares (and hence nonnegative on $\PP$) of degree at most $t$. Moreover, it is known that $\fmax \geq f_t^{sos} \geq f_{t+1}^{sos}$ and $\lim_{t \rightarrow \infty} f_t^{sos} = \fmax$ \cite{lasserre2011NewLook}.

The authors of \cite{deKlerk2015} extended this idea using Handelman's theorem. They defined  

\[ f_t^{han} := \max_{\sigma \in \mathcal{H}_t} \left\{ \int_\PP f(x) \sigma(x) \d x \mid \int_\PP \sigma(x) \d x = 1\right\}, \]
where $\mathcal{H}_t$ is the set of Handelman polynomials on $\PP$ of degree at most $t$. They also show $f_t^{han}$ converges to $\fmax$ as $t \rightarrow \infty$. 

The roles played by sums of squares polynomials and Handelman polynomials in polynomial optimization are very similar. Our approach to Handelman's theorem is not focused on using it for optimization, but rather as a decomposition of a nonnegative polynomial function. For more information on sum of squares type optimization techniques, see \cite{deKlerk2015, lasserre2009momentsBook, Lasserre01globaloptimization, lasserre2002semidefinite, lasserre2011NewLook, Lasserre2000929}.

\subsection{Complexity of the problem}
\label{ch:bg:sh:poly-complexity}

In Chapter \ref{ch:polynomialOptimization}, the main results are presented with precise statements on an algorithm's complexity. In this section, we review the complexity of polynomial optimization and explain why we always fix the dimension in complexity statements. First we review the complexity of maximizing a polynomial $f(x) \in \R[x_1, \dots, x_d]$ over a polytope $\PP \subset \R^d$. That is, consider the complexity of the problems
\begin{equation*}
\begin{split}
\max & \; f(x) \\
 & x \in \PP.
\end{split}
\hspace{2em}\text{ and }\hspace{2em}
\begin{split}
\max & \; f(x) \\
  x \in & \PP \cap \Z^d.
\end{split}
\end{equation*}
Both exact optimization problems are NP-hard, as they include two well known NP-complete problems: the max-cut problem, and the maximum stable set problem. The max-cut problem can be reduced to a quadratic optimization problem over the integer points of the 0-1 hypercube: $[0, 1]^d \cap \Z^d$ \cite{Hastad99someoptimal}. The maximum stable set problem can be reduced to a quadratic optimization problem over the standard simplex $\Delta_n = \{ x \in \R^d \mid \sum x_i = 1, x \geq 0\}$ \cite{motzkin-straus}. Furthermore, problems in fixed dimension can still be hard. For example, the problem of minimizing a degree-4 polynomial over the lattice points of a convex polytope in dimension two is NP-hard. See \cite{koeppeComplexity2012} for a nice survey on the complexity of optimization problems.

Because the exact optimization problems are difficult, even when the domain is a cube or simplex, practical algorithms instead approximate the maximum. However, this too can be hard without some further restrictions. When the dimension is allowed to vary, approximating the optimal objective value of polynomial programming problems is still hard. For instance, Bellare and Rogaway \cite{Bellare:1995} proved that if $P \neq NP$ and $\epsilon \in (0,\frac{1}{3})$ is fixed, there does not exist an algorithm $\mathcal{A}$ for continuous quadratic programming over polytopes that produces an approximation $f_\mathcal{A}$ in polynomial time where $|f_\mathcal{A} - \fmax| < \epsilon |\fmax - \fmin|$. 

Next we give three definitions of approximation algorithm common within combinatorial optimization for nonnegative objective functions $f(x)$.

\begin{definition}
\hfill
\begin{enumerate}
\item An algorithm $\mathcal{A}$ is an \emph{$\epsilon$-approximation algorithm} for a maximization problem with optimal cost $\fmax$, if for each instance of the problem of encoding length n, $\mathcal{A}$ runs in polynomial time in $n$ and returns a feasible solution with cost $f_\mathcal{A}$ such that $f_\mathcal{A} \geq (1-\epsilon) \fmax$. Note that epsilon in this case is fixed, and is not considered an input to the algorithm.
\item A family of algorithms $\mathcal{A}_\epsilon$ is a \emph{polynomial time approximation scheme (PTAS)} if for every error parameter $\epsilon > 0$, $\mathcal{A}_\epsilon$ is an $\epsilon$-approximation algorithm and its running time is polynomial in the size of the instance for every fixed $\epsilon$. Here, the algorithms are parameterized by $\epsilon$, but $\mathcal{A}_\epsilon$ is still efficient when $\epsilon$ is fixed.
\item A family of $\epsilon$-approximation algorithms $\mathcal{A}_\epsilon$ is a \emph{fully polynomial time approximation scheme (FPTAS)} if the running time of $\mathcal{A}_\epsilon$ is polynomial in the encoding size of the instance and $1/\epsilon$. In this case, $\epsilon$ is an input to the algorithm, and its complexity term is polynomial in $1/\epsilon$.
\end{enumerate}
\end{definition}

In light of these complexity results and definitions, Chapter \ref{ch:polynomialOptimization} describes a FPTAS algorithm for polynomial optimization over a polytope, when the dimension $d$ is fixed. Note that we use the usual input size where everything is measured in the binary encoding. However exponents, such as a polynomial's degree or the integer $k$ in $f(x)^k$, must encoded in unary, otherwise their values cannot be computed in polynomial time. To see this, consider the problem of computing $2^x$. The encoding length of this number is $x$ in binary, which is a polynomial size when $x$ is measured in unary. If $x$ is encoded in binary, then its length is $\log_2(x)$, and so the length of $2^x$ is exponential in the binary encoding size of $x$.

\section{Ehrhart Polynomials}
\label{sec:bg:knap}

%(for a reference about computational complex analysis see \cite{henrici3,henrici1,henrici2}.)

This section introduces Chapter \ref{ch:knapsack}. In Section \ref{ch:bg:sec:ehrhart:theory}, we review the two most important theorems in Ehrhart theory. Then in Section \ref{ch:bg:sec:ehrhart:knapsack-polytopes} we set up the main question Chapter \ref{ch:knapsack} addresses.

\subsection{Two classic theorems}
\label{ch:bg:sec:ehrhart:theory}

Let $\PP$ be a rational polytope, meaning each vertex is in $\Q^d$. The Ehrhart function counts the number of integer points in $\PP$ as $\PP$ is scaled by an \emph{nonnegative integer} number $t$:

\[L(\PP, t) := | \{x \in \Z^d \mid x \in t\cdot \PP \}|. \]

If $\PP$ is given by a vertex representation, $t\cdot \PP$ is given by multiplying each vertex by the scalar $t$. If $\PP$ has the description $Ax \leq b$, then $t\cdot\PP$ is $Ax \leq tb$. The next theorem states that $L(\PP, t)$ has remarkable structure, and is the cornerstone to an area of combinatorics called Ehrhart Theory.

\begin{theorem}[Ehrhart \cite{ehrhart1962geometrie}]
For any polytope $\PP$ with integer vertices, $L(\PP, t)$ is a polynomial in $t$, where the degree is equal to the dimension of $\PP$, the leading term in the polynomial is equal to the volume of $\PP$, and the constant term is 1.
\end{theorem}

\begin{example}
Consider the polytope that is a square in $\R^2$ given as the convex hull of the vertices $(0,0)$, $(0,1)$, $(1,0)$, and $(1,1)$. Computing a few values gives $L(\PP, 0) = 1$, $L(\PP, 1) = 4$, $L(\PP, 2) = 9$. The Ehrhart polynomial is $L(\PP, t) = (t+1)^2$.
\end{example}

When the polytope has rational vertices, the Ehrhart polynomial can still be defined, but now the coefficients may no longer be constants, they may be periodic functions. 

\begin{definition}
A \emph{quasi-polynomial} of degree $k$ is a polynomial that has the from $q(t) = c_d(t)t^d + c_{d-1}(t)t^{d-1} + \cdots + c_0(t)$, where $c_i(t)$ is a periodic function with integral period. Quasi-polynomials can be written in many ways. For instance given a collection of polynomials $\{p_i(t)\}$ and a period $s$, $q(t) = p_i(t)$ where $i = t \mod s$ is a quasi-polynomial. In Chapter \ref{ch:knapsack}, we use another representation: \emph{step functions}.
\end{definition}

The other well-known theorem in Ehrhart theory applies to rational polytopes. 

\begin{theorem}[Ehrhart \cite{ehrhartbook}]
For any polytope $\PP$ with rational vertices, $L(\PP, t)$ is a quasi-polynomial in $t$, where the degree is equal to the dimension of $\PP$, the leading term in the polynomial is equal to the volume of $\PP$, and the constant term is 1.
\end{theorem}

There are a few software packages for computing the Ehrhart polynomial. One of the best has to be \latte\footnote{Available under the GNU General Public
  License at \url{https://www.math.ucdavis.edu/~latte/}.}  \cite{latteintegrale}. \latte can also be used within some other popular software tools such as \texttt{Polymake} \cite{polymake-software} and \texttt{Sage} \cite{sage}. Another good package is \texttt{barvinok} \cite{barvinok-noversion}.

For introductions to Ehrhart Theory, see  \cite{ barvinokzurichbook, beckrobins, ehrhartbook, stanley}. For generalizations when the dilation factor $t$ is real, see \cite{so-called-paper-2, linke:rational-ehrhart}. \latte can also compute the quasi-polynomial of rational polytopes when $t$ is a real number.
\subsection{Special case: knapsack polytopes}
\label{ch:bg:sec:ehrhart:knapsack-polytopes}

Given sequence $\a = [\alpha_1,\alpha_2,\dots,\alpha_{N+1}]$  of $N+1$ positive integers, Chapter \ref{ch:knapsack} seeks to compute the the combinatorial function $E(\a; t)$ that counts the non-negative integer solutions of the equation $\alpha_1x_1+\alpha_2 x_2+\cdots+\alpha_{N} x_{N}+\alpha_{N+1}x_{N+1}=t$, where the right-hand side~$t$ is a varying non-negative integer. This is precisely the Ehrhart quasi-polynomial of the polytope 
\[\PP = \left\{ \ x \in \R^{N+1} \mid \sum_{i=1}^{N+1}  x_i = 1, x_i \geq 0\right\}.\]
And so $E(\a; t)$ is a quasi-polynomial function in the variable $t$ of degree $N$. The polytope $\PP$ can be called by a few different names such as a knapsack polytope and $N$-dimensional simplex. 

The Ehrhart function over this polytope also has a few names and computing it as a close formula or evaluating it for specific $t$ is relevant in several other areas of mathematics. For example, in combinatorial number theory, this function is known as Sylvester's \emph{denumerant}.
In the combinatorics literature the denumerant has been studied extensively (see e.g.,\cite{agnarsson,beckgesselkomatsu, comtetbook,lisonek,riordanbook}. The denumerant plays an important role in integer optimization too\cite{kellereretalbook,martellotothbook}, where the problem is called an \emph{equality-constrained knapsack}.  In combinatorial number theory and the theory of partitions, the problem appears in relation to the \emph{Frobenius problem} or the \emph{coin-change problem} of finding the largest value of $t$ with $E(\a; t)=0$ (see \cite{wagonetal, kannanfrobenius, ramirezalfonsinbook} for details and algorithms).  Authors in the theory of numerical semigroups have also investigated the so called \emph{gaps} or \emph{holes} of the function (see \cite{hemmeckeetal} and references therein), which are values of $t$ for which $E(\a; t)=0$, i.e., those positive integers $t$ which cannot be represented by the $\alpha_i$.  For $N=1$ the number of gaps is $(\alpha_1-1)(\alpha_2-1)/2$ but for larger $N$ the problem is quite difficult.

Unfortunately, computing $E(\a; t)$ or evaluating it are very challenging computational problems. Even deciding whether $E(\a; t)>0$ for a given $t$, is a well-known (weakly) NP-hard problem. Computing $E(\a; t)$ for a given~$t$, is $\#P$-hard. Computing the Frobenius number is also known to be NP-hard \cite{ramirezalfonsinbook}.  Despite this, when the dimension $N$ is fixed (the number of variables is fixed), the Frobenius number can be computed in polynomial time \cite{barvinokwood, kannanfrobenius}. Also, when $N+1$ is fixed, the entire quasi-polynomial $E(\a; t)$ can be computed in polynomial time as a special case of a well-known result of Barvinok \cite{bar}. There are several papers exploring the practical computation of the Frobenius numbers (see e.g., \cite{wagonetal} and the many references therein).

One main result in Chapter \ref{ch:knapsack} is a polynomial time algorithm for approximating the Ehrhart polynomial by computing the highest-degree $k+1$ terms of the Ehrhart quasi-polynomial. Unlike methods for computing the entire polynomial $E(\a; t)$ in polynomial time, our method allows the dimension to vary; however, we fix $k$ to obtain a polynomial time algorithm. The algorithm takes advantage of three key things: the simple structure of the knapsack polytope, a geometric reinterpretation of some rational generating functions in terms of lattice points in polyhedral cones, and Barvinok's \cite{bar} fast decomposition of a polyhedral cone into unimodular cones. 

Chapter \ref{ch:knapsack} is also closely related to some other works. In \cite{so-called-paper-1}, the authors presented a polynomial-time algorithm to compute the first $k+1$ coefficients of $L(\PP, t)$ for any \emph{simple polytope} (a $d$-dimensional polytope where each vertex is adjacent to exactly $d$ edges) given by its rational vertices. In \cite{barvinok-2006-ehrhart-quasipolynomial}, the entire quasi-polynomial over a simplex is computed for a fixed periodicity value. These two papers use 
the geometry of the problem very strongly, while the methods of Chapter \ref{ch:knapsack} are based  the number-theoretic  structure of the knapsack.

  \section{MILP Heuristics}
\label{sec:bg:milp}
   
A mixed integer linear program (MILP) is a problem in the form  

\begin{align*}
\min \quad & c^Tx \\
\text{such that } & Ax \leq b\\
& x_i \in \Z \text{ for } i \in I
\end{align*}
 
 where $c \in \R^d$, $A \in \R^{m \times d}$, $b \in \R^m$, and $I$ is an index set for the integer variables. The problem is to find values for the $d$ variables $x_i$  that satisfy the linear constraints $Ax \leq b$ and the integer constraints $x_i \in \Z$ for $i \in I$ while minimizing the objective function $c^Tx$. The $A$, $b$, and $c$ are given as constants and $x$ is the only vector of variables. If $I = \emptyset$, meaning that the problem contains no integer variables, this problem reduces to a linear program. Linear programs are some of the best understood mathematical objects in optimization, and there exist a large collection of commercial-grade software to solve these problems. Furthermore, linear programs can be solved in polynomial time. 
 
 Once some of the variables in a linear program are forced to be integer, the optimization problem becomes a mixed integer problem. MILP have extremely useful modeling power and have been successfully used to solve problems in mining and refinery, air transport, telecommunications, finance, water supply management, chip design and verification, network design---basically all  scheduling and planning problems, see  for instance \cite{milpApplicationXpress}.
 
 The success of solving MILP problems despite being NP-Hard \cite{Schrijver2003Combinatorial-O} can be attributed to tricks and techniques that transform the MILP problem into smaller linear programming problems. The common tools include pre-solvers, branch-and-bound, cutting planes,  and heuristics \cite{fischettiMILP, grossmanReview}. A MILP heuristic is any algorithm that attempts to find a MILP solution, or somehow speeds the process up. A heuristic does not have to have any mathematical logic to why it is a good idea. For example, one simple heuristic for solving MILP problems is to compute the linear programming relaxation. Let $x^*$ be a solution to 
\begin{align*}
\min \quad & c^Tx \\
\text{such that } & Ax \leq b\\
& x_i \in \R \text{ for every } i.
\end{align*}
 Some of the integer variables $x_i \in I$ might have non-integer-valued solutions in the relaxation. A candidate feasible solution can be obtained by rounding the fractional values to integer values: $\floor{x_i^*}$ for $i\in I$. The rounded linear relaxation solution may or may not satisfy $Ax \leq b$. If it does, then this rounding heuristic just produced a feasible point to the MILP problem. 
 
 Heuristics themselves come in many types and can appear in every stage of the solving process. There are heuristics for finding an initial feasible point, producing better feasible points, and for directing the search of the optimizer. See \cite{berthold2006primal} for a nice short summary of the different classes of MILP heuristics. Chapter \ref{ch:dancingLinks} develops a special heuristic for problems with \emph{set partitioning constraints}. 

\section{Using computational geometry in neuro-immune communication}
\label{sec:bg:spleen}

Chapter \ref{ch:spleen} is based on a collaboration between the author and researchers at the University of California, Davis, School of Veterinary Medicine. The focus is on using tools from data science, image processing, and computational geometry to help identify diseases in the spleen. In this section, we give the medical background of the project. 

%\subsection{A science-fiction future}

Imagine a scenario where someone is sick and they visit the doctor. Instead of getting a pill to take, they are connected to a machine that sends electrical current to their body. The electricity is targeted to specific organs or parts of the nervous or lymphatic systems connected to their immune system. Under the electric current, their body's immune system is improved, and the disease is cured! Imagine a world where we can regulate neural impulses to repair the body. Such techniques could be used, for example, to treat diabetes by restoring insulin producing cells. Weight loss could be achieved by  regulating food intake by making the body think it is full or by removing cravings for junk food.  Imagine treating diseases ranging from hypertension to pulmonary diseases with electricity. 

This is not electroshock therapy or something out of Frankenstein, but an area of medical research called electroceuticals \cite{famm2013drug}. The body's neurons make a great system for delivering medical treatment. First, almost all organs and biological functions are regulated through circuits of neurons which function through electrical charges. In addition, the neural network is discrete; it is composed of nerve bundles and fiber tracts which interconnect individual cells. Hence an electrical current could be precisely applied to an area of the body. 

Some electroceuticals have already been developed. Pacemakers and defibrillators already use the body's nervous system to save lives.  Sacral-nerve stimulation can restore bladder control in people with paraplegia, and vagus-nerve stimulation can be applied to epilepsy and rheumatoid arthritis. What separates these technologies from future techniques is that they do not target specific areas of the nervous system. It is easy to send a signal to 100,000 nerve fibers, but not to just one. Electroceuticals of the future could be applied on precise areas for targeted results. 

Many advances must be made before the full power of electroceuticals  can be used. For instance, we must be able to map the body's network of nerve cells, we must know how to excite only a few nerve fibers at a time, and we must understand how the frequency of the electrical current should change over the course of treatment. For much more detailed description of what electroceuticals can cure and what are the problems we must first solve, see \cite{famm2013drug, metz2005takes, pavlov2012vagus}. 

\subsection{Spleen}

In the spleen and other secondary lymph organs, it has been suggested that $CD4^+ChAT^+$ T-cells (henceforth referred to as just T-cells) directly synapse with neurons; however, there is little data to support or refute this idea \cite{felten1987noradrenergic, tournier2003neuro}. My collaborators in the  School of Veterinary Medicine at the University of California, Davis have started addressing this knowledge gap by defining the mechanisms of communication and rigorously mapping the nature of the interaction between  T-cells, neurons, and target effector cells in mice during various stages of intestinal inflammation. The functional ramifications of how T-cells and nerve cells interact  are vast. A synaptic model could imply that T-cells communicate with neurons, forming intimate associations with these cells, ceasing their surveying of the spleen. In addition, this tight association would suggest that unless released from such synapses,  T-cells would have limited mobility and would exert a highly localized regulatory effect. This “lymphocyte-neuron synapse” notion has led to the hypothesis that immune cells can be innervated and controlled by specific neurons to reduce morbidity in patients suffering from immunopathologies, including inflammatory bowel disease \cite{andersson2012new, famm2013drug}. On the other hand, it could be that the T-cells do not form tight bonds with neurons and are always free to move. This is called the \emph{diffusion model}, and its implications mean the nervous system cannot be easily used to program their function or activation. 

My collaborators at the UC Davis School of Veterinary Medicine, Colin Reardon and Ingrid Brust-Mascher, have used an advanced imaging technique \cite{chung2013clarity, chung2013structural}  to generate detailed mappings of neuro-immune interfaces with subcellular resolution. They have the capability to image the entire lymph node or spleen volume. Mapping these interactions and defining how neural stimulation influences a response from lymphocytes in mice is vital in understanding how cellular communication occurs in mice, and eventually humans. The T-cells form a critical component of the inflammatory reflex by significantly inhibiting morbidity and mortality in septic shock. These cells were proposed to interface directly with  sympathetic neurons, and were shown to be in close proximity to these neurons in the spleen. It has been suggested that a complete neuro-immune circuit not only requires input from neurons \cite{rosas2011acetylcholine}, but that this communication occurs in a classical synapse \cite{felten1987noradrenergic, tournier2003neuro}. A dissenting view has emerged in neurobiology showing that a classical synapse is not always required for neural signaling \cite{agnati2006volume, muller2014cell, nguyen2010vivo, sarter2009phasic, tsai2009phasic}.  Despite these high profile studies demonstrating neural regulation of innate immune function, a functional circuit tracing in the spleen with and without  inflammation has not been performed. 

%MC previous paragrpah: they, following paragraphs: we. Which one is it? 
Preliminary data suggests that most  T-cells are not in intimate contact with sympathetic neurons. Although we observed T-cells within $5 \mu$m of  sympathetic neurons, the majority of these cells do not appear to be intimately associated, as would be required in a synapse. Our data set also demonstrates the advantage of 3D image reconstruction to assess localization. While  T-cells and fibers appear to be co-localized in some orientations, rotation along the axis of the neuron reveals no intimate association. Our data does not directly refute the possibility of synaptic transmission, but it suggests that most T-cells communicate with sympathetic neurons through another mechanism.

However, we still want to study how strongly the T-cells and neurons connect. It may be true that most T-cells are not located near the sympathetic neurons, and hence do not form synaptic connections to them, but there are still clusters of T-cells that seem to connect to the neurons. It might be the case that the distribution of where the T-cells are located changes as the spleen responds to intestinal inflammation. For example, maybe T-cells are more likely to from synaptic bonds with the neurons during inflammation, and then diffuse when the inflammation is over. The advanced imaging techniques allow us to see how the  physical relationships between the cells changes as the spleen responds to the inflammation. Images showing the location of T-cells and neurons  do not show a clear distinction between healthy and sick spleens. This can be due to a few factors: our data sets are large and hard to visualize, imaging is not perfect and so our data contains noise, and it is possible that there is little relationship between T-cell location and spleen inflammation. Just because it is hard to visualize  how the location of T-cells and neurons change with inflammation does not mean we cannot detect a difference. This is where the project leaves the sphere of biology and enters the world of data science. 

\subsection{The data science question}
%MC we wary of switching between past and present

My collaborators have mapped the location of T-cells and nerve cells in the spleen, forming two sets of points in $\R^3$. The two main questions now are: 
\begin{enumerate}
\item How does the distance between these two sets of points change during inflammation?
\item Can we use the distribution of these point sets to predict inflammation?
\end{enumerate}
Our approach is to compute the distance from each T-cell to a nervous cell. Looking at the distribution of these distance values should reveal if T-cells are closer to nervous cells or not. Hence the problem is to perform a classical supervised clustering analysis on the distance distributions across different spleen samples. How we did this is one of the main topics in Chapter \ref{ch:spleen}. 

Before we can compute the distance distributions, the T-cells and nerve cells must first be identified. As it turns out, this is a hard problem. While it is true that we have sophisticated imaging tools and software, our data is still extremely noisy. To accurately extract the location of these cells, we use the software tool Imaris \cite{imaris-software}.  Despite being the best tool for the job, it cannot process the entire spleen at once without identifying T-cells and nerve cells at locations that are physically impossible! 
%Another complication is that we are interested in a special kind of T-cell, but the images that have been collected so far cannot distinguish between these T-cells and other kinds of T-cells. To extract the cells we care about in Imaris, we run a function in Imaris that takes a parameter. This parameter cannot be , and B cells as well. So the green spots are actually all these cells, and the cells of real interest are some unknown subset of them. This doesn't change the math at all, but I think it's something that should be clarified because I feel like your explanation gives the impression that we know exactly where these cells are in the images.
Instead, the spleen data must be divided into many small regions, and correct parameters to Imaris functions must be picked for each region. The result is that it is an extremely labor- and time-intensive process to clean up the image data in Imaris. Hence a major problem is how to automate some of the data cleaning within Imaris. This problem leads us into using topics in distance geometry for data cleaning. Chapter \ref{ch:spleen} also explores some of the tools we used and developed to significantly reduce this project's  time to completion.

%\subsection{Relevance to human data}

%This research is both pertinent and timely. Despite scant understanding of the basic immunological processes at work, clinical trials have already begun with implantation of neural stimulation devices to ameliorate rheumatoid arthritis and inflammatory bowel disease \cite{andersson2012new, famm2013drug}.

%Data supporting either model, synapse or diffusion, will be critical to the development of new technologies. Such data could therefore drastically change the landscape of therapeutics designed to limit inflammation in a number of diseases and significantly improve human health.

%We have elected to pursue studies in C57BL/6 mice due to the ability to experimentally manipulate these animals. Although discoveries in mice do not always lead to translatable therapeutics in human patients, the immune system in mouse models of disease is a valid predictor of human immune responses in several diseases \cite{takao2015genomic}. From these findings, we believe that the general principles of neuro-immune communication are likely to be conserved between mouse and human. 

\chapter{Integration}
\label{ch:Integration}   

The theory of integrating a polynomial over a polytope by using continuous generating functions has been developed over a series a papers \cite{baldoni-berline-deloera-koeppe-vergne:integration, deloera:software-exact-integration-polynomials, brandon-handelman-paper}. These papers develop different generating functions and algorithms depending on the form the integrand takes. All the methods are based on the continuous generating function for $\int_\PP e^{\ll x, \ell \rr}\d x$ (see Section \ref{ch:bg:sec:integration-stuff} for an introduction). In this section, we seek to organize the results by the integrand type. 

For this chapter, $\PP$ will be a $d$-dimensional rational polytope given by an inequality description $Ax \leq b$ where $A \in \Q^{n \times d}$ and $b \in \Q^n$. Let $f(x) \in \Q[x_1, \dots, x_d]$ be a polynomial. Our goal is to compute $\int_\PP f(x) \d x$. In \cite{baldoni-berline-deloera-koeppe-vergne:integration}, the focus is on the case when the polytope $\PP$ is a simplex and the integrand is a power of a linear form (e.g., $(x_1 + 2x_2 - x_3)^5$), or a product of linear forms (e.g., ${(x_1 + 2x_2)^2(3x_1 + x_3)^4}$). Then in \cite{deloera:software-exact-integration-polynomials}, the case when $\PP$ is an arbitrary full-dimensional polytope and the integrand is a power of a linear form is developed. Most recently in \cite{brandon-handelman-paper}, integration of products  of affine functions (e.g., ${(x_1 + 2x_2 + 3)^2(3x_1 + x_3 - 1)^4}$) was developed. The difference between products of linear forms and products of affine functions is that the latter is allowed to have constant terms in the factors.

Section \ref{sec:integrand:power-of-linear-forms} covers the discussion on integrating a polynomial $f(x)$ by decomposing it into a sum of powers of linear forms. This section contains highlights from \cite{deloera:software-exact-integration-polynomials}, except Section \ref{sec:intplf:domain:simplex} which contains results from \cite{baldoni-berline-deloera-koeppe-vergne:integration}. 

Section \ref{sec:integration-prducts-affine-functions} covers integrating a product of affine functions. It is a partial summary of \cite{brandon-handelman-paper}, except Section \ref{sec:intpaf:domain:simplex} which contains slightly generalized results from \cite{baldoni-berline-deloera-koeppe-vergne:integration}.

The integration methods we focus on will depend on decomposing the polyhedral domain $\PP$ in two different ways: triangulating $\PP$ or triangulating the tangent cones of $\PP$. This chapter describes the integration algorithms on both domain types. 

Polyhedral triangulation plays a role in every method we develop for integration. For general background in triangulations, see \cite{DRStriangbook, leeRegularTriangulations, Pfeifle2003, rambau2002topcom} and the software packages \cite{4ti2, cddlib-094a, polymake-software, topcom-software, sage}.

\section{Integrand: powers of linear forms}
\label{sec:integrand:power-of-linear-forms}

In this section, we focus on computing $\int_\PP f(x) \d x$ by decomposing $f(x)$ into a sum of power of linear forms: 

\[f(x) = \sum_\ell c_\ell \ll x, \ell \rr^{m_\ell}\]
where $\ell \in \Q^d$, $c_\ell \in \Q$ and $m_\ell \in \N$. 

This has been done in \cite{deloera:software-exact-integration-polynomials} by using the following formula on each monomial of $h$. If $x^{m}= x_1^{m_1}x_2^{m_2}\cdots x_d^{m_d}$, then

\begin{multline} 
x^{m}  = \frac{1}{|m|!} \sum_{0\leq p_i\leq m_i}(-1)^{|m|-(p_1+\cdots+p_d)}
  \binom{m_1}{p_1}\cdots \binom{m_d}{p_d}(p_1 x_1+\cdots+p_d x_d)^{|m|},
  \label{eq:decomp-powerlinform}
\end{multline}
where $|m| = m_1+\cdots+m_d$. Using this formula on a degree $D$ polynomial in fixed dimension $d$ results in at most $O(D^{2d})$ linear forms.

It is worth noting that Equation \eqref{eq:decomp-powerlinform} does \emph{not} yield an optimal decomposition. The problem of finding a decomposition with the smallest possible number of summands is known as the \emph{polynomial Waring problem} \cite{alexanderhirschowitz, brambillaottaviani, Carlini20125}. A key benefit of Equation \eqref{eq:decomp-powerlinform} is that it is explicit, is computable over $\Q$, and is sufficient for generating a polynomial-time algorithm on when the dimension is fixed \cite{baldoni-berline-deloera-koeppe-vergne:integration}. However, Equation \eqref{eq:decomp-powerlinform} has a major problem: it can produce many terms. For example, a monomial in $10$ variables and total degree $40$ can have more than $4$ million terms. See Table \ref{tabel:count-linear-forms}. Nevertheless, in the next two sections we will develop polynomial time algorithms for integrating a power of a linear form over a polytope. 

\begin{table}
\centering
\caption{Average number of powers of linear forms plus or minus one standard deviation necessary to express one monomial in $d$ variables, averaged over 50 monomials of the same degree}
\label{tabel:count-linear-forms}
\tabcolsep 2.5pt
\small
\begin{tabular}{c*8{r}}
	\toprule	
	
	& \multicolumn{6}{c}{Monomial Degree}\\
		\cmidrule(r){2-7}

\multicolumn{1}{c}{$d$}&   \multicolumn{1}{c}{5} &  \multicolumn{1}{c}{10} &  \multicolumn{1}{c}{20} &  \multicolumn{1}{c}{30} &  \multicolumn{1}{c}{40} &  \multicolumn{1}{c}{50} \\ 
\hline
3 & $14 \pm 3$ & $(6.6 \pm 1.2)\cdot10^{1}$ & $(4.0 \pm 0.5)\cdot10^{2}$ & $(1.2 \pm 0.1)\cdot10^{3}$ & $(2.7 \pm 0.2)\cdot10^{3}$ & $(5.2 \pm 0.2)\cdot10^{3}$  \\ 
4 & $16 \pm 5$ & $(1.1 \pm 0.2)\cdot10^{2}$ & $(1.1 \pm 0.2)\cdot10^{3}$ & $(4.5 \pm 0.6)\cdot10^{3}$ & $(1.3 \pm 0.2)\cdot10^{4}$ & $(3.0 \pm 0.2)\cdot10^{4}$  \\ 
5 & $19 \pm 4$ & $(1.5 \pm 0.4)\cdot10^{2}$ & $(2.2 \pm 0.6)\cdot10^{3}$ & $(1.2 \pm 0.3)\cdot10^{4}$ & $(4.7 \pm 0.7)\cdot10^{4}$ & $(1.4 \pm 0.2)\cdot10^{5}$  \\ 
6 & $20 \pm 5$ & $(2.0 \pm 0.6)\cdot10^{2}$ & $(4.1 \pm 1.2)\cdot10^{3}$ & $(3.2 \pm 0.8)\cdot10^{4}$ & $(1.5 \pm 0.3)\cdot10^{5}$ & $(5.2 \pm 0.6)\cdot10^{5}$  \\ 
7 & $21 \pm 5$ & $(2.4 \pm 0.9)\cdot10^{2}$ & $(6.7 \pm 2.4)\cdot10^{3}$ & $(7.1 \pm 2.1)\cdot10^{4}$ & $(4.0 \pm 1.0)\cdot10^{5}$ & $(1.7 \pm 0.3)\cdot10^{6}$  \\ 
8 & $21 \pm 5$ & $(2.9 \pm 0.9)\cdot10^{2}$ & $(1.1 \pm 0.5)\cdot10^{4}$ & $(1.4 \pm 0.5)\cdot10^{5}$ & $(9.8 \pm 2.7)\cdot10^{5}$ & $(4.8 \pm 1.1)\cdot10^{6}$  \\ 
10 & $24 \pm 5$ & $(3.5 \pm 1.1)\cdot10^{2}$ & $(2.1 \pm 0.9)\cdot10^{4}$ & $(4.1 \pm 1.6)\cdot10^{5}$ & $(4.5 \pm 1.7)\cdot10^{6}$ & $(3.1 \pm 1.0)\cdot10^{7}$ \\ 

	\bottomrule
\end{tabular}
\end{table}

\subsection{Domain: cone decomposition}
\label{sec:intplf:domain:cone}

We now consider powers of linear forms instead of exponentials.  Similar to
$I(P, \ell)$ in Section \ref{ch:bg:sec:integration-stuff}, we now let $L^M(P, \ell)$ be the meromorphic extension of the function defined by
$$ L^M(P, \ell) = \int_{P} {\langle \ell, x \rangle}^M  \d{x}$$ for those
$\ell$ such that the integral exists. 
To transfer what we know about integrals of exponentials to those of powers of
linear forms, 
we can consider the formula of Proposition~\ref{prop:integral-exp-simplicial}
as a function of the auxiliary parameter $t$:  
\begin{equation}\label{eq:formula-with-t}
  \int_{s+C} e^{\langle t \ell, x \rangle}  \d{x}=\vol(\Pi_C)
  e^{\langle t \ell, s \rangle}\prod_{i=1}^d \frac{1}{\langle- t \ell,
    u_i \rangle}. 
\end{equation}
Using the series expansion of the left in the variable~$t$, we wish to
recover the value of the integral of $\langle \ell, x \rangle^M$ over
the cone (by value, we mean a real number or a meromorphic function as explained in Section \ref{ch:bg:sec:integration-stuff}). This is the coefficient of $t^M$ in the expansion; to
compute it, we equate it to the Laurent series expansion around $t=0$
of the right-hand-side expression, which is a meromorphic function
of~$t$. Clearly
$$\vol(\Pi_C) e^{\langle t \ell, s \rangle}  \prod_{i=1}^d \frac1{\langle- t
  \ell, u_i \rangle}=\sum_{n=0}^\infty t^{n-d} \frac{\langle \ell, s
  \rangle^n}{n!}\cdot \vol(\Pi_C) \prod_{i=1}^d\frac1{\langle-\ell, u_i \rangle}.$$

\begin{corollary} \label{oneconeoneplf}
  For a linear form $\ell$ and a simplicial cone $C$ generated by rays $u_1,u_2,\dots u_d$ with vertex $s$ and $\ll \ell, u_i \rr \neq 0$,

\begin{equation} \label{above}
L^M(s+C, \ell) %%= \int_{s+C}  {\langle \ell, x \rangle}^M  \d{m}
=\frac{M!}{(M+d)!} \vol(\Pi_C) \frac{(\langle \ell, s \rangle)^{M+d}}{\prod_{i=1}^d  \langle -\ell, u_i \rangle}.
\end{equation}
\end{corollary}

\begin{corollary} \label{cor:the-lawrence-method}
For any triangulation $\mathcal D_s$ of the tangent cone~$C_s$ at each of the
vertices~$s$ of the polytope~$P$ we have 
\begin{equation}
	L^M(P, \ell)% := \int_P \langle \ell, x \rangle^M \d{m}
        =\sum_{s\in V(P)} \sum_{C \in \mathcal D_s} L^M(s+C)(\ell).
\end{equation}
\end{corollary}

Notice that when $\PP$ is a polytope, $L^M(P, \ell)$ is holomorphic in $\ell$, while each summand in the last corollary is meromorphic in $\ell$. Hence the singularities in $L^M(s+C, \ell)$ cancel out in the sum. 

We say that $\ell$ is \emph{regular}  if $\langle \ell, u_i \rangle \neq 0$ for 
 every ray $u_i$ of a cone $C$. So if $\ell$ is not orthogonal to any ray of every simplicial cone in $\mathcal D_s$, then Corollary \ref{cor:the-lawrence-method} gives an explicit formula for evaluating the exact value of $\int_\PP \ll \ell, x \rr^M \d x.$ We will also say that $\ell$ is \emph{regular on $\PP$}, if it is regular for every tangent cone of $\PP$, which implies it is regular on every cone $C \in \mathcal{D}_s$ at each vertex $s \in V(\PP)$.
 
Otherwise when $\ell$ is not regular,  there is a nearby perturbation which is regular. To obtain it,
we use  $\ell + \hat \epsilon$ where $\hat \epsilon = \epsilon a$ is any linear form with $a \in \R^n$ 
such that $\langle -\ell - \hat \epsilon, u_i \rangle \neq 0$ for all $u_i$. Notice that $L^M(\PP, \ell + \hat \epsilon)$ is holomorphic in  $\hat \epsilon$ when $\PP$ is a polytope, but for some cones $C$, $L^M(s+C, \ell + \hat \epsilon)$ is meromorphic in $\epsilon$; hence, the singularities cancel out in the sum in Corollary \eqref{cor:the-lawrence-method}, so taking the limit at as $\epsilon$ goes to zero is possible. This means when $\ell$ is not regular, we can just collect the coefficient of $\epsilon^0$ in the series expansion of $L^M(s+C, \ell + \hat \epsilon)$ and compute $L^M(s + C, \ell)$ as
\begin{equation}\label{eq:residue-with-form-perturbation}
\frac{M!}{(M+d)!} \vol(\Pi_C) \Res_{\epsilon=0}\frac{(\langle \ell + \hat \epsilon, s \rangle)^{M+d}}{\epsilon \prod_{i=1}^d  \langle - \ell - \hat \epsilon, u_i \rangle}.
\end{equation}
We use the residue residue operator $\Res$ as a shorthand to mean to the coefficient of $\epsilon^0$ in the series expansion of 
\[
\frac{(\langle \ell + \hat \epsilon, s \rangle)^{M+d}}{\prod_{i=1}^d  \langle - \ell - \hat \epsilon, u_i \rangle}.
\]
about $\epsilon = 0$.

Next we recall some useful facts on complex analysis (see, e.g., \cite{Henrici} for details).  As we observed, there is a pole at $\epsilon=0$ for our univariate rational 
function  given in Formula \eqref{eq:residue-with-form-perturbation} of Corollary \ref{oneconeoneplf}.  
Recall that if a univariate rational function $f(\epsilon)=p(\epsilon)/q(\epsilon)$ has Laurent series expansion 
$f(\epsilon)=\sum_{k=-m}^{\infty} a_k \epsilon^k,$ the residue is defined as $a_{-1}$.
Given a rational function $f(\epsilon)$ with a pole at~$\epsilon=0$ there are a variety of well-known techniques to extract the value of the residue. 
For example,    if $\epsilon=0$ is a simple pole ($m=1$), then $\Res_{\epsilon=0} (f)=\frac{p(0)}{q'(0)}$. 
Otherwise, when $\epsilon=0$ is a pole of order $m>1$, we can write $f(\epsilon)=\frac{p(\epsilon)}{\epsilon^m q_1(\epsilon)}.$
Then  expand $p,q_1$ in powers of $\epsilon$  with $p(\epsilon)=a_0+a_1\epsilon+a_2\epsilon^2+\dots$ and $q_1(\epsilon)=b_0+b_1\epsilon+b_2\epsilon^2+\dots$.
This way the Taylor expansion of $p(\epsilon)/q_1(\epsilon)$ at $\epsilon_0$ is $c_0+c_1\epsilon+c_2\epsilon^2+c_3\epsilon^3+\dots$, where
$ c_0=\frac{a_0}{b_0}$, and $c_k=\frac{1}{b_0}(a_k-b_1c_{k-1}-b_2c_{k-2}-\dots-b_kc_0)$. Thus  we recover the residue  
$\Res_{\epsilon=0}(f)=c_{m-1}$. We must stress that the special structure of the rational functions in Corollary \ref{oneconeoneplf} can be exploited to speed up computation further rather than using this general methodology. Summarizing the above discussion results in the next theorem.

\begin{theorem}
\label{thm:both-parts-integrate-plf}
Let $\PP$ be a full dimensional polytope with vertex set $V(\PP)$, and $C_s$ be the cone of feasible directions at vertex $s$. Let $\mathcal D_s$ be a triangulation of $C_s$. Then
\begin{enumerate}
\item if $\ell$ is regular, meaning for every vertex $s$ of $\PP$ and every ray $u$ in $\mathcal{D}_s$ we have $\ll \ell, u \rr \neq 0$, then 

\[\int_\PP \ll \ell, x \rr^M \d x = \sum_{s \in V(P)} \sum_{C \in \mathcal{D}_s} \frac{M!}{(M+d)!} \vol(\Pi_C) \frac{(\langle \ell, s \rangle)^{M+d}}{\prod_{i=1}^d  \langle -\ell, u_i \rangle}, \]
\item otherwise, pick an $a \in \R^d$ that is regular on $\PP$ and then

\[\int_\PP \ll \ell, x \rr^M \d x = \sum_{s \in V(P)} \sum_{C \in \mathcal{D}_s} \frac{M!}{(M+d)!} \vol(\Pi_C) \Res_{\epsilon=0}\frac{(\langle \ell + \epsilon \cdot a, s \rangle)^{M+d}}{\epsilon \prod_{i=1}^d  \langle - \ell - \epsilon \cdot a, u_i \rangle}. \]
\end{enumerate}

In both cases, when the dimension $d$ is fixed, the integral can be computed in polynomial time in the usual binary encoding of $\PP$, $\ell$, and unary encoding of $M$. 

\end{theorem}

\begin{proof}
The only thing left to show is the statement about the complexity. Because the dimension is fixed, there are polynomial many cones $C_s$ and simplicial cones $\mathcal{D}_s$. Hence the only thing that has to be shown is that the residue can be computed in polynomial time. We will do this by multiplying truncated power series. For a fixed cone $C \in \mathcal{D}_s$ with rays $\{u_1, \dots, u_d\}$, let $I_1 := \{ i \mid \ll \ell, u_i \rr = 0\}$ and $I_2 := \{ i \mid \ll \ell, u_i \rr  \neq 0\}.$ Then 

\[\Res_{\epsilon=0}\frac{(\langle \ell + \epsilon \cdot a, s \rangle)^{M+d}}{\epsilon \prod_{i=1}^d  \langle - \ell - \epsilon \cdot a, u_i \rangle} = 
\frac{(\langle \ell + \epsilon \cdot a, s \rangle)^{M+d}}{\epsilon} \cdot 
\frac{1}{\prod_{i \in I_1} -\epsilon \ll a, u_i \rr} \cdot
\frac{1}{\prod_{i \in I_2} \ll -\ell, u_i\rr + \epsilon \ll -a, u_i \rr}
\]

Let $h_i(\epsilon)$ be the series expansion of $\frac{1}{\ll -\ell, u_i\rr + \epsilon\ll -a, u_i \rr}$ about $\epsilon=0$ up to degree $|I_1|$ for each $i \in I_2$, which can be done via the generalized binomial theorem:

\[h_i(\epsilon) := \sum_{j = 0}^{|I_1|} (-1)^j (\epsilon\ll -a, u_i \rr)^j(\ll -\ell, u_i\rr)^{-1-j}. \]

Let $h_0(\epsilon)$ be the expansion of $(\langle \ell + \epsilon \cdot a, s \rangle)^{M+d}$. Let $H := h_0(\epsilon) \cdot \prod_{i \in I_2} h_i(\epsilon)$ be the resulting polynomial product in $\epsilon$ truncated at degree $|I_1|$. This is can be done in polynomial time by  Lemma \ref{lemma:poly-mult}. Then the desired residue is simply the coefficient of $\epsilon^{|I_1|}$ in $H$ times the coefficient $\prod_{i \in I_1} \frac{1}{-\ll a, u_i \rr}$. 
\end{proof}

Algorithm \ref{alg:integratePolytopeTangentCone} explicitly states the polynomial time algorithm for integrating a polynomial over a polytope by decomposing the polynomial into powers of linear forms, and by decomposing the polytope into simplicial cones. Each main step---decomposing a polynomial into a sum of powers of linear forms, decomposing $\PP$ into simple cones, and computing the integral over each cone---is done in polynomial time when $d$ is fixed.

\begin{algorithm}                      
\caption{ Integrate by decomposing a polynomial into a sum of powers of linear forms and by triangulating a polytope's tangent cones}
\label{alg:integratePolytopeTangentCone}
\begin{algorithmic}                    
\REQUIRE $f(x) \in \Q[x_1, \dots, x_d]$, $d$-dimensional polytope $\PP \subset \R^d$
\ENSURE $\int_\PP f(x) \d x$
\STATE Use Equation \eqref{eq:decomp-powerlinform} to decompose each monomial in $f$ into a sum of powers of linear forms. 
\STATE Let $F$ be the resulting collection of linear forms. 
\STATE Let $T = \emptyset$
\FORALL{ $s \in V(P)$}
	\STATE Let $\mathcal{D}_s$ be a triangulation of $C_s$
	\FORALL{ $C \in \mathcal{D}_s$}
		\STATE $T \gets T \cup \{s + C \}$
	\ENDFOR
\ENDFOR
\STATE Pick $a\in \R^d$ so that for each $s + C \in T$, $\ll a, u_i\rr \neq 0$, where the $u_i$ are the rays of cone $C$
\STATE $\mathtt{sum} \gets 0$
\FORALL{linear forms $c \langle \ell, x \rangle^M$ in $F$}
	\FORALL{$ s + C \in T$}
		\IF{$\ell$ is regular on $C$}
			\STATE $\mathtt{sum} \gets \mathtt{sum} + c \cdot L^M(s+C, \ell)$, computed using Corollary \ref{oneconeoneplf}
		\ELSE
			\STATE $\mathtt{sum} \gets \mathtt{sum} + c \cdot L^M(s+C, \ell)$, by computing a residue as outlined in Theorem \ref{thm:both-parts-integrate-plf}
		\ENDIF
	\ENDFOR
\ENDFOR
\RETURN $\mathtt{sum} $
\end{algorithmic}
\end{algorithm}

\subsection{Domain: a full dimensional simplex}
\label{sec:intplf:domain:simplex}

Suppose now  that $\Delta \subset \R^d$ is a $d$-dimensional simplex with vertices $s_1, s_2, \dots, s_{d+1}$ and $\ell \in \R^d$.  We say that $\ell$ is
\emph{regular} for the simplex~$\Delta$ if it is not orthogonal to any of the edges of
the simplex, meaning $\ll \ell, s_i - s_j\rr \neq 0$ for every $i \neq j$.

\begin{theorem}[Corollary 12 and 13 in \cite{baldoni-berline-deloera-koeppe-vergne:integration}]
\label{thm:integration-simplex-plf}
  Let $\Delta \subset \R $ be a d-simplex with vertices $s_1, \dots, s_{d+1}$, and let $\ell \in \R$.
  
\begin{enumerate}
\item If $\ell$ is regular on $\Delta$, i.e., $\langle
  \ell, s_i\rangle \neq \langle \ell, s_j \rangle$ for any
  pair $i\neq j$, then we have the following relation.
  
  \begin{equation*}
  L^M(\Delta, \ell)
  =\int_\Delta {\langle \ell, x \rangle}^M \d x = d!\vol(\Delta, \d m)\frac{M!}{(M+d)!}
  \Big(\sum_{i=1}^{d+1}\frac{ \langle \ell , s_i
    \rangle^{M+d}}{\prod_{j\neq i} \langle \ell, s_i- s_j \rangle}\Big).
 \end{equation*}
  
\item If $\ell$ is not regular on $\Delta$, let $K\subseteq\{1,\dots,d+1\}$ be an index set of the different poles
$\langle \ell ,s_k\rangle$, and for $k\in K$ let $m_k$ denote the order of the pole, i.e.,
\begin{displaymath}
  m_k = \#\bigl\{\, i\in\{1,\dots,d+1\} : \langle \ell ,s_i\rangle = \langle \ell ,s_k\rangle \,\bigr\}. 
\end{displaymath}

Then we have the following relation.
\begin{equation*}
  %%L^M(\Delta)(\ell)=
  \int_{\Delta} {\langle \ell, x \rangle}^M  \d x =
 d!\vol(\Delta, \d m) \frac{M!}{(M+d)!}\sum_{k\in K} \Res_{\epsilon=0} \frac{(\epsilon + \langle
    \ell,  s_k \rangle)^{M+d}}
  {\epsilon^{m_k} {\prod\limits_{\substack{i\in K\\ i\neq k}} {(\epsilon +
      \langle \ell,  s_k-s_i\rangle )}^{m_i}} }.
\end{equation*}

\end{enumerate}
\end{theorem}

These evaluations of $L^M(\Delta, \ell)$ allows us to show that integrating a power of a linear form over a simplex can be done efficiently. The next theorem is a simplified statement of Theorem 2 in \cite{baldoni-berline-deloera-koeppe-vergne:integration} and the alternative proof immediately gives itself to an algorithmic implementation. 

\begin{theorem}
\label{thm:my-proof-plf-simplex}
Fix the dimension $d$. Then there is a polynomial time algorithm for computing $L^M(\Delta, \ell)$ when both $\Delta$ and $\ell$ are encoded in binary and $M \in \N$ is encoded in unary. 
\end{theorem}
\begin{proof}
When $\ell$ is regular on $\Delta$, Theorem \ref{thm:integration-simplex-plf} makes this clear. 

When $\ell$ is not regular, at most $|K| < d$ residues must be computed. Each residue can be computed by multiplying univariate polynomials and truncating. For $k \in K$, the coefficient of $\epsilon^{m_k-1}$ needs to be computed in the series expansion of  

\[
\frac{(\epsilon + \langle
    \ell,  s_k \rangle)^{M+d}}
  { {\prod\limits_{\substack{i\in K\\ i\neq k}} {(\epsilon +
      \langle \ell,  s_k-s_i\rangle )}^{m_i}} } \text{ for each } k \in K.\] 

This can be done in polynomial time by applying Lemma \ref{lemma:poly-mult}. For a fixed $k$, let $h_i(\epsilon)$ be the series expansion of $1/(\epsilon + \ll \ell, s_k - s_i\rr)^{m_i}$ about $\epsilon = 0$ up to degree $m_k -1$ in $\epsilon$ for each $i \in K$ and $i\neq k$. This can be done using the generalized binomial theorem. Also let $h_0(\epsilon)$ be the polynomial expansion of $(\epsilon + \ll \ell, s_k \rr)^{M+d}$. The product \[h_0(\epsilon) \prod\limits_{\substack{i\in K\\ i\neq k}} h_i(\epsilon)\] truncated at degree $m_k -1$ in $\epsilon$ can be computed in polynomial time using Lemma \ref{lemma:poly-mult}. The residue is then the coefficient of $\epsilon^{m_k-1}$ in the truncated product.
\end{proof}

Algorithm \ref{alg:integratePLFPolytopeTriangulation} summarizes how to integrate a polynomial via decomposing it into a sum of powers of linear forms and by using integration formulas for simplices. Notice that the algorithm runs in polynomial time when the dimension is fixed. This is because for a polynomial of degree $D$, Equation \eqref{eq:decomp-powerlinform} produces at most $\binom{D+d}{d}\binom{D+d}{d} = O(D^{2d})$ linear forms, and each linear form can be integrated in polynomial time in $D$. Also, the number of simplices in a triangulation of a polytope is polynomial in fixed dimension.

\begin{algorithm}                      
\caption{ Integrate by decomposing a polynomial into a sum of powers of linear forms and by triangulating a polytope}
\label{alg:integratePLFPolytopeTriangulation}
\begin{algorithmic}                    
\REQUIRE $f(x) \in \Q[x_1, \dots, x_d]$, d-dimensional polytope $\PP \subset \R^d$
\ENSURE $\int_\PP f(x) \d x$
\STATE Use Equation \eqref{eq:decomp-powerlinform} to decompose each monomial in $f$ into a sum of powers of linear forms. 
\STATE Let $F$ be the resulting collection of linear forms. 
\STATE Let $T$ be a list of simplices in a triangulation of $P$
\STATE $\mathtt{sum} = 0$
\FORALL{simplices $\Delta$ in $T$}
	\FORALL{linear forms $c \langle \ell, x \rangle^M$ in $F$}
		\STATE $\alpha \gets c \cdot \int_\Delta \ll \ell, x\rr^M \d x$, computed by using the method outlined in Theorem \ref{thm:my-proof-plf-simplex}
		\STATE $\mathtt{sum}  \gets \mathtt{sum}  + \alpha$
	\ENDFOR
\ENDFOR
\RETURN $\mathtt{sum} $
\end{algorithmic}
\end{algorithm}

\subsection{Examples}
\label{section-examples-of-each}
Before continuing, let's highlight the power of encoding integral values  by rational function identities. For regular linear forms the integration formulas are given by sums of rational functions which we read from the geometry at vertices and possibly a cone decomposition method. Consider a pentagon $P$ with vertices $(0,0), (2,0), (0,2), (3,1),$ and $(1, 3)$ as in Figure \ref{fig:pentagon}.

\begin{figure}
 	\centering
	\includegraphics[width=0.3\textwidth]{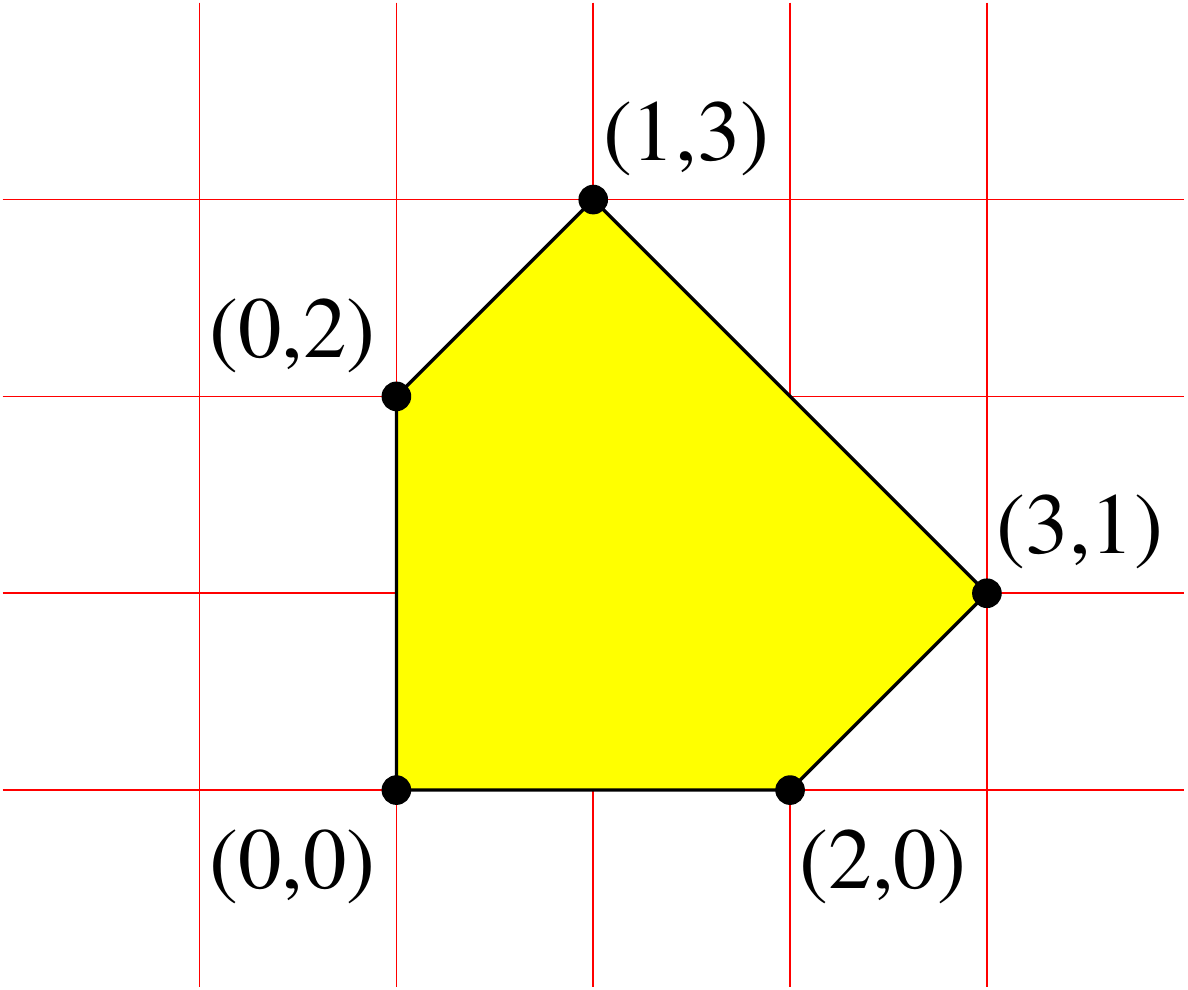}                
	\caption{A pentagon}
	\label{fig:pentagon}
\end{figure}

Then the rational function giving the value of $\int_{P} (c_1 x +c_2 y)^{M} \d{x} \d{y}$ by using the cone decomposition method is
\begin{displaymath}
	\frac{M!}{(M+2)!} \, \left( {\frac { \left( 2\,c_{{1}} \right) ^{M+2}}{c_{{1}} \left( 
-c_{{1}}-c_{{2}} \right) }}+4\,{\frac { \left( 3\,c_{{1}}+c_{{2}}
 \right) ^{M+2}}{ \left( c_{{1}}+c_{{2}} \right)  \left( 2\,c_{{1}}-2
\,c_{{2}} \right) }}+4\,{\frac { \left( c_{{1}}+3\,c_{{2}} \right) ^{M
+2}}{ \left( c_{{1}}+c_{{2}} \right)  \left( -2\,c_{{1}}+2\,c_{{2}}
 \right) }}+{\frac { \left( 2\,c_{{2}} \right) ^{M+2}}{ \left( -c_{{1}
}-c_{{2}} \right) c_{{2}}}} \right).
\end{displaymath}

This rational function expression encodes \emph{every integral} of the form $\int_{P} (c_1 x +c_2 y)^{M} \d{x} \d{y}$. For example, if we let $M=0$, then the integral is equal to the area of the pentagon, and the rational function simplifies to a number by simple high-school algebra:
\begin{displaymath}
\frac{1}{2} \left( 4\,{\frac {c_{{1}}}{-c_{{1}}-c_{{2}}}}+4\,{\frac { \left( 3\,c_{{1}}+c
_{{2}} \right) ^{2}}{ \left( c_{{1}}+c_{{2}} \right)  \left( 2\,c_{{1}
}-2\,c_{{2}} \right) }}+4\,{\frac { \left( c_{{1}}+3\,c_{{2}} \right) 
^{2}}{ \left( c_{{1}}+c_{{2}} \right)  \left( -2\,c_{{1}}+2\,c_{{2}}
 \right) }}+4\,{\frac {c_{{2}}}{-c_{{1}}-c_{{2}}}} \right)  = 6.
 \end{displaymath}

Hence the area is $6$. When $M$ and $(c_1,c_2)$ are given and $(c_1,c_2)$ is not perpendicular to any of the edge directions we can simply plug in numbers to the rational function. 
For instance, when $M=100$ and $(c_1=3,c_2=5)$ the answer is a fraction with numerator equal to
\begin{align*}
227276369386899663893588867403220233833167842959382265474194585 \\
3115019517044815807828554973991981183769557979672803164125396992
\end{align*}
and denominator equal to $1717$. When $(c_1,c_2)$ is perpendicular to
an edge direction, we encounter (removable) singularities in the rational functions,
thus using complex residues we can do the evaluation.  Note that those linear forms that 
are perpendicular to some edge direction form a measure zero set inside
a hyperplane arrangement.

Now we give an example of each method.

\subsubsection{Using the triangulation method}
Take the problem of integrating the polynomial $x+y$ over the triangle $\Delta$ with vertices $s_1=(1,1), s_2=(0,1),$ and $s_3=(1,0)$ in Figure~\ref{fig:example-triangle}.

\begin{figure}
    \centering
    \begin{subfigure}[b]{0.48\textwidth}
    		\centering
        \includegraphics[scale=1.5]{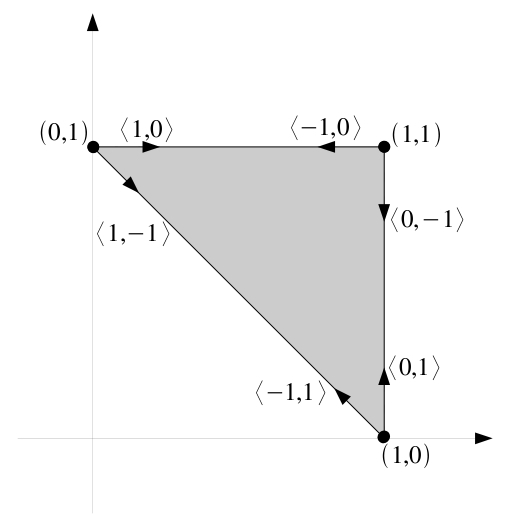}
        \caption{Triangulation method}
        \label{fig:example-triangle}
    \end{subfigure}
    ~ %\quad %add desired spacing between images, e. g. ~, \quad, \qquad, \hfill etc. 
    \begin{subfigure}[b]{0.48\textwidth}
    		\centering
        \includegraphics[scale=1.5]{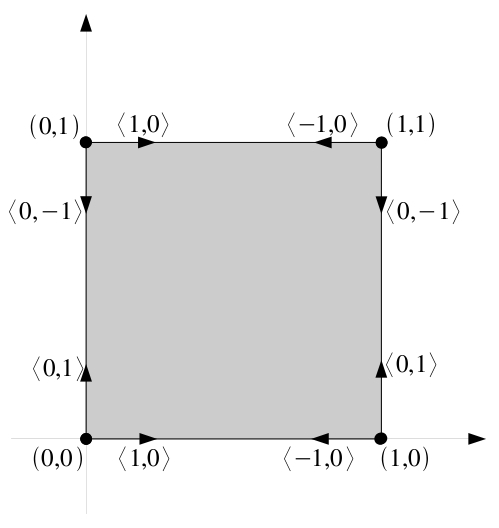}
        \caption{Cone decomposition method}
        \label{fig:example-square}
    \end{subfigure}

    \caption{Example polytopes}\label{fig:example-polytopes}
\end{figure}

The polynomial is already a power of a linear form, and the polytope is a simplex. Because $\ell = (1,1)$ is not regular (it is perpendicular to the edge spanned by $s_2$ and $s_3$), we have to build the index set $K$. Note $\langle \ell, s_1 \rangle = 2, \langle \ell, s_2 \rangle = 1,$ and  $\langle \ell, s_3 \rangle = 1$; pick $K = \{1,2\}$ with $m_1 = 1, m_2 = 2$. We proceed below with this choice, but note that we have a choice in picking the indices and we could have instead $K=\{1,3\}$. This would yield a different decomposition of the generating function. Note also that the decomposition of the power of a linear form is not necessarily unique either.
We now need to compute two values:

Vertex $s_1$: We are not dividing by zero, we can simply plug vectors into Corollary~\ref{oneconeoneplf}, \[\frac{\langle \ell, s_1 \rangle ^ 3}{\langle \ell, s_1  - s_2 \rangle ^ 2} = 8.\]

Vertex $s_2$: Here, we need to compute a residue. 

\begin{displaymath}
 \Res_{\epsilon=0}\frac{(\epsilon + \langle \ell, s_2 \rangle)^{1+2}}{\epsilon^2 (\epsilon + \langle \ell, s_2 - s_1 \rangle)} = \Res_{\epsilon=0}\frac{(\epsilon +1)^{1+2}}{\epsilon^2 (\epsilon -1)} = -4.
\end{displaymath}

Finally, $\int_{\Delta} (x+y) \d{x}\d{y} = 2! \times \frac{1}{2} \times \frac{1!}{3!} ( 8 - 4) = 2/3$.

\subsubsection{Using the cone decomposition method}
Next, integrate the polynomial $x$ over the unit square in Figure~\ref{fig:example-square} using the cone decomposition algorithm. Let $s_1 = (0,1)$, $s_2 = (0,0)$, $s_3 = (1,0)$, and $s_4 = (1,1)$. The polynomial is already a power of a linear form, so $\ell = (1,0)$. The polytope has four vertices that we need to consider, and each tangent cone is already simplicial. The linear form 
$\ell$ is not regular at any vertex. We let the reader verify that the residue-based calculation gives the value zero for the integrals on the corresponding cones at vertex $s_1$ and $s_2$. We only do in detail the  calculation for  vertex $s_3 = (1,0)$. At this vertex, the rays are $u_1 = (0,1), u_2=(-1,0)$. Because $\langle \ell, u_1 \rangle = 0$, we need a perturbation vector $\hat \epsilon$ so that when $\ell := \ell + \hat \epsilon$, we do not divide by zero on any cone (we have to check this cone and the next one). Pick $\hat \epsilon = (\epsilon, \epsilon)$. Then the integral on this cone is
	
\begin{displaymath}
 \frac{M!}{(M+d)!} \vol(\Pi_C) \Res_{\epsilon=0}\frac{(1 + \epsilon)^{1+2}}{\epsilon (-\epsilon) (1 + \epsilon)} = \frac{1!}{(1+2)!} \times 1 \times -2 = -2/6.
\end{displaymath}

Vertex $s_4 = (1,1)$: The rays are $u_1 = (-1,0), u_2= (0,-1)$. Again, we divide by zero, so we perturb $\ell$ by the same $\hat \epsilon$. The integral on this cone  is
\begin{displaymath}
 \frac{M!}{(M+d)!} \vol(\Pi_C) \Res_{\epsilon=0}\frac{(1 + 2\epsilon)^{1+2}}{\epsilon (\epsilon) (1 + \epsilon)} = \frac{1!}{(1+2)!} \times 1 \times 5 = 5/6.
\end{displaymath}

The integral $\int_P x  \d{x}\d{y} = 0 + 0 -2/6 + 5/6 = 1/2$ as it should be.

\subsection{How the software works: burst tries}
\label{section-software}

We implemented the two algorithms for integrating a polynomial (via a decomposition into powers of linear forms) over a polytope in C++ as part of the software package \latteInt \cite{latteintegrale}. Originally \latte was  developed in 2001 as software to study lattice points of convex polytopes~\cite{latte1}. The algorithms used combinations of geometric and symbolic computation. Two key data structures are rational generating functions and cone decompositions, and it was the first ever implementation of Barvinok's algorithm. \latte was improved in 2007 with various software and theoretical modifications, which increased speed dramatically.  This version was released under the name \texttt{LattE macchiato}; see \cite{koeppe:irrational-barvinok}. In 2011, \latteInt was released which included the computation of exact integrals of polynomial functions over convex polyhedra. The new integration functions are C++ implementations of the algorithms provided in \cite{baldoni-berline-deloera-koeppe-vergne:integration} with additional technical improvements. A key distinction between \latteInt and other software tools is that that the exact value of the integrals are computed since the implementation uses exact rational arithmetic. The code of this software is freely available at \cite{latteintegrale} under the GNU license.

As we see from the algorithmic descriptions above, one big step is performing manipulation of truncated power series. In this section, we discuss a data structure that improved our running time for multiplying polynomials. We focus here on how to store and multiply polynomials. For more on how on how we implemented the integration algorithms for powers of linear forms, see \cite{deloera:software-exact-integration-polynomials} and the user manual from the website \cite{latteintegrale}.

Our initial implementation used a pair of linked lists for polynomials and sums of powers of linear forms. For polynomials, each node in the lists contained a monomial. The problem with this data structure was that it was too slow  because it lacked any sort of ordering, meaning we had to traverse over every term in the list to find out if there was a duplicate. Hence when multiplying two polynomials, each resulting monomial in the product required a full list transversal in the inner for-loop of Algorithm \ref{alg:demo-poly-mult}. In \cite{baldoni-berline-deloera-koeppe-vergne:integration}, the integration over simplices was first implemented in \maple, and so there was no control over the data structures used to store the data. 
\begin{algorithm}                      
\caption{ Multiplying two polynomials}
\label{alg:demo-poly-mult}
\begin{algorithmic}                    
\REQUIRE $p_1, p_2 \in \Q[x_1, \dots, x_d]$
\ENSURE $p_1\cdot p_2$
\STATE $p \gets 0$
\FORALL{monomials $m_1$ of $p_1$}
\FORALL{monomials $m_2$ of $p_2$}
	\STATE \texttt{INSERT}$(p, m_1\cdot m_2)$
\ENDFOR
\ENDFOR
\RETURN $p$
\end{algorithmic}
\end{algorithm}

We then switched to using  \emph{burst tries}, a data structure designed to have \emph{cache-efficient storage and search}, due to the fact that they are prefix trees with sorted arrays of stored elements as leaves~\cite{trip-burst-tries-gastineaau-2006}. Such a data structure is performance-critical when computing residues, as a comparison with a linked-list implementation showed. In our implementation, each node corresponds to a particular  dimension or variable and contains the maximal and minimal values of the  exponent on this dimension.  The node either points to another node a level deeper in the tree or to a list of sorted elements.

Figure~\ref{fig:burst-trie} presents a simplified example from \cite{trip-burst-tries-gastineaau-2006} on how burst tries work. We want to save polynomials in the from $x^iy^jz^k$, let $R$ be the root of the tree. $R$ contains a list of ranges for the power on the $x$ term and points to monomials with an $x$ term that has a power in this range. For example, the 2nd column of $R$ points to all monomials with the degree of the $x$ term greater or equal to 1 and less than 2. Thus $R_1$ contains all monomials in the from $x^1y^jz^k$. $R_1$ contains one list per monomial where only the powers of $y$ and $z$ are saved with trailing zeros removed and the coefficient is at the end of the list. In this example, $R_1$ contains $5x^1y^1z^0 + 6x^1y^1z^1$.

\begin{figure}[htb]
\begin{center}
\leavevmode
\includegraphics[width=0.7\textwidth]{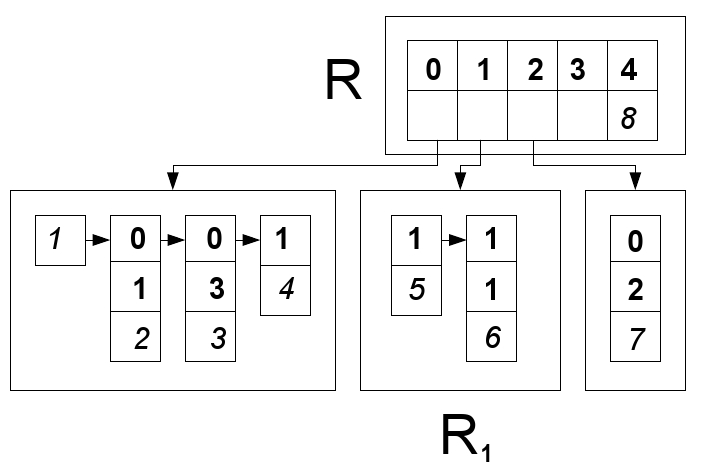}
\end{center}
\caption{Burst trie holding $1+2z + 3z^3 + 4y + 5xy + 6xyz +7x^2z^2 + 8x^4$ with a threshold of 5}
\label{fig:burst-trie}
\end{figure}

Whenever a list becomes too large (we imposed a limit of 10 elements per list), it splits on the first variable where the maximal and minimal exponents differ.  This process is called ``bursting'' and doing so ensures that finding potential duplicates is a very fast procedure, while making traversal of the entire tree take negligibly more time. For example, add $9x + 10xy^4 + 11xy^3z$ to the burst trie in Figure~\ref{fig:burst-trie}. All of these new elements are added to $R_1$. Now add $12xy^3$. This element is added to $R_1$, which now has $6$ elements, but the threshold is $5$ and so $R_1$ is bursted to $\tilde R_1$, see Figure~\ref{fig:burst-trie-bursting}. We must find the first variable in the list $y, z$ that has different exponents in $R_1$. This turns out to be $y$. Then $\tilde R_1$ now contains the exponents of $y$ from $0$ to $5$. $R_2$ now contains all monomials in the form $xyz^k$ and $R_3$ contains all the monomials in the form $xy^3z^k$. See \cite{trip-burst-tries-gastineaau-2006} for a complete introduction. 

\begin{figure}[htb]
\begin{center}
\leavevmode
\includegraphics[width=0.7\textwidth]{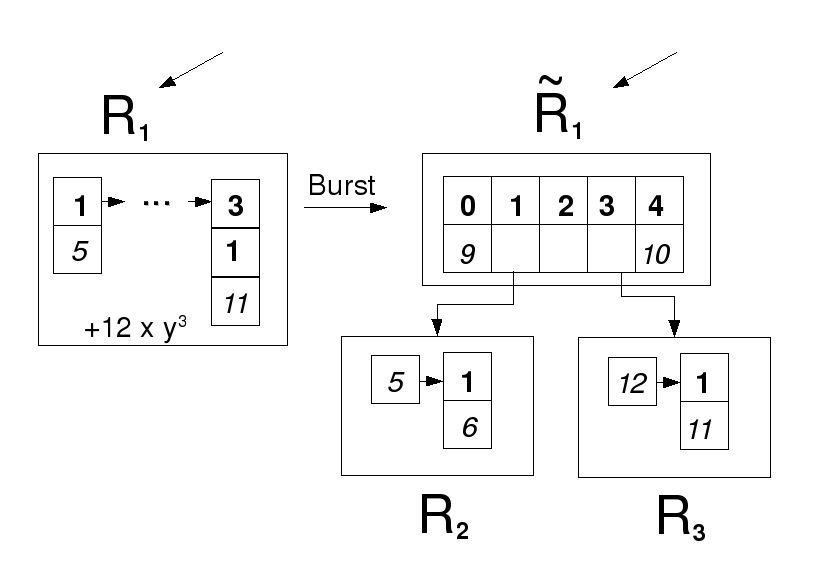}
\end{center}
\caption{$R_1$ is bursted after adding $6$ elements. $\tilde R_1$ holds $5xy + 6xyz + 9x + 10xy^4 + 11xy^3z + 12xy^3$}
\label{fig:burst-trie-bursting}
\end{figure}

\subsection{Should one triangulate or cone decompose?}
\label{section-triangulate-or-decompose}

We have developed two polynomial time algorithms in Sections \ref{sec:intplf:domain:cone} and \ref{sec:intplf:domain:simplex} for integrating a polynomial over a polytope. The two methods differ on how the domain $\PP$ is decomposed. One could triangulate the whole polytope, or integrate over each tangent cone.  However, each cone must be decomposed into simplicial cones. This is the trade-off: do one (large and possible costly) polytope triangulation, or many smaller cone triangulations. The number of simplices in a triangulation and the number of simplicial cones in a polytope decomposition can significantly differ. Depending on the polytope, choosing the right method can determine its practicality.

We highlight here a fraction of experiments performed in \cite{deloera:software-exact-integration-polynomials, extratables}. We describe two experiments. First, we compare the two integration algorithms on random polynomials and random polytopes, and second we compare the algorithms on other numerical software.

Our experimental results agree with
\cite{bueler-enge-fukuda-2000:exact-volume} in showing that triangulating the
polytope is better for polytopes that are ``almost simplicial'' while cone
decomposition is faster for simple polytopes.

\subsubsection{Integration over random polytopes}

For each $v \in \{9, 10, 11, 12, 	17, 22, 27, 32\}$, we created 50 random polytopes in dimension $7$ by taking the convex hull of random points using {\tt Polymake} \cite{polymake-software} and ensuring the resulting polytope had $v$ vertices. We call these the primal polytopes.  When zero is not in the interior of the polytope, we translated the centroid to the origin before constructing the dual polytope. Because of the construction method, most primal polytopes are simplicial and the duals are mostly simple polytopes. For each degree $D \in \{1, 2, 5, 10, 20, 30, 40, 50 \}$, we also constructed a set of $50$ random monomials in $7$ variables of degree $D$. Then each polytope was integrated over by one random monomial in degree $D$. Meaning, for each polytope class with $v$ vertices, and for each $D$, 50 integration tests were performed. The same test was then performed with the dual polytopes. 

We only report those tests where both the triangulation and cone-decomposition method finished under 600 seconds. We define the \emph{relative time difference} as the time taken by the triangulation method \emph{minus} the time taken by the cone-decomposition method, all divided by the time of the triangulation method. Note that when the triangulation method is faster we obtain a negative number. We will use this quantity throughout.

Figure \ref{fig:integration-random-graphs5} displays a histogram on three axes. The first horizontal axis is the relative time difference between the two integration methods. The second horizontal axis shows the degrees of monomials and finally the vertical axis presents the number of random polytopes in dimension $7$. The height of a particular solid bar in position $(a_k,b^*)$ tallies the number of random polytopes for which the relative time difference between the two algorithms, when integrating a monomial of degree $b^*$, was between $a_{k-1}$ and $a_k$ with $a_k$ included in that bar. Thus, the bars with negative relative time difference should be counted as experiments where triangulation is faster. 

There is one histogram for the primal polytopes and one for the dual polytopes. The row color corresponds to a degree class.
Note that for the primal polytopes, which are simplicial, the triangulation method is faster than the cone decomposition method. In contrast, the cone decomposition method is slightly better for the simple polytopes. More tables are available in \cite{deloera:software-exact-integration-polynomials, extratables}.

Our experiments on integrating monomials have the same qualitative behavior as those of \cite{bueler-enge-fukuda-2000:exact-volume} for volume computation (polynomial of degree zero): the triangulation 
method is faster for simplicial polytopes (mass on histograms is highly concentrated on negative relative time differences)  while the cone 
decomposition is faster for simple polytopes (mass on histograms is concentrated on positive relative time differences).

\begin{figure}
    \centering
    \begin{subfigure}[b]{1.0\textwidth}
    		\centering
        \includegraphics[scale=0.75]{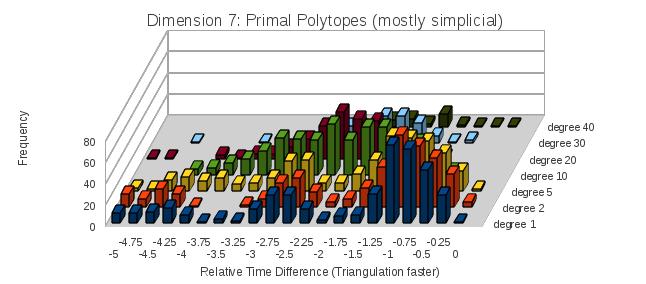}
        \label{fig:ip7}
    \end{subfigure}
    
    %\quad %add desired spacing between images, e. g. ~, \quad, \qquad, \hfill etc. 
    \begin{subfigure}[b]{1.0\textwidth}
    		\centering
        \includegraphics[scale=0.75]{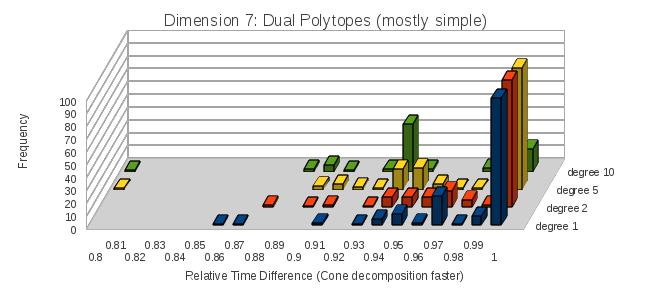}
        \label{fig:id7}
    \end{subfigure}

    \caption{Histogram of the relative time difference between the triangulation and cone-decomposition methods for integrating  over random polytopes in dimension 7}\label{fig:integration-random-graphs5}
\end{figure}

\subsubsection{Numerical methods}

There are two general classes of algorithms for finding integrals over polytopes: numerical and exact. Numerical algorithms approximate the valuation on the polytope and involve error bounds, whereas exact algorithms do not contain a theoretical error term. However, exact algorithms may contain errors when they use finite digit integers or use floating-point arithmetic. In order to sidestep this problem, \latteInt uses NTL's arbitrary length integer and rational arithmetic \cite{ntl-5.4} compiled with the GNU Multiple Precision Arithmetic Library \cite{gmp-4.1.4}. The obvious downside to exact arithmetic is speed, but this cost is necessary to obtain exact answers. In this section, we compare our exact algorithms with \cubpack, a Fortran~90 library which estimates the integral of a function (or vector of functions) over a collection of $d$-dimensional hyper-rectangles and simplices \cite{CUBPACK}. This comparison is very interesting because \cubpack uses an adaptive grid to seek better performance and accuracy. 

All integration tests with \cubpack in dimension $d$ were done with a
product of linear forms with a constant term over a random
$d$-dimensional simplex where the absolute value of any coordinate in
any vertex does not exceed 10. For example, we integrated a product of
inhomogeneous linear forms such as $(\frac{1}{5} + 2x -
\frac{37}{100}y)(2 - 5x)$ over the simplex with vertices $(10,0),
(9,9), (1,1)$ and $(0,0$.  In Table~\ref{tabel:CUBPACK-time-table}, \latte was
run 100 times to get the average running time, while \cubpack was run
1000 times due to variance.  Both the dimension and number of linear
forms multiplied to construct the integrand were varied.

\begin{table}
\centering
\caption{Average Time for \latteInt and \cubpack for integrating products of inhomogeneous linear forms over simplices.} %latte on top vs cube on bottom
\label{tabel:CUBPACK-time-table}
\tabcolsep 3pt
\footnotesize
%\small
\begin{tabular}{lcrrrrrrrrrrr}
\toprule 
 & &\multicolumn{10}{c}{Number of linear factors} \\ 
\cmidrule(r){3-12} 
$d$ & Tool & \multicolumn{1}{c}{1}  & \multicolumn{1}{c}{2}  & \multicolumn{1}{c}{3}  & \multicolumn{1}{c}{4}  & \multicolumn{1}{c}{5}  & \multicolumn{1}{c}{6}  & \multicolumn{1}{c}{7}  & \multicolumn{1}{c}{8} & \multicolumn{1}{c}{9} & \multicolumn{1}{c}{10} \\ 
\hline 
\multirow{2}{*}{2} & \latte   & \timeFaster{0.0001} & \timeFaster{0.0002} & \timeFaster{0.0005} & \timeFaster{0.0008} & \timeFaster{0.0009} & \timeFaster{0.0019} & \timeSlower{0.0038} & \timeSlower{0.0048} & \timeSlower{0.0058} & \timeFaster{0.0089} \\
                   & \cubpack & \timeSlower{0.0027} & \timeSlower{0.0014} & \timeSlower{0.0016} & \timeSlower{0.0022} & \timeSlower{0.0064} & \timeSlower{0.0052} & \timeFaster{0.0014} & \timeFaster{0.0002} & \timeFaster{0.0026} & \timeSlower{0.0213} \\
\hline 
\multirow{2}{*}{3} & \latte   & \timeFaster{0.0002} & \timeFaster{0.0005} & \timeFaster{0.0009} & \timeFaster{0.0016} & \timeFaster{0.0043} & \timeFaster{0.0073} & \timeFaster{0.0144} & \timeFaster{0.0266} & \timeFaster{0.0453} & \timeFaster{0.0748} \\
                   & \cubpack & \timeSlower{0.0134} & \timeSlower{0.0145} & \timeSlower{0.0018} & \timeSlower{0.0054} & \timeSlower{0.0234} & \timeSlower{0.0219} & \timeSlower{0.0445} & \timeSlower{0.0699} & \timeSlower{0.1170} & \timeSlower{0.2420} \\
\hline 
\multirow{2}{*}{4} & \latte   & \timeFaster{0.0003} & \timeFaster{0.0012} & \timeFaster{0.0018} & \timeSlower{0.0044} & \timeSlower{0.0121} & \timeFaster{0.0274} & \timeFaster{0.0569} & \timeFaster{0.1094} & \timeFaster{0.2247} & \timeFaster{0.4171} \\
                   & \cubpack & \timeSlower{0.0042} & \timeSlower{0.0134} & \timeSlower{0.0028} & \timeFaster{0.0019} & \timeFaster{0.0076} & \timeSlower{0.5788} & \timeSlower{4.7837} & \timeSlower{4.3778} & \timeSlower{22.3530} & \timeSlower{54.3878} \\
\hline 
\multirow{2}{*}{5} & \latte   & \timeFaster{0.0005} & \timeFaster{0.0008} & \timeTie{0.0048}    & \timeFaster{0.0108} & \timeSlower{0.0305} & \timeFaster{0.0780} & \timeFaster{0.0800} & \timeNotComputed{   0.00} & \timeNotComputed{   0.00} & \timeNotComputed{   0.00} \\
                   & \cubpack & \timeSlower{0.0013} & \timeSlower{0.0145} & \timeTie{0.0048}    & \timeSlower{0.0217} & \timeFaster{0.0027} & \timeSlower{37.0252} & \timeSlower{128.2242} & \timeNotComputed{   0.00} & \timeNotComputed{   0.00} & \timeNotComputed{   0.00} \\
\bottomrule 
\end{tabular} 
\end{table}

As shown in Table~\ref{tabel:CUBPACK-time-table}, \latteInt tends to take less time, especially when the number of forms and dimension increases.  The table does not show the high variance that \cubpack has in its run times.  For example, the 5-dimensional test case with 6 linear forms 
had a maximum running time of 2874.48 seconds, while the minimum running time was 0.05 seconds on a different random simplex.  This contrasted starkly with \latteInt, which had every test be within 0.01 (the minimum time discrepancy recognized by its timer) of every other test case.

\begin{table}
\centering
\caption{ \texttt{CUBPACK} scaling with increased relative accuracy. ``Relative Error'' is a user-specified parameter of \texttt{CUBPACK};  ``Expected Error'' is
an estimate of the absolute error, produced by \texttt{CUBPACK}'s error estimators. Finally, the ``Actual Error'' is the difference of \texttt{CUBPACK}'s result to the
exact integral computed with \latteInt.}
\label{tabel:cubpack-vs-latte}
\tabcolsep 3pt
\small
\begin{tabular}{crrrrr}
\toprule 
\multicolumn{1}{c}{Relative Error} & \multicolumn{1}{c}{Result}  & \multicolumn{1}{c}{Expected Error}  & \multicolumn{1}{c}{Actual Error}  & \multicolumn{1}{c}{\# Evaluations}  & \multicolumn{1}{c}{Time (s)}  \\ 
\midrule
$10^{-2}$ & 1260422511.762 & 9185366.414 & 94536.015 & 4467 & 0.00 \\
%\hline
$10^{-3}$ & 1260507955.807 & 1173478.333 & 9091.974 & 9820 & 0.01 \\
%\hline
$10^{-4}$ & 1260516650.281 & 123541.490 & 397.496 & 34411 & 0.04 \\
%\hline
$10^{-5}$ & 1260517042.311 & 12588.455 & 5.466 & 104330 & 0.10 \\
%\hline
$10^{-6}$ & 1260517047.653 & 1257.553 & 0.124 & 357917 & 0.31 \\
%\hline
$10^{-7}$ & 1260517047.691 & 126.042 & 0.086 & 1344826 & 1.16 \\
%\hline
$10^{-8}$ & 1260517047.775 & 12.601 & 0.002 & 4707078 & 4.15 \\
%\hline
$10^{-9}$ & 1260517047.777 & 1.260 & $<10^{-3}$ & 16224509 & 14.09 \\
%\hline
$10^{-10}$ & 1260517047.777 & 0.126 & $<10^{-3}$ & 55598639 & 48.73 \\
\bottomrule 
\end{tabular} 
\end{table}

\cubpack differs from \latteInt in that since it is based on numerical
approximations, one can ask for different levels of precision.
Table~\ref{tabel:cubpack-vs-latte} illustrates how \cubpack scales with
requested precision on a single, 4-dimensional, 10 linear form test case. It seems that \cubpack scales linearly with the inverse of the requested
precision---10 times the precision requires about 3 times the work. 
% Table~\ref{tabel:cubpack-vs-latte} seems to indicate that the best precision for the amount of time used is at approximately $10^{-6}$.
All reported tests were done by expanding the multiplication of linear forms,
and coding a Fortran~90 function to read in the resulting polynomial and
evaluate it for specific points.

\section{Integrand: products of affine functions}
\label{sec:integration-prducts-affine-functions}

In Section \ref{sec:integrand:power-of-linear-forms}, polynomial time algorithms for integrating a polynomial over a polytope was developed, when the dimension is fixed. The key step was using Equation \eqref{eq:decomp-powerlinform} to decompose a polynomial into a sum of powers of linear forms. That equation has a big problem: it can produce millions of linear forms for polynomials with modest degree, see Table \ref{tabel:count-linear-forms}. The large number of summands is our motivation for exploring an alternative decomposition.

In this section we focus on decomposing a polynomial into a sum of produces of affine functions and develop polynomial time algorithms for integrating over products of affine functions, similar to the last section. 

To build such a polynomial decomposition, we seek an application of the next theorem. 

\begin{theorem}[Handelman \cite{handelman1988}]\label{theoHan}
Assume that $g_1,\dots,g_n\in\R[x_1, \dots, x_d]$ are linear polynomials and that the semialgebraic set
\begin{equation}\label{eqK}
S=\{x\in\R^d \mid g_1(x)\geq 0,\dots,g_n(x)\geq 0\}
\end{equation}
is compact and has a non-empty interior. Then any polynomial $f\in\R[x_1, \dots, x_d]$ strictly positive on $S$ can be written as
$f(x) = \sum_{\alpha\in\N^n} c_\alpha g_1^{\alpha_1}\cdots g_n^{\alpha_n}$
for some nonnegative scalars $ c_{\alpha}$.
\end{theorem}

Note that this theorem is true when $S$ is a polytope $\PP$, and the polynomials $g_i(x)$ correspond to the rows in the constraint matrix $b - Ax \geq 0$. In the case of the hypercube $\PP=[0, 1]^d$, this result was shown  earlier by Krivine \cite{Krivine1964}. See \cite{Castle20091285, deKlerk2015, monique2014} for a nice introduction to the Handelman decomposition. The Handelman decomposition is only guaranteed to exist if the polynomial is strictly greater than zero on $\PP$, and the required degree of the Handelman decomposition can grow as the minimum of the polynomial approaches zero  \cite{sankaranarayanan2013lyapunov}. Not much is known if redundant inequalities in $\PP$'s description can help. 

Let the \emph{degree} of a Handelman decomposition be $\max |\alpha|$, where the maximum is taken over all the exponent vectors $\alpha$ of $g_i(x)$ that appear in a decomposition.  For a degree $t$ decomposition, the goal is to find $c_\alpha \geq 0$ and a $s \in \R$ such that 
\[f(x) + s = \sum_{\alpha\in \Z_{\geq 0}^n \;:\; |\alpha|\le t} c_{\alpha} g^\alpha.\]
Adding the unknown constant shift $s$ to $f(x)$ has three important consequences.

\begin{enumerate}
\item $-s$ will be a lower bound for $\fmin$. So if $f$ is negative on $\PP$, $f(x)+s$ will be positive on $\PP$, allowing us to use Theorem \ref{theoHan}.
\item If the minimum value of $f$ is zero, $f$ might not have a Handelman decomposition for any $t$. By adding a shift, a decomposition is guaranteed to exist for some $t$. 
\item If the minimum value of $f$ is positive, but small, $f(x)$ might only have Handelman decompositions for large degree $t$. By adding a shift, we can find a Handelman decomposition of smaller size.
\end{enumerate} 

If one uses a large shift $s$ so that $f(x)+s$ is positive on its domain, there no known general bound for how large the Handelman degree $t$ has to be for $f(x)+s$ to have a Handelman decomposition. Likewise, it is not known that by fixing the Handelman degree $t$, a shift $s$ can always be found so that $f(x)+s$ has a Handelman decomposition, but we have not run into this problem in our experiments. For a nice review of known Handelman degree bounds for some special cases see \cite{monique2014}.

If we expand the right hand side of $f(x) + s = \sum_{|\alpha| \leq t} c_\alpha g^\alpha$ into monomials, and force the coefficients of monomials on both sides of the equality to be equal, the results would be a linear system in the $c_\alpha$ and $s$. Hence we seek a solution to the \emph{linear program}
	\begin{align}
	\label{eq:sparse-lp-handelman}
	\min & \;s + \sum_\alpha c_\alpha\\ \nonumber
	& A_Hc_\alpha = b \\ 	\nonumber
	& a_0^Tc_\alpha -s = 0\\ \nonumber
	& s \text{ free}, c_\alpha \geq 0, 
	\end{align}
where the objective has been chosen so that $-s$ is close to $f_{\mathrm{min}}$ and to force a sparse Handelman decomposition of order $t$. It is common practice to use $\norm{\cdot}_1$ as a proxy for sparse solutions \cite{elad2010sparse}. Notice that $A_H$ has ${t+ n\choose n} = O(t^{n})$ columns if the number of facts $n$ of $\PP$ is bounded, and $A_H$ has ${t + d \choose d} = O(t^d)$ rows if the dimension $d$ is fixed. Therefore, if a polynomial had a degree $t$ Handelman decomposition, then the decomposition can be found in polynomial time in $t$ when $d$ and $n$ are fixed by solving a linear program.

Consider the example $f(x) = x^2 -x$ on $[-1,1] \subset \R$, which cannot have a Handelman decomposition as it takes negative values. 
We seek a solution to \[f(x) + s = c_{2,0} (x + 1)^2 + c_{1,1} (1 - x) (x + 1) + c_{0,2}(1 - x)^2 + c_{1,0} (x + 1) + c_{0,1}(1 - x).\]
Solving the linear program results in $c_{0,2} = 3/4$, $c_{2,0}=1/4$, and $s=1$.

Algorithm \ref{alg:integrate-using-handelman-affine-functions} illustrates how the Handelman decomposition can be used to integrate $f(x)$. 

\begin{algorithm}                      
\caption{ Integrate $f(x)$ by Handelman decomposition}
\label{alg:integrate-using-handelman-affine-functions}
\begin{algorithmic}                    
\REQUIRE A polynomial $f(x) \in \Q[x_1, \dots, x_d]$ of degree $D$ that has a degree $D$ Handelman representation, and $d$-dimensional polytope $\PP \subset \R^d$
\ENSURE $\int_\PP f(x) \d x$
\STATE Solve the linear program \eqref{eq:sparse-lp-handelman} and obtain \[f(x) + s = \sum_{\alpha_1 + \cdots + \alpha_n \leq D} c_\alpha g^\alpha\]

\STATE $\mathtt{sum} \gets \int_\PP \sum_{\alpha} c_\alpha g^\alpha \d x$, by methods in Section \ref{sec:intpaf:domain:cone} or \ref{sec:intpaf:domain:simplex}
\STATE $\mathtt{vol} \gets \vol(\PP)$, using any integration method where the power of the integrand is zero
\RETURN $\mathtt{sum} - s \cdot \mathtt{vol}$
\end{algorithmic}
\end{algorithm}

With a method of decomposing a polynomial into a sum of products of affine functions possible, we describe polynomial time algorithms for integrating such integrands in the next two sections. 

\subsection{Domain: cone decomposition}
\label{sec:intpaf:domain:cone}

The next proposition outlines a method for integrating a product of affine functions over a polytope, where the polytope is decomposed into cones like the method in Section \ref{sec:intplf:domain:cone}.

\begin{theorem}[proposition 5.6 in \cite{brandon-handelman-paper}]
\label{thm:continuous-affnie-products-linear-forms-cones}
When the dimension $d$ and number of factors $n$ is fixed, the value of 
\[\int_{\PP} \frac{(\langle \ell_1, x\rangle + r_1)^{m_1} \cdots \langle (\ell_n, x\rangle + r_n)^{m_n}}{m_1!\cdots m_n!} \d x\]
can be computed in polynomial time in $M:=\sum_{i=1}^nm_i$ and the size of the input data.
\end{theorem}

\begin{proof}
We will compute the polynomial

\[ \sum_{p_1+\cdots+p_n \leq M} \left( \int_{\PP} \frac{(\langle \ell_1, x\rangle + r_1)^{p_1} \cdots (\langle \ell_n, x\rangle + r_n)^{p_n}}{p_1!\cdots p_n!}  \d x \right) t_1^{p_1}\cdots t_n^{p_n}. \]
in polynomial time in $M$ and the input data, which contains the desired value as a coefficient of a polynomial in $t_1, \dots, t_n$.

We start with the exponential integral in the indeterminate $\ell$ from Section \ref{ch:bg:sec:integration-stuff}:

\[\int_{\PP} e^{\langle \ell, x\rangle} \d x
=
\sum_{s \in V(\PP)}\sum_{C \in D_s} \mathrm{vol}(\Pi_C) e^{\ll \ell, s\rr}  \prod_{i=1}^d \frac{1}{-\ll \ell, u_i\rr}.
\]

Note that because the dimension is fixed, the number of vertices and the number of simplicial cones at each feasible cone of $\PP$ is polynomial in the input size of $\PP$.

First pick $\ell_{n+1} \in \Q^d$ so that $\ll \ell_{n+1}, u \rr \neq 0$ for every ray $u$ in the simplicial cones $D_s$ at each vertex $s$. The set of points $\ell_{n+1}$ that fail this condition have measure zero, so $\ell_{n+1}$ can be picked randomly. Next, replace $\ell$ with $\ell_1t_1 + \cdots + \ell_{n+1}t_{n+1}$. To simplify notation, let $\ve t = (t_1, \dots, t_n)$, $\ve a_s = ( \ll \ell_1, s \rr, \dots,  \ll \ell_{n}, s \rr)$, $\ve r:= (r_1, \dots,  r_n)$, $\ve b_i := (\ll \ell_1, u_i \rr,  \dots, \ll \ell_{n}, u_i \rr)$, $\beta_i := \ll \ell_{n+1}, u_i \rr$, and $\boldsymbol{\ell_x} = (\ll \ell_1, x\rr, \dots, \ll \ell_{n}, x\rr)$. Then

\[\int_{\PP} e^{ \ll \boldsymbol{\ell_x}, \ve t\rr  + \ll \ve  r , \ve t \rr + \ll \ell_{n+1}, x\rr t_{n+1}}\d x
=
\sum_{s \in V(\PP)}\sum_{C \in D_s} \mathrm{vol}(\Pi_C)   \prod_{i=1}^d \frac{e^{\ll \ve a_s, \ve t \rr}e^{\ll \ve r, \ve t\rr}e^{\ll \ell_{n+1}, s \rr t_{n+1}}}{-\ll \ve b_i, \ve t \rr - \beta_it_{n+1}}.
\]

Note that the integrand of the left hand side is $e^{\ll \ell_{n+1}, x\rr t_{n+1}} \prod_{i=1}^n e^{(\ll \ell_i, x\rr + r_i)t_i}$, which when expanded in series form contains the desired result. We will compute the series expansion of each summand in the right hand side of the above equation in $t_1, \dots, t_n$ up to total degree $M$ where the power of $t_{n+1}$ is zero, and do this in polynomial time in $M$. 

Let $h_i$ be the series expansion of $e^{(\ll \ell_i, s\rr + r_i)t_i}$ up to degree $M$ in $t_i$ for $1 \leq i \leq n$. Let $h_{n + i}$ be the series expansion of 

\[\frac{1}{-\ll \ve b_i, \ve t \rr - \beta_it_{n+1}} = \sum_{k=0}^\infty (-1)^k (\ll \ve b_i, \ve t \rr)^k (\beta_{ij} t_{n+1})^{-1-k}. \]
in $t_i, \dots, t_n$ up to total degree $M$ using the generalized binomial  theorem. Applying Lemma \ref{lemma:poly-mult} to $H_1 := \prod_{i=1}^{n+d} h_i$ results in the series expansion of 

\[\prod_{i=1}^d \frac{e^{\ll \ve a_s, \ve t \rr}e^{\ll \ve r, \ve t\rr}}{-\ll \ve b_i, \ve t \rr - \beta_it_{n+1}}.
\]
up to total degree $M$ in $t_1, \dots, t_{n}$ and where the power of $t_{n+1}$ at most ranges from $-d(M+1)$ to $-d$. 

Let $h_{n+d+1}$ be the series expansion of $e^{\ll \ell_{n+1}, s \rr t_{n+1}}$ up to degree $d(M+1)$ in $t_{n+1}$. Next, use Lemma \ref{lemma:poly-mult} one last time while treating $t_{n+1}$ as a coefficient and truncating at total degree $M$ to compute $H_2 := H_1 h_{n+d+1}$. Any term where the power of $t_{n+1}$ is not zero can be dropped because $I(\PP, \ell_1t_1 + \cdots + \ell_{n+1}t_{n+1})$ is holomorphic in $t_{n+1}$. Repeating this calculation for every simplicial cone results in 

\[ \sum_{p_1+\cdots+p_n \leq M} \left( \int_{\PP} \frac{(\langle \ell_1, x\rangle + r_1)^{p_1} \cdots (\langle \ell_n, x\rangle + r_n)^{p_n}}{p_1!\cdots p_n!}  \d x \right) t_1^{p_1}\cdots t_n^{p_n}. \]
\end{proof}

\subsection{Domain: a full dimensional simplex}
\label{sec:intpaf:domain:simplex}

In \cite{baldoni-berline-deloera-koeppe-vergne:integration}, a polynomial time algorithm for integrating a product of linear forms was developed. In this section, we will extend the original proof to the slightly more general setting of integrating a product of affine functions. 

\begin{theorem}
\label{thm:continuous-affnie-products-linear-forms-simplex}
Let the dimension $d$ and number of factors $n$ be fixed. Let $\Delta \subset \R^d$ be a full dimensional simplex with vertices $s_1, \dots, s_{d+1}$, and let $\ell_1, \dots, \ell_n \in \R^d$. Then the value of 

\[\int_\Delta \frac{(\ll \ell_1, x\rr +r_1)^{m_1}\cdots (\ll \ell_n, x\rr +r_n)^{m_n}}{m_1!\cdots m_n!} \d x\]

can be computed in polynomial time in $M:=m_1 + \cdots + m_n$ and the usual input size.
\end{theorem}

\begin{proof}
Similar to the proof of \ref{thm:continuous-affnie-products-linear-forms-cones}, we instead compute the polynomial 
\[sum_{p_1 + \cdots p_n \leq M} \left(\int_\Delta \frac{(\ll \ell_1, x\rr +r_1)^{p_1}\cdots (\ll \ell_n, x\rr +r_n)^{p_n}}{p_1!\cdots p_n!} \d x \right) t_1^{p_1}\cdots t_n^{p_n}.
\]
We start with Lemma 8 in \cite{baldoni-berline-deloera-koeppe-vergne:integration} and write

\begin{equation}
\label{eq:lemma8HowToIntSimplex}
\int_\Delta e^{\ll \ell, x \rr}\d x = d! \vol(\Delta) \sum_{k\in \N^{d+1}} \frac{\ll \ell, s_1\rr^{k_1} \cdots \ll \ell, s_{d+1} \rr^{k_{d+1}} }{(|k| + d)!},
\end{equation}
where $|k|:=k_1 + \cdots k_{d+1}$. Replace $\ell$ with $\ell_1t_1 + \cdots + \ell_n t_n$ and multiply both sizes by $e^{r_1t_1 + \cdots +  r_nt_n}$ in Equation \eqref{eq:lemma8HowToIntSimplex}. 
To simplify notation, let  $\ve t = (t_1, \dots, t_n)$, $\ve r = (r_1, \dots, r_n)$, $\boldsymbol{\ell_{s_i}} = ( \ll \ell_1, s_i \rr, \dots, \ll \ell_n, s_i \rr)$, then we have

\begin{equation}
\label{eq:lemma8HowToIntSimplex2}
\int_\Delta e^{ \ll \ell_1, x \rr t_1 + \cdots + \ll \ell_n, x\rr t_n + \ll \ve r, \ve t\rr} \d x
= d! \vol(\Delta) e^{\ll \ve r, \ve t\rr} \sum_{k\in \N^{d+1}} \frac{\ll \boldsymbol{\ell_{s_1}}, \ve t \rr^{k_1} \cdots \ll \boldsymbol{\ell_{s_{d+1}}}, \ve t \rr^{k_{d+1}} }{(|k| + d)!}.
\end{equation} 
Notice that the left hand now becomes 

\[ \int_\Delta e^{(\ll \ell_1, x \rr + r_1)t_n + \cdots (\ll \ell_n, x \rr + r_n)t_n }\d x = \sum_{(p_1,\dots, p_n) \in \N^{n}} \left(\int_\Delta \frac{(\ll \ell_1, x\rr +r_1)^{p_1}\cdots (\ll \ell_n, x\rr +r_n)^{p_n}}{p_1!\cdots p_n!} \d x \right) t_1^{p_1}\cdots t_n^{p_n}, \]
which contains the desired polynomial. Hence we seek to compute the series expansion of the right hand side of Equation \eqref{eq:lemma8HowToIntSimplex2} up to total degree $M$ in $t_1, \dots, t_n$. 

Let $g_i(t_i)$ be the series expansion of $e^{r_it_i}$ about $t_i=0$ up to degree $M$ in $t_i$ for $1 \leq i \leq n$. Using Lemma \ref{lemma:poly-mult}, let $H_1$ be the product $d!\vol(\Delta)\cdot g_1(t_1) \cdots g_n(t_n)$ truncated at total degree $M$ in $t_1, \dots t_n$, which is done in polynomial time in $M$. 

Notice that the sum in Equation \eqref{eq:lemma8HowToIntSimplex2} only needs to run for $k$ such that $k_1 + \cdots + k_n \leq M$. This produces $\binom{M + n}{ n }$ summands, which is $O(M^n)$ as $n$ is fixed. Let 
\[h_k(\ve t) :=\frac{\ll \boldsymbol{\ell_{s_1}}, \ve t \rr^{k_1} \cdots \ll \boldsymbol{\ell_{s_{d+1}}}, \ve t \rr^{k_{d+1}} }{(|k| + d)!}, \]
where $|k| \leq M$. The expansion of each $\ll \boldsymbol{\ell_{s_i}}, \ve t \rr^{k_i}$ is a homogenous polynomial of degree $k_i$, which has $\binom{k_i + n -1}{n-1} = O(k_i^{n-1})$ terms. Multiplying everything out in $h_k(\ve t)$ results is at most $O(M^{(n-1)(d+1)})$ terms. Hence $h_k(\ve t)$ can be expanded into a polynomial in polynomial time in $M$.

Using Lemma \ref{lemma:poly-mult} again, the product $H_1 \sum_{k \in \N^n, |k| \leq M} h_k(\ve t)$, truncated at total degree $M$ can be computed in polynomial time. 
\end{proof}

\subsection{Benefits of Handelman decomposition}
\label{sec:handelman-example}

In Sections  \ref{sssec:handelman-same-form}, \ref{sssec:handelman-few-terms}, \ref{sssec:handelman-sparse}, we discuss three benefits of the Handelman decomposition: our algorithms for integrating one Handelman term have reusable computations, the decomposition can have few terms, and it can have sparse LP solutions.

\subsubsection{Handelman terms have the same form}
\label{sssec:handelman-same-form}
The polynomial time algorithms outlined in Theorems \ref{thm:both-parts-integrate-plf} and \ref{thm:my-proof-plf-simplex} compute the integral of just one power of a linear form, while the algorithms outlined in Theorems \ref{thm:continuous-affnie-products-linear-forms-cones} and \ref{thm:continuous-affnie-products-linear-forms-simplex} compute all the integrals of produces of affine functions up to a given degree. We stress the fact that every term in a Handelman decomposition of a polynomial $f(x)$ is in the same form: $g_1^{\alpha_1} \cdots g_n^{\alpha_n}$ where the $g_i$ corresponds to the rows of the polytope's $\PP$ inequality constraints $b - Ax \geq 0$. The only part that changes between Handelman terms in the powers $\alpha_1, \dots, \alpha_n$. This means, the integral of every Handelman term can be computed via one series computation from Theorems \ref{thm:continuous-affnie-products-linear-forms-cones} and \ref{thm:continuous-affnie-products-linear-forms-simplex}. Hence one big benefit of using the Handelman decomposition over a power of a linear form decomposition, is that the computation in Theorems \ref{thm:continuous-affnie-products-linear-forms-cones} and \ref{thm:continuous-affnie-products-linear-forms-simplex} can be applied to many Handelman terms at once.

\subsubsection{Handelman decompositions can have fewer terms}
\label{sssec:handelman-few-terms}
Next we compare between decomposing a random polynomial into a sum of powers of a linear form using Equation (\ref{eq:decomp-powerlinform}) and the Handelman method. We constructed a set of random polynomials in dimensions 3, 4, and 5 of total degree ranging from 3 to 8. For each dimension and degree pair, five random polynomials where constructed. Each polynomial was $20\%$ dense, meaning, $20\%$ of the coefficients of each polynomial out of the possible $\binom{d+D}{d}$ were nonzero. For the Handelman decomposition, the polytope $\PP$ was picked to be the box $[-1, 1]^d$, and the objective function enforces a sparse decomposition and minimizes $s$. Figure \ref{fig:percent-improvement-handelman-plf} illustrates how much better the Handelman decomposition can be over the power of linear form formula. The figure plots the average percent change between the number of terms in each method. 

For example, looking at the degree 8 polynomials, the Handelman decomposition had about $40\%$ fewer terms than the power of linear forms formula among the 15 test polynomials (five in dimension 3, 4, and 5, each). Among the 15 test polynomials, the best case had about $50\%$ fewer terms while the worst case has about $30\%$ fewer terms. 

\begin{figure}[h!]
  \caption{Percent improvement of the number of terms from a Handelman decomposition versus the number of terms from the power of linear form formula. Bars reflect min and max percent improvements.}
  \centering
    \includegraphics[width=0.90\textwidth]{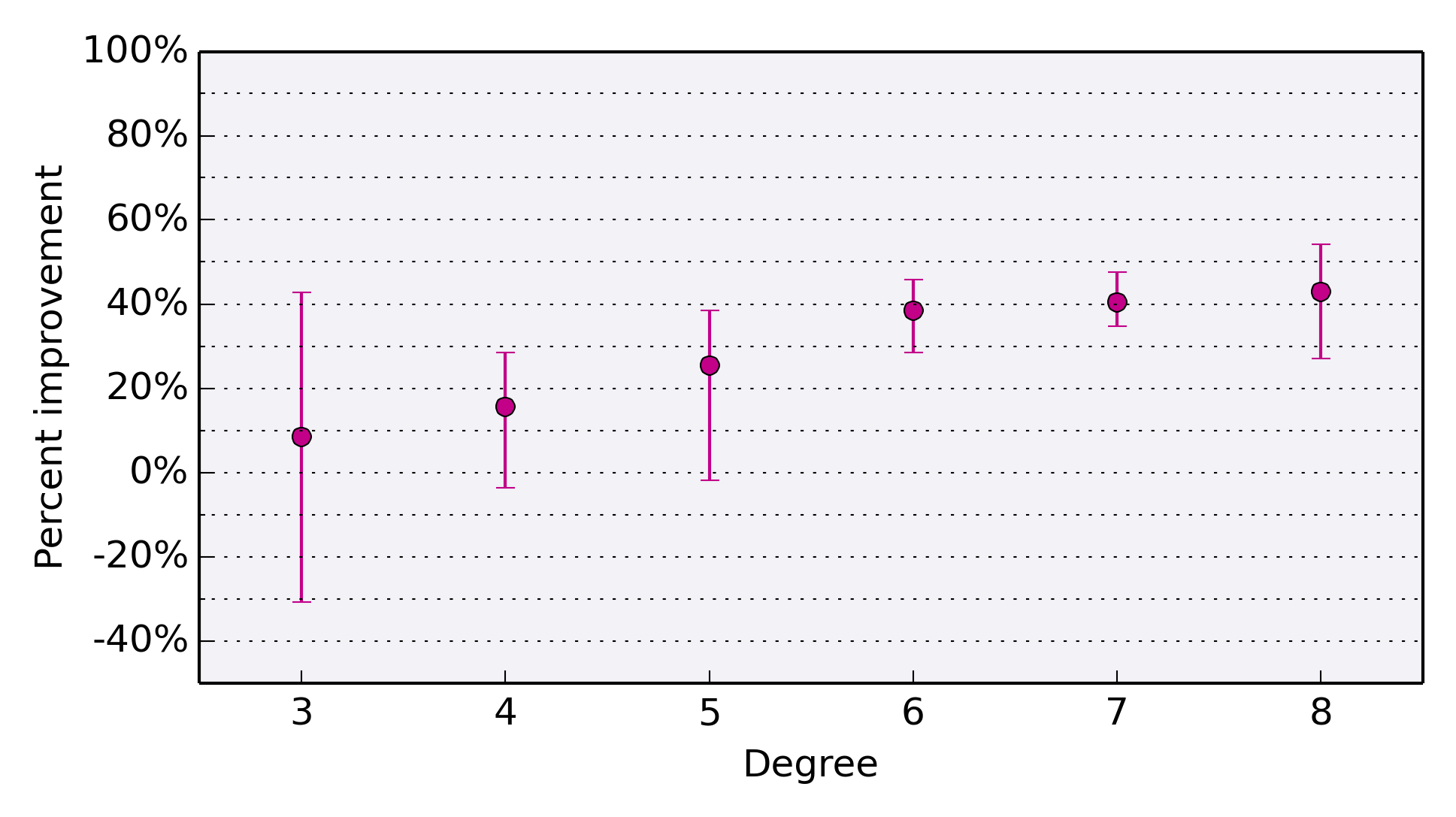}
    \label{fig:percent-improvement-handelman-plf}
\end{figure}

\begin{figure}[h!]
  \caption{Average number of terms between the Handelman decomposition and power of linear form decomposition.}
  \centering
    \includegraphics[width=0.90\textwidth]{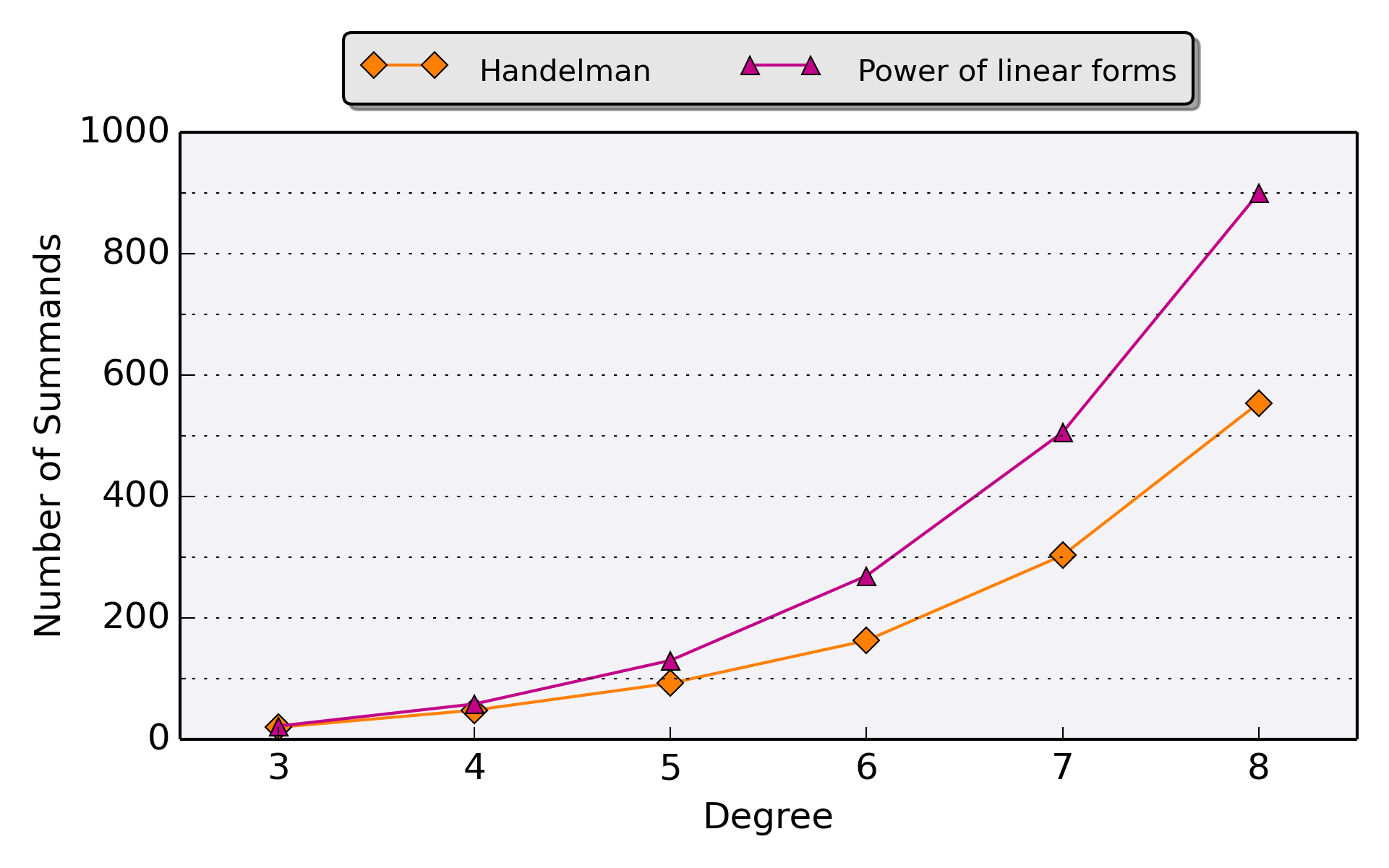}
    \label{fig:raw-improvement-handelman-plf}
\end{figure}

\begin{figure}[h!]
  \caption{Percent improvement of the average number of terms from a Handelman decomposition found by solving a linear program where the objective contains a sparsity term versus the number of nonzeros of a generic basic feasible solution. Bars reflect min and max percent improvements.}
  \centering
    \includegraphics[width=0.90\textwidth]{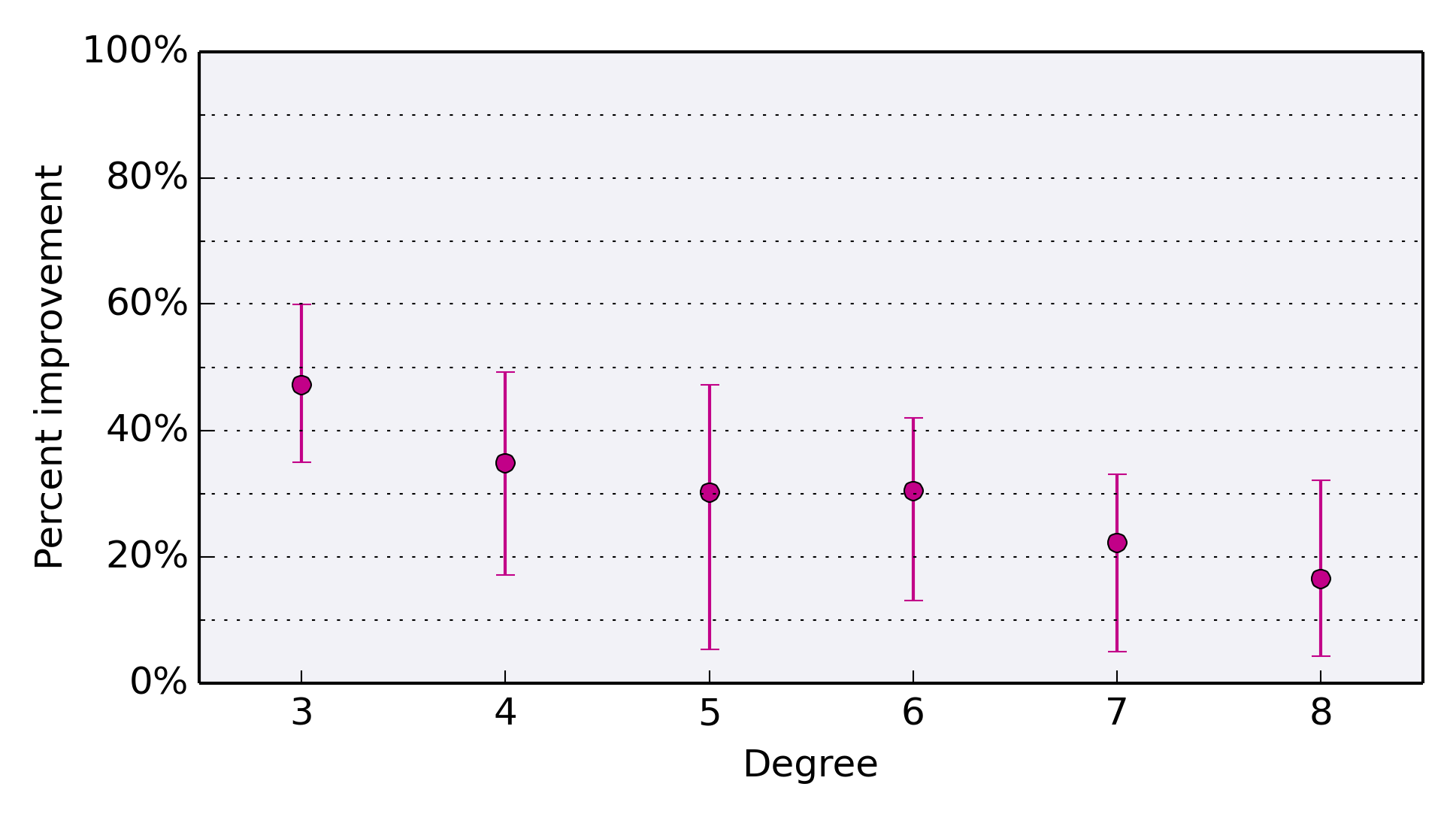}
    \label{fig:percent-improvement-handelman-sparsity}
\end{figure}

One interesting fact is that there are one or two examples in degree 3, 4, and 5 where the Handelman decomposition had a few more terms than the power of linear form decomposition (resulting in a negative percent improvement). However, Figure \ref{fig:raw-improvement-handelman-plf} reveals that the difference in the number of terms between both methods is small in these dimensions. Hence the benefit of the Handelman decomposition is less important for the low degree polynomials. However, the Handelman decomposition also discovers a lower bound to shift $f(x)$ to make it nonnegative on $\PP$.

\subsubsection{Handelman decompositions can be sparse}
\label{sssec:handelman-sparse}
Finally, we consider the effect of including a sparse term in the objective function of the linear program in Equation \eqref{eq:sparse-lp-handelman}. Figure \ref{fig:percent-improvement-handelman-sparsity} shows the average percent improvement between the number of linear forms found from our linear program and the number of nonzeros a generic basic feasible solution would have. For example, in degree $8$ the Handelman linear program solutions contained about $20\%$ fewer terms than a generic basic feasible solution would have, and the sparsest solution contained $30\%$ fewer terms while the densest solution had $5\%$ fewer terms than a generic basic solution.  Every linear program is degenerate. This shows that it could be worthwhile to build an objective function that controls sparsity. Of course, any objective function that is not just $\min s$ can result in $-s$ being a poor lower bound for $\fmin$.

One troubling fact with searching for a Handelman decomposition of order $t$ on a polytope $\PP$ with $n$ facets is that a linear program of size $\binom{t+d}{d} \times \binom{t+n}{n}$ needs to be solved. This brute force search quickly becomes impractical for large dimensions or with complicated polytopes. One way to reduce this cost when $n$ is large is to contain $\PP$ within a larger simpler polytope like a box or simplex, or to decompose $\PP$ into simpler polytopes. Another idea is to use row and column generation methods for solving the linear program. 

Table \ref{table:handelman-formula-comparison} lists the key points made about our two polynomial decomposition methods.

\begin{table}
\begin{tabularx}{\linewidth}{>{\parskip1ex}X@{\kern4\tabcolsep}>{\parskip1ex}X}
\toprule
\hfil\bfseries Handelman decomposition
&
\hfil\bfseries Decomposition using Equation \eqref{eq:decomp-powerlinform}
\\\cmidrule(r{3\tabcolsep}){1-1}\cmidrule(l{-\tabcolsep}){2-2}

%% PROS, seperated by empty line or \par
\begin{itemize}[leftmargin=*]
\item harder to find decomposition
\item fewer terms
\item domain dependent
\item only exists for positive polynomials on $\PP$, but this can be worked around
\item $s$ gives lower bound on $\fmin$ on $\PP$ 
\end{itemize}
&

%% CONS, seperated by empty line or \par
\begin{itemize}[leftmargin=*]
\item easy to find decomposition
\item many terms
\item domain independent
\item works on any polynomial
\end{itemize}

\\\bottomrule
\end{tabularx}
\caption{Comparison between two polynomial decompositions}
\label{table:handelman-formula-comparison}
\end{table}

\chapter{Polynomial Optimization}
\label{ch:polynomialOptimization}

In this chapter, we focus on optimizing a polynomial $f(x)$ over a polytope $P$. In particular, we explore the discrete and continuous optimization problems:

\begin{equation*}
\begin{split}
\max & \; f(x) \\
 & x \in \PP \cap \Z^d, 
\end{split}
\end{equation*}
and
\begin{equation*}
\begin{split}
\max & \; f(x) \\
 & x \in \PP.
\end{split}
\end{equation*}

We denote this maximum by $\fmax$. The exact optimization problems are hard. When the dimension is allowed to vary, approximating the optimal objective value of polynomial programming problems is still hard. See Section \ref{ch:bg:sh:poly-complexity} for a review of the complexity for these problems. Hence in this chapter, we develop  approximation algorithms for these optimization problems that run in polynomial time when the dimension is fixed. 

Our methods are general and do not assume $f$ is convex nor that it has any other special properties. At first, it will be necessary to assume $f(x)$ is nonnegative on $P$, but this restriction can be resolved. At the heart of our methods is the need to efficiently evaluate (when the dimension is fixed) $\sum_{x \in \PP\cap \Z^d} f(x)^k$ or $\int_{\PP} f(x)^k\d x$ for a given $k \in \N$. There are many references in the literature on how to compute these sums and integrals, see \cite{Baldoni2011WeightedEhrhart, baldoni-berline-deloera-koeppe-vergne:integration, bronstein2005symbolic, CUBPACK, deloera:software-exact-integration-polynomials}. Our focus is on methods that use \emph{generating functions} \cite{barvinokzurichbook}. 

When the dimension is fixed, there are many other polynomial time approximation algorithms for polynomial programming. For instance Parrilo \cite{parrilo2003SDP}, building on work by Lasserre \cite{Lasserre01globaloptimization} and Nesterov \cite{nesterov2000sqFunctinalSystems}, developed such methods using sums of squares optimization. These methods were further developed in \cite{anjos2011handbook, marshall2008positive}. Other methods for polynomial optimization have been developed in \cite{Behrends2015, Claudia2014, deKlerk2006210, Thanh2013, tawarmalani2002convexification, baron}.

Handelman's theorem (see Section \ref{sec:integration-prducts-affine-functions}) has also been used to optimize a polynomial in \cite{laurentsurvey, monique2014, ParriloSturmfels2003, sankaranarayanan2013lyapunov}. Using Handelman to maximize a polynomial directly would require finding a Handelman decomposition of large degree, which becomes more difficult as degree increases. When we use Handelman's theorem, it is only for producing a small degree Handelman decomposition of $f(x)$ to use with the polynomial approximation theorems.

Section \ref{ch:po:intopt} is a overview of \cite{deloera-hemmecke-koeppe-weismantel:intpoly-fixeddim} for optimizing a polynomial over $\PP \cap \Z^d$. As a side note, the authors of \cite{deloera-hemmecke-koeppe-weismantel:mixedintpoly-fixeddim-fullpaper} used the methods in \cite{deloera-hemmecke-koeppe-weismantel:intpoly-fixeddim} to develop efficient algorithms for optimizing a polynomial over the continuous domain $\PP$ and a mixed-integer domain $\PP \cap (\R^{d_1} \times \Z^{d_2})$. Then in Section \ref{ch:po:contopt}, a dedicated algorithm for the continuous optimization problem is developed, which was developed in \cite{brandon-handelman-paper}.

\section{Prior work: integer optimization of polynomials}
\label{ch:po:intopt}
The \emph{summation method} for optimization uses the elementary relation \[\max\{s_1, \dots, s_N\} = \lim_{k \rightarrow \infty} \sqrt[k]{s_1^k | \cdots + s_N^k}\] which holds for any finite set $S = \{s_1, \dots, s_N \}$ of nonnegative real numbers. This relation can be viewed as an approximation result for $\ell_k$-norms. Let $f(x) \in \Q[x_1, \dots, x_d]$ be a polynomial in $d$ variables that is nonnegative over $\PP$, which is a full-dimensional polytope in $\R^d$. Our goal is to solve the problem

\begin{equation}
\label{equ:discreteOpt}
\begin{split}
\max & \; f(x) \\
 & x \in \PP \\
 & x_i \in \Z.
\end{split}
\end{equation}

This is an NP-hard problem, so instead we seek lower and upper bounds for the maximum. Because $f$ is nonnegative on $\PP \cap \Z^d$, let $\norm{f}_k^k := \sum_{x \in \PP \cap \Z^d} f(x)^k$, and $\norm{f}_\infty = \max_{x \in \PP \cap \Z^d} f(x)$. Using the above norm-limit idea, we get the bounds

\[
\norm{f}_k/\sqrt[k]{N} \leq \norm{f}_\infty \leq \norm{f}_k,
\]
where $N := | \PP \cap \Z^d|$.

\begin{example}
Let $f(x) = x_1^2x_2 - x_1x_2$ with $x_1 \in [1,3]\cap \Z$ and $x_2\in [1,3] \cap \Z$. Then $f$ takes $N=9$ nonnegative values on its domain: $\{0, 2, 6, 0, 4, 12, 0, 6, 18\}$. Computing the bounds for different $k$ gives: 

\begin{center}
\begin{tabular}{ c  c c }

  $k$ & $\norm{f}_k/\sqrt[k]{9}$ & $\norm{f}_k$\\
  \hline
  10 & 14.47 & 18.03  \\
  20 & 16.12 & 18.00  \\
  30 & 16.72 & 18.00  \\
  40 & 17.03 & 18.00  \\    
\end{tabular}
\end{center}

Because $f$ has integer coefficients, for $k=40$ the bounds $17.03 \leq \fmax \leq 18$ implies $\fmax = 18$.

\end{example}

This example illustrates that $\fmax$ can be approximated by computing $\norm{f}_k$. However, to obtain a polynomial time algorithm, enumerating the lattice points in $\PP \cap \Z^d$ must be avoided. When the domain is a box, $\norm{f}_k$ can be computed by first expanding $g(x):= f(x)^k$, and then for each resulting monomial, $c x_1^{m_1}\cdots x_d^{m_d}$, computing $\sum_{x \in \PP \cap \Z^d} c x_1^{m_1}\cdots x_d^{m_d}$ by repeated application of Bernoulli's formula (which is also sometimes called Faulhaber's formula). The latter is the explicit polynomial in $n$ of degree $p+1$ for $F(n,p):=\sum_{j=1}^n j^p$, see \cite{ConwayBookOfNumbers}.

\begin{example}
Let $f(x) = x_1^2x_2 - x_1x_2$ with $x_1 \in [-5,6]\cap \Z$ and $x_2\in [1,3] \cap \Z$. Find  $\norm{f}_1$.

We will compute the sum monomial-by-monomial. The degree of the variables are $1$ and $2$, so $F(n,1)=(n^2 + n)/2$ and $F(n,2)=(2n^3 + 3n^2 +n)/6$. Notice that the domain of $x_1$ does not start at $1$, so we will need to evaluate $F(n,1)$ and $F(n,2)$ at different points and subtract or add.

\textit{summing the first monomial:} 
\begin{align*}
\sum_{x_1=-5}^{6} \sum_{x_2=1}^{3} x^2_1x_2 &= \sum_{x_1=-5}^{6} x_1^2 \sum_{x_2=1}^{3} x_2 \\
&= (F(5,2) + F(6,2)) F(3,1)\\
&= (55 + 91) 6 \\
&= 876
\end{align*}

\textit{summing the second monomial:} 
\begin{align*}
\sum_{x_1=-5}^{6} \sum_{x_2=1}^{3} -x_1x_2 &= -\sum_{x_1=-5}^{6} x_1 \sum_{x_2=1}^{3} x_2 \\
&= -(-F(5,1) + F(6,1)) F(3,1)\\
&= -(-15 + 21) 6\\
&= -36
\end{align*}

Thus, $\norm{f}_1 = 876 - 36 = 840$.

\end{example}

A lower and upper bounds for the discrete case are reviewed below.

\begin{theorem}[Theorem 1.1 in \cite{deloera-hemmecke-koeppe-weismantel:intpoly-fixeddim}]
\label{thm:first-fptas}
Let the number of variables $d$ be fixed. Let $f(x)$  be a polynomial of maximum total degree $D$ with integer coefficients, and let $\PP$ be a convex rational polytope defined by linear inequalities in $d$ variables. Assume $f(x)$ is nonnegative over $\PP$. There is an increasing sequence of lower bounds $\{L_k\}$ and a decreasing sequence of upper bounds $\{U_k\}$ to the optimal value of 
\[ \max  f(x) \text{ such that }  x \in \PP \cap \Z^d\]
that has the following properties:
\begin{enumerate}
\item The bounds are given by
\[ L_k := \sqrt[k]{\frac{\sum\limits_{\alpha \in \PP\cap \Z^d} f(\alpha)^k}{ |\PP \cap \Z^d|}} \leq \max\{f(\alpha) \mid \alpha \in \PP \cap \Z^d\} \leq  \sqrt[k]{\sum\limits_{\alpha \in \PP\cap \Z^d} f(\alpha)^k} := U_k. \]
\item $L_k$ and $U_k$ can be computed in time polynomial in
$k$, the input size of $\PP$ and $f$, and the total degree $D$. 
\item The bounds satisfy the following inequality:
\[U_k - L_k \leq \fmax \left(\sqrt[k]{|\PP\cap \Z^d|} - 1\right). \]
\item In addition, for $k = (1 + 1/\epsilon)\log(|\PP\cap \Z^d|)$, $L_k$ is a $(1-\epsilon)$-approximation to the optimal value $\fmax$ and it can be computed in time polynomial in the input size, the total degree $D$, and $1/\epsilon$. Similarly, $U_k$ gives a $(1 + \epsilon)$- approximation to $\fmax$. Moreover, with the same complexity, one can also find a feasible lattice point that approximates an optimal solution with similar quality.
\end{enumerate}
\end{theorem}

Therefore, computing an approximation for Problem \eqref{equ:discreteOpt} via Theorem \ref{thm:first-fptas} requires summing the evaluation of a polynomial over the lattice points of a polytope. In order to compute $\sum_{\alpha \in \PP \cap \Z^d} f(\alpha)^k$, the authors in \cite{deloera-hemmecke-koeppe-weismantel:intpoly-fixeddim} suggest computing the expansion $g(x) := f(x)^k$ and then evaluating the generating function for $\sum_{\alpha \in \PP \cap \Z^d} g(\alpha)z^\alpha$ at $z=1$ by taking derivatives of the generating function for $\sum_{\alpha \in \PP \cap \Z^d} g(\alpha)z^\alpha$. They show this can be done in polynomial time when the dimension is fixed; however, is hard to implement. We will illustrate a simpler method that is similar to the ideas of Chapter \ref{ch:Integration}.

To use our generating functions for summing a polynomial over the lattice points of a polytope, the polynomial $g(x)= f(x)^k$ is first decomposed. Again, there are two options, $g(x)$ can be decomposed into a sum of powers of linear forms (via Equation \eqref{eq:decomp-powerlinform}) or a sum of products of affine functions (via Handelman's theorem in Section \ref{sec:integration-prducts-affine-functions}). Both of these decompositions can be constructed in polynomial time. 

What follows is the author's original method for computing the sum of a Handelman term. The method also produces an algorithm for summing a power of a linear form over a polytope by setting $n=1$ and $r_1=0$.

\begin{proposition}
\label{prop:discrete-affnie-products-linear-forms}
When the dimension $d$ and number of factors $n$ is fixed, the value of 
\[\sum_{x \in \PP\cap \Z^d} \frac{(\langle \ell_1, x\rangle + r_1)^{m_1} \cdots (\langle \ell_n, x\rangle + r_n)^{m_n}}{m_1!\cdots m_n!} \]
can be computed in polynomial time in $M:=\sum_{i=1}^nm_i$ and the size of the input data.
\end{proposition}

In the proof below, instead of computing the value for a fixed sequence of powers $m_1, \dots, m_n$, we will compute the polynomial

\[ \sum_{ p_1+\cdots + p_{n} \leq M} \left(\sum_{x \in \PP \cap \Z^d} \frac{\langle (\ell_1, x\rangle + r_1)^{p_1} \cdots (\ll \ell_n, x\rr + r_n)^{p_n}}{p_1!\cdots p_{n}!}\right) t_1^{p_1}\cdots t_{n}^{p_{n}},\]

which includes the desired value as the coefficient of the monomial $t_1^{m_1}\cdots t_{n}^{m_{n}}$. This is ideal when using a Handelman decomposition because terms in the Handelman decomposition only differs in the exponents $m_1, \dots, m_n$. When $n=1$, the extra terms that are computed are unavoidable. The following proof will use a tangent cone decomposition of the polytope $\PP$.

\begin{proof}
Because the dimension $d$ is fixed, we can use Barvinok's Algorithm to write down

\[\sum_{x \in \PP \cap \Z^d} e^{\ll \ell, x\rr} = \sum_{i \in I} \epsilon_i \frac{e^{\langle \ell, v_i \rangle}}{\prod_{j=1}^d ( 1 - e^{\langle \ell, u_{ij}\rangle})}\] 

where $\epsilon_i \in \{-1,1\}, v_i \in \Z^d$, and $u_{ij} \in \Z^d$, and  $I$ is some index set whose size is polynomially bounded in the input size of $\PP$. The $v_i$ correspond to the integer point of an unimodular cone and the $u_{ij}$ are the cone's rays. This equality only holds when $\ell$ is an indeterminate. If $\ell$ is in $\Q^d$ and orthogonal to some $u_{ij}$, some of the above fractions are singular. However, the singularity is removed in the sum as the expression is holomorphic in $\ell$. See Section \ref{ch:bg:sec:integration-stuff} for a review of these statements. 

To work around the possibility that one of the fractions is singular when $\ell$ is evaluated, we will replace $\ell$ with $\ell_1t_1 +\cdots+ \ell_nt_n + \ell_{n+1}t_{n+1}$ where $\ell_{n+1}$ has been picked so that $\ll \ell_{n+1}, u_{ij} \rr \neq 0$ for all $u_{ij}$. Notice that the set of $\ell_{n+1}$ that fail this condition have measure zero, and so $\ell_{n+1}$ can be picked randomly. After doing this, we get

\[\sum_{x \in \PP \cap \Z^d} e^{\langle \ell_1, x \rangle t_1+ \cdots + \langle \ell_{n+1}, x \rangle t_{n+1} } =\sum_{i \in I} \epsilon_i \frac{e^{\langle \ell_1, v_i \rangle t_1+ \cdots + \langle \ell_{n+1}, v_i \rangle t_{n+1}}}{\prod_{j=1}^d ( 1 - e^{\langle \ell_1, u_{ij}\rangle t_1 + \cdots + \langle \ell_{n+1}, u_{ij}\rangle t_{n+1}})},\]
where no fraction is singular. 

To make the notation cleaner, let $\ve t = (t_1, \dots, t_n)$, $\ve a_i = (\langle \ell_1, v_i \rangle, \dots, \langle \ell_n, v_i \rangle)$, $\ve r = (r_1, \dots, r_n)$, $\ve b_{ij} = (\langle \ell_1, u_{ij}\rangle, \dots, \langle \ell_n, u_{ij}\rangle)$, $\beta_{ij} :=  \langle \ell_{n+1}, u_{ij}\rangle$, and $\boldsymbol{\ell}_x = (\ll \ell_1, x\rr, \dots, \ll \ell_n, x\rr)$. After multiplying both sides by $e^{\ll \ve r, \ve t\rr}$, we have:

\begin{equation}
\label{equ:proof-discrete-gen-fun}
\sum_{x \in \PP \cap \Z^d} e^{\ll \boldsymbol{\ell}_x, \ve t\rr + \ll \ve r, \ve t \rr + \ll \ell_{n+1}, x \rr t_{n+1} } = \sum_{i \in I} \epsilon_i \frac{e^{\ll \ve a_i, \ve t \rr}e^{\ll \ve r, \ve t \rr }e^{\ll \ell_{n+1}, v_i\rr t_{n+1}}}{\prod_{j=1}^d ( 1 - e^{\ll \ve b_{ij}, \ve t \rr + \beta_{ij} t_{n+1}})}. 
\end{equation}

The left hand side of Equation \ref{equ:proof-discrete-gen-fun} is 

\[ \sum_{x \in \PP \cap \Z^d} \prod_{i=1}^n e^{\ll \ell_i, x\rr t_i + r_i} e^{\ll \ell_{n+1}, x\rr t_{n+1}} \]
Replacing each exponential function with its Taylor series and expanding the product results in the series expansion in $t_1, \dots, t_{n+1}$ of the left hand side of Equation \ref{equ:proof-discrete-gen-fun}. Truncating the series at total degree $M$ contains exactly the desired polynomial plus monomials that have a $t_{n+1}$ factor which can be dropped. Hence the proof is complete once the series expansion of the right had side of Equation \ref{equ:proof-discrete-gen-fun} can be done in polynomial time. Next write each summand as a product of four terms:

\begin{multline}
 \frac{e^{\ll \ve a_i, \ve t \rr}e^{\ll \ve r, \ve t \rr }e^{\ll \ell_{n+1}, v_i\rr t_{n+1}}}{\prod_{j=1}^d ( 1 - e^{\ll \ve b_{ij}, \ve t \rr + \beta_{ij} t_{n+1}})}
=  \left(e^{\ll \ve a_i + \ve r, \ve t \rr}\right)
\left(\prod_{j=1}^d \frac{(\ll \ve b_{ij}, \ve t \rr + \beta_{ij} t_{n+1})}{ ( 1 - e^{\ll \ve b_{ij}, \ve t \rr + \beta_{ij} t_{n+1}})}\right) \\ \times
\left(\prod_{j=1}^d \frac{1}{\ll \ve b_{ij}, \ve t \rr + \beta_{ij} t_{n+1}}\right) 
e^{\ll \ell_{n+1}, v_i\rr t_{n+1}}. 
\end{multline}

Let $h_j$ be the series expansion of $e^{(\ll \ell_j, v_i\rr + r_j)t_j}$ in $t_i$ up to degree $M$ for $ 1 \leq j \leq n$. Let $h_{n+j}$ be the series expansion of 

\[\frac{(\ll \ve b_{ij}, \ve t \rr + \beta_{ij} t_{n+1})}{ ( 1 - e^{\ll \ve b_{ij}, \ve t \rr + \beta_{ij} t_{n+1}})}\]
up to total degree $M$ in $t_1, \dots, t_{n+1}$ for $ 1 \leq j \leq d$. This can be done in polynomial time as this is the generating function for the Bernoulli numbers, see Lemma \ref{lemma:bernoulli-numbers}. By Lemma \ref{lemma:poly-mult}, the product $H_1 := \prod_{j=1}^{n+d} h_j$  truncated at total degree $M$ is done in polynomial time in $M$ as $n$ and $d$ are fixed. 

Using the generalized binomial theorem we have,
\[\frac{1}{\ll \ve b_{ij}, \ve t \rr + \beta_{ij} t_{n+1}} = \sum_{k=0}^\infty (-1)^k (\ll \ve b_{ij}, \ve t \rr)^k (\beta_{ij} t_{n+1})^{-1-k}. \]
Notice that $(\ll \ve b_{ij}, \ve t \rr)^k$ is a polynomial in $t_1,\dots,t_n$ of total degree $k$. Let $h_{n+d+j}$ represent this sum truncated at $k=M$ for $1 \leq j \leq d$. Note that $h_{n+d+j}$ is a polynomial of total degree $M$ in $t_1, \dots, t_n$, and the power of $t_{n+1}$ ranges from $-1-M$ to $-1$ at most. 

Treating $t_{n+1}$ as a coefficient (that is, ignoring its power when computing the degree), we compute the product $H_2 := H_1 \cdot \prod_{j=1}^d h_{n+d+j}$ using Lemma \ref{lemma:poly-mult}. $H_2$ is a polynomial in $t_1, \dots, t_n$ of total degree $M$, and the power of $t_{n+1}$ at most ranges from $-d(M+1)$  to $M-d$. 

Finally, let $h_{n+2d+1}$ be the series expansion of $e^{\ll \ell_{n+1}, v_i\rr t_{n+1}}$ in $t_{n+1}$ up to degree $d(M+1)$. Using Lemma \ref{lemma:poly-mult} to compute $H_3 := H_2 \cdot h_{n+2d+1}$ results in polynomial in of total degree $M$ in $t_1, \dots, t_n$, and the power of each $t_{n+1}$ is nonnegative. 

Dropping terms where the power of $t_{n+1}$ is positive, and repeating the calculation $|I|$ times results in

\[ \sum_{ p_1+\cdots p_{n} \leq M} \left(\sum_{x \in \PP \cap \Z^d} \frac{\langle (\ell_1, x\rangle + r_1)^{p_1} \cdots (\ll \ell_n, x\rr + r_n)^{p_n} }{p_1!\cdots p_{n}!}\right) t_1^{p_1}\cdots t_{n+1}^{p_{n}}.\]

\end{proof}

For completeness, we explicitly show how to compute the Bernoulli number generating function. 

\begin{lemma}
\label{lemma:bernoulli-numbers}
Let $\ve t = (t_1, \dots, t_n)$, and $\ve c = (c_1, \dots, c_n)$. The series expansion of 
\[\frac{\ll \ve c, \ve t\rr}{ 1 - e^{\ll \ve c, \ve t\rr}}\]
up to total degree $M$ in $t_1, \dots, t_{n}$ can be done in polynomial time in $M$ when $n$ is fixed.
\end{lemma}

\begin{proof}
From \cite{ConwayBookOfNumbers}, we start with

\[\frac{x}{1 - e^{-x}} = \sum_{k=0}^\infty B_k \frac{(-x)^k}{k!}, \]
where the $B_k$ are the Bernoulli numbers of the first kind. 

Then expanding the polynomial
\[ - \sum_{k=0}^M B_k \frac{(\ll \ve c, \ve t\rr)^k}{k!} \]
in $t_1, \dots, t_n$ yields the degree $M$ truncation of the series expansion of 
\[\frac{\ll \ve c, \ve t\rr}{ 1 - e^{\ll \ve c, \ve t\rr}}.\]
Because $n$ is fixed, the expansion can be done in polynomial time in $M$. $B_k$ can be computed in time $O(k^2)$ by the Akiyama--Tanigawa algorithm, which is reproduced in Algorithm \ref{alg:Akiyama-Tanigawa}.
\end{proof}

% https://en.wikipedia.org/wiki/Bernoulli_number#Algorithmic_description
\begin{algorithm}                      
\caption{Akiyama--Tanigawa algorithm for first Bernoulli numbers \cite{KanekoBernoulli}}
\label{alg:Akiyama-Tanigawa}
\begin{algorithmic}                    
\REQUIRE $k$
\ENSURE $B_k$

\STATE Let $A$ be an array of length $k+1$ with index starting at zero
\FOR{$i$ from 0 to $k$}
	\STATE $A[i] \gets 1/(i+1)$
	\FORALL{$j$ from $i$ by $-1$ to 1}
		\STATE $A[j-1] \gets j\cdot ( A[j-1] - A[j])$
	\ENDFOR
\ENDFOR
\IF{$k = 1$}
	\RETURN $-A[0]$
\ENDIF
\RETURN $A[0]$
\end{algorithmic}
\end{algorithm}

\section{Continuous Optimization}
\label{ch:po:contopt}
This section addresses the problem of approximating the continuous optimization problem

\begin{equation}
\label{equ:continuousOpt}
\begin{split}
\max & \; f(x) \\
 & x \in \PP \\
 & x_i \in \R,
\end{split}
\end{equation}
using the same style of bounded developed in Section \ref{ch:po:intopt}. Theorem \ref{theorem:continuous-lkuk-bounds} is the continuous analog of Theorem \ref{thm:first-fptas} for optimizing a polynomial $f$ over the continuous domain $\PP$. This section is devoted to its proof, and is a highlight of \cite{brandon-handelman-paper}.

We note that Theorem \ref{theorem:continuous-lkuk-bounds} involves integrating the polynomial $f(x)$. Integral kernel or moment methods produce an approximation to $\fmax$ via computing $\int f(x)\d \mu_k$ where the measure $\mu_k$ usually involves a special subset of  nonnegative functions, like sum-of-squares polynomials. Such  methods have been developed in \cite{deKlerk2015, lasserre2009momentsBook, Lasserre01globaloptimization, lasserre2002semidefinite, lasserre2011NewLook, Lasserre2000929}. Our method is slightly different as our measure is always the standard Lebesgue measure, while our integrand is $f(x)^k$.

\begin{theorem}
\label{theorem:continuous-lkuk-bounds}
Let the number of variables $d$ be fixed. Let $f(x)$  be a polynomial of maximum total degree $D$ with rational coefficients, and let $\PP$ be a full-dimensional convex rational polytope defined by linear inequalities in $d$ variables. Assume $f(x)$ is nonnegative over $\PP$. Then there is an increasing sequence of lower bounds $\{L_k\}$ and a decreasing sequence of upper bounds $\{U_k\}$ for $k \geq k_0$ to the optimal value of 
$ \max  f(x) \text{ for }  x \in \PP$
such that the bounds have the following properties:
\begin{enumerate}
\item For each $k > 0$, the lower bound is given by
\[ L_k := \sqrt[k]{\frac{\int_{\PP} f(x)^k \d x}{ \vol(\PP)}}. \] 

\item Let $M$ be the maximum width of $\PP$ along the $d$ coordinate directions, $\epsilon' := \frac{d}{d+k}$, and let $\mathcal{L}$ be a Lipschitz constant of $f$ satisfying
\[ |f(x) - f(y)| \leq \mathcal{L} \norm{x-y}_\infty \text{ for } x,y  \PP.\]
Then for all $k$ such that $k \geq k_0$ where $k_0 = \max\{1, d(\frac{\fmax}{M\mathcal{L}} -1)\}$, the upper bound is given by

\[U_k := \left( \frac{\int_\PP f(x)^k \d x}{\vol(\PP)} \right)^{1/(d+k)}  \left(\frac{M\mathcal{L}}{\epsilon'}\right)^{d/(d+k)} \frac{1}{(1-\epsilon')^{k/(d+k)}}.
\]
\item The bounds $L_k$ and $U_k$ can be computed in time polynomial in $k$, the input size of $\PP$ and $f$. 

\item Let $\epsilon > 0$, and $U$ be an arbitrary initial upper bound for $\fmax$. Then there is a $k \geq k_0$ such that 
\begin{enumerate} 
\item $U_k - L_k \leq \epsilon\fmax$ where $k$ depends polynomially on $1/\epsilon$, linearly in $U/M\mathcal{L}$, and logarithmically on $UM\mathcal{L}$, and 
\item for this choice of $k$, $L_k$ is a $(1-\epsilon)$--approximation to the optimal value $\fmax$ and $U_k$ is a $(1 + \epsilon)$--approximation to $\fmax$.
\end{enumerate}
\end{enumerate}
\end{theorem}

\begin{proof}
A polynomial time algorithm for computing $\int_P f(x)^k \d x$ is developed in Chapter \ref{ch:Integration}. 

We now focus on proving the bounds. The lower bound $L_k$ is immediate. For the upper bound, we take inspiration from the standard proof that $\lim_{p \rightarrow \infty} \norm{f}_p = \norm{f}_\infty$. Let  
$\epsilon' > 0$
, $\PP^{\epsilon'}:= \{ x \in \PP \mid f(x) \geq (1-\epsilon')\fmax\}$ and $[\PP^{\epsilon'}]$ denote the characteristic function of $\PP^{\epsilon'}$. Then we have 
\[(1-\epsilon')\fmax \cdot \vol(\PP^{\epsilon'}) \leq \int_{\PP^{\epsilon'}} f(x) \d x = \int_{\PP} f(x)[\PP^{\epsilon'}] \d x \leq \left( \int_\PP f(x)^k \d x \right)^{1/k} \vol(\PP^{\epsilon'})^{1/q}\]
where the last inequality comes from H\"older's inequality with $\frac{1}{k} + \frac{1}{q} = 1$.

Rearranging the first and last expressions results in 
\[ \left( \int_\PP f(x)^k \d x \right)^{1/k} \geq (1-\epsilon')\fmax \cdot \vol(\PP^{\epsilon'})^{1/k}.\]

We now seek a lower bound for $\vol(\PP^{\epsilon'})$. Because $f$ is a polynomial and $\PP$ is compact, $f$ has a Lipschitz constant $\mathcal{L}$ ensuring $|f(x) - f(y)| \leq \mathcal{L} \norm{x-y}_\infty$ for $x,y \in \PP$. Let $x^*$ denote an optimal point where $f(x^*) = \fmax$ and let \[B_\infty(x^*,r) := \{x \in \R^d : \norm{x-x^*}_\infty \leq r \}\] be a closed  ball of radius $r$ centered at $x^*$. Then for all $x$ in the ball $B_\infty(x^*, \epsilon' \fmax / \mathcal{L})$, $f(x) \geq (1-\epsilon')\fmax$. Therefore,

\[\vol(P^{\epsilon'}) = \vol\{x \in P \mid f(x) \geq (1-\epsilon')\fmax\} \geq \vol\{x \in P \cap B_\infty(x^*,\epsilon' \fmax / \mathcal{L}) \}. \]

\begin{figure}
        \centering
        \begin{subfigure}[b]{0.3\textwidth}
                \includegraphics{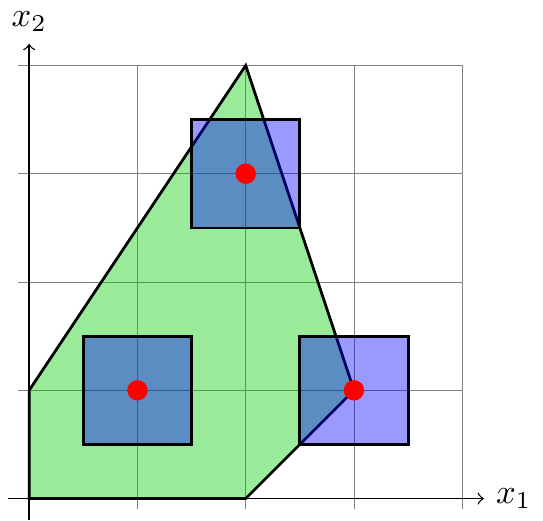}
                \caption{}
                \label{fig:pentagon-without}
        \end{subfigure}%
        \hspace{1in} %add desired spacing between images, e. g. ~, \quad, \qquad, \hfill etc.
        \begin{subfigure}[b]{0.3\textwidth}
                \includegraphics{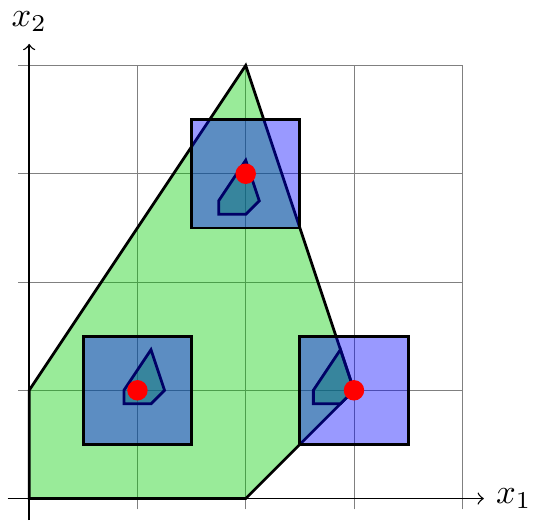}
                \caption{}
                \label{fig:pentagon-with}
        \end{subfigure}
\caption{(A) Pentagon with three possible points $x^*$ that maximize $f(x)$ along with the $\ell_\infty$ balls $B_\infty(x^*, 1/2)$. It might not be true that the ball intercepts the polytope. (B) However, a scaled pentagon can be contained in the ball and the original polytope.}
\label{fig:pentagons-scaled-balls}
\end{figure}

Notice that the ball $B(x^*,\epsilon' \fmax / \mathcal{L})$ may not be contained in $\PP$, see Figure \ref{fig:pentagon-without}. We need to lower bound the fraction of the ball that intercepts $\PP$ where $x^*$ could be any point in $\PP$. To do this, we will show how a scaled $\PP$ can be contained in $\PP^{\epsilon'}$ as in Figure \ref{fig:pentagon-with}. Consider the invertible linear transformation $\Delta : \R^d \rightarrow \R^d$ given by $\Delta x := \delta(x-x^*) + x^*$ where $\delta > 0$ is some unknown scaling term. Notice that this function is simply a scaling about the point $x^*$ that fixes the point $x^*$. We need to pick $\delta$ such that $\Delta \PP \subseteq \PP \cap B_\infty(x^*,\epsilon' \fmax / \mathcal{L})$.

\begin{itemize}
\item If $\Delta \PP \subseteq B_\infty(x^*,\epsilon' \fmax / \mathcal{L})$, then for $x \in \PP$
\begin{align*}
\norm{\Delta x - x^*}_\infty &= \norm{\delta (x-x^*) + x^* -x^*}_\infty \\
&= \delta \norm{ x- x^*}_\infty \\
&\leq \frac{\epsilon' \fmax}{\mathcal{L}}.  \\
\end{align*}

Let $M$ be the maximum width of $\PP$ along the $d$ coordinate axes, then 
 $0 < \delta \leq \frac{\epsilon' \fmax}{M\mathcal{L}}$ is a sufficient condition for  $\Delta \PP \subseteq B_\infty(x^*, \epsilon' \fmax/\mathcal{L})$.

\item Let $x \in \PP$, and write $x$ and $x^*$ as a convex combination of the vertices of $\PP$. That is, let $x=\sum_{i=1}^N \alpha_i v_i$ and $x^*=\sum_{i=1}^N \beta_i v_i$ where $\sum_{i=1}^N \alpha_i = 1$, $\sum_{i=1}^N \beta_i = 1$, $\alpha_i \geq 0$, and $\beta_i \geq 0$. Then 

\begin{align*}
\Delta x &= \delta \left(\sum_{i=1}^N \alpha_i v_i- \sum_{i=1}^N \beta_i v_i\right) + \sum_{i=1}^N \beta_i v_i \\
&= \sum_{i=1}^N (\delta \alpha_i + \beta_i - \delta \beta_i) v_i \\
\end{align*}

No matter what $\delta$ is, $\sum_{i=1}^N (\delta \alpha_i + \beta_i - \delta \beta_i) = 1$. Forcing $\delta \alpha_i +  (1 - \delta)\beta_i \geq 0$ for each $i$ as $x$ and $x^*$ varies over the polytope $\PP$ is equivalent to $\delta \leq 1$. 

Hence, $0 < \delta \leq 1$ a sufficient condition for $\Delta \PP \subseteq \PP$. 

\end{itemize}

Therefore, $\vol\{x \in P \cap  B_\infty(x^*,\epsilon' \fmax / \mathcal{L}) \} \geq \vol(\Delta \PP) = \left(\min\{1,\frac{\epsilon' \fmax}{M\mathcal{L}}\}\right)^d \vol(\PP),$ and finally, 

\[ \left( \int_\PP f(x)^k \d x \right)^{1/k} \geq (1-\epsilon')\fmax \cdot \vol(\PP)^{1/k} \left(\min\{1,\frac{\epsilon' \fmax}{M\mathcal{L}}\}\right)^{d/k}.\]

As the above inequality is true for all $0 < \epsilon' \leq 1$, we want to pick $\epsilon'$ to maximize the function $\phi(\epsilon') := (1-\epsilon') \left(\min\{1,\frac{\epsilon' \fmax}{M\mathcal{L}}\}\right)^{d/k}$. The maximum of $(1-\epsilon') \left(\frac{\epsilon' \fmax}{M\mathcal{L}}\right)^{d/k}$ occurs at $\epsilon' = \frac{d}{d+k}$. Hence the maximum of $\phi(\epsilon')$ occurs at $\epsilon' = \frac{d}{d+k}$ if this is less than $\frac{M\mathcal{L}}{\fmax}.$ Otherwise if $\epsilon' \geq \frac{M\mathcal{L}}{\fmax}$, the maximum of $\phi(\epsilon')$ is at $\epsilon' = \frac{M\mathcal{L}}{\fmax}$. Therefore  the maximum of $\phi(\epsilon')$ occurs at $\epsilon' =  \min\{\frac{d}{d+k}, \frac{M\mathcal{L}}{\fmax}\}$. Enforcing $\frac{d}{d+k} \leq \frac{M\mathcal{L}}{\fmax}$ is equivalent to $d(\frac{\fmax}{M\mathcal{L}} - 1) \leq k$. With this choice of $\epsilon'$ and value restriction on $k$, solving for $\fmax$ yields the desired formula for $U_k$.

The proof of part (4) is completed in the next three lemmas.
\end{proof}

\begin{lemma}
\label{lemma:loge}
$\frac{1}{\log(1+\epsilon)} \leq 1 + \frac{1}{\epsilon}$ for $0 < \epsilon \leq 1.$
\end{lemma}
\begin{proof}
It is enough to show that $\phi(\epsilon):= (1 + \frac{1}{\epsilon})\log(1+\epsilon) \geq 1$ for $0 < \epsilon \leq 1$. This follows because $\phi(\epsilon)$ is increasing on $0 < \epsilon \leq 1$ and $\lim_{\epsilon \rightarrow 0} \phi(\epsilon) = 1$. 
\end{proof}

%\begin{lemma}
%\label{lemma:logz}
%$\log(z)/z \leq 1/\sqrt{z}$ for $z \geq 1$.
%\end{lemma}
%\begin{proof}
%We show that $\phi(z):= \log(z) / \sqrt{z} \leq 1$ for all $z \geq 1$.
%Notice that $\phi(1) = 0 < 1$. Next we maximize $\phi$. Solving for $\frac{d}{dz}\phi(z) =0$ gives $z = e^2$. Then $\phi(e^2) = \frac{2}{\sqrt{e^2}} < 1$. Finally, $\lim_{z \rightarrow \infty} \phi(z) = 0$. 
%\end{proof}

\begin{lemma}
\label{lemma:logdelta}
For every $\delta > 0$, there is a $c_\delta > 0$ such that \[\frac{\log(z+1)}{z+1} \leq c_\delta (z+1)^{- \frac{1}{1+\delta}} \] for all $z > 0$.

In particular, if $\delta = 0.1$, then $c_\delta$ can be set to $4.05$.
\end{lemma}
\begin{proof}
Because \[\lim_{z \rightarrow \infty} \frac{(z+1)^{1 - \frac{1}{1+\delta}}}{\log(1+z)} = \infty,\]
there exist a $c > 0$ and $z_0 \geq 0$ such that $ \log(z+1) \leq c(z+1)^{1 - \frac{1}{1+\delta}}$ for $z \geq z_0$. As $(z+1)^{1 - \frac{1}{1+\delta}}$ is positive for $z \geq 0$, there exist a sufficiently large constant, $c_\delta$, such that $ \log(z+1) \leq c_\delta(z+1)^{1 - \frac{1}{1+\delta}}$ for $z \geq 0$.

Now let $\delta = 0.1$. We show the minimum of $\phi(z) := 4.05(z+1)^{1/11} - \log(z+1)$ is positive. Notice that $\phi(0) > 0$, and $\lim_{z \rightarrow \infty} \phi(z) = \infty$. The only zero of $\phi'(z)$ occurs at $z^* = \left(\frac{11}{4.05}\right)^{11} -1$. Finally, $\phi(z^*) > 0$.

\end{proof}

\begin{lemma}
\label{lemma:epsilon-polynom-in-k}
Let $\epsilon > 0$, and $U$ be an arbitrary initial upper bound for $\fmax$. Then there is a $k \geq 1$ such that $U_k - L_k \leq \epsilon\fmax$ where $k$ depends polynomially on $1/\epsilon$, linearly in $U/M\mathcal{L}$, and logarithmically on $UM\mathcal{L}$. 
Moreover, for this choice of $k$, $L_k$ is a $(1-\epsilon)$--approximation to the optimal value $\fmax$ and $U_k$ is a $(1 + \epsilon)$--approximation to $\fmax$. 
\end{lemma}

\begin{proof}

%Let \[ k = \left \lceil{\max\left\{d\left(\frac{U}{M\mathcal{L}}-1\right), \frac{d}{(\epsilon + 1)^{1/3} -1}, 3d \log(UM\mathcal{L})\left( 1 + \frac{1}{\epsilon}\right), 9d\left(1 + \frac{1}{\epsilon}\right)^2\right\}}\right \rceil.\]

Let \[ k = \left \lceil{\max\left\{d\left(\frac{U}{M\mathcal{L}}-1\right), \frac{d}{(\epsilon + 1)^{1/3} -1}, 3d \log(UM\mathcal{L})\left( 1 + \frac{1}{\epsilon}\right), O\left(\left(1 + \frac{1}{\epsilon}\right)^{1+\delta}\right)\right\}}\right \rceil,\] where the last term is understood to mean for every $\delta > 0$, there exist a $c_\delta > 0$ so that $k$ should be larger than $d\left( (3c_\delta)^{1+\delta} \left( 1 + \frac{1}{\epsilon}\right)^{1+\delta} - 1\right)$. In particular, if $\delta = 0.1$, then $c_\delta$ could be $4.05$.

The first term ensures the formula for $U_k$ holds because \[k \geq d\left(\frac{U}{M\mathcal{L}}-1\right) \geq d\left(\frac{\fmax}{M\mathcal{L}}-1\right).\] In terms of a algorithmic implementation, other upper bounds for $d\left(\frac{\fmax}{M\mathcal{L}}-1\right)$ can be used; for instance, because $\fmax - \fmin \leq M\mathcal{L}$, $d\left(\frac{\fmax}{M\mathcal{L}}-1\right) \leq \frac{d\fmin}{M\mathcal{L}} \leq \frac{df(x_0)}{M\mathcal{L}}$ where $x_0$ is any point in the domain. 

Notice that $\frac{1}{(\epsilon + 1)^{1/3} -1} = \frac{1}{\epsilon} + 1 + O(\epsilon)$, so $k$ is bounded by a polynomial in $1/\epsilon$. We proceed with an ``$\epsilon/3$'' argument.

Because $k \geq \frac{d}{(\epsilon + 1)^{1/3} -1} $, after rearranging we have

\[ (\epsilon + 1)^{1/3} \geq  \frac{d}{k} + 1 = \frac{d+k}{k}. \]

Using the logarithm produces the bounds
\begin{equation}
\label{eq:k-bounds-part1}
\frac{\log(\epsilon + 1)}{3} \geq  \log\left( \frac{d+k}{k}\right) \geq \frac{k}{d+k}  \log\left( \frac{d+k}{k}\right). 
\end{equation}

Next, because $k \geq 3d \log(UM\mathcal{L})\left( 1 + \frac{1}{\epsilon}\right)$, we can rearrange this to

\[ k \geq 3d \log(UM\mathcal{L})\left( 1 + \frac{1}{\epsilon}\right) \geq \frac{3d \log(UM\mathcal{L})}{\log(\epsilon + 1)} \geq \frac{d[ \log(UM\mathcal{L}) - \log(\epsilon + 1)/3]}{\log(\epsilon + 1)/3} \]

where the second inequality comes from Lemma \ref{lemma:loge}, and the third inequality comes from the fact that $\log(\epsilon+1) > 0$. After further rearranging of the first and last expressions, we get 

\begin{equation}
\label{eq:k-bounds-part2}
 \frac{\log(\epsilon + 1)}{3} \geq \frac{d}{d+k}\log(UM\mathcal{L}). 
\end{equation}

Finally, fix $\delta > 0$, then because $k \geq d\left( (3c_\delta)^{1+\delta} \left( 1 + \frac{1}{\epsilon}\right)^{1+\delta} - 1\right)$, rearranging terms results in 

\[\frac{k}{d}+1 \geq (3c_\delta)^{1+\delta} \left(1+\frac{1}{\epsilon}\right)^{1+\delta}. \]

Then Lemma \ref{lemma:loge} implies

\[\frac{k}{d}+1 \geq \frac{(3c_\delta)^{1+\delta}}{\log(1+\epsilon)^{1+\delta}},\] and after solving for the log term gives

\[ \frac{\log(1+\epsilon)}{3} \geq  \frac{c_\delta}{\left(\frac{k}{d}+1\right)^{\frac{1}{1+\delta}}}.\]

Then Lemma \ref{lemma:logdelta} with $z = k/d$ yields

\begin{equation}
\label{eq:k-bounds-part3}
 \frac{\log(1+\epsilon)}{3} \geq \frac{\log(1 + \frac{k}{d})}{1+\frac{k}{d}} = \frac{d}{d+k}\log\left(\frac{d+k}{d}\right).
\end{equation}
 
%OLD WAY
%Finally, because $k \geq 9d(1 + \frac{1}{\epsilon})^2,$
%
%\[ k \geq 9d\left(1 + \frac{1}{\epsilon}\right)^2 \geq  \frac{d}{\left(\log(\epsilon + 1)/3\right)^2} \geq \frac{d(1 - (\log(\epsilon + 1)/3)^2)}{(\log(\epsilon + 1)/3)^2},\]
%
%where the second inequality comes from using Lemma $\ref{lemma:loge}$, and the third inequality is true because $\log(\epsilon+1) \geq 0$. Rearranging the first and last expressions gives
%\[  (d+k)(\log(\epsilon + 1)/3)^2 \geq d,  \]
%
%
%or 
%
%\[  \frac{\log(\epsilon + 1)}{3} \geq  \left( \frac{d}{d+k} \right)^{1/2}. \]
%
%
%Then we have
%
%\begin{equation}
%\label{eq:k-bounds-part3}
%\frac{\log(\epsilon + 1)}{3} \geq  \left( \frac{d}{d+k} \right)^{1/2}
%\geq \frac{d}{d+k}\log\left(\frac{d+k}{d}\right);
%\end{equation}
%where the last inequality comes from using Lemma \ref{lemma:logz} with $z = \frac{d+k}{d}$. 

Using Equations \eqref{eq:k-bounds-part1}, \eqref{eq:k-bounds-part2}, and \eqref{eq:k-bounds-part3} results in the inequality

\begin{align*}
\log(\epsilon + 1) &\geq \frac{d}{d+k} \log\left(\frac{d+k}{d}\right)  +\frac{d}{d+k}\log(UM\mathcal{L}) + \frac{k}{d+k}  \log\left( \frac{d+k}{k}\right) \\
&= \frac{d}{d+k} \log\left(\frac{UM\mathcal{L}(d+k)}{d} \right) + \frac{k}{d+k}\log\left(\frac{d+k}{k}\right),  
\end{align*}
which implies 

\[ \epsilon \geq   U^{\frac{d}{d+k} } \left(\frac{M\mathcal{L}(d+k)}{d} \right)^{\frac{d}{d+k}}   \left( \frac{d+k}{k}\right)^{\frac{k}{d+k}}  - 1. \]

Finally, because $\frac{1}{d+k}-\frac{1}{k} = \frac{1}{k}\cdot \frac{d}{d+k}$ we have that

\begin{align*}
U_k - L_k  &= \left(  \left(\frac{\int_\PP f^k \d x}{\vol(\PP)}\right)^{\frac{1}{d+k}} \left(\frac{M\mathcal{L}}{\epsilon'} \right)^{\frac{d}{d+k}}   \left( \frac{1}{1-\epsilon'}\right)^{\frac{k}{d+k}}   \right) - \left(\frac{\int_\PP f^k \d x}{\vol(\PP)}\right)^{1/k} \\
& = \left(\frac{\int_\PP f^k \d x}{\vol(\PP)}\right)^{1/k} \left(  \left(\frac{\int_\PP f^k \d x}{\vol(\PP)}\right)^{\frac{1}{d+k}-\frac{1}{k}} \left(\frac{M\mathcal{L}}{\frac{d}{d+k}}  \right)^{\frac{d}{d+k}}   \left(  \frac{1}{1-d/(d+k)}\right)^{\frac{k}{d+k}}   - 1\right)  \\
& \leq \fmax \left(  \left(\frac{\int_\PP f^k \d x}{\vol(\PP)}\right)^{\frac{1}{d+k}-\frac{1}{k}} \left(\frac{M\mathcal{L}}{\frac{d}{d+k}}  \right)^{\frac{d}{d+k}}   \left(  \frac{1}{1-d/(d+k)}\right)^{\frac{k}{d+k}}   - 1\right)  \\
&\leq \fmax \left(  U^{\frac{d}{d+k}} \left(\frac{M\mathcal{L}(d+k)}{d}  \right)^{\frac{d}{d+k}}   \left( \frac{d+k}{k}\right)^{\frac{k}{d+k}}   - 1\right) \\
&\leq \fmax \cdot \epsilon.
\end{align*}

$L_k$ is a $(1-\epsilon)$--approximation to $\fmax$ because

\[ \fmax \leq U_k = L_k + (U_k - L_k) \leq L_k +\epsilon \fmax \]

and $U_k$ is a $(1+\epsilon)$--approximation to $\fmax$ because

\[ U_k - \epsilon \fmax \leq U_k + (L_k - U_k) = L_k \leq \fmax.\]

\end{proof}

An important input to our approximation bounds is a Lipschitz constant $\mathcal{L}$ satisfying $|f(x) - f(y)| \leq \mathcal{L} \norm{x-y}_\infty$ for all $x, y \in \PP$. One natural way to compute this constant is by maximizing $\norm{\nabla f(x)}_1$ on $\PP$. However this is a difficult problem in general. Instead, one can compute a potentially larger constant by following the next lemma, which produces a Lipschitz constant of polynomial size and in linear time.

\begin{lemma}[Lemma 11 in \cite{deloera-hemmecke-koeppe-weismantel:mixedintpoly-fixeddim-fullpaper}] 
\label{lemma:lemma_11}
Let $f$ be a polynomial in $d$ variables with maximum total degree $D$. Let $c$ denote the largest absolute value of a coefficient of $f$. Then there exists a Lipschitz constant $\mathcal{L}$ such that $|f(x) - f(y)| \leq \mathcal{L} \norm{x-y}_\infty$ for all $|x_i|, |y_i| \leq M$. The constant $\mathcal{L}$ is $O(D^{d+1}cM^D)$, and is computable in linear time in the number of monomials of $f$.
\end{lemma}

Before closing this section, we remark that Theorem \ref{theorem:continuous-lkuk-bounds} could also be applied for global maximization. If a polynomial has a global maximum, then there is a $M > 0$ such that the maximum is contained in the box $\PP := [-M, M]^d$. If $M$ can be computed or estimated, this reduces the unbounded problem to maximizing $f(x)$ over $\PP$. 
%%%%%%%%%%%%%%%%%%%%%%%%%%%%%%%%%%%%%%%%%%%%%%%%%%%%%%%%%%%%%%%
%%%%%%%%%%%%%%%%%%%%%%%%%%%%%%%%%%%%%%%%%%%%%%%%%%%%%%%%%%%%%%%
%%%%%%%%%%%%%%%%%%%%%%%%%%%%%%%%%%%%%%%%%%%%%%%%%%%%%%%%%%%%%%%
%%%%%%%%%%%%%%%%%%%%%%%%%%%%%%%%%%%%%%%%%%%%%%%%%%%%%%%%%%%%%%%

\subsection{An example}

We now illustrate our bounds on an example. Let $f(x) = x_1^2x_2 - x_1x_2$ with $x_1 \in [1,3]$ and $x_2\in [1,3]$. Note that $f(x)$ is nonnegative on its domain $\PP = [1,3]^2$. 

For the Lipschitz constant, we could easily maximize $\norm{\nabla f(x)}_1$ on $\PP$ which for this problem is $21$. However let us use the  bound produced by Lemma \ref{lemma:lemma_11} which gives $\mathcal{L} = 33$.

Using our new bounds from Theorem \ref{theorem:continuous-lkuk-bounds}, we have $L_k \leq \fmax \leq U_k$
where
\[ L_k:=\left( \frac{\int_\PP f(x)^k \d x}{\vol(P)} \right)^{1/k}\] and \[U_k := \left( \frac{\int_\PP f(x)^k \d x}{\vol(\PP)} \right)^{1/(d+k)}  \left(\frac{M\mathcal{L}}{\epsilon'}\right)^{d/(d+k)} \frac{1}{(1-\epsilon')^{k/(d+k)}}.
\]
Here, $M = 2$, $\epsilon' = \frac{2}{2+k}$, and $\mathcal{L} = 33$. Then for different values of $k$, the bounds are:
\begin{center}
\begin{tabular}{ c  c c }

  $k$ & $L_k$ & $U_k$\\
  \hline
  10 & 11.07 & 23.40  \\
  20 & 13.22 & 20.75  \\
  30 & 14.27 & 19.84  \\
  40 & 14.91 & 19.38  \\    
\end{tabular}
\end{center}

Next we want to apply Lemma \ref{lemma:epsilon-polynom-in-k} when $\epsilon=0.1$. Using $U = 19.38$, we see that $k$ has to be at least 
\[ k = \lceil \max\{-1.4, 62.0, 472.2,  434.1\}\rceil\] to guarantee that $U_k - L_k \leq \fmax \cdot \epsilon = 1.8$. However for this example   $k=126$ suffices: $U_{126} - L_{126} = 1.7$.

The integration of $f(x)^k$ can easily be done with the fundamental theorem of calculus because the domain is a box. For general rational full-dimensional polytopes, any method in Chapter \ref{ch:Integration} can be used.  

\subsection{Summary of the algorithm}

Notice that both Theorem \ref{thm:first-fptas} and \ref{theorem:continuous-lkuk-bounds} require the polynomial $f(x)$ to be nonnegative on the domain $\PP$. A shift, $s \in \R$, could be added to $f(x)$ so that $f(x) + s$ is nonnegative on $\PP$. We see that any $s$ such that  $s \geq -\fmin$ will work. 

To find such a shift $s$, we could follow the suggestion in \cite{deloera-hemmecke-koeppe-weismantel:mixedintpoly-fixeddim-fullpaper} and use linear programming to find the range of each variable $x_i$ and compute $f(x) \geq -rCM^D = -s$, where $r$ is the number of monomials of $f(x)$, $C$ is the largest absolute value of each coefficient in $f(x)$, $D$ is the total degree of $f(x)$, and $M$ is the largest bound on the variables. Another way is to compute a Handelman decomposition for $f(x)$ on $\PP$.

Algorithm \ref{alg:knorm-handelman} gives the final connection between optimizing a polynomial, computing integrals, and the Handelman decomposition. 

\begin{algorithm}
\caption{Computing $\int_\PP (f(x)+s)^k\d x$ via Handelman for Theorem \ref{theorem:continuous-lkuk-bounds}}
\label{alg:knorm-handelman}
\begin{flushleft}
Input: A polynomial $f(x) \in \Q[x_1, \dots, x_d]$ of degree $D$, polytope $\PP$, and $k$.

Output: $s$ such that $f(x)+s \geq 0$ on $\PP$, and $\int_\PP (f(x)+s)^k \d x$.
\end{flushleft}

\begin{enumerate}
\item Let $t \leftarrow D$
\item  Find $s \in \R$ and $c_\alpha \geq 0$ such that $f(x)+s = \sum_{|\alpha| \leq t} c_\alpha g^\alpha$ by solving a linear program
\item Expand $h(x) := (\sum_{|\alpha| \leq t} c_\alpha g^\alpha)^k$ to get a new sum of products of affine functions
\item Integrate each Handelman monomial in $h(x)$ by the methods in Section \ref{sec:integration-prducts-affine-functions}
\end{enumerate}
\end{algorithm}

\begin{theorem}
Fix the dimension $d$ and the number of facets $n$ in $\PP$. If $f(x)$ has a Handelman decomposition of order bounded by $D$, then Algorithm \ref{alg:knorm-handelman} runs in time polynomial in $k$, $D$, and the input size of $f$ and $\PP$.
\end{theorem}

\begin{proof}
Because $f(x)$ had a Handelman decomposition of order bounded by $D$, let $t=D$. The linear program that is solved has a matrix with dimension at most $O(D^d) \times O(D^n)$. Expanding $(\sum_{|\alpha| \leq t} c_\alpha g^\alpha)^k$ has at most ${ kD + n \choose n} = O((kD)^n)$ Handelman terms. Finally, each term can be integrated in time polynomial in $kD$. 
\end{proof}

Notice that Steps (2) and (3) in Algorithm \ref{alg:knorm-handelman} can be swapped, resulting in the same time complexity. For example, $h(x) := (f(x) +s)^k$ could first be expanded, and then a Handelman decomposition could be found for $h(x)$.

\chapter{Top coefficients of the Ehrhart polynomial for knapsack polytopes}
\label{ch:knapsack}

This chapter is mostly a highlight of \cite{baldoni-et-al:denumerant-full-paper}. Let $\a=[\alpha_1,\alpha_2,\ldots, \alpha_{N},\alpha_{N+1}]$ be a sequence of positive integers.  If $t$ is a non-negative integer, we denote by $E(\a; t)$ the number of solutions in non-negative integers of the equation \[\sum_{i=1}^{N+1} \alpha_i x_i=t.\]  In other words, $E(\a; t)$ is the same as the number of partitions of the number $t$ using the parts $\alpha_1,\alpha_2,\ldots, \alpha_{N},\alpha_{N+1}$ (with repetitions allowed). 

We will use the notation $f(a_1, \dots a_n; x_1, \dots, x_m)$ to stress that we will think of the $a_i$ as parameters to the function $f$ while the $x_i$ are the important variables. 
The combinatorial function $E(\a; t)$ was called by J.~Sylvester the \emph{denumerant}.  The denumerant $E(\a;t)$ has a beautiful structure: it has been known since the times of Cayley and Sylvester that $E(\a; t)$ is in fact a \emph{quasi-polynomial}, i.e., it can be written in the form $E(\a;  t)= \sum_{i=0}^{N} E_{i}(t)t^{i}$, where $E_i(t)$ is a periodic function of $t$. In other words, there exists a positive integer $Q$ such that for $t$ in the coset $q+Q\Z$, the function $E(\a;  t)$ coincides with a polynomial function of $t$.  This chapter presents an algorithm to compute individual coefficients of this function and explores their periodicity. Sylvester and Cayley first showed that the coefficients $E_i(t)$ are periodic functions having period equal to the least common  multiple of $\alpha_1,\ldots,\alpha_{d+1}$ (see \cite{beckgesselkomatsu,bellET} and references therein).  In 1943,  E.\,T.~Bell gave a simpler proof and remarked that the period $Q$ is  in the worst case given by the least common multiple of the $\alpha_i$,  but in general it can be smaller. A classical observation that goes  back to I.~Schur is that when the list $\a$ consist of relatively  prime numbers, then asymptotically

$$ E(\a;  t) \approx \frac{t^N}{N!\, \alpha_1\alpha_2\cdots \alpha_{N+1}} \quad \text{as the number} \ t \rightarrow \infty. $$

\begin{example}\label{example1}
Let $\a=[6,2,3].$  Then on each of the cosets $q+6\Z$, the  function
$E(\a;  t)$ coincides with a polynomial $E^{[q]}(t)$. Here are the
corresponding polynomials. 
\begin{align*}
E^{[0]}(t)&=\tfrac{1}{72}t^2+\tfrac{1}{4}t+1, &
E^{[1]}(t)&=\tfrac{1}{72}t^2+\tfrac{1}{18}t-\tfrac{5}{72}, \\
E^{[2]}(t)&=\tfrac{1}{72}t^2+\tfrac{7}{36}t+\tfrac{5}{9}, & 
E^{[3]}(t)&=\tfrac{1}{72}t^2+\tfrac{1}{6}t+\tfrac{3}{8}, \\
E^{[4]}(t)&=\tfrac{1}{72}t^2+\tfrac{5}{36}t+\tfrac{2}{9}, & 
E^{[5]}(t)&=\tfrac{1}{72}t^2+\tfrac{1}{9}t+\tfrac{7}{72}. 
\end{align*}

Then the number of nonnegative solutions in $x \in \Z^3$ to $6x_1 + 2x_2 + 3x_3 = 10$ is given by \[E^{[10 \bmod 6]}(10) = E^{[4]}(10) = 3.\]
\end{example}

It is interesting to note that the coefficients of the polynomial $E^{[t \bmod 6]}(t)$ may not be integer, but the evaluation will always be integer as they count integer solutions to an equation. Also, the coefficient of the highest degree (the $t^{N}$ term) is constant. In other examples, more  $t^m$ terms could be constant. 

Naturally, the function $E(\a;  t)$ is equal to $0$ if $t$ does not belong to the lattice $\sum_{i=1}^{N+1} \Z \alpha_i\subset \Z$ generated by the integers $\alpha_i$. Note that if $g$ is the greatest common divisor of the $\alpha_i$ (which can be computed in polynomial time), and $\a/g=[\frac{\alpha_1}{g},\frac{\alpha_2}{g},\ldots, \frac{\alpha_{N+1}}{g}]$ the formula $E(\a;  gt)=E(\a/g, t)$ holds, hence we assume that the numbers $\alpha_i$ span $\Z$ without changing the complexity of the problem. In other words, we assume that the greatest common divisor of the $\alpha_i$ is equal to $1$. With this assumption, there is a large enough integer $F$ such that for any $t \geq F$, $E(\a)(t)>0$ and  there is a largest $t$ for which $E(\a;  t)=0$.

The focus of this chapter is on developing the polynomial time algorithm that is stated in Theorem \ref{theo:knap-main}. 

\begin{theorem}
\label{theo:knap-main}
Given any fixed integer~$k$, there is a polynomial time algorithm to compute the highest $k+1$ degree
terms of the quasi-polynomial $E(\a; t)$, that is
$$\mathrm{Top}_kE(\a; t)=\sum_{i=0}^k E_{N-i}(t) t^{N-i}.$$
The coefficients are recovered as \emph{step polynomial} functions of $t$.
\end{theorem}

As noted in Section \ref{sec:bg:knap}, computing $E(\a;  t)$ for a given~$t$, is $\#P$-hard, and computing the polynomial $E^{[q]}(t)$ is NP-hard. Despite this difficulty, when $N+1$ is fixed, the entire quasi-polynomial $E(\a;  t)$ can be computed in polynomial time \cite{barvinokwood, kannanfrobenius}. The above theorem instead allows $N+1$ to vary but fixes $k$, computing just the top $k$ terms of the quasi-polynomial $E(\a;  t)$ in polynomial time.

 Note that the number~$Q$ of cosets for $E(\a;  t)$ can be exponential in the binary encoding size of the problem, and thus it is impossible to list, in polynomial time, the polynomials~$E^{[q]}(t)$ for all the cosets~$q+Q\Z$.  That is why to obtain a polynomial time algorithm, the output is presented in the format of \emph{step polynomials}, 
which we define next. 

  \begin{enumerate}
  \item 
    Let $\fractional{s}:=s-\floor{s} \in [0,1)$ 
    for $s\in\R$, where 
    $\floor{s}$ denotes the largest integer smaller or equal
    to~$s$. The function $\fractional{s+1}=\fractional{s}$ is a periodic
    function of $s$ modulo $1$.
  \item If $r$ is rational with denominator $q$, the function $T\mapsto
    \fractional{rT}$ is a function of $T\in \R$ periodic modulo~$q$.  A
    function of the form $T\mapsto \sum_i c_i \fractional{r_iT}$ will be called a
    \emph{(rational) step linear function}.  If all the $r_i$ have a common denominator $q$,
    this function is periodic modulo~$q$.
   
  \item Then consider the algebra generated over~$\Q$ by
    such functions on~$\R$. An element $\phi$ of this algebra can be written
    (not in a unique way) as
    %%%$$\phi(T)=\sum _{J} c_J \prod_{j\in J}\fractional{r_{j} T}^{n_j}.$$
    $$ \phi(T) = \sum_{l=1}^L c_l \prod_{j=1}^{J_l} \fractional{r_{l,j} T}^{n_{l,j}}.$$
    Such a function $\phi(T)$ will be called a \emph{(rational) step polynomial}.
  \item We  will say that the step polynomial $\phi$ is of \emph{degree} (at most) $u$ if $\sum_j
    n_{l,j}\leq u$ for each index~$l$ occurring in the formula for
    $\phi$. As a side note, this notion of degree only induces a filtration, not a grading, on the algebra of step polynomials, because there exist
      polynomial relations between step linear functions and therefore several
      step-polynomial formulas with different degrees may represent the same
      function.
    We will say that $\phi$ is of \emph{period} $q$ if all the rational
    numbers $r_j$ have common denominator $q$.
  \end{enumerate}

In Example~\ref{example1}, instead of the $Q=6$ polynomials $E^{[0]}(t),
\dots, E^{[5]}(t)$ that we wrote down, we could write a single closed formula,
where the coefficients of powers of~$t$ are step polynomials in~$t$:
$$ {\frac {1}{72}}\,{t}^{2}+\left( \frac{1}{4}-\frac{\{-\frac{t}{3}\}}{6}-\frac{\{\frac{t}{2}\}}{6}\right) \, t + \left (1-\frac{3}{2}\, \{-\tfrac{t}{3}\}- \frac{3}{2}\,\fractional{\tfrac{t}{2}}+\frac{1}{2}\, \left( \fractional{-\tfrac{t}{3}} \right) ^{2}+ \fractional{-\tfrac{t}{3}}\fractional{\tfrac{t}{2}}+\frac{1}{2}\, \left( \fractional{\tfrac{t}{2}} \right) ^{2} \right).$$
For larger $Q$, one can see that this step polynomial representation is much more economical than writing
the individual polynomials for each of the cosets of the period~$Q$.

%%%%%%%%%%%%%%%%%%%%%%%%%%%%%%%%%%%%%%%%%%%%%%%%%
%%%%%%%%%%%%%%%%%%%%%%%%%%%%%%%%%%%%%%%%%%%%%%%%%
%%%%%%%%%%%%%%%%%%%%%%%%%%%%%%%%%%%%%%%%%%%%%%%%%
%%%%%%%%%%%%%%%%%%%%%%%%%%%%%%%%%%%%%%%%%%%%%%%%%
%%%%%%%%%%%%%%%%%%%%%%%%%%%%%%%%%%%%%%%%%%%%%%%%%
%%%%%%%%%%%%%%%%%%%%%%%%%%%%%%%%%%%%%%%%%%%%%%%%%
%%%%%%%%%%%%%%%%%%%%%%%%%%%%%%%%%%%%%%%%%%%%%%%%%

\section{The residue formula for $E(\a;  t)$}

Let us begin fixing some notation.
If $\phi(z)\,\d{z}$ is a meromorphic one form  on $\C$, with a pole at $z=\zeta$, we write
$$\Res_{z=\zeta}\phi(z)\,\d{z}=\frac{1}{2\pi i}\int_{C_\zeta} \phi(z) \,\d{z},$$ 
where $C_\zeta$ is a small circle around the pole $\zeta$.
 If $\phi(z)=\sum_{k\geq k_0} \phi_k z^k$ is a Laurent series in $z$, we denote by $\res_{z=0}$ the coefficient of $z^{-1}$ of $\phi(z)$.
  Cauchy's formula implies that
 $\res_{z=0} \phi(z)=\Res_{z=0} \phi(z)\,\d{z}$.

Let $\a=[\alpha_1,\alpha_2,\ldots, \alpha_{N+1}]$  be a list of integers. Define
 $$F(\a;  z):=\frac{1}{\prod_{i=1}^{N+1}(1-z^{\alpha_i})}.$$

 Denote by  $\mathcal P=\bigcup_{i=1}^{N+1}\{\,\zeta\in \C: \zeta^{\alpha_i}=1\,\}$
 the set of poles of the meromorphic function $F(\a)$ and  by $p(\zeta)$ the order of
 the pole $\zeta$ for $\zeta\in \mathcal P$.

Note that because the $\alpha_i$ have greatest common divisor~$1$, we
have $\zeta=1$ as a pole of order ${N+1}$, and the other poles have
order strictly smaller.

\begin{theorem}[Theorem 2.1 in \cite{baldoni-et-al:denumerant-full-paper}]
\label{theo:sum}
Let $\a=[\alpha_1,\alpha_2,\ldots, \alpha_{N+1}]$  be a list of  integers with 
greatest common divisor equal to $1$, and let  $$F(\a; z):=\frac{1}{\prod_{i=1}^{N+1}(1-z^{\alpha_i})}.$$
If $t$ is a non-negative integer, then
\begin{equation}
  E(\a; t)=-\sum_{\zeta\in \mathcal P} \Res_{z=\zeta} z^{-t-1}F(\a; z)\,\d{z}
  \label{eq:Ea-as-sum-of-residues}
\end{equation}
and the $\zeta$-term 
of this sum is a quasi-polynomial function of $t$ with degree less than or
equal to $p(\zeta)-1$. 

\end{theorem}

\subsection{Poles of high and low order}

Given an integer $0\leq k\leq N$,
we partition  the set of poles $\mathcal P$ in two disjoint sets according to
the order of the pole:
$$\mathcal P_{> N-k}=\{\,\zeta: p(\zeta)\geq N+1-k\,\},\qquad
\mathcal P_{\leq N-k} =\{\,\zeta: p(\zeta)\leq {N}-k\,\}.$$

\begin{example}\label{ex:poles}

\hfill
  \begin{enumerate}
  \item  Let $\a= [98, 59, 44, 100]$, so $N=3$, and let $k=1$.
    Then $\mathcal P_{> N-k}$ consists of poles of order greater than~$2$. Of
    course $\zeta=1$ is a pole of order $4$.  Note that $\zeta=-1$ is a pole of
    order $3$.  So $\mathcal P_{> N-k}=\{\,\zeta : \zeta^2=1\,\}$.
  %% \item Let $\a= [6, 2,3]$, so $N=2$, and let $k=1$.
  %%   Then $\mathcal P_{> N-k}=\mathcal P_{> 1}$ is the union of $\{\,\zeta :
  %%   \zeta^2=1\,\}$ and $\{\,\zeta : \zeta^3=1\,\}$.
  \item Let $\a= [6,2,2,3,3]$, so $N=4$, and let $k=2$.
    Let $\zeta_6 = \e^{2\pi i/6}$ be a primitive 6th root of unity. 
    Then $\zeta_6^6 = 1$ is a pole of order~$5$, $\zeta_6$ and $\zeta_6^5$ are poles of
    order~1, and $\zeta_6^2$, $\zeta_6^3 = -1$, $\zeta_6^4$ are poles of
    order~3. Thus $\mathcal P_{> N-k}=\mathcal P_{> 2}$ is the union of $\{\,\zeta :
    \zeta^2=1\,\} = \{-1,1\} $ and $\{\,\zeta : \zeta^3=1\,\} = \{
    \zeta_6^2, \zeta_6^4, \zeta_6^6 = 1\}$.
  \end{enumerate}
\end{example}

According to the disjoint decomposition $\mathcal P=\mathcal P_{\leq N-k}\cup \mathcal P_{> N-k}$, we write
\begin{align*}
E_{\mathcal P_{> N-k}}(t)&=-\sum_{\zeta\in \mathcal P_{> N-k}} \Res_{z=\zeta}
z^{-t-1}F(\a; z)\,\d{z} \\
\intertext{and}  
E_{\mathcal P_{\leq N-k}}(t)&=-\sum_{\zeta\in \mathcal P_{\leq N-k}} \Res_{z=\zeta} z^{-t-1}F(\a; z)\,\d{z}.
\end{align*}
The following proposition is a direct consequence of Theorem \ref{theo:sum}.
 \begin{proposition}
We have
$$E(\a; t)=E_{\mathcal P_{> N-k}}(t)+E_{\mathcal P_{\leq N-k}}(t),$$
where the function  $E_{\mathcal P_{\leq N-k}}(t)$ is a quasi-polynomial function in the variable $t$ of degree  strictly less than $N-k$.
\end{proposition}

Thus  for the purpose of computing ${\rm Top}_kE(\a; t)$ it is sufficient to compute the function
$E_{\mathcal P_{> N-k}}(t)$. Notice that in order to compute $E_{\mathcal P_{> N-k}}(t)$, we need to compute a residue for every $\zeta \in \mathcal P_{> N-k}$. It turns out many of these residue calculations can be grouped together. We will do this by looking at a poset structure of $\mathcal P_{> N-k}$ and by using generating functions. We discus this in the next two sections.

%This function is computable in polynomial
%time, as stated in the main result of our paper:
%\begin{theorem}\label{theo:complexity}
%  Let $k$ be a fixed number.  Then the coefficient functions of the quasi-polynomial function
%  $E_{\mathcal P_{> N-k}}(t)$ are computable in polynomial time as step
%  polynomials of~$t$.
%\end{theorem}
%
%We prove the theorem in the rest of this section and the next.

\section{Using the poset structure of the $\alpha_i$} \label{posetpoles}

We first rewrite our set ${\mathcal P_{> N-k}}$.  Note
that if $\zeta$ is a pole of order $\geq p$, this means that there exist at
least $p$ elements $\alpha_i$ in the list $\a$ so that $\zeta^{\alpha_i}=1$. But if
$\zeta^{\alpha_i}=1$ for a set~$I\subseteq\{1,\dots,N+1\}$ of indices $i$, this is
equivalent to the fact that $\zeta^f=1$, for $f$ the greatest common divisor
of the elements $\alpha_i, i\in I$.

Now let $\mathcal I_{> N-k}$ be the set of subsets of $\{1, \dots, N+1\}$ of cardinality greater than $N-k$. Note that when $k$ is fixed,
the cardinality of $\mathcal I_{> N-k}$ is a polynomial function of~$N$.
For each subset $I\in
\mathcal I_{> N-k}$, define $f_I$ to be the greatest common divisor of the corresponding sublist $\alpha_i$, $i\in I$.  Let $\CG_{>N-k}(\a)=\{\,f_I : I \in \mathcal I_{>
  N-k}\,\}$ be the set of integers so obtained and let $G(f)\subset\C^\times$ be the group of $f$-th roots of unity, 
$$G(f)=\{\,\zeta\in \C: \zeta^f=1\,\}.$$ The set $\{\, G(f) : f \in
\CG_{>N-k}(\a)\,\}$ forms a poset $\tilde{P}_{>N-k}$ (partially ordered set)
with respect to reverse inclusion.  That is, $G(f_i) \preceq_{\tilde{P}_{>N-k}} G(f_j)$ if $G(f_j) \subseteq G(f_i)$ (the $i$ and $j$ become swapped). Notice $G(f_j) \subseteq G(f_i) \Leftrightarrow f_j $ divides $f_i$. Even if $\tilde{P}_{>N-k}$ has a unique minimal element, we add an element $\hat{0}$ such that $\hat{0} \preceq G(f)$ and call this new poset $P_{>N-k}$.

In terms of the group $G(f)$ we have thus $\mathcal P_{>
  N-k}=\bigcup_{f\in \CG_{>N-k}(\a)} G(f)$. This is, of course, not a
disjoint union, but using the inclusion--exclusion principle, we can write
the indicator function of the set  $\mathcal P_{> N-k}$ as a linear combination
of indicator functions of the sets $G(f)$:
$$[\mathcal P_{> N-k}]=\sum_{f\in \CG_{>N-k}(\a)} \mu_{>N-k}(f) [G(f)],$$ where $\mu_{>N-k}(f) := -\mu'_{>N-k}(\hat{0},G(f))$ and $\mu'_{>N-k}(x,y)$ is the standard M\"obius function for the poset $P_{>N-k}$:
\begin{align*}
\mu'_{>N-k}(s,s) &= 1 && \forall s \in P_{>N-k}, \\
\mu'_{>N-k}(s,u) &= -\sum\limits_{s \preceq t \prec u} \mu'_{>N-k}(s,t) && \forall s \prec u\text{ in } P_{>N-k}.
\end{align*}
For simplicity,  $\mu_{>N-k}$ will be called the \emph{M\"obius function} for the poset $P_{>N-k}$ and will be denoted simply by $\mu(f)$. We also have the relationship
\begin{align*}
\mu(f) &= -\mu'_{>N-k}(\hat{0},G(f)) \\
&= 1 + \sum\limits_{\hat{0} \prec G(t) \prec G(f)} \mu'_{>N-k}(\hat{0},G(t))\\
&= 1 - \sum\limits_{\hat{0} \prec G(t) \prec G(f)} -\mu'_{>N-k}(\hat{0},G(t))\\
&= 1 - \sum\limits_{\hat{0} \prec G(t) \prec G(f)} \mu(t).
\end{align*}
\begin{example}[Example~\ref{ex:poles}, continued]~
  \begin{enumerate}
  \item Here we have $\mathcal I_{>N-k} = \mathcal I_{>2} = \bigl\{
    \{\oldstylenums1, \oldstylenums2, \oldstylenums3\},
    \{\oldstylenums1,\oldstylenums2,\oldstylenums4\},
    \{\oldstylenums1,\oldstylenums3,\oldstylenums4\},
    \{\oldstylenums2,\oldstylenums3,\oldstylenums4\},\allowbreak
    \{\oldstylenums1,\oldstylenums2,\oldstylenums3,\oldstylenums4\} \bigr\}$ 
    %%%% \oldstylenums for indices
    and $\mathcal G_{>N-k}(\a) = \{ 1, 1, 2, 1, 1\} = \{1,2\}$.  Accordingly, 
    $\mathcal P_{> N-k} = G(1) \cup G(2)$.  The poset $P_{>2}$ is
    
    \begin{figure}[ht]
        \centering
    \includegraphics{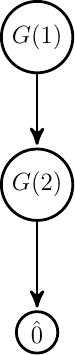}
%\begin{tikzpicture}[scale=0.50,transform shape,->,>=stealth',shorten >=1pt,auto,node distance=3cm,
%	  thick,main node/.style={circle,draw,font=\sffamily\Large\bfseries}]
%	  \node[main node] (1) {$G(1)$};
%	  \node[main node] (2) [below of =1] {$G(2)$};
%	  \node[main node] (3) [below  of =2] {$\hat{0}$};
%	
%	  \path[every node/.style={font=\sffamily\small}]
%		 (1) edge node  {} (2)
%		(2) edge node {} (3);
%\end{tikzpicture}
\end{figure}

    The arrows denote subsets, that is $G(1) \subset G(2)$ and $\hat{0}$ can be identified with the unit circle. The M\"obius function~$\mu$ is simply given by $\mu(1) = 0$, $\mu(2) = 1$, and so
    $[\mathcal P_{>N-k}] = [G(2)]$.
  %% \item We have $\mathcal I_{>N-k} = \mathcal I_{>1} = \bigl\{
  %%   \{\oldstylenums1, \oldstylenums2\},
  %%   \{\oldstylenums1,\oldstylenums3\},
  %%   \{\oldstylenums2,\oldstylenums3\},
  %%   \{\oldstylenums1,\oldstylenums2,\oldstylenums3\} \bigr\}$ 
  %%   %%%% \oldstylenums for indices
  %%   and thus $\mathcal G_{>N-k}(\a) = \{ 2, 3, 1, 1 \} \allowbreak =
  %%   \allowbreak \{1,2,3\}$. Hence
  %%   $\mathcal P_{> N-k} = G(1) \cup G(2)\cup G(3) = \{1\} \cup \{-1,1\} \cup
  %%   \{\zeta_3, \zeta_3^2, 1\}$, where $\zeta_3 = \e^{2\pi i/3}$ is a primitive
  %%   3rd root of unity. The patching function~$\mu$ is then $\mu(3)=1$, $\mu(2)=1$,
  %%   $\mu(1)=-1$, and thus 
  %%   $[\mathcal P_{>N-k}] = -[G(1)] + [G(2)] + [G(3)]$. 
  \item Now $\mathcal I_{>N-k} = \mathcal I_{>2} = \bigl\{
    \{\oldstylenums1, \oldstylenums2, \oldstylenums3\},
    \{\oldstylenums1,\oldstylenums2,\oldstylenums4\},\dots,\allowbreak
    \{\oldstylenums3,\oldstylenums4,\oldstylenums5\},\allowbreak
    \{\oldstylenums1, \oldstylenums2, \oldstylenums3, \oldstylenums4\},\allowbreak
    \{\oldstylenums1, \oldstylenums2, \oldstylenums3, \oldstylenums5\},\allowbreak
    \{\oldstylenums1,\oldstylenums2,\oldstylenums4, \oldstylenums5\},\allowbreak
    \{\oldstylenums1,\oldstylenums3,\oldstylenums4, \oldstylenums5\},\allowbreak
    \{\oldstylenums2,\oldstylenums3,\oldstylenums4, \oldstylenums5\},\allowbreak
    \{\oldstylenums1,\oldstylenums2,\oldstylenums3,\oldstylenums4,\oldstylenums5\}
    \bigr\}$  
    %%%% \oldstylenums for indices
    and thus $\mathcal G_{>N-k}(\a) = \{ 2, 3, 1, 1 \} \allowbreak =
    \allowbreak \{1,2,3\}$. Hence
    $\mathcal P_{> N-k} = G(1) \cup G(2)\cup G(3) = \{1\} \cup \{-1,1\} \cup
    \{\zeta_3, \zeta_3^2, 1\}$, where $\zeta_3 = \e^{2\pi i/3}$ is a primitive
    3rd root of unity. 
    
    \begin{figure}[ht]
    \centering
    \includegraphics{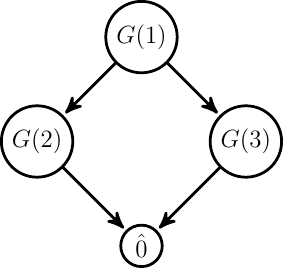}
%\begin{tikzpicture}[scale=0.50,transform shape,->,>=stealth',shorten >=1pt,auto,node distance=3cm,
%	  thick,main node/.style={circle,draw,font=\sffamily\Large\bfseries}]
%	  \node[main node] (1) {$G(1)$};
%	  \node[main node] (2) [below left of =1] {$G(2)$};
%	  \node[main node] (3) [below right of =1] {$G(3)$};
%	  \node[main node] (4) [below  right of =2] {$\hat{0}$};
%	
%	  \path[every node/.style={font=\sffamily\small}]
%		 (1) edge node  {} (2)
%		      edge node {} (3)
%		(2) edge node {} (4)
%		(3) edge node {} (4);
%\end{tikzpicture}
\end{figure}    
    
    The M\"obius function~$\mu$ is then $\mu(3)=1$, $\mu(2)=1$,
    $\mu(1)=-1$, and thus 
    $[\mathcal P_{>N-k}] = -[G(1)] + [G(2)] + [G(3)]$.
  \end{enumerate}
\end{example}

%For fixed $k$, all the data above can be computed in polynomial time in
%function of the data $\a$. The greatest common divisor of a set of
%integers is computed in polynomial time.  Finally the M\"obius function
%$\mu(f)$ is computed in polynomial time, because there are polynomially many levels of the poset being considered.

\begin{theorem}
\label{theorem:knap:mobius-function}
Given a list $\a=[\alpha_1, \dots, \alpha_{N+1}]$ and a fixed integer $k$, then the values for the M\"obius function for the poset $P_{>N-k}$ can be computed in polynomial time.
\end{theorem}
\begin{proof}
First find the greatest common divisor  of all sublists of  the list $\a$ with size greater than $N-k$. Let $V$ be the set of integers obtained from all such greatest common divisors. We note that each node of the poset $P_{>N-k}$ is a group of roots of unity $G(v)$. But it
is labeled by a non-negative integer $v$.

Construct an array $M$ of size $|V|$ to keep the value of the M\"obius function. Initialize $M$ to hold the M\"obius values of infinity: $M[v] \leftarrow \infty$ for all $v \in V$. Then call Algorithm \ref{alg:findMobius} below with $\text{findM\"obius}(1,V,M)$.

\begin{algorithm}   
\caption{ findM\"obius($n$, $V$, $M$)}

\begin{algorithmic}[1]                    
\REQUIRE $n$: the label of node $G(n)$ in the poset $\tilde{P}_{>N-k}$
\REQUIRE $V$: list of numbers in the poset $\tilde{P}_{>N-k}$
\REQUIRE $M$: array of current M\"obius values computed for $P_{>N-k}$
\ENSURE updates the array $M$ of M\"obius values
\IF{$M[n] < \infty$}
	\RETURN
\ENDIF
\STATE $L \leftarrow \{\, v \in V  : n \mid v\,\} \setminus \{n\}$
\IF{ $L = \emptyset $}
	\STATE $M[n] \leftarrow 1$
	\RETURN
\ENDIF
\STATE $M[n] \leftarrow 0$
\FORALL {$v \in L$}
	\STATE findM\"obius($v, L, M$)
	\STATE $M[n] \leftarrow M[n] + M[v]$
\ENDFOR
\STATE $M[n] \leftarrow 1- M[n]$

\end{algorithmic}
\label{alg:findMobius}
\end{algorithm}

Algorithm \ref{alg:findMobius} terminates because the number of nodes $v$ with $M[v] = \infty$ decreases to zero in each iteration. 
To show correctness, consider a node $v$ in the poset $P_{N-k}$. If $v$ covers $\hat{0}$, then we must have $M[v] = 1$ as there is no other $G(w)$ with  $G(f) \subset G(w)$. Else if $v$ does not cover $\hat{0}$, we set $M[v]$ to be 1 minus the sum $\sum\limits_{w :\; v \mid w} M[w]$ which guarantees that the poles in $G(v)$ are only counted once because $\sum\limits_{w :\; v \mid w} M[w]$  is how many times $G(v)$ is a subset of another element that has already been counted.

The number of sublists of $\a$ considered is $\binom{N}{1}+\binom{N}{2}+\cdots+\binom{N}{k} = O(N^k)$, which is a polynomial for $k$ fixed. For each sublist, the greatest common divisor of a set of integers is computed in polynomial time. Hence $|V| = O(N^k)$. Notice that lines $4$ to $14$ of Algorithm \ref{alg:findMobius} are executed at most $O(|V|)$ times as once a $M[v]$ value is computed, it is never recomputed. The number of additions on line $12$ is $O(|V|^2)$ while the number of divisions on line $4$ is also $O(|V|^2)$. Hence this algorithm finds the M\"obius function in $O(|V|^2) = O(N^{2k})$ time where $k$ is fixed.
\end{proof}

Let us define  for  any positive integer $f$  $$E(\a, f; t)=-\sum_{\zeta:\  \zeta^f=1}  \Res_{z=\zeta} z^{-t-1} F(\a; z)\,\d{z}.$$

\begin{proposition}\label{knap-poset-prop}
  Let $k$ be a fixed integer, then 
\begin{equation}\label{eq:Emu}
   E_{\mathcal P_{> N-k}}(t)=\sum_{f\in \CG_{>N-k}(\a)}
     \mu(f)  E(\a,f; t).
   \end{equation}
\end{proposition}

\subsection{Partial summary so far}
As a partial summary, for a fixed integer $k$, our original goal is to compute \[\mathrm{Top}_kE(\a; t)=\sum_{i=0}^k E_{N-i}(t) t^{N-i}, \] in polynomial time where $E_{N-i}(t)$ are step polynomials. So far we have shown that this is equal to $E_{\mathcal P_{> N-k}}(t)$, which in tern is equal to $E_{\mathcal P_{> N-k}}(t)=\sum_{f\in \CG_{>N-k}(\a)} \mu(f)  E(\a,f; t).$ The set $\CG_{>N-k}(\a)$ and the values $\mu(f)$ for $f\in \CG_{>N-k}(\a)$ can be computed in polynomial time in the input size of $\a$ when $k$ is fixed. The next section illustrates how to compute $E(\a,f; t)$ efficiently.  

%%%%%%%%%%%%%%%%%%%%%%%%%%%%%%%%%%%%%%%%%%%%%%%%%
%%%%%%%%%%%%%%%%%%%%%%%%%%%%%%%%%%%%%%%%%%%%%%%%%
%%%%%%%%%%%%%%%%%%%%%%%%%%%%%%%%%%%%%%%%%%%%%%%%%
%%%%%%%%%%%%%%%%%%%%%%%%%%%%%%%%%%%%%%%%%%%%%%%%%
%%%%%%%%%%%%%%%%%%%%%%%%%%%%%%%%%%%%%%%%%%%%%%%%%
%%%%%%%%%%%%%%%%%%%%%%%%%%%%%%%%%%%%%%%%%%%%%%%%%
%%%%%%%%%%%%%%%%%%%%%%%%%%%%%%%%%%%%%%%%%%%%%%%%%

\section{Polyhedral  reinterpretation of the generating function $E(\a,f; t)$}

To complete the proof of Theorem \ref {theo:knap-main} we need only to prove the following proposition.

\begin{proposition}
   For any integer $f\in \CG_{>N-k}(\a)$, the coefficient functions of the
     quasi-polynomial function $E(\a,f; t)$ and hence $E_{\mathcal P_{>
         N-k}}(t)$ are computed in polynomial time as step polynomials of~$t$.
\end{proposition}

By Proposition \ref{knap-poset-prop} we know we need to compute the value of $E(\a,f; t)$. Our goal now is to 
demonstrate that this function can be thought of as the generating function of the lattice points inside
a convex cone. This is a key point to guarantee good computational bounds. Before we can do that
we review some preliminaries on generating functions of cones. We recall the notion of generating 
functions of cones; see also \cite{so-called-paper-1}.

\subsection{Review of exponential summation with real shift $\ve s$}
We will define a \emph{lattice} as the subgroup of all linear combinations of a basis of $\R^r$ with integer coefficients.  This material is similar to Section \ref{ch:bg:sec:integration-stuff}, but for a general lattice $\Lambda$ instead of just $\Z^r$. 

Let $V=\R^r$ provided with a lattice $\Lambda$, and let $V^*$ denote the dual
space. A \emph{(rational) simplicial cone}  $\coneC{} =  \R_{\geq0}\ve w_1+\dots+\R_{\geq0}\ve w_r$ is
a cone generated by $r$ linearly independent vectors~$\ve w_1,\dots,\ve w_r$ of~$\Lambda$.
We consider the semi-rational affine cone $\ve s+\coneC{}$, $\ve s\in V$. %The cone   $\coneC{}$ is a (rational) simplicial  cone.
Let $\vexi\in V^*$ be a dual vector such that $\ll\vexi,\ve w_i\rr <0, \ 1\leq i\leq r.$ Then the sum
$$S(\ve s+\coneC{},\Lambda; \vexi)=\sum_{\ve n \in   (\ve s+\coneC{})\cap\Lambda}
\e^{\langle \vexi,\ve n\rangle}$$ is summable and defines an analytic function of $\vexi$.  It is well known that
this function extends to a meromorphic function of $\vexi\in V^*_\C$. We still
denote this meromorphic extension by $S(\ve s+\coneC{},\Lambda; \vexi)$.

\begin{example}\label{ex:dim1}
  Let $V=\R$  with lattice $\Z$,  $\coneC{}=\R_{\geq 0}$, and $s\in  \R$.
  Then $$S(s+\R_{\geq0},\Z; \xi)=\sum_{n\geq s} \e^{n \xi}=\e^{\ceil{s}\xi}\frac{1}{1-\e^{\xi}}.$$
  Using the function 
  $\fractional{x}=x - \floor{x}$, we find $\ceil{s} = s + \fractional{-s}$ and
  can write
  \begin{equation}\label{eq:dim1}
    \e^{-s\xi}S(s+\R_{\geq0},\Z; \xi)=\frac{\e^{\fractional{-s}\xi}}{1-\e^{\xi}}. 
  \end{equation}
\end{example}

Recall the following result:

\begin{theorem}
 Consider the semi-rational affine cone $\ve s+\coneC{}$ and the lattice $\Lambda$. The series $S(\ve s+\coneC{},\Lambda; \vexi)$ is a meromorphic function of $\vexi$  such that $\prod_{i=1}^r \ll \vexi,\ve w_i\rr  \cdot\allowbreak S(\ve s+\coneC{},\Lambda; \vexi)$ is
  holomorphic in a neighborhood of $\ve0$.
\end{theorem}

Let $\ve t\in \Lambda$. Consider the translated cone $\ve t+\ve
s+\coneC{}$ of
$\ve s+\coneC{}$ by $\ve t$. Then we have the covariance formula
\begin{equation}\label{eq:covariance}
  S(\ve t+\ve s+\coneC{},\Lambda; \vexi)=\e^{\ll\vexi,\ve t\rr }  S(\ve s+\coneC{},\Lambda; \vexi).
\end{equation}

%% Thus the function $s\mapsto \e^{-\ll s,\vexi\rr } S(s+\coneC{},\Lambda; \vexi)$ is a
%% function on $V/\Lambda$ whose values are meromorphic functions of~$\vexi$.

Because of this formula, 
it is convenient to introduce the following function.
\begin{definition} \label{def:useful} 
  Define the function $$M(\ve s,\coneC{},\Lambda; \vexi):=\e^{-\ll \vexi, \ve s\rr }
  S(\ve s+\coneC{},\Lambda; \vexi).$$
\end{definition}
Thus the function $\ve s\mapsto M(\ve s,\coneC{},\Lambda; \vexi)$ is a
function of $\ve s\in V/\Lambda$ (a periodic function of $\ve s$) whose values are
meromorphic functions of $\vexi$.
It is interesting to introduce this modified function since, as seen in
Equation \eqref{eq:dim1} in Example~\ref{ex:dim1}, its dependence in  $\ve s$
is via step linear functions of $\ve s.$ 

There is a very special and important case when the function $M(\ve s,\coneC{},\Lambda; \vexi)=\e^{-\ll \vexi, \ve s\rr }
  S(\ve s+\coneC{},\Lambda; \vexi)$ is easy to write down. A \emph{unimodular} cone, is a
cone $\coneU$ whose primitive generators $\ve g_i^\coneU$ form a basis of the
lattice $\Lambda$.  We introduce the following notation.
\begin{definition}
  Let $\coneU{}$ be a unimodular cone with primitive generators $\ve g_i^\coneU$ and
  let $\ve s\in V$. 
  Then, write  $\ve s=\sum_i s_i \ve g_i^\coneU$, with $s_i\in \R$, and         define
  $$\smallstep{\ve s}_\coneU{}=\sum_i \fractional{-s_i} \ve g_i^\coneU.$$
\end{definition}
Thus 
$\ve s+\smallstep{\ve s}_\coneU{}=\sum_i \ceil{s_i} \ve g_i^\coneU$.  Note that if $\ve t\in \Lambda$, 
then $\smallstep{(\ve s+\ve t)}_\coneU{}=\smallstep{\ve s}_\coneU{}$. Thus,
$\ve s\mapsto
\smallstep{\ve s}_\coneU{}$ is a function on $V/\Lambda$  with value in $V$. 
For any $\vexi\in V^*$, we then find
\begin{equation*}
S(\ve s+\coneU,\Lambda; \vexi)= \e^{\ll \vexi,\ve
  s\rr}\e^{\ll\vexi,\smallstep{\ve s}_\coneU\rr}\frac{1}{\prod_j(1-\e^{\ll \vexi,\ve g_j^\coneU\rr})}
\end{equation*}
and thus
\begin{equation}\label{eq:M-uni}
M(\ve s,\coneU,\Lambda; \vexi)= \e^{\ll\vexi,\smallstep{\ve
    s}_\coneU\rr}\frac{1}{\prod_j(1-\e^{\ll \vexi,\ve g_j^\coneU\rr})}. 
\end{equation}

For a general cone $\coneC{}$,
we can decompose its indicator function $[\coneC{}]$ as a signed sum of
indicator functions of unimodular cones, 
$\sum_\coneU \epsilon_\coneU [\coneU]$, modulo indicator functions of
cones containing lines. As shown 
by Barvinok (see \cite{bar} for the original source and \cite{barvinokzurichbook} for a great new exposition),  
if the dimension~$r$ of~$V$ is fixed, this decomposition can be computed in polynomial time.
Then we can write 
$$S(\ve s+\coneC{},\Lambda; \vexi)=\sum_\coneU \epsilon_\coneU\, S(\ve s+\coneU,\Lambda; \vexi).$$
Thus we obtain, using Formula (\ref{eq:M-uni}), 
\begin{equation}\label{formula:M}
M(\ve s, \coneC{},\Lambda; \vexi)=\sum_{\coneU} \epsilon_\coneU\, \e^{ \ll
  \vexi,\smallstep{\ve s}_\coneU\rr } \frac{1}{\prod_j (1-\e^{\ll \vexi,\ve g_j^\coneU\rr })}.
\end{equation}
Here $\coneU$ runs through all the unimodular cones occurring in the decomposition of $\coneC{}$, and
the $\ve g_j^\coneU\in \Lambda$ are the corresponding generators of  the unimodular cone $\coneU.$

\begin{remark}\label{rem:change-to-standard-lattice}
For computing explicit examples, it is convenient to make a change of variables that
leads to computations in the standard lattice~$\Z^r$.  Let $B$ be the matrix
whose columns are the generators of the lattice~$\Lambda$; then
$\Lambda=B\Z^r$.  
\begin{align*}
  M(\ve s,\coneC{},\Lambda; \vexi)
  &= \e^{-\ll \vexi, \ve s\rr } \sum_{\ve n \in (\ve s+\coneC{})\cap B\Z^r}
  \e^{\langle \vexi,\ve n\rangle} \\
  &= \e^{-\ll B^\T \vexi, B^{-1} \ve s\rr} \sum_{\ve x \in  B^{-1}(\ve
    s+\coneC{})\cap\Z^r} \e^{\langle B^\T\vexi,\ve x\rangle} 
  = M( B^{-1}\ve s,B^{-1} \coneC{}, \Z^r; B^\T\vexi).
\end{align*}
\end{remark}

\subsection{Rewriting $E(\a,f; t)$}
Currently, we have
\begin{align*}
\mathrm{Top}_kE(\a; t)&= E_{\mathcal P_{> N-k}}(t) = \sum_{i=0}^k E_{N-i}(t) t^{N-i}\\
&=\sum_{f\in \CG_{>N-k}(\a)} \mu(f)  E(\a,f; t) \\
&= \sum_{f\in \CG_{>N-k}(\a)} \mu(f)  \left(-\sum_{\zeta:\  \zeta^f=1}  \Res_{z=\zeta} z^{-t-1} F(\a; z)  \right).
\end{align*}

In the next few sections, $E(\a,f; t)$ will be rewritten in terms of lattice points of simplicial cones. This will require some suitable manipulation of the initial form of $E(\a,f; t)$. 

To start with, we want to write the Ehrhart polynomial $\sum_{i=0}^k E_{N-i}(t)t^{N-i}$ as a function of two variables $t$ and $T$: $\sum_{i=0}^k E_{N-i}(T)t^{N-i}$. That is, we use $T$ to denote the variable of the periodic coefficients and $t$ to be the variable of the polynomial. To do this, define the  function
\[\CE(\a,f; t,T)=-\res_{z=\zeta} z^{-t-1} \zeta^{t}\sum_{\zeta \colon   \zeta^f=1} \frac{\zeta^{-T}}{\prod_{i=1}^{N+1} (1-z^{\alpha_i})}.\]
Notice that the $T$ variable is periodic modulo $f$. By evaluating  at $T=t,$ we obtain
\begin{equation}\label{eval}
E(\a,f; t)=\CE(\a,f; t,T) \big|_{T=t}.
\end{equation}
It will be helpful to perform a change of variables and let $z=\zeta \e^x,$ then changing coordinates in residue and computing $\d{z}=z\, \d{x}$  we have that

$$\CE(\a,f; t,T)=-\res_{x=0} \e^{-tx} \sum_{\zeta:\  \zeta^f=1} \frac{\zeta^{-T}}{\prod_{i=1}^{N+1} (1-\zeta^{\alpha_i}\e^{\alpha_i x})}.$$

\begin{definition}\label{FE} Let $k$ be fixed.
For $f\in \mathcal G_{>N-k}(\a)$,   define
$$\CF(\a,f,T; x):=\sum_{\zeta:\  \zeta^f=1} \frac{\zeta^{-T}}{\prod_{i=1}^{N+1} (1-\zeta^{\alpha_i}\e^{\alpha_i x})},$$
and 
$$E_i(f; T):=-\res_{x=0}\frac{(- x)^i}{i!} \CF(\a,f,T; x).$$
\end{definition}
Then $$\CE(\a,f; t,T)=-\res_{x=0} \e^{-tx} \CF(\a,f,T; x).$$

The  dependence in $T$ of $\CF(\a,f,T; x)$ is through $\zeta^T$. As
$\zeta^f=1$,  the function  $\CF(\a,f,T; x)$ is a periodic function of $T$
modulo $f$ whose values are meromorphic functions of $x$.  Since the pole in $x$ is of order at most $N+1$, we can rewrite $\CE(\a,f; t,T)$ in terms of  $E_i(f; T)$ and prove:

\begin{theorem}\label{E} Let $k$ be fixed.
Then for  $f\in \mathcal G_{>N-k}(\a)$  we  can write
$$\CE(\a,f; t,T)=\sum_{i=0}^N  t^i E_i(f; T)$$ with $E_i(f; T)$  a step
polynomial  of degree  less than or equal to $N-i$ and  periodic  of $T$ modulo
$f$.  This step polynomial can be computed in polynomial time.
\end{theorem}

It is now clear that once we have proved Theorem \ref{E}, then the proof of Theorem \ref {theo:knap-main} will follow. Writing everything out, for  $m$ such that $0\leq m\leq N$, the coefficient of  $t^{m}$ in the Ehrhart quasi-polynomial is given by
\begin{equation}\label{eq:Ehrhartcoeff}
 E_m(T)= - \res_{x=0}\frac{(-x)^{m}}{m!}\sum_{f\in \CG_{>m}(\a) } \mu(f)\sum_{\zeta: \  \zeta^f=1}\frac{\zeta^{-T}}{\prod_i(1-\zeta^{\alpha_i}\e^{\alpha_i x})}.
\end{equation}
As an example, we see that $E_N$ is indeed independent of $T$ because
$\CG_{>N}(\a) = \{1\}$; thus  $E_N$ is a constant. 
We now concentrate on writing the function 
$\CF(\a,f,T; x)$ more explicitly.

\begin{definition}\label{def:H}
For a list $\a$ and integers $f$ and $T$, define   meromorphic functions of $x\in \C$  by:
$${\mathcal B}(\a,f; x):=\frac{1}{\prod_{i\colon f\mid \alpha_i}(1-\e^{\alpha_i x})},$$
$${\mathcal S}(\a,f,T; x):=\sum_{\zeta:  \  \zeta^f=1} \frac{\zeta^{-T}}{\prod_{i : f\nmid \alpha_i} (1-\zeta^{\alpha_i}\e^{\alpha_i x})}.$$
\end{definition}
Thus we have
$$\CF(\a,f,T; x)={\mathcal B}(\a,f; x)\, {\mathcal S}(\a,f,T; x),$$
and
\begin{align*}
\sum_{i=0}^k E_{N-i}(T) t^{N-i}
&=\sum_{f\in \CG_{>N-k}(\a)} \mu(f)  \CE(\a,f; t,T) \\
&= \sum_{f\in \CG_{>N-k}(\a)} \mu(f)  \left(-\res_{x=0} \e^{-tx} {\mathcal B}(\a,f; x)\cdot {\mathcal S}(\a,f,T; x)\right).
\end{align*}

To compute the residue, we need to compute the series expansion about $x=0$. This can be done by multiplying the series expansions of $e^{-tx}$, ${\mathcal B}(\a,f; x),$ and ${\mathcal S}(\a,f,T; x)$. The series expansion of ${\mathcal B}(\a,f; x)$ is easy to do because it is related to the Bernoulli numbers, see Remark \ref{remark:knap:bernoulli}. We do not want to work with ${\mathcal S}(\a,f,T; x)$ in its current form because it contains (possibly irrational) roots of unity.

In the next section, the expression we obtained will allow us to compute $\CF(\a,f,T)$ by relating ${\mathcal S}(\a,f,T)$ to a generating function of a cone.  This cone will have fixed dimension when $k$ is fixed.

\begin{remark}
\label{remark:knap:bernoulli}
Let $x$ and $\epsilon$ be two variables, and $a, b \in \R$. In particular, $a$ or $b$ could be zero, but not at the same time. Then the Laurent series expansion of $\frac{1}{1 - \e^{ax + b\epsilon}}$ can be computed using the generating function for the Bernoulli numbers. Note that

\[\frac{1}{1 - \e^{ax + b\epsilon}} = \frac{ax + b\epsilon}{1 - \e^{ax + b\epsilon}} \times \frac{1}{(ax + b\epsilon)}.\]
The series expansion of the first term in the product can be done using Lemma \ref{lemma:bernoulli-numbers}. If $a \neq 0$, the second term can be expanded using the Binomial theorem:

\[\frac{1}{(ax + b\epsilon)} = \sum_{i=0}^\infty (ax)^{-1-k}(b\epsilon)^k\]
\end{remark}

\subsection{${\mathcal S}(\a,f,T; x)$ as the generating function of a cone in fixed dimension}

To this end, let $f$ be an integer from $\mathcal G_{>N-k}(\a)$.  By
definition, $f$ is the greatest common divisor of a sublist of $\a$.  Thus
the greatest common divisor of $f$ and the elements of $\a$ which are
\emph{not} a multiple of $f$ is still equal to~$1$.
Let $J=J(\a,f)$ be the set of indices $i\in\{1,\dots,N+1\}$ such that $\alpha_i$ is
indivisible by~$f$, i.e., $f \nmid \alpha_i$.  Note that 
$f$ by definition is the greatest common divisor of all except at most $k$
of the integers $\alpha_j$.  Let $r$
denote the cardinality of~$J$; then $r\leq k$. Let $V_J=\R^J$ and let $V_J^*$ denote
the dual space. We will use the standard basis of $\R^J,$ and we denote by $\R^J_{\geq 0}$ the standard cone of elements in $\R^J$ having non-negative coordinates.  We also define the sublist $\a_J = [\alpha_i]_{i\in J}$ of elements
of~$\a$ indivisible by~$f$ and view it as a vector in $V_J^*$ via the standard basis. 
\begin{definition}
  For an integer $T$, define the meromorphic function of
  $\vexi\in V_J^*$, 
  $$Q({\a},f,T; \vexi):=\sum_{\zeta:\  \zeta^{f}=1}
  \frac{\zeta^{-T}}{\prod_{j\in J(\a,f)} (1-\zeta^{\alpha_j}\e^{\xi_j})}.$$
\end{definition}
\begin{remark}
  \label{rem:restrict-to-SafT}
  Observe that $Q({\a},f,T)$ can be restricted at $\vexi=\a_J x$,
  for $x\in\C$ generic, to give $\mathcal S(\a,f,T; x).$
\end{remark}

%% We will prove that $Q({\a},f,T; \vexi)$ is the generating function of a cone and more importantly that it restrict with some care to $S(\a,f)$. This will make the link we need to prove Theorem  \ref {theo:complexity}  and will be achieved  defining  the appropriate lattice. The last section  contains a   summary of all the steps.

We  find that $Q({\a},f,T; \vexi)$ is the discrete generating function of  an affine shift of the standard cone $\R^J_{\geq 0}$ relative to a certain lattice in $V_J$
which we define as:
\begin{equation}\label{lattice}
  \Lambda({\a},f):=\biggl\{\,\ve y \in \Z^J : \langle \a_J, \ve y\rangle = \sum_{j\in J} y_j\alpha_j\in\Z f\,\biggr\}. 
\end{equation}%\end{definition}
Consider the map $\phi\colon \Z^J \to \Z/\Z f$, $\ve y\mapsto \langle \a, \ve
y\rangle + \Z f$.  Its kernel is the lattice~$\Lambda(\a,f)$. Because the
greatest common divisor of $f$ and the elements of $\a_J$ is~$1$, by Bezout's
theorem there exist $s_0\in\Z$ and $\ve s\in\Z^J$ such that $1=\sum_{i\in J}
s_i \alpha_i+s_0 f$.  Therefore, the map~$\phi$ is surjective, and therefore the
index $|\Z^J:\Lambda(\a,f)|$ equals~$f$.

\begin{theorem}\label{th:as-lattice-genfun}
  Let $\a=[\alpha_1,\dots,\alpha_{N+1}]$ be a list of positive integers and $f$ be the
  greatest common divisor of a sublist of~$\a$.  
  Let $J=J(\a,f) = \{\, i%\in\{1,\dots,N+1\} 
  : f \nmid \alpha_i \,\}.$ 
  %%%Assuming that the greatest common divisor of $f$ and the $\alpha_i$ is~$1$.\sum_{\ve n \in   (\ve s+\coneC{})\cap\Lambda}
  %%%%%% (this follows from $f$ gcd of sublist!)
  Let $s_0\in\Z$ and $\ve s\in\Z^J$ such that $1=\sum_{i\in J} s_i
  \alpha_i+s_0 f$ using Bezout's theorem.
  Consider $\ve s=(s_i)_{i\in J}$ as an element of $V_J = \R^J.$ Let $T$ be an integer,  and  $\vexi=(\xi_i)_{i\in J} \in V_J^*$ with  $ \xi_i <0.$ Then $$Q({\a},f,T; \vexi)=f\, \e^{\ll \vexi,T \ve s\rr }\sum_{\ve n \in   (-T\ve s+ \R_{\geq 0}^J)\cap\Lambda(\a,f)}\e^{\langle \vexi,\ve n\rangle}$$\end{theorem}

\begin{remark}
The function  $Q(\a,f,T; \vexi)$ is a function of $T$ periodic modulo $f$.
Since $f \Z^J$ is contained in $\Lambda(\a,f)$, the element $f \ve s$ is in
the lattice  $\Lambda(\a,f)$, and we see that the right hand side is also a
periodic function of $T$  modulo $f$. 
\end{remark}

\begin{proof}[Proof of Theorem~\ref{th:as-lattice-genfun}]
  Consider $\vexi\in V_J^*$ with $\xi_j<0$. 
  Then we can write the equality
  $$\frac{1} {\prod_{j\in J} (1-\zeta^{\alpha_j}\e^{\xi_j})}=\prod_{j\in J} 
  \sum_{n_j=0}^{\infty} \zeta^{n_j\alpha_j} \e^{n_j\xi_j}.$$
  So $$Q(\a,f,T; \vexi)=
  \sum_{\ve n \in \Z_{\geq0}^J} \Bigl(\sum_{\zeta \colon \zeta^{f}=1} \zeta^{\sum_j
    n_j\alpha_j-T}\Bigr)\e^{\sum_{j\in J} n_j\xi_j}.$$

We note that $\sum_{\zeta: \ \zeta^{f}=1}\zeta^m$ is zero except if $m\in \Z f$, when
this sum is equal to~$f$.  Then we obtain that
$Q({\a},f,T)$ is the sum over $\ve n\in \Z_{\geq0}^J$ such that     $\sum_j n_j \alpha_j-T\in \Z f$.
The equality $1=\sum_{j\in J} s_j \alpha_j+s_0 f$ implies that
 $T\equiv \sum_{j} t s_j\alpha_j $ modulo $f$, and  the condition
 $\sum_j n_j \alpha_j-T\in \Z f$ is equivalent to the condition
$\sum_{j}(n_j-Ts_j)\alpha_j \in \Z f$.

We see that the point  $\ve n - T\ve s$ is in the 
lattice $\Lambda(\a,f)$ as well as in the cone $-T\ve s+\R^J_{\geq0}$ (as $n_j\geq
0$). Thus the claim.
\end{proof}

By definition of the meromorphic functions $ S(-T \ve
  s+\R^J_{\geq0},\Lambda(\a,f); \; \vexi) $    and 
  $M(-T\ve s, \R_{\geq 0}^J,\Lambda(\a,f); \; \vexi), $
we obtain the following equality.

\begin{corollary}
$$Q({\a},f,T; \vexi)= f\ M(-T\ve s, \R_{\geq 0}^J,\Lambda(\a,f); \; \vexi).$$
\end{corollary}

Using Remark~\ref{rem:restrict-to-SafT} we thus obtain by restriction to
$\vexi=\a_J x$ the following equality.

\begin{corollary} \label{add}
  $$\CF(\a,f,T; x) =f\,  M(-T\ve s,\R_{\geq0}^J,\Lambda(\a,f); \; \a_J x)\prod_{j\colon f \mid \alpha_j}\frac{1}{1-\e^{\alpha_jx}}.$$
\end{corollary}

\subsection{Unimodular decomposition in the dual space}

The cone $\R_{\geq0}^J$ is in general not unimodular with respect to the
lattice $\Lambda(\a,f)$.  By decomposing $\R_{\geq0}^J$ in cones
$\coneU$ that are unimodular with respect to~$\Lambda(\a,f)$, modulo cones
containing lines, we can write
$$M\bigl(-T\ve s,\R_{\geq0}^J,\Lambda(\a, f)\bigr)=\sum_\coneU \epsilon_\coneU M(-T\ve
s,\coneU,\Lambda),$$
where $\epsilon_\coneU\in\{\pm1\}$. 
This decomposition can be computed using Barvinok's algorithm in polynomial
time for fixed~$k$ because the dimension $|J|$ is at most $k$.

\begin{remark}
For this particular cone and lattice, this decomposition modulo cones
containing lines is best done using the 
``dual'' variant of Barvinok's algorithm, as introduced
in \cite{BarviPom}.  This is in contrast to the
``primal'' variant described in \cite{Brion1997residue,koeppe-verdoolaege:parametric}; see also
\cite{so-called-paper-2} for an exposition of Brion--Vergne decomposition and
its relation to both decompositions. 
To explain this, let us determine the index of the cone~$\R_{\geq0}^J$ in the
lattice $\Lambda=\Lambda(\a,f)$;  the worst-case complexity of the signed cone
decomposition is bounded by a polynomial in the logarithm of this index. 
%\tgreen{Updated references to decomposition as requested by referee. --Matthias}

Let $B$ be a matrix whose columns form a basis of~$\Lambda$, so
$\Lambda=B\Z^J$.  Then $|\Z^J:\Lambda| = \mathopen|\det B\mathclose| = f$. By
Remark~\ref{rem:change-to-standard-lattice}, we find
$$
M(-T\ve s,\R_{\geq0}^J,\Lambda; \; \vexi)
= M(-T B^{-1}\ve s,B^{-1} \R_{\geq0}^J, \Z^J; \; B^\T\vexi).
$$
Let $\coneC$ denote the cone $B^{-1} \R_{\geq0}^J$, which is generated by the
columns of~$B^{-1}$.  Since $B^{-1}$ is not integer in general, we find
generators of~$\coneC$ that are primitive vectors of~$\Z^J$ by scaling each of
the columns by an integer.  Certainly $\mathopen|\det B\mathclose| B^{-1}$ is
an integer matrix, and thus we find that the index of the cone~$\coneC$ is
bounded above by $f^{r-1}$.  We can easily determine the exact index as
follows.  For each $i\in J$, the generator $\ve e_i$ of the original cone~$\R_{\geq0}^J$
needs to be scaled so as to lie in the lattice~$\Lambda$.  The smallest
multiplier $y_i\in\Z_{>0}$ such that $\langle \a_J, y_i \ve e_i\rangle \in \Z
f$ is $y_i = \lcm(\alpha_i, f) / \alpha_i$.  Thus the index of~$\R^J_{\geq0}$ in $\Z^J$ is the
product of the $y_i$, and finally the index of~$\R^J_{\geq0}$ in~$\Lambda$ is
$$\frac1{|\Z^r:\Lambda|}\prod_{i\in J} \frac{\lcm(\alpha_i, f)}{\alpha_i}
= \frac1f \prod_{i\in J} \frac{\lcm(\alpha_i, f)}{\alpha_i}.
$$

Instead we consider the dual cone, $\coneC^\circ = \{\, \veeta \in V^*_J : \ll
\veeta,\ve y\rr \geq 0\text{ for $\ve y\in \coneC$} \,\}$.  We have
$\coneC^\circ = B^{\T} \R^J_{\geq0}$.  Then the index of the dual cone
$\coneC^\circ$ equals~$\mathopen|\det B^\T\mathclose| = f$, which is much smaller than $f^{r-1}$.

Following \cite{DyerKannan97}, we now compute a decomposition of
$\coneC^\circ$ in cones~$\coneU^\circ$ that are unimodular with respect to~$\Z^J$,
modulo lower-dimensional cones,
\begin{alignat*}{2}
[\coneC^\circ] &\equiv \sum_\coneU \epsilon_\coneU [\coneU^\circ] &\quad& \text{(modulo
  lower-dimensional cones)}.
\intertext{Then the desired decomposition follows:}
[\coneC] &\equiv \sum_\coneU \epsilon_\coneU [\coneU] && \text{(modulo
  cones with lines)}.
\end{alignat*}
Because of the better bound on the index of the cone on the dual side, the
worst-case complexity of the signed decomposition algorithm is reduced.  This
is confirmed by computational experiments.

\end{remark}

\begin{remark}
\label{remark:epsilon-term}
  Although we know that the meromorphic function $M(-T\ve s,\R_{\geq0}^J,\Lambda(\a,f); \; \vexi)$
restricts via $\vexi=\a_J x$ to a meromorphic function of a single variable $x$,
it may happen that the individual functions 
$M(-T\ve s,\coneU,\Lambda(\a,f); \; \vexi)$ do not restrict.  In other words,
the line $\a_J x$ may be entirely contained in the set of
poles. If this is  the case, we can  compute  (in polynomial time) a regular
vector $\vebeta\in\Q^J$ so that, for $\epsilon\neq 0, $ the deformed vector $(\a_J+\epsilon \vebeta)x$ is not a pole of any of the  functions $M(-T\ve s, \coneU,\Lambda(\a,f); \; \vexi)$ occurring.  
We then consider the meromorphic functions $\epsilon \mapsto M(-T\ve s, \coneU,\Lambda(\a,f); \; (\a_J+\epsilon \vebeta)x)$ and their Laurent
expansions at $\epsilon=0$ in the variable~$\epsilon$. We then add the constant terms of these expansions (multiplied by $\epsilon_\coneU$). This is the value of
$M(-T\ve s, \R_{\geq0}^J,\Lambda(\a,f); \; \vexi)$ 
at the point $\vexi=\a_J x$.
\end{remark}

\subsection{The periodic dependence in \boldmath$T$}

Now let us analyze the dependence in $T$ of 
the functions $M(-T\ve s,\coneU,\Lambda(\a,f))$, where $\coneU$ is a unimodular cone.  
Let the generators be $\ve g_i^\coneU$, so the elements $\ve g_i^\coneU$ form a basis
of the lattice $\Lambda(\a,f)$. Recall that the lattice $f \Z^r$ is contained
in $\Lambda(\a,f)$. Thus as $\ve s\in \Z^r$, we have 
$\ve s=\sum_i s_i \ve g_i^\coneU$ with $fs_i\in \Z$
and hence
$\smallstep{T\ve s}_\coneU=\sum_i \fractional{-Ts_i} \ve g_i^\coneU$
 with $\fractional{-Ts_i}$  a function of $T$ periodic modulo~$f$.

Thus the function $T\mapsto \smallstep{T\ve s}_\coneU$ is a  step linear function, modulo $f$, with value in $V$.
We then write
$$M(-T\ve s,\coneU,\Lambda(\a,f); \vexi)= \e^{\langle\vexi,\fractional{T \ve
    s}_\coneU\rangle}\prod_{j=1}^r \frac{1}{1-\e^{\ll\vexi,\ve g_j\rr }}.$$
Recall that by Corollary \ref{add},
$$\CF(\a,f,T; x) =f\,  M(-T\ve s,\R_{\geq0}^J,\Lambda(\a,f); \; \a_J x)\prod_{j\colon f \mid \alpha_j}\frac{1}{1-\e^{\alpha_jx}}.$$
Thus this is a meromorphic function of the variable $x$ of the form: $$\sum_\coneU \e^{l_\coneU(T)x} \frac{h(x)}{x^{N+1}},$$
where $h(x)$ is
holomorphic in $x$ and $l_\coneU(T)$ is a step linear function of $T$, modulo $f$.
Thus to compute
$$E_i(f; T)=\res_{x=0}\frac{(- x)^i}{i!} \CF(\a,f,T; x)$$
we only have to expand the function $x\mapsto \e^{l_\coneU(T)x}$  up to   the
power $x^{N-i}$.  This expansion can be done in polynomial time.
We thus see that, as stated in Theorem~\ref{E},
$E_i(f; T)$ is a step polynomial of degree less than or equal to $N-i$,
which is periodic  of $T$ modulo~$f$.  This completes the proof of
Theorem~\ref{E} and thus the proof of Theorem~\ref{theo:knap-main}.

%%%%%%%%%%%%%%%%%%%%%%%%%%%%%%%%%%%%%%%%%%%%%%%%%
%%%%%%%%%%%%%%%%%%%%%%%%%%%%%%%%%%%%%%%%%%%%%%%%%

\section{Periodicity of coefficients}

By Schur's result, it is clear that the  coefficient $E_N(t)$ of the highest degree term is just an explicit constant. Our analysis of the high-order
  poles of the generating function associated to $E(\a; t)$ allows us
  to decide what is the highest-degree coefficient of $E(\a; t)$ that
  is not a constant function of $t$ (we will also say that the coefficient is \emph{strictly periodic}).
  
\begin{theorem} \label{theo:firstperiodico}
Given a list of non-negative integer numbers $\a=[\alpha_1, \dots,
  \alpha_{N+1}]$, let $\ell$ be the greatest integer for which there exists a sublist $\a_J$ with $|J|=\ell$, such that its greatest common divisor is not $1$. Then for $k\geq \ell$ the coefficient of degree $k$ is a constant while the coefficient of degree $\ell-1$ of the quasi-polynomial $E(\a; t)$  is strictly periodic. Moreover, if the numbers $\alpha_i$ are given with their prime factorization, then detecting $\ell$ can be done in polynomial time.

\end{theorem}

\begin{example} We apply the theorem above to investigate the 
question of periodicity of the denumerant coefficients
in the case of the classical partition problem
$E([1,2,3,\dots,m]; t)$. It is well known that this coincides with the
classical problem of finding the number of partitions of the integer
$t$ into at most $m$ parts, usually denoted $p_m(t)$ (see
\cite{andrewsbook}).  In this case, Theorem \ref{theo:firstperiodico}
predicts indeed that the highest-degree  coefficient
of the partition function $p_m(t)$ which is non-constant  is the coefficient of the term of degree $\lceil m/2
\rceil$.  This follows from the theorem because the even numbers in
the set $\{1,2,3,\dots,m\}$ form the largest sublist with gcd two. \end{example}

Now that we have the main algorithmic result we can prove some consequences to the description of the periodicity of the coefficients.  In this section, we determine the largest $i$  with a non-constant coefficient $E_i(t)$ and we give a polynomial time algorithm for computing it. This will complete the proof of Theorem \ref{theo:firstperiodico}.

\begin{theorem}\label{thm:mobiusFan}
Given as input a list of integers $\a=[\alpha_1, \dots, \alpha_{N+1}]$ with
their prime factorization $\alpha_i = p_1^{a_{i1}}p_2^{a_{i2}} \cdots
p_n^{a_{in}}$, there is a polynomial time algorithm to find all of the
largest sublists where the greatest common divisor is not
one. Moreover, if $\ell$ denotes the size of the largest sublists with greatest common divisor different from one, 
then (1) there are polynomially many such sublists, (2) the poset $\tilde{P}_{>\ell-1}$  is a fan (a poset with a maximal element and adjacent atoms), and
 (3) the M\"obius function for $P_{>\ell-1}$ is $\mu(f) =1$ if $G(f) \neq G(1)$ and $\mu(1) = 1 - (|\CG_{>\ell-1}(\a)|-1)$.
 \end{theorem}

\begin{proof} Consider the matrix $A=[a_{ij}]$. Let $c_{i_1},
\dots, c_{i_k}$ be column indices of $A$ that denote the columns that contain the largest number of non-zero elements among the columns. Let $\a^{(c_{i_j})}$  be the sublist of $\a$ that corresponds to the rows of $A$ where column $c_{i_j}$ has a non-zero entry. Each $\a^{(c_{i_j})}$ has greatest common divisor different from one. If $\ell$ is the size of the largest sublist of $\a$  with greatest common divisor different from one, then there are $\ell$ many $\alpha_i$'s that share a common prime. Hence each column $c_{i_1}$ of $A$ has $\ell$ many non-zero elements. Then each $\a^{(c_{i_j})}$ is a largest sublist where the greatest common divisor is not one. Note that more than one column index $c_i$ might produce the same sublist $\a^{(c_{i_j})}$. The construction of $A$, counting the non-zero elements of each column, and forming the sublist indexed by each $c_{i_j}$ can be done in polynomial time in the input size. 

To show the poset $\tilde{P}_{>\ell-1}$  is a fan, let $\CG =\{1, f_1, \dots, f_m \}$ be the set of greatest common divisors of
sublists  of size $>\ell -1$. Each $f_i$ corresponds to a greatest common divisor of a sublist $\a^{(i)}$ of $\a$ with size $\ell$. We cannot have $f_i \mid f_j$ for $i \neq j$ because if $f_i \mid f_j$, then  $f_i$ is also the greatest common divisor of $\a^{(i)} \cup \a^{(j)}$, a contradiction to the maximality of $\ell$. Then the M\"obius function is $\mu(f_i) = 1$, and $\mu(1) = 1-m.$ 

As an aside, $\gcd(f_i, f_j) = 1$ for all $f_i \neq f_j$ as if $\gcd(f_i, f_j) \neq 1$, then we can take the union of the sublist that produced $f_i$ and $f_j$ thereby giving a larger sublist with greatest common divisor not equal to one, a contradiction. 
\end{proof}

\begin{example} $[2^2 7^4 41^1, 2^1 7^2 11^1, 11^4, 17^3]$ gives the matrix
\[
\begin{pmatrix}
2 & 4 & 0 & 0 & 1 \\
1 & 2 & 1 & 0 & 0\\
0 & 0 & 4 & 0 & 0\\
0 & 0 & 0 & 3 & 0
\end{pmatrix}
\]
where the columns are the powers of the primes indexed by $(2, 7, 11, 17, 41)$. We see the largest sublists that have gcd not equal to one are $[2^2 7^4 41^1, 2^1 7^2 11^1]$ and $[2^1 7^2 11^1, 11^4]$. Then $\CG =\{1, 2^17^2, 11\}$. The poset $P_{>1}$ is
\begin{center}

 \includegraphics{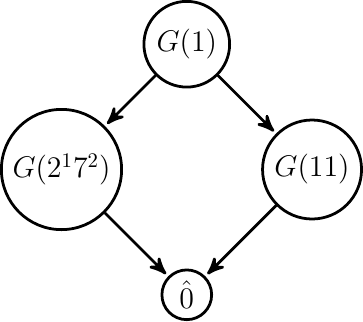}
%\begin{figure}[ht]
%\begin{tikzpicture}[scale=0.60,transform shape,->,>=stealth',shorten >=1pt,auto,node distance=3cm,
%	  thick,main node/.style={circle,draw,font=\sffamily\Large\bfseries}]
%	  \node[main node] (1) {$G(1)$};
%	  \node[main node] (2) [below left of=1] {$G(2^17^2)$};
%	  \node[main node] (3) [below  right of=1] {$G(11)$};
%	  \node[main node] (4) [below left of = 3]{$\hat{0}$};
%	
%	  \path[every node/.style={font=\sffamily\small}]
%		 (1) edge node  {} (2)
%			edge node {} (3)
%		(2) edge node  {} (4)
%       (3) edge node  {} (4);
%\end{tikzpicture}
%\end{figure}
\end{center}
and $\mu(1) = -1$, $\mu(11) = \mu(2^17^2) = 1$.
\end{example}

\begin{proof}[Proof of Theorem~\ref{theo:firstperiodico}]
 Let $\ell$ be the greatest integer for which there exists a sublist $\a_J$ with $|J|=\ell$, such that its gcd $f$ is not $1$. Then for $m \geq \ell$ the coefficient of  degree $m$, $E_m(T)$, is constant because in Equation \eqref{eq:Ehrhartcoeff},  $\CG_{>m}(\a) = \{1\}$. Hence $E_m(T)$ does not depend on $T$. We now focus on $E_{\ell -1}(T)$. To simplify Equation \eqref{eq:Ehrhartcoeff}, we first compute the $\mu(f)$ values.
 
  \begin{lemma}
For $\ell$ as in Theorem \ref{theo:firstperiodico}, the poset  $ \CG_{>\ell-1}(\a)$ is a fan, with one maximal element $1$ and adjacent elements $f$ which are pairwise coprime. In particular, $\mu(f)=1$ for $f\neq 1$.
    \end{lemma}
    \begin{proof}
        Let $\a_{J_1}$, $\a_{J_2}$ be two  sublists of length $\ell$ with gcd's $f_1\neq f_2$ both not equal to $1$. If $f_1$ and $f_2$ had a nontrivial common divisor $d$, then the list $\a_{J_1\cup J_2}$ would have a gcd not equal to $1$, in contradiction with its length being strictly greater than $\ell$.
    \end{proof}
 
Next we recall a fact about Fourier series and use it to show that each
term in the summation over $f\in \CG_{>\ell-1}(\a)$ in Equation
\eqref{eq:Ehrhartcoeff} has smallest period equal to~$f$.

\begin{lemma}\label{lemma:FourierPeriod}
        Let $f$ be a positive integer and let $\phi(t)$ be a periodic function on $\Z/f\Z$  with Fourier expansion
        \[
        \phi(t)=\sum_{n=0}^{f-1} c_n \e^{2i\pi \inlinefrac{nt}{f}}.
        \]
     If $c_n\neq 0$ for some $n$ which is coprime to $f$ then $\phi(t)$ has smallest period equal to $f$.
    \end{lemma}
    \begin{proof}
       Assume  $\phi(t)$ has period $m$ with $f=qm$ and $q>1$. We write its Fourier series as a function of period $m$.
      $$
        \phi(t)=\sum_{j=0}^{m-1} c'_j \e^{2i\pi \inlinefrac{jt}{m}}= \sum_{j= 0}^{m-1}c'_j \e^{2i\pi \inlinefrac{(jq)t}{f}}.
        $$
     By uniqueness of the Fourier coefficients, we have $c_n=0$  if $n$ is not a multiple of $q$ (and $c_{qj}=c'_j$). 
     
We claim that if $n$ is coprime to $f$, then $n$ is not a multiple of $q$. This is true by considering the contrapositive. 
     
     Hence, $c_n = 0$ if $n$ is coprime to $f$, a contradiction.
    \end{proof} 
Theorem~\ref{theo:firstperiodico} is thus the consequence of the following lemma.

   \begin{lemma}\label{lemma:fterm}
     Let $f\in \CG_{>\ell-1}(\a)$.
     The term in the summation over $f$ in \eqref{eq:Ehrhartcoeff} 
     has smallest period $f$ as a function of $T$.
   \end{lemma}
 \begin{proof}
    For $f=1$, the statement is clear. Assume $f \neq 1$. We observe that the $f$-term in \eqref{eq:Ehrhartcoeff} is a periodic function (of period $f$) which is \emph{given as the sum of its Fourier expansion} and is written as $\sum_{n=0}^{f-1}c_n \e^{-2i\pi \inlinefrac{nT}{f}}$  where
    $$
    c_n=- \res_{x=0}\frac{(-x)^{\ell -1}}{(\ell-1)!\,\prod_j\bigl(1-\e^{-2i\pi \inlinefrac{n \alpha_j}{f}}\e^{\alpha_j x}\bigr)}.
    $$
    Consider a coefficient for which  $n$  is coprime to $f$. We decompose the product according to whether $f$ divides $\alpha_j$ or not. The crucial observation is that there are exactly $\ell$ indices $j$ such that $f$ divides $\alpha_j$, because of the maximality assumption on $\ell$.   Therefore $x=0$ is a simple pole and the residue is readily computed. We obtain
    $$
   c_n= \frac{(-1)^{\ell-1}}{(\ell-1)!}
   \cdot\frac{1}{\prod_{j: f \nmid \alpha_j}\bigl(1-\e^{2i\pi \inlinefrac{n \alpha_j}{f} }\bigr)}
   \cdot\frac{1}{\prod_{j: f\mid\alpha_j}\alpha_j }.
    $$
    Thus $c_n\neq 0$ for an $n$ coprime with $f$. By Lemma \ref{lemma:FourierPeriod}, each $f$-term has minimal period $f$. 
 \end{proof} 
  As the various numbers~$f$ in  $\CG_{>\ell-1}(\a)$ different from $1$ are
  pairwise coprime and the corresponding terms have minimal period~$f$, $E_{\ell-1}(T)$ has minimal
  period $\prod\limits_{f \in \CG_{>\ell-1}(\a)}f > 1$.
  This completes the proof of Theorem~\ref{theo:firstperiodico}.
\end{proof}

%%%%%%%%%%%%%%%%%%%%%%%%%%%%%%%%%%%%%%%%%%%%%%%%%
%%%%%%%%%%%%%%%%%%%%%%%%%%%%%%%%%%%%%%%%%%%%%%%%%
%%%%%%%%%%%%%%%%%%%%%%%%%%%%%%%%%%%%%%%%%%%%%%%%%
%%%%%%%%%%%%%%%%%%%%%%%%%%%%%%%%%%%%%%%%%%%%%%%%%
%%%%%%%%%%%%%%%%%%%%%%%%%%%%%%%%%%%%%%%%%%%%%%%%%
%%%%%%%%%%%%%%%%%%%%%%%%%%%%%%%%%%%%%%%%%%%%%%%%%
%%%%%%%%%%%%%%%%%%%%%%%%%%%%%%%%%%%%%%%%%%%%%%%%%

\section{Summary of the algorithm  to compute top coefficients} 

In this section, we give a detailed outline of the main algorithm. Given a sequence of integers $\a$ of length ${N+1}$, we wish to compute the top $k+1$ coefficients of the quasi-polynomial $E(\a; t)=\sum_{i=0}^{N} E_{i}(t) t^{i}$ of degree $N$. Note that $\a$ is not a set and is allowed to have repeated numbers.  Recall that
$$E(\a; t)=\sum_{i=0}^{N} E_{i}(t) t^{i}$$
where $E_{i}(t)$ is a periodic function of $t$ modulo some period $q_i$.  We
assume that greatest common divisor of the list $\a$ is~$1$.

\begin{enumerate}
\item For every subsequence of $\a$ of length greater than $N-k$, compute the greatest common divisor. Assign these values to the set $\CG_{>N-k}(\a)$ (ignoring duplicates). Note that $1 \in \CG_{>N-k}(\a).$
\item For each $f \in \CG_{>N-k}(\a)$, compute $\mu(f)$ by Theorem \ref{theorem:knap:mobius-function}.
\end{enumerate}

Recall some notation
\begin{itemize}
\item $J=J(\a,f) = \{\, i \in\{1,\dots,N+1\}   : f \nmid \alpha_i \,\}$,  and $\a_J =[\alpha_i]_{i \in J}$
\item $s_0\in\Z$ and $\ve s\in\Z^J$ such that $1=\sum_{i\in J} s_i \alpha_i+s_0 f$ (as a superscript, we mean $\Z^J = \Z^{|J|}$),
\item $\Lambda({\a},f):=\biggl\{\,\ve y \in \Z^J : \langle \a_J, \ve y\rangle = \sum_{j\in J} y_j\alpha_j\in\Z f\,\biggr\}$.
\end{itemize}
 Then we seek to compute
\begin{align*}
\sum_{i=0}^k E_{N-i}(T) t^{N-i}
&=\sum_{f\in \CG_{>N-k}(\a)} \mu(f)  {\mathcal E}(\a,f; t, T) \\
&= \sum_{f\in \CG_{>N-k}(\a)} \mu(f)  \left( - \Res_{x=0} e^{-tx} {\mathcal F}(\a, f, T; x) \right) \\
&= \sum_{f\in \CG_{>N-k}(\a)} \mu(f)  \left( - \Res_{x=0} e^{-tx} \cdot f \cdot  M(-T\ve s,\R_{\geq0}^J,\Lambda(\a,f); \; \a_J x) \cdot \prod_{j\colon f \mid \alpha_j}\frac{1}{1-\e^{\alpha_jx}}\right). \\
\end{align*}

We return to the algorithm by describing how to compute the term in the big parentheses above for every fixed $f$ value.

\begin{enumerate}
\setcounter{enumi}{2}
\item To compute $s \in \Z^J$, let $A$ be the $(|J|+1) \times 1$ matrix given by $A = (\a_J, f)^T$. Compute the Hermite Normal Form of $A$ \cite{schrijver}. Let $H \in \Z^{|J| +1 \times 1}$ be the Hermite Normal Form of $A$, and let $U \in \Z^{|J| +1 \times |J| +1 }$ be a unimodular matrix such that $UA = H$. Note that $H$ will have the form $H = (1, 0, \dots, 0)^T$. If $u:=(u_1, \dots, u_{|J|}, u_{|J|+1})$ denotes the first row of $U$, then  $u\cdot A = 1$. So $\ve s \in \Z^J$ is the first $|J|$ elements of $u$.
\item To compute a basis for $\Lambda({\a},f),$ we again use the Hermite Normal Form. Let $A$ be the column matrix of $\a_J$. Let $H \in \Z^{J \times 1},$ and $U \in \Z^{J \times J}$ be the Hermite Normal Form of $A$ and unimodular matrix $U$, respectively, such that $UA=H$. $H$ will have the form $H=(h, 0, 0, \dots, 0)^T$, where $h$ is the smallest positive integer that can be written as an integer combination of $\a_J$. To make sure $h$ is a multiple of $f$, scale the first row of $U$ by $\frac{f}{\gcd(h, g)}$. Let $\tilde U$ be this scaled $U$ matrix. Then $\hat U A = (\lcm(h,f), 0, \dots, 0)^T$, where $\lcm$ is the lowest common multiple, and so the rows of $\tilde U$ form a basis for $\Lambda({\a},f).$ Let $B = \tilde U^T$, then $\Lambda({\a},f) = B\Z^J$.
\item Let $B$ be the matrix whose columns are the generators of the lattice~$\Lambda$ as in the last step. Then $B^{-1}\R_{\geq0}^J$ is the $|J|$-dimensional cone in $\R^J$ generated by the nonnegative combinations of the columns of $B^{-1}$. Using Barvinok's \emph{dual}  algorithm \cite{bar, barvinokzurichbook, BarviPom, Brion1997residue}, the cone $B^{-1}\R_{\geq0}^J$ can be written as a signed sum of unimodular cones. Let $U_B$ be the set of resulting unimodular cones and $\epsilon_\coneU \in \{-1, 1\}$ for each $u \in U_B$. Also, for $\coneU \in U_B$, let $g_i^u$ be the corresponding rays of the cone $\coneU$ for $i=1, \dots, |J|$. Then
\begin{align*}
M(-T\ve s,\R_{\geq0}^J,\Lambda(\a,f); \; \a_J x) &= M(-T B^{-1} \ve s,B^{-1}\R_{\geq0}^J,\Z^J; \; B^{T}\a_J x)\\
&= \sum_{\coneU \in U_B} \epsilon_\coneU M(-T B^{-1} \ve s, \coneU, \Z^J; \; B^{T}\a_J x)\\
&= \sum_{\coneU \in U_B} \epsilon_\coneU\, \e^{ \ll
  B^{T}\a_J x,\{T B^{-1}\ve s\}_\coneU\rr } \frac{1}{\prod_j (1-\e^{\ll B^{T}\a_J x,\ve g_j^\coneU\rr })}.
\end{align*}
To compute $\{T B^{-1}\ve s\}_\coneU$, write $B^{-1}\ve s$ in the basis given by the rays of cone $\coneU$: $B^{-1}\ve s = \gamma_1 g_1^u + \cdots \gamma_{|J|} g_{|J|}^U$. Then $\{T B^{-1}\ve s\}_\coneU$ is the vector $(\{\gamma_1 T\}, \dots, \{\gamma_{|J|} T\})$. Note that $T$ is a symbolic variable, so every calculation involving the function $\{x\}$ must be done symbolically in some data structure. 

Although $M(-T\ve s,\R_{\geq0}^J,\Lambda(\a,f); \; \a_J x)$ is well defined at $\a_Jx$, it may happen that $\ll B^{T}\a_J x,\ve g_j^\coneU\rr = 0$ for some rays $g_i^\coneU$ at some cone $\coneU$. That is, one of the cones could be singular at $\a_Jx$. To fix this, as noted in Remark \ref{remark:epsilon-term}, we replace $\a_J$ with $\a_J + \epsilon \beta$ such that $\ll B^{T}\beta x,\ve g_j^\coneU\rr \neq 0$. The set of $\beta$ that fail this condition has measure zero, so $\beta$ can be picked randomly. After computing the limit as $\epsilon$ goes to zero, the new variable $\epsilon$ is eliminated. To compute the limit, it is enough to find the series expansion of $\epsilon$. Positive powers of $\epsilon$ in the series can be ignored because they will vanish is the limit, while negative powers of $\epsilon$ can also be dropped because they are guaranteed to cancel out in the summation over $U_B$. Hence it is enough to compute the series expansion in $\epsilon$ and only keep the coefficient of $\epsilon^0$. Then finally, 

\[M(-T\ve s,\R_{\geq0}^J,\Lambda(\a,f); \; \a_J x) 
= \sum_{\coneU \in U_B} \epsilon_\coneU\ \Res_{\epsilon=0} \frac{1}{\epsilon}  \e^{ \ll
  B^{T}(\a_J + \beta \cdot \epsilon) x,\{T B^{-1}\ve s\}_\coneU\rr } \frac{1}{\prod_j (1-\e^{\ll B^{T}(\a_J + \beta \cdot \epsilon) x,\ve g_j^\coneU\rr })}
  \]
  
\item The series expansion of 
\[ f \cdot  \left(\sum_{\coneU \in U_B} \epsilon_\coneU\ \Res_{\epsilon=0} \frac{1}{\epsilon}  \e^{ \ll
  B^{T}(\a_J + \beta \cdot \epsilon) x,\{T B^{-1}\ve s\}_\coneU\rr } \frac{1}{\prod_j (1-\e^{\ll B^{T}(\a_J + \beta \cdot \epsilon) x,\ve g_j^\coneU\rr })} \right) \prod_{j\colon f \mid \alpha_j}\frac{1}{1-\e^{\alpha_jx}}\]
can be computed by first finding the Laurent series expansion at $x=0$ and at $\epsilon=0$ of each of the terms
\begin{itemize}
\item $\e^{ \ll
  B^{T}\a_J x,\{T B^{-1}\ve s\}_\coneU\rr }$,
\item $\e^{ \ll
  B^{T}\beta \cdot \epsilon x,\{T B^{-1}\ve s\}_\coneU\rr }$,
\item  $\frac{1}{1-\e^{\ll B^{T}(\a_J + \beta \cdot \epsilon) x,\ve g_j^\coneU\rr }}$, and
\item $\frac{1}{1-\e^{\alpha_jx}}$, 
\end{itemize}
by using the Taylor expansion of $e^x$ and Remark \ref{remark:knap:bernoulli}, and then by multiplying each series together. This Laurent series  starts at $x^{-N-1}$, and so the coefficient of $x^{-N-1+i}$ in this Laurent series  contributes to $E_{N-i}(T)$ (after being multiplied by $\mu(f) \cdot \frac{(-1)^i}{i!}$). Therefore it is enough to compute at most the first $k$ terms of any partial Laurent series.
\end{enumerate}

%%%%%%%%%%%%%%%%%%%%%%%%%%%%%%%%%%%%%%%%%%%%%%%%%
%%%%%%%%%%%%%%%%%%%%%%%%%%%%%%%%%%%%%%%%%%%%%%%%%
%%%%%%%%%%%%%%%%%%%%%%%%%%%%%%%%%%%%%%%%%%%%%%%%%
%%%%%%%%%%%%%%%%%%%%%%%%%%%%%%%%%%%%%%%%%%%%%%%%%
%%%%%%%%%%%%%%%%%%%%%%%%%%%%%%%%%%%%%%%%%%%%%%%%%
%%%%%%%%%%%%%%%%%%%%%%%%%%%%%%%%%%%%%%%%%%%%%%%%%
%%%%%%%%%%%%%%%%%%%%%%%%%%%%%%%%%%%%%%%%%%%%%%%%%

\section{Experiments} 
\label{experiments}
This chapter closes with an extensive collection of computational experiments (Section~\ref{experiments}). 
We constructed a dataset of over 760 knapsacks and show our new algorithm is the fastest available method for computing 
the top $k$ terms in the Ehrhart quasi-polynomial.  Our implementation of the
new algorithm is made available as a part of the free software
\latteintegrale \cite{latteintegrale}, version 1.7.2.\footnote{Available under the GNU General Public
  License at \url{https://www.math.ucdavis.edu/~latte/}.}

We first wrote a preliminary implementation of  our algorithm in \maple, which we call \mapleKnapsack in the following.
Later we developed a faster implementation in C++, which is referred to as \latteKnapsack in the following (we use the term 
knapsack to refer to the Diophantine problem  $\alpha_1x_1+\alpha_2 x_2+\cdots+\alpha_{N} x_{N}+\alpha_{N+1}x_{N+1}=t$).
Both implementations are released as part of the software package 
\latteintegrale \cite{latteintegrale}, version 1.7.2.\footnote{Available under the GNU General Public
  License at \url{https://www.math.ucdavis.edu/~latte/}.
  The \maple code \mapleKnapsack is also available separately at \url{https://www.math.ucdavis.edu/~latte/software/packages/maple/}.
} 

We report on two different benchmarks tests: 
\begin{enumerate}
\item We test the performance of the implementations 
\mapleKnapsack%
  \footnote{Maple usage:
    \maplecode{coeff\_Nminusk\_knapsack($\langle\mathit{knapsack\
        list}\rangle$, t, $\langle\mathit{k\ value}\rangle$)}.} and 
  \latteKnapsack \footnote{Command line usage:
    \shellcode{dest/bin/top-ehrhart-knapsack -f $\langle\mathit{knapsack\
        file}\rangle$ -o $\langle\mathit{output\ file}\rangle$ -k
      $\langle\mathit{k\ value}\rangle$}.},
  and also the implementation of the algorithm from \cite{so-called-paper-1},
  which refer to as \coneApx \footnote{Command line usage: \shellcode{dest/bin/integrate
      --valuation=top-ehrhart --top-ehrhart-save=\allowbreak$\langle\mathit{output\
        file}\rangle$ --num-coefficients=$\langle\mathit{k\ value}\rangle$
      $\langle\mathit{LattE\ style\ knapsack\ file}\rangle$}.},
	  on a collection of over $750$ knapsacks. The latter algorithm can compute the weighted Ehrhart quasi-polynomials for simplicial polytopes, and hence it is more general than the algorithm we present in this chapter, but this is the only other available algorithm for computing coefficients
 directly. Note that the implementations of the \mapleKnapsack algorithm and the main computational part of the \coneApx algorithm are in \maple, making comparisons between the two easier. 

 \item Next, we run our algorithms on a few knapsacks that have been studied in the literature. We chose these examples because some of these problems are considered difficult in the literature.  We also present a comparison with other available software that can also compute information of the denumerant $E_\a(t)$:
   the codes \emph{CTEuclid6} \cite{guocexin} and \emph{pSn}
 \cite{sillszeilberger}.\footnote{Both codes can be downloaded from the
   locations indicated in the respective papers.  Maple scripts that
   correspond to our tests of these codes are 
   available at
   \url{https://www.math.ucdavis.edu/~latte/software/denumerantSupplemental/}.} 
 These codes use mathematical ideas that are different from those used in this chapter. 
\end{enumerate}

All computations were performed on a 64-bit Ubuntu machine with 64 GB of RAM and eight Dual Core AMD Opteron 880 processors.

\subsection{\mapleKnapsack vs.\@ \latteKnapsack vs.\@ \coneApx}

Here we compare our two implementations with the \coneApx algorithm from \cite{so-called-paper-1}. We constructed a test set of  768 knapsacks. For each $3 \leq d \leq 50$, we constructed four families of knapsacks:
\begin{description}
	\item[random-3] Five random knapsacks in dimension $d-1$ where $a_1 = 1$ and the other coefficients, $a_2, \dots, a_{d}$, are $3$-digit random numbers picked uniformly
	\item[random-15] Similar to the previous case, but with a $15$-digit random number
	\item[repeat] Five knapsacks in dimension $d-1$ where $\alpha_1=1$ and all the other $\alpha_i$'s are the same $3$-digit random number. These produce few poles and have a simple poset structure. These are among the simplest knapsacks that produce periodic coefficients.
	\item[partition] One knapsack in the form $\alpha_i = i$ for $ 1 \leq i \leq d$.
\end{description}

For each knapsack, we successively compute the highest degree terms of the
quasi-polyno\-mial, with a time limit of $200$ CPU~seconds for each coefficient.  
Once a term takes longer than $200$ seconds to compute, we skip the remaining
terms, as they are harder to compute than the previous ones.  
We then count the maximum number of terms of the quasi-polynomial, starting from the
highest degree term (which would, of course, be trivial to compute), that can be
computed subject to these time limits.  
Figures \ref{fig:random3}, \ref{fig:random15}, \ref{fig:repeat},
\ref{fig:partition} show these maximum numbers of terms for the random-3,
random-15, repeat, and partition knapsacks, respectively. 
For example, in Figure \ref{fig:random3}, for each of the five random 3-digit
knapsacks in ambient dimension $50$, the \latteKnapsack method computed at
most $6$ terms of an Ehrhart polynomial, the \mapleKnapsack computed at most
four terms, and the \coneApx method computed at most the trivially computable highest
degree term. %%(for some examples it even failed to compute that). --suppressed
%\tgreen{Reworded this paragraph regarding the 200 seconds of time limit.--Matthias}

In each knapsack family, we see that each algorithm has a ``peak'' dimension
where after it, the number of terms that can be computed subject to the time limit quickly decreases; for the \latteKnapsack method, this is around dimension $25$ in each knapsack family. In each family, there is a clear order to which algorithm can compute the most: \latteKnapsack computes the most coefficients, while the \coneApx method computes the least number of terms.  In Figure \ref{fig:repeat}, the simple poset structure helps every method to compute more terms, but the two \maple scripts seem to benefit more than the \latteKnapsack method. 

Figure \ref{fig:partition} demonstrates the power of the LattE implementation. Note that a knapsack of this particular form in dimension $d$ does not start to have periodic terms until around $d/2$. Thus even though half of the coefficients are only constants we see that the \mapleKnapsack code cannot compute past a few periodic term in dimension $10$--$15$ while the \latteKnapsack method is able to compute the entire polynomial.

In  Figure \ref{fig:ratioRandom15} we plot the average speedup ratio between
the \mapleKnapsack and \coneApx implementations along with the maximum and
minimum speedup ratios (we wrote both algorithms  in \maple). The ratios are
given by the time it takes \coneApx to compute a term, divided by the time it
takes \mapleKnapsack to compute the same term, where both times are between
$0$ and $200$ seconds. For example, among all the terms computed in dimension
$15$ from random $15$-digit knapsacks, the average speedup between the two
methods was $8000$, the maximum ratio was $20000$, and the minimum ratio was
$200$. We see that in dimensions $3$--$10$, there are a few terms for which
the \coneApx method was faster than the \mapleKnapsack method, but this only
occurs for the highest degree terms. Also, after dimension $25$, there is
little variance in the ratios because the \coneApx method is only computing
the trivial highest term. Similar results hold for the other knapsack families, and so their plots are omitted. 

\begin{figure}
  \centering
    \includegraphics[width=0.95\textwidth]{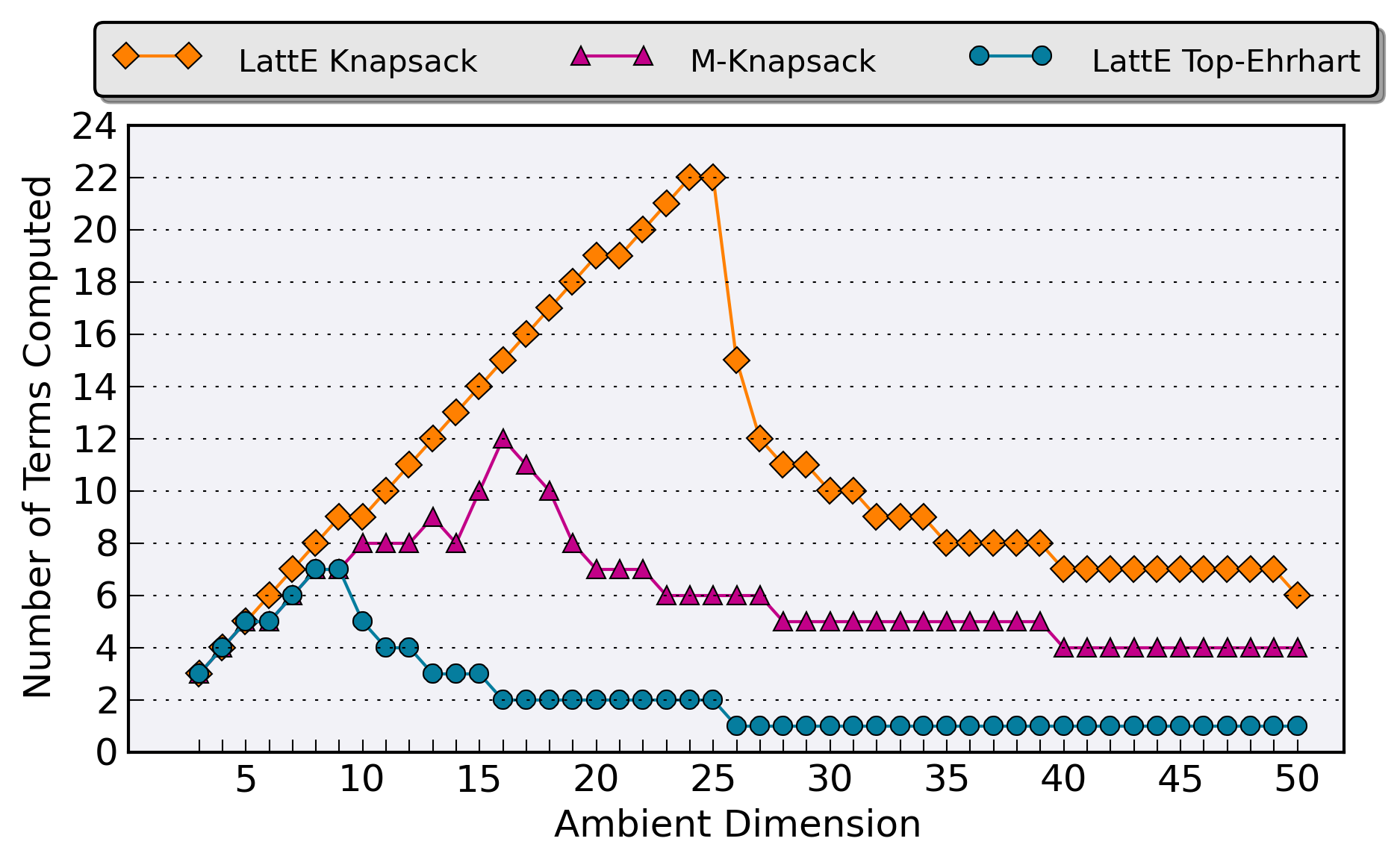}
 \caption{Random 3-digit knapsacks: Maximum number of coefficients each algorithm can compute where each coefficient takes less than 200 seconds.}
 \label{fig:random3}
\end{figure}

\begin{figure}
  \centering
    \includegraphics[width=0.95\textwidth]{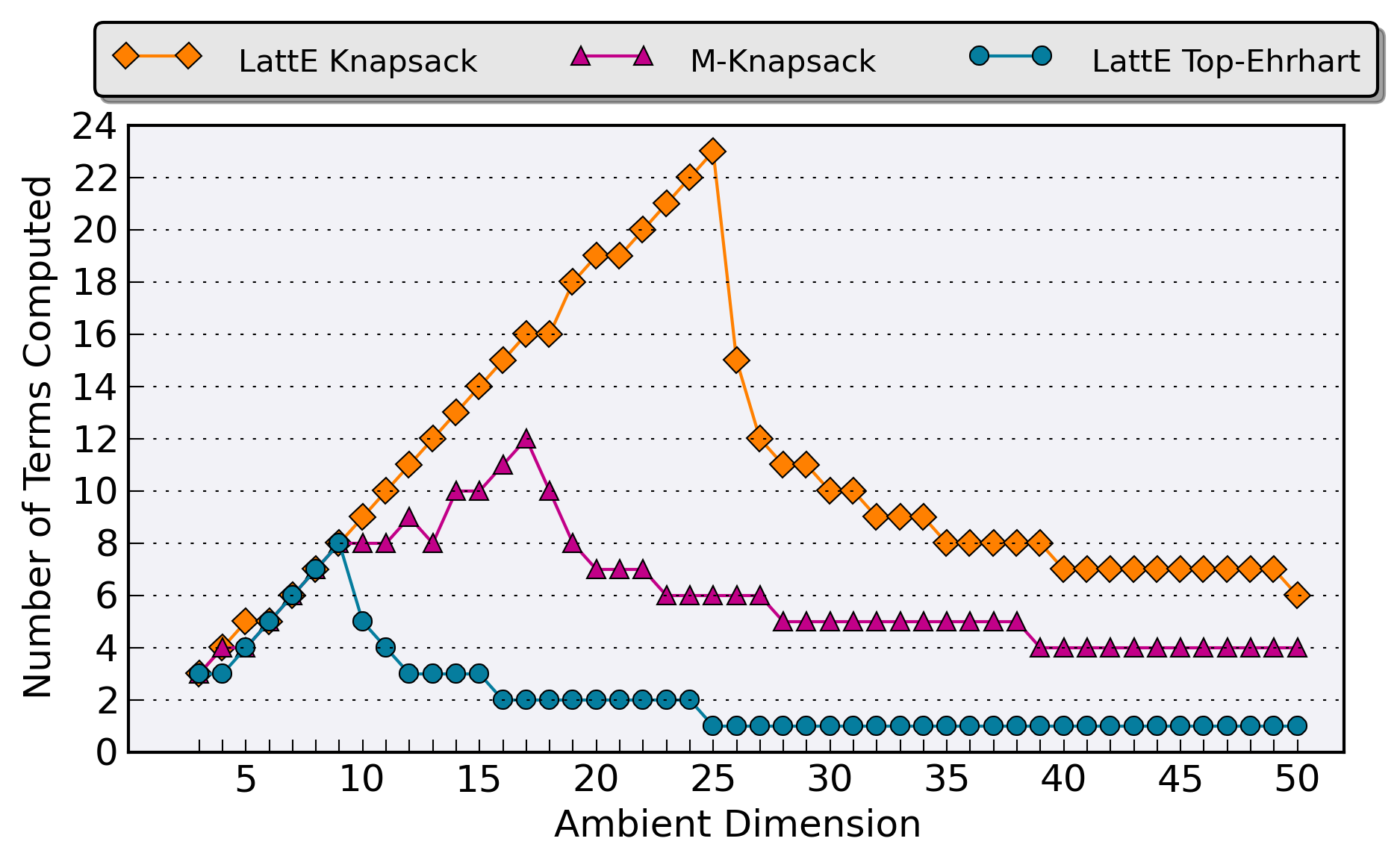}
 \caption{Random 15-digit knapsacks: Maximum number of coefficients each algorithm can compute where each coefficient takes less than 200 seconds.}
 \label{fig:random15}
\end{figure}

\begin{figure}
  \centering
    \includegraphics[width=0.95\textwidth]{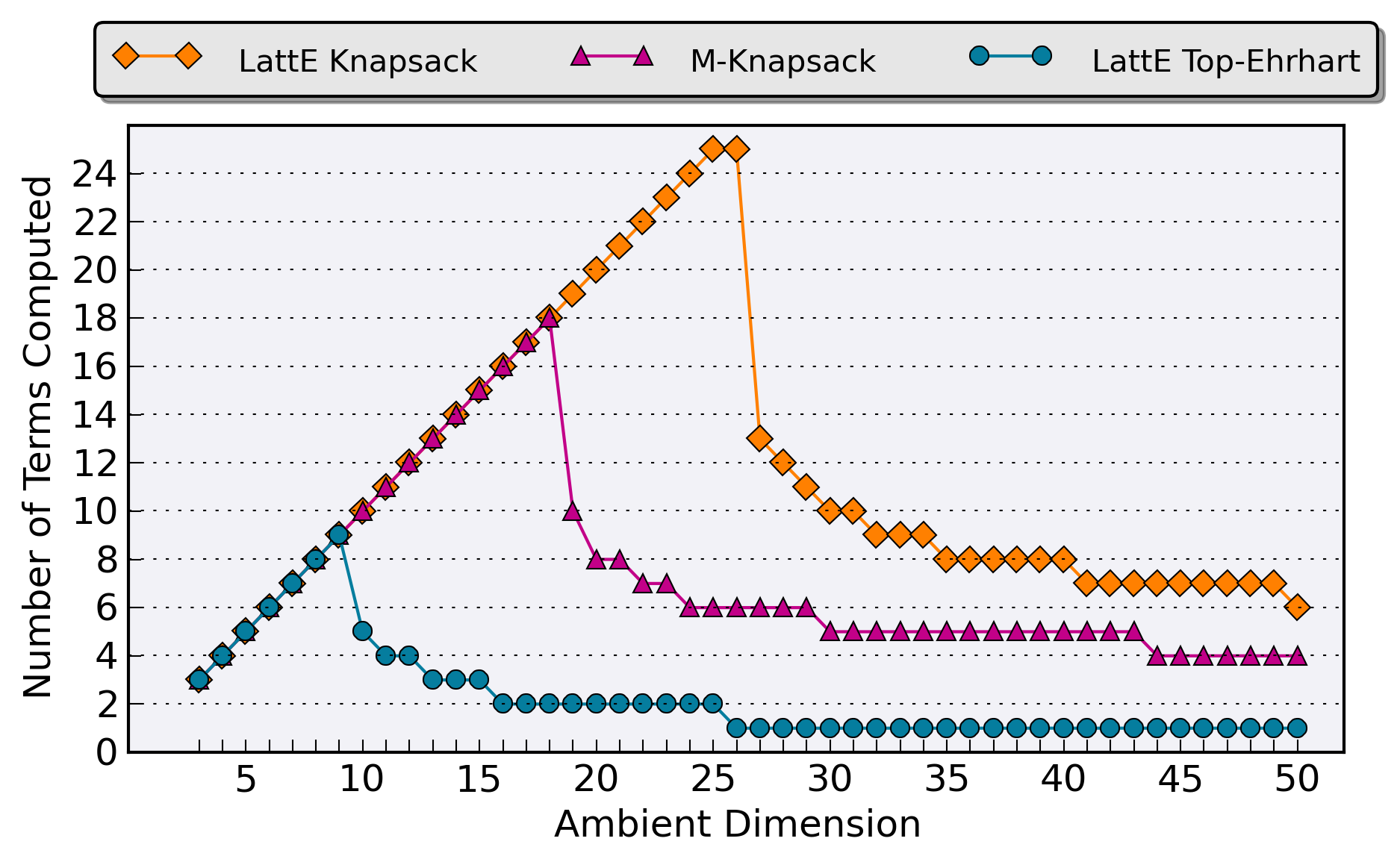}
 \caption{Repeat knapsacks: Maximum number of coefficients each algorithm can compute where each coefficient takes less than 200 seconds.}
    \label{fig:repeat}
\end{figure}

\begin{figure}
  \centering
    \includegraphics[width=0.95\textwidth]{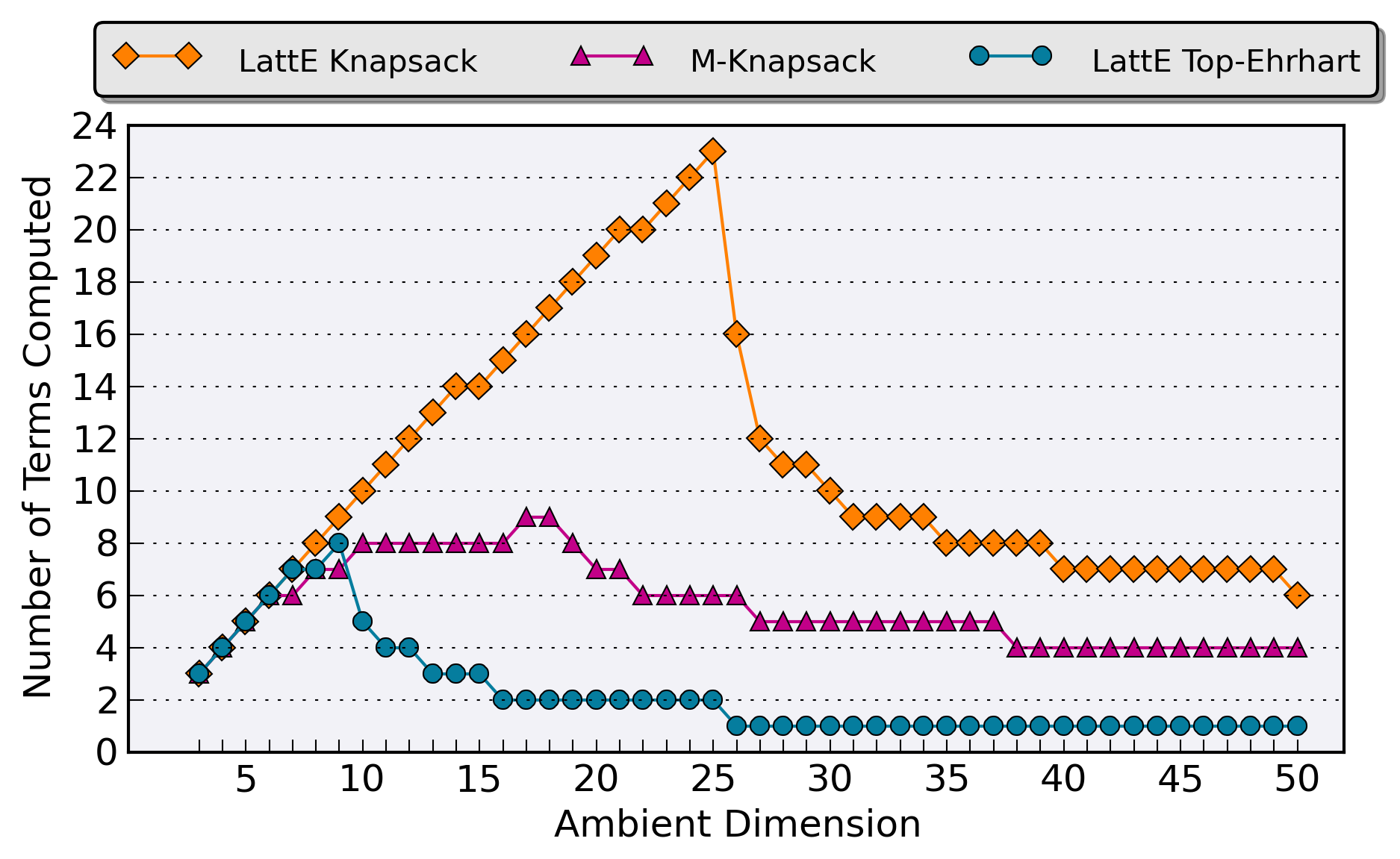}
 \caption{Partition knapsacks: Maximum number of coefficients each algorithm can compute where each coefficient takes less than 200 seconds.}
    \label{fig:partition}
\end{figure}

\begin{figure}
  \centering
    \includegraphics[width=0.95\textwidth]{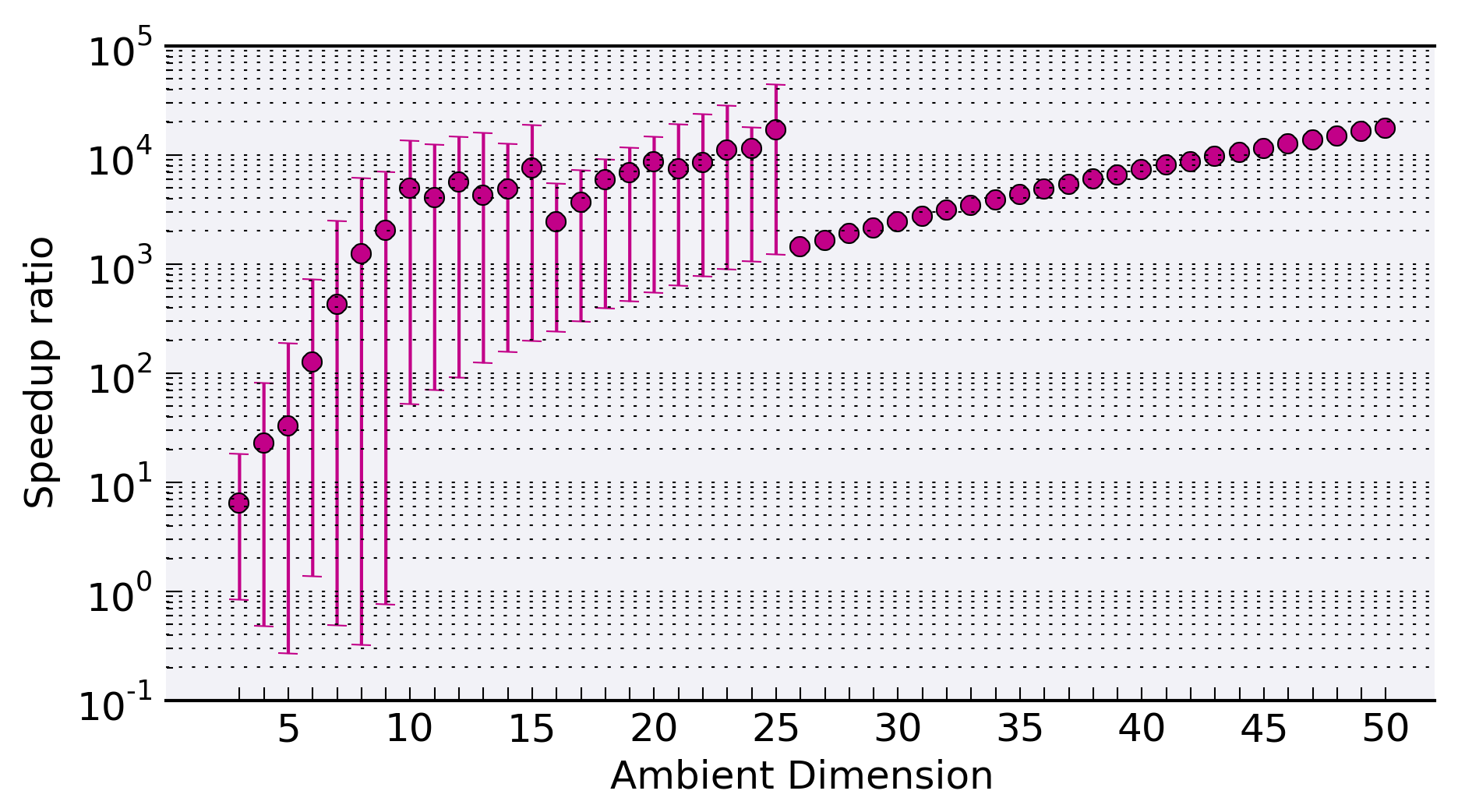}
 \caption{Average speedup ratio (dots) between the \mapleKnapsack and \coneApx codes along with maximum and minimum speedup ratio bounds (vertical lines) for the random 15-digit knapsacks.}
 \label{fig:ratioRandom15}
\end{figure}

\subsection{Other examples}

Next we focus on ten problems listed in Table \ref{tab:examples}. Some of these selected problems have been studied before in the literature \cite{aardallenstra,latte1,GuoceXin2004,guocexin}. Table \ref{tab:coneApxCTE} shows the time in seconds to compute the entire denumerant using the \mapleKnapsack, \latteKnapsack and \coneApx codes with two other algorithms: \emph{CTEuclid6} and \emph{pSn}.

 The \emph{CTEuclid6} algorithm \cite{guocexin}  computes the lattice point
 count of a polytope, and supersedes an earlier algorithm in
 \cite{GuoceXin2004}.\footnote{Maple usage:
   \maplecode{CTEuclid($F(\a; x)/x^b$, t, [x])}; where $b = \alpha_1 + \cdots
   + \alpha_{N+1}$.}  
Instead of using Barvinok's algorithm to construct unimodular cones, the main idea used by the \emph{CTEuclid6} algorithm to find the constant term in the generating function $F(\a; z)$ relies on recursively computing partial fraction decompositions to construct the series. Notice that the \emph{CTEuclid6} method only computes the number of integer points in one dilation of a polytope and not the full Ehrhart polynomial. We can estimate how long it would take to find the Ehrhart polynomial using an interpolation method by computing the time it takes to find one lattice point count times the periodicity of the polynomial and degree. Hence, in Table \ref{tab:coneApxCTE}, column ``one point'' refers to the running time of finding one lattice point count, while column ``estimate'' is an estimate for how long it would take to find the Ehrhart polynomial by interpolation. We see that the \emph{CTEuclid6} algorithm is fast for finding the number of integer points in a knapsack, but this would lead to a slow method for finding the Ehrhart polynomial. 
 
 The \emph{pSn} algorithm of \cite{sillszeilberger} computes the entire
 denumerant by using a partial fraction decomposition based
 method.\footnote{Maple usage:
   \maplecode{QPStoTrunc(pSn($\langle\mathit{knapsack\
       list}\rangle$,n,$j$),n)}; where $j$ is the smallest value in $\{100,
   200, \allowbreak 500, \allowbreak 1000, \allowbreak 2000, 3000\}$ that produces an answer.} More precisely the quasi-polynomials are represented as a function $f(t)$ given by $q$ polynomials $f^{[1]}(t), f^{[2]}(t), \dots,f^{[q]}(t)$ 
such that $f(t)=f^{[i]}(t)$ when $t \equiv i \pmod{q}$. To find the coefficients of the $f^{[i]}$ their method finds the  first few terms of the Maclaurin 
expansion of the partial fraction decomposition to find enough evaluations of those polynomials
and then recovers the coefficients of each the $f^{[i]}$ as a result of solving a linear system. This algorithm goes back to Cayley and it was  implemented in \maple. Looking at Table \ref{tab:coneApxCTE}, we see that the \emph{pSn} method is competitive with \latteKnapsack for knapsacks $1, 2, \dots, 6$, and beats \latteKnapsack in knapsack $10$. However, the \emph{pSn} method is highly sensitive to the number of digits in the knapsack coefficients, unlike our \mapleKnapsack and \latteKnapsack methods. For example,  the knapsacks $[1,2,4,6,8]$ takes 0.320 seconds to find the full Ehrhart polynomial, $[1,20,40, 60, 80]$ takes 5.520 seconds, and $[1, 200, 600, 900, 400]$ takes 247.939 seconds. Similar results hold for other three-digit knapsacks in dimension four. However, the partition knapsack $[1,2,3,\dots, 50]$ only takes 102.7 seconds. Finally, comparing the two \maple scripts, the \coneApx method outperforms the \mapleKnapsack method.

Table \ref{tab:coneApxCTE}  ignores one of the main features of our algorithm: that it can compute just the top $k$ terms of the Ehrhart polynomial. In Table \ref{tab:top3and4times}, we time the computation for finding the top three and four terms of the Ehrhart polynomial on the knapsacks in Table \ref{tab:examples}. We immediately see that our \latteKnapsack method takes less than one thousandth of a second in each example. Comparing the two \maple scripts, \mapleKnapsack greatly outperforms \coneApx. Hence, for a fixed $k$, the \latteKnapsack is the fastest method.

In summary, the \latteKnapsack is the fastest method for computing the top $k$ terms of the Ehrhart polynomial. The \latteKnapsack method can also compute the full Ehrhart polynomial in a reasonable amount of time up to  around dimension $25$, and the number of digits in each knapsack coefficient does not significantly alter performance. However, if the coefficients each have one or two digits, the \emph{pSn} method is faster, even in large dimensions. 

\begin{table}[t]
\centering
\caption{Ten selected instances}\label{tab:examples}
\begin{tabular}{ll}  
\toprule
Problem & Data                                      \\
\midrule
 \#1    & $[8,12,11]$                               \\
 \#2    & $[5,13,2,8,3]$                            \\
 \#3    & $[5,3,1,4,2]$                             \\
 \#4    & $[9,11,14,5,12]$                          \\
 \#5    & $[9,10,17,5,2]$                           \\
 \#6    & $[1,2,3,4,5,6]$                           \\
 \#7    & $[12223,12224,36674, 61119,85569]$        \\
 \#8    & $[12137, 24269,36405,36407,48545,60683]$  \\
 \#9    & $[20601,40429,40429,45415,53725,61919,64470,69340,78539,95043]$ \\
 \#10   & $[5, 10, 10, 2, 8, 20, 15, 2, 9, 9, 7, 4, 12, 13, 19]$ \\
\bottomrule
\end{tabular}
\end{table}

\begin{table}[t]
\sisetup{
  output-exponent-marker = \text{e},
  table-format=+1.4e+2,
  exponent-product={},
  retain-explicit-plus
}
\centering\small
\caption{Computation times in seconds for finding the full Ehrhart polynomial
  using five different methods.}  \label{tab:coneApxCTE}
\begin{tabular}{
l
S[table-format=1.3,table-number-alignment=center]
S[table-format=1.3,table-number-alignment=center]
S[table-format=1.3,table-number-alignment=center]
S[table-format=1.3,table-number-alignment=center]
S[table-format=1.3e+2,table-number-alignment=center]
S[table-format=1.3,table-number-alignment=right]
%S[tabformat=1.3,tabnumalign=center]
%S[tabformat=1.3,tabnumalign=center]
%S[tabformat=1.3,tabnumalign=center]
%S[tabformat=1.3,tabnumalign=center]
%S[tabformat=1.3e+2,tabnumalign=center]
%S[tabformat=1.3,tabnumalign=right]
}
\toprule
&  &  &  & \multicolumn{2}{c}{\emph{CTEuclid6}}  & \\
% \cmidrule(l){2-2}\cmidrule(l){3-3}\cmidrule(l){4-4}
\cmidrule(l){5-6} % \cmidrule(l){7-7}
& \latteKnapsack & \mapleKnapsack & \coneApx & {One point} &  {estimate} & \emph{pSn}\\
\midrule
\#1    & 0 & 0.316  & 0.160 & 0.004   & 3.168  &0.328\\
\#2    &0.03 & 5.984  & 2.208 & 0.048 & 347.4 & 0.292\\
\#3    &0.02 & 4.564  & 0.148 & 0.031 & 9.60   &0.212\\
\#4    &0.08 & 18.317 & 3.884 & 0.112 & 7761.6 &0.496\\
\#5    &0.06 & 15.200 & 3.588 & 0.096 & 734.4  &0.392\\
\#6    &0.11 & 37.974 & 8.068 & 0.088 & 31.68  &0.336\\
\#7    &0.19 & 43.006 & 8.424 & 0.436 & 9.466e+20 & {$>$30min}\\
\#8    &1.14 & 1110.857&184.663&2.120 & 8.530e+20 &{$>$30min}\\
\#9    &{$>$30min} & {$>$30min} & {$>$30min} & {$>$30min} & {$>$30min}  & {$>$30min} \\
\#10   &{$>$30min} & {$>$30min} & {$>$30min} & 142.792 & 1.333e+9 &2.336\\
\bottomrule
\end{tabular}
\end{table}

%Notes: the number of integer points when the RHS b-value = sum of the coefficients.
%\#1      & 1\\
%\#2    & 59\\
%\#3    & 84\\
%\#4    & 16\\
%\#5     & 42\\
%\#6     & 331\\
%\#7     & 11\\
%\#8    & 15\\
%\#9    & ?\\
%\#10   & 139208366\\

\begin{table}
%\sisetup{
% table-space-text-post=\-,    % leave space for a ‘%’
%  table-align-text-post=false, % push ‘%’ next to the number
%}
\centering\small
\caption{Computation times in seconds for finding the top three and four  terms of the Ehrhart polynomial}  \label{tab:top3and4times}
\begin{tabular}
{lcc
S[table-format=1.3,table-number-alignment=center]
%S[tabformat=1.3,tabnumalign=center]
cc
S[detect-all,table-format=3.3,table-text-alignment=right,table-number-alignment=right]
%S[obeyall,tabformat=3.3,tabtextalign=right,tabnumalign=right]
}  
\toprule

& \multicolumn{3}{c}{Top 3 coefficients}  & \multicolumn{3}{c}{Top 4 coefficients}   \\ 
\cmidrule(l){2-4}\cmidrule(l){5-7}
  & \emph{LattE} &  \mapleKnapsack & \emph{LattE} &  \emph{LattE} & \mapleKnapsack & \emph{LattE} \\
 & \emph{Knapsack} & & \emph{Top-Ehrhart} & \emph{Knapsack} & & \emph{Top-Ehrhart} \\
\midrule
\#1    & 0 & 0.305   & 0.128  & -- &  --   &  {--} \\
\#2    & 0 & 0.004   & 0.768  & 0  & 0.096 & 1.356\\
\#3    & 0 & 0.004   & 0.788  & 0  & 0.080 & 1.308\\
\#4    & 0 & 0.003   & 0.792  & 0  & 0.124 & 1.368\\
\#5    & 0 & 0.004   & 0.784  & 0  & 0.176 & 1.424\\
\#6    & 0 & 0.004   & 1.660  & 0  & 0.088 & 2.976\\
\#7    & 0 & 0.004   & 0.836  & 0  & 0.272 & 1.652\\
\#8    & 0 & 0.068   & 1.828  & 0  & 0.112 & 3.544\\
\#9    & 0 & 0.004   & 18.437 & 0  & 0.016 & 59.527\\      
\#10   & 0 & 0.012   & 142.104& 0  & 0.044 & 822.187\\
\bottomrule
\end{tabular}
\end{table}

\chapter{MILP heuristic for finding feasible solutions}
\label{ch:dancingLinks} 
      
   In this chapter, we describe the author's  2014 summer internship project at SAS Institute where he developed a Mixed-Integer Linear Programming (MILP) heuristic for finding feasible solutions to MILP problems. The focus of this chapter is to describe a data structure for finding a feasible solution to a system of set partition constraints. We use this data structure to find feasible solutions to more general MILP problems. See Section \ref{sec:bg:milp} for a short background on mixed integer linear programming.
   
\section{Motivation}

A set partitioning constraint is a linear equality constraint that only includes binary variables and every coefficient, including the constant, is one (e.g., $x_1 + x_2 = 1$, $x_i \in \{0,1\}$). Despite having a very specific form, set partitioning constraints have very useful modeling power and often play an important role in MILP problems. There are many examples where these constraints play a role \cite{Balas76, Garfinkel69, garfinkel1972integer, marsten1981exact, vemuganti1999applications}. This section gives a quick introduction via one example.

\subsection{Example}
Consider the following motivating example. Suppose you are responsible for scheduling fabrication jobs to machines in a job shop. Assume the following.
\begin{enumerate}
\item There are $N$ jobs and $M$ machines, each with slightly different characteristics.
\item A job can only be processed on one machine. Once a job has started on a machine, it cannot be preempted or paused.
\item Each job $i \in N$ takes $t_i \in \R$ hours to complete.
\item No machine is allowed to work for more than $8$ hours in a day to give time for maintenance.
\item Job $i$ has a production cost of $c_{ij}$ if processed on machine $j \in M$.
\item Each machine $j$ has a fixed start up expense of $c_j$. Meaning, if at least one job is assigned to machine $j$, $c_j$ must be paid as an expense. If no job is assigned to machine $j$, then there is no expense.
\end{enumerate}

We give one mathematical encoding of the problem. Let us encode on which machine a job is scheduled by a binary variable:

\[ x_{ij} := \begin{cases}
1 & \text{ if job } i \text{ is scheduled on machine } j \\
0 & \text{ else}.
\end{cases} \]
Then requirement (2) above can be encoded in the constraints \[\sum_j x_{ij} = 1 \; \text{ for each job } i.\]
Notice that this produces $|N|$ set partitioning constraints. Requirement (4) can be encoded as \[ \sum_i t_i x_{ij} \leq 8\; \text{for each machine } j. \] Let us introduce another set of binary variables to model the fixed start up expense: 
\[ y_{j} := \begin{cases}
1 & \text{ if machine } j \text{ is used} \\
0 & \text{ else}.
\end{cases} \]
To force $y_j$ to be 1 when at least one job is scheduled on machine $j$ can be done with $|M|\cdot|N|$ constraints: \[x_{ij} \leq y_j \; \text{ for each } i \text{ and } j.\]
Finally, the total expense is then the sum of two terms: the production expenses $c_{ij}$ and set-up expenses $c_j$. We need to minimize \[\sum_j c_j y_j + \sum_{ij} c_{ij} x_{ij}. \]

This example shows two important properties that a lot of MILP problems have. First, notice that the number of set partitioning constraints ($\sum_j x_{ij} = 1$) is small compared to the total number of constraints in the problem. However, these are the ``most important" because once the $x_{ij}$ have been fixed, they force values for everything (or almost everything in the general case). Second, once the $x_{ij}$ have been set, their  feasibility can be easily checked. 

We will use these two properties to form the basis of a MILP heuristic in Section \ref{sec:ch:sas:how-to-use-dacing-links}. The focus of this chapter will be generating feasible solutions to the set partitioning constraints. 

\subsection{MIPLIB}

One of the most publicized benchmarks is the MIPLIB 2010 benchmark set \cite{miplib2010}. This set contains 361  MILP problem instances sorted into three groups: 1) 220 easy problems that can be solved within one hour using any  commercial MILP solver, 2) 59 hard problems that cannot be solved within an hour, and 3) 82 currently unsolved problems. The benchmark contains real-world problems from both industry and academia. 

Out of the benchmark, 133 or $37\%$ of the problems contain  set partitioning constraints. So the MILP heuristic that will be described in this chapter is usable by over a third of the problems in the MIPLIB set! Said differently, it is worth it to develop a heuristic for this special problem structure as it is extremely common. 

\section{Dancing Links}

In this section we review the \emph{Dancing Links algorithm} as described in \cite{dlx}. The algorithm finds a feasible solution to a system of set partition equations by running an exhaustive enumeration over the binary variables. The key feature of the algorithm is a data structure for doing two things: 1) propagating new variable fixings when a variable has been set, 2) bookkeeping everything so that the decision to set a variable to a given value can be undone.

The key idea that makes propagating the implications of setting a variable to a value, and being able to undo all the implications quickly lies in a doubly linked list, a standard  data structure \cite{cormen2009introduction}. Consider the linked list in Figure \ref{fig:dlx-llex-1}. The squares in the figures are the nodes in the linked list, and arrows denote who a node points to in the list. Let $L[2]$ denotes what node two points to on the left, and let $R[2]$ denote what node two points to on the right, then it is easy to remove node two from the list. All we have to do is set $R[L[2]] \gets R[2]$ and $L[R[2]] \gets L[2]$ to produce Figure \ref{fig:dlx-llex-2}. If we were really deleting node two, we should also remove what it is pointing to and delete it. Instead, lets leave node two alone; meaning, we want node two to keep pointing to its old neighbors. If we wanted to insert node two, it would be easy, as node two points to the nodes for which we have to update their pointers. This is the genius behind the dancing links algorithm. It is easy to remove a node by simply updating a set of pointers, and if we want to quickly undo this, the deleted node remembers which nodes to update. 

One side affect of this double link structure is that deleted nodes cannot be reinserted in any order, they must be removed and then inserted in a first-out last-in order. If we first remove node two and then node three, we must add node three back in before adding node two. Figure \ref{fig:dlx-llex} illustrates how the linked list can become corrupted if this order is not kept. Donald Knuth used this idea to develop a data structure that could quickly backtrack in solving the set partitioning problem via brute-force enumeration in \cite{dlx}.  

\begin{figure}
	\captionsetup[subfigure]{slc=off,margin={1cm,0cm}}
	\parbox[b]{0.5\textwidth}{\includegraphics[width=\hsize]{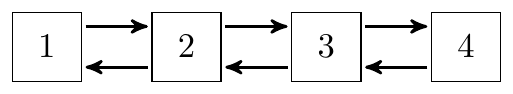}}%
	\parbox[b]{0.5\textwidth}{\subcaption{State of initial of the doubly linked list.}\label{fig:dlx-llex-1}}\\[1ex]
	\parbox[b]{0.5\textwidth}{\includegraphics[width=\hsize]{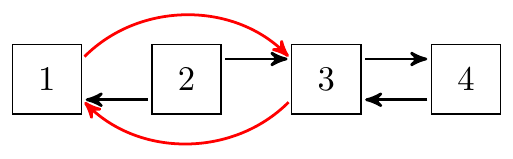}}%
	\parbox[b]{0.5\textwidth}{\subcaption{After removing node $2$. Notice that node $2$ points to its old neighbors. }\label{fig:dlx-llex-2}}\\[1ex]
	\parbox[b]{0.5\textwidth}{\includegraphics[width=\hsize]{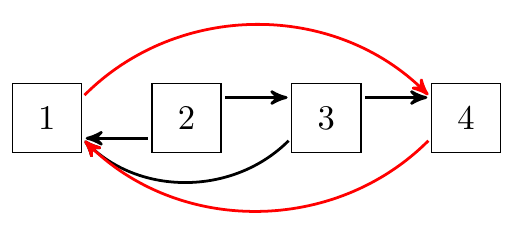}}%
	\parbox[b]{0.5\textwidth}{\subcaption{After removing node 3. Reading the list forward or backward results in only seeing nodes 1 and 4.}\label{fig:dlx-llex-3}}
	\parbox[b]{0.5\textwidth}{\includegraphics[width=\hsize]{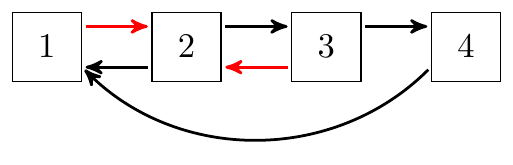}}%
	\parbox[b]{0.5\textwidth}{\subcaption{Inserting node 2 in, before inserting node 3 back in. Reading the list forwards results in $\{1, 2, 3, 4\}.$ Reading the list backwards results in $\{1, 4\}.$}\label{fig:dlx-llex-4}}	
    \caption{The problem with not using first-out last-in ordering.}
    \label{fig:dlx-llex}
\end{figure}

Next we will illustrate his data structure and algorithm on an example.
 \begin{align*}
  (A) \hspace{1em} 1 & =   x_2 + x_4   \\
  (B) \hspace{1em} 1 & =  x_3 +x_5 \\
  (C) \hspace{1em} 1 & = x_1 + x_3 \\
  (D) \hspace{1em} 1 & = x_1 + x_2 + x_3 
 \end{align*}
 
Before describing the data structure behind dancing links, lets see how a brute force search for a solution can be performed. Lets assume $x_2 = 1$ and see if a solution is possible. The equations become
 \begin{align*}
  (A) \hspace{1em} 1 & =   1 + x_4   \\
  (B) \hspace{1em} 1 & =  x_3 +x_5 \\
  (C) \hspace{1em} 1 & = x_1 + x_3 \\
  (D) \hspace{1em} 1 & = x_1 + 1 + x_3.
 \end{align*} 
This implies that $x_2 = x_3 = x_4 =0$. Inserting these values produces a contradiction in equation $(C)$,
 \begin{align*}
  (A) \hspace{1em} 1 & =   1 + 0   \\
  (B) \hspace{1em} 1 & =  0 +x_5 \\
  (C) \hspace{1em} 1 & = 0 + 0 \\
  (D) \hspace{1em} 1 & = 0 + 1 + 0.
 \end{align*} 

We see that the assumption $x_2=1$ produces a contradiction, and so if a solution to this system, $x_2$ must equal zero. The data structure in the dancing links algorithm will easily keep track of which variables are assigned values, propagate variable implications, and undo things when contradictions are produced.
  
We now collect the original equations into the rows of a matrix, and because the constant term is always one, it can be ignored, see Figure \ref{fig:dxl-initial-matrix}. We then associate a doubly linked circular list to each column and row of the matrix, see Figure \ref{fig:dxl-initial}. The first column contains row-headings for each row. The first data node in the header column is a special ``root" like object that is the starting point for any list iteration. Each data node $n$ will have five arrays of pointers associated with it: $L[n], R[n], U[n], D[n], H[n]$. These correspond to a node's left, right, up, and down neighbors in the linked list structures. $H[n]$ is a pointer back to a node in the header row (not illustrated in Figure \ref{fig:dxl-initial}). The root header object does not have a left or right linked list attached to it. Also, every header row node does not use the $H[n]$ pointer.

\begin{figure}
    \centering  
    \begin{subfigure}[b]{0.5\textwidth}
 \[
 \begin{pmatrix}
 A & & x_2 & & x_4 & \\
 B & & & x_3 & & x_5 \\
 C & x_1 & & x_3 & & \\
 D & x_1 & x_2 & x_3 & &
\end{pmatrix}
\]
\vspace{0.85in}
        \caption{Matrix form}
        \label{fig:dxl-initial-matrix}
    \end{subfigure}
    ~ %add desired spacing between images, e. g. ~, \quad, \qquad, \hfill etc. 
    \begin{subfigure}[b]{0.5\textwidth}
        \includegraphics[width=\textwidth]{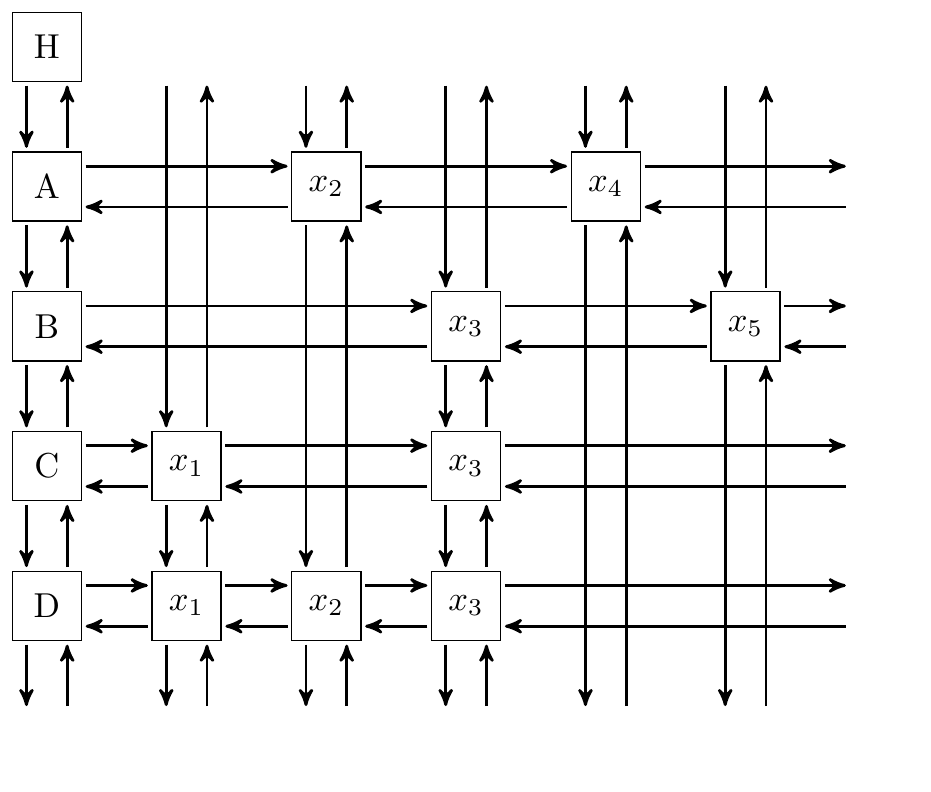}
        \caption{The initial data structure}
        \label{fig:dxl-initial}
    \end{subfigure}
    \caption{The initial state of the dancing links data structure}\label{fig:animals}
\end{figure}

\begin{algorithm}
\caption{Dancing Links Algorithm}
\label{alg:dlx}
\begin{justify}
Procedure \texttt{SEARCH}
\end{justify}
\begin{algorithmic}
\STATE Let $H$ be the root node.
\STATE Let $S$ be a stack of nodes that are nonnegative.
\STATE If $D[H] = H$, print the current solution from stack $S$ and stop.
\STATE Pick a row $r$ to process (say $r \gets D[H]$)
\STATE Run procedure \texttt{COVER}$(r).$
\FOR{ $n \gets R[r], R[R[r]], \dots$, while $n \neq r$}
	\STATE $S.$\texttt{PUSH}$(n)$
	\FOR{$j \gets U[n], U[U[n]], \dots$, while $j \neq n$}
		\STATE Run procedure \texttt{COVER}$(H[j]).$
	\ENDFOR
	\STATE Run procedure \texttt{SEARCH}
	\STATE $n \gets S.$\texttt{POP}$()$, $r \gets H[n]$
	\COMMENT{Search failed, undo last step}
	\FOR{$j \gets D[n], D[D[n]], \dots$, while $j \neq n$}
		\STATE Run procedure \texttt{UNCOVER}$(H[j]).$
	\ENDFOR
\ENDFOR
\STATE Run procedure \texttt{UNCOVER}$(r)$ and stop.
\end{algorithmic}
\begin{justify}
Procedure \texttt{COVER}$(n)$
\end{justify}
\begin{algorithmic}
\STATE $D[U[n]] \gets D[n]$, and $U[D[n]] \gets U[n]$.
\FOR{ $i \gets R[n], R[R[n]], \dots$, while $i \neq n$}
	\FOR{$j \gets U[i], U[U[i]], \dots$, while $j \neq i$}
		\STATE $L[R[j]] \gets L[j]$, and $R[L[j]] \gets R[j]$
	\ENDFOR
\ENDFOR
\end{algorithmic}
\end{algorithm}

The outline of the method is given in Algorithm \ref{alg:dlx}. We now illustrate the workings of the algorithm on our example. We start by calling procedure \texttt{SEARCH}. First, let us pick row $A$ to process. The first thing we do is run procedure \texttt{COVER} on the row. This procedure has the job of removing $x_2$ and $x_4$ from the list. By this, we mean it will be impossible to start at the root node and iterate over the list at reach a $x_2$ or $x_4$ node. See Figure \ref{fig:dxl-coverA}. The outer for-loop then loops over the variables in the $A$ row and iteratively attempts to assign them to $1$. When a node (variable) is assigned to one, it is recorded by inserting it into the solution stack $S$, so at any point $S$ will contain which variables are nonzero. It is important that $S$ is a stack of nodes in the linked list structure, and not just a stack of which variables are one. This allows us to \texttt{POP} a node from $S$ and know what variable it represents, and what row it came from. At this point we add row $A$'s $x_2$ node to the solution stack $S$. The first inner for-loop then loops over the rows that have an $x_2$ and calls the \texttt{COVER} procedure on them. This has the effect of propagating the implications of setting $x_2=1$, in that row $D$ is removed (because it is satisfied) and variables $x_1$ and $x_3$ have been removed (set to zero). The first step in covering row $D$ is to remove the $x_1$ node, see Figure \ref{fig:dxl-coverDx1}. The second step in covering row $D$ is to remove the $x_3$ node, see Figure \ref{fig:dxl-coverDx3-contradiction}. Next we finally enter our first recursive call to \texttt{SEARCH}. Upon doing so we should pick row $C$ to process next. We immediately see that row $C$ has no nodes left, and hence our current variable assignment $x_2=1$ leads to an infeasible system. 

\begin{figure}[H]
    \centering
    \begin{subfigure}[b]{0.45\textwidth}
        \includegraphics[width=\textwidth]{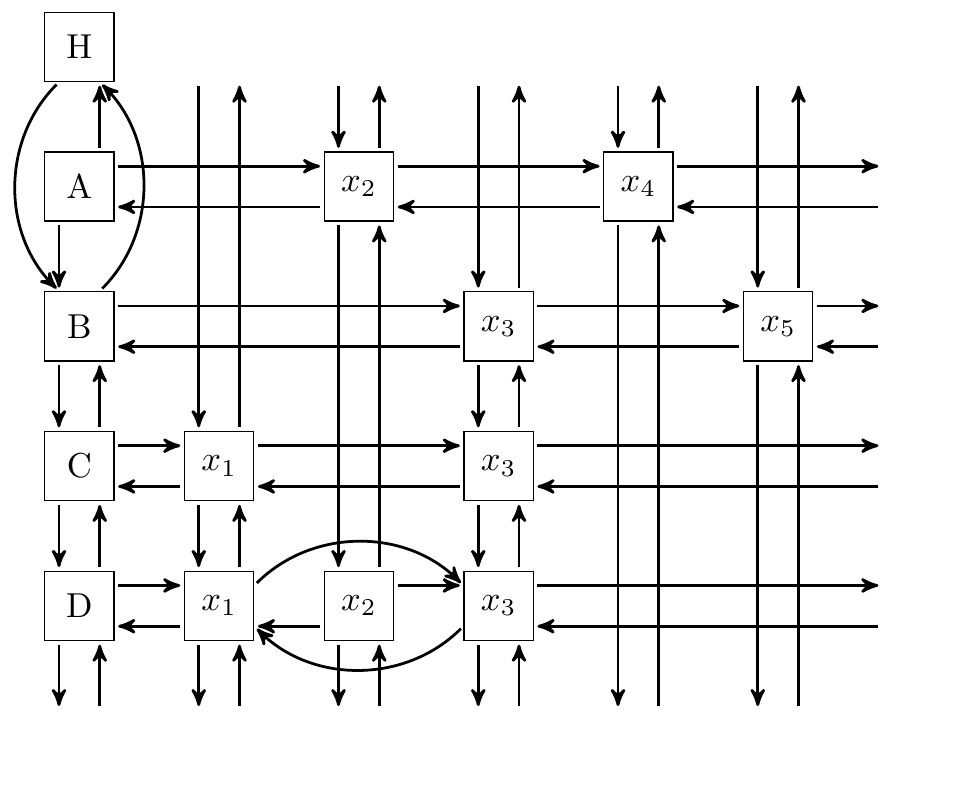}
        \caption{After running the cover procedure on row $A$. Will first pick $x_2=1$.}
        \label{fig:dxl-coverA}
    \end{subfigure}
    \quad %add desired spacing between images, e. g. ~, \quad, \qquad, \hfill etc. 
    \begin{subfigure}[b]{0.45\textwidth}
        \includegraphics[width=\textwidth]{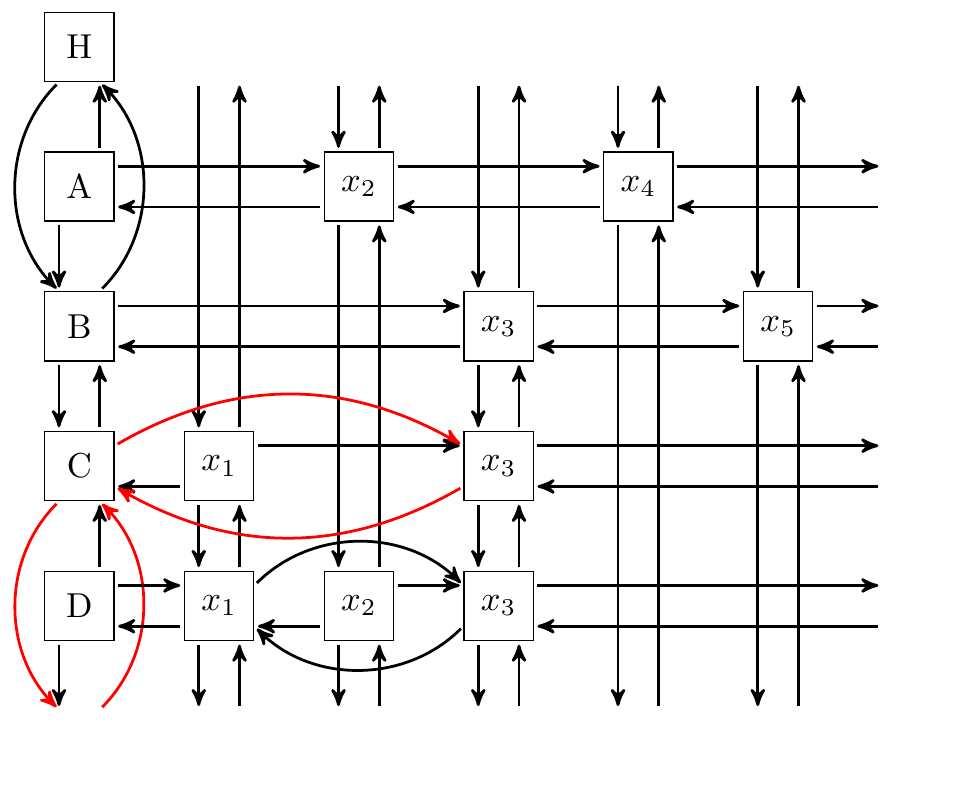}
        \caption{Propagating $x_2 = 1$: Cover row $D$. Step one is to remove $x_1$ from the list.}
        \label{fig:dxl-coverDx1}
    \end{subfigure}

    \caption{Selecting $x_2=1$.}\label{fig:animals}
\end{figure}

\begin{figure}[H]
    \centering
    \begin{subfigure}[b]{0.45\textwidth}
        \includegraphics[width=\textwidth]{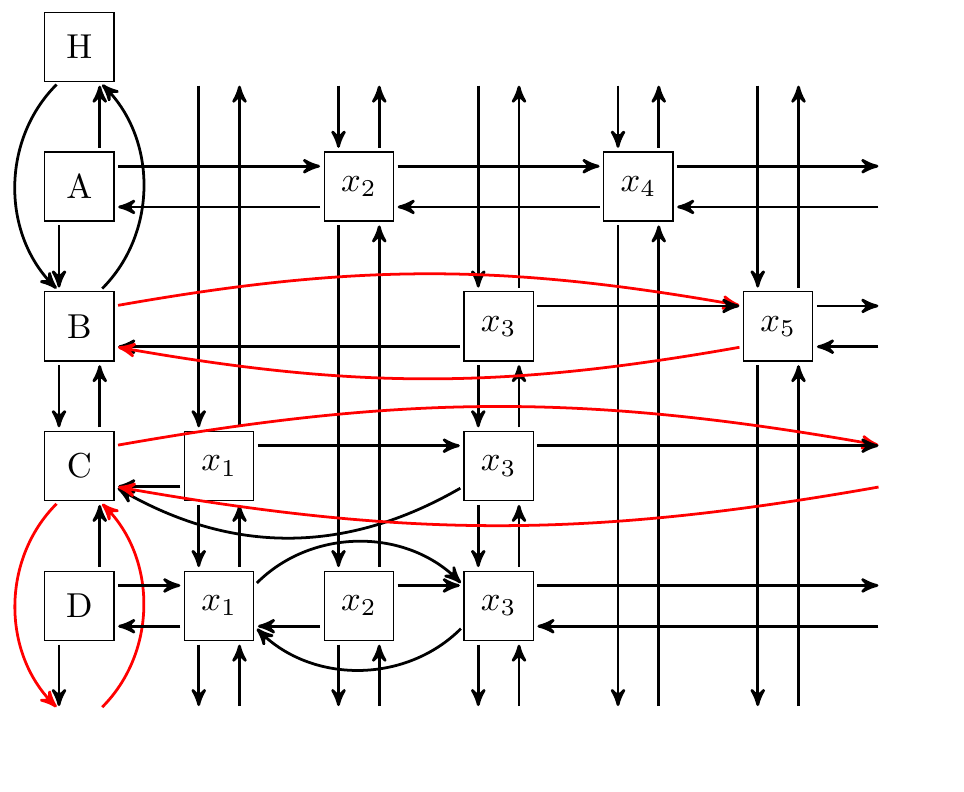}
        \caption{Propagating $x_2 = 1$: Cover row $D$. Step two is to remove $x_3$ from the list. This produces the contradiction in row $C$.}
        \label{fig:dxl-coverDx3-contradiction}
    \end{subfigure}
    \quad %add desired spacing between images, e. g. ~, \quad, \qquad, \hfill etc. 
    \begin{subfigure}[b]{0.45\textwidth}
        \includegraphics[width=\textwidth]{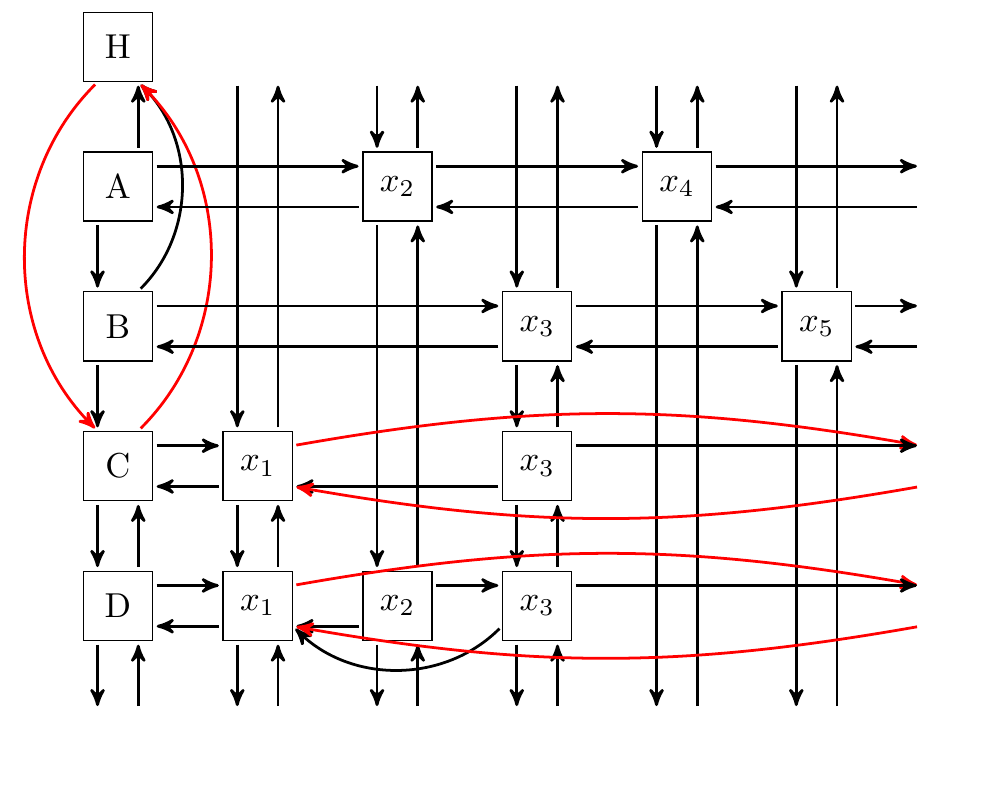}
        \caption{After covering row $B$.}
        \label{fig:dxl-coverAx4-CoverB}
    \end{subfigure}

    \caption{Getting a contradiction and undoing it.}\label{fig:animals}
\end{figure}

We continue stepping in the algorithm to see how it handles this. First, \texttt{COVER} is called on row $C$ which only has the effect of removing itself from the home column. The outer for-loop does not perform any iterations as $R[C] = C$. The \texttt{UNCOVER} procedure is not reproduced here, but it has the task of undoing what the \texttt{COVER} procedure does. Hence \texttt{UNCOVER} on row $C$ just puts it back into the first column's list. The second call to \texttt{SEARCH} now terminates and returns control to the first \texttt{SEARCH} call. Node $A$'s $x_2$ is removed from the stack $S$. \texttt{UNCOVER} is called on row $D$ and the data structure is put back into the initial state after \texttt{COVER} was called on row $A$ (see Figure \ref{fig:dxl-coverA}). Node $A$'s $x_4$ is then set to one an inserted in $S$. It appears in no other row, so no additional \texttt{COVER} operations are done. Then \texttt{SEARCH} is called again (the recursive tree is at depth one now). Let us pick row $B$ to process, and start by calling \texttt{COVER} on row $B$, see Figure \ref{fig:dxl-coverAx4-CoverB}. The first element in the row is selected and $x_3 = 1$, and $B$'s $x_3$ is inserted in $S$. As $x_3$ appears in rows $C$ and $D$, \texttt{COVER} is also called on those rows, see Figure \ref{fig:dxl-final}. This concludes the bound propagation of $x_3 =1$, and \texttt{SEARCH} is recursively called again. However, the linked list starting at the root only contains the root, and so this procedure terminates. The state of the stack $S$ contains $x_4$ and $x_3$, implying a solution of $x_4 = x_3 = 1$ with the other variables set to zero.

\begin{figure}
    \centering
    \begin{subfigure}[b]{0.45\textwidth}
        \includegraphics[width=\textwidth]{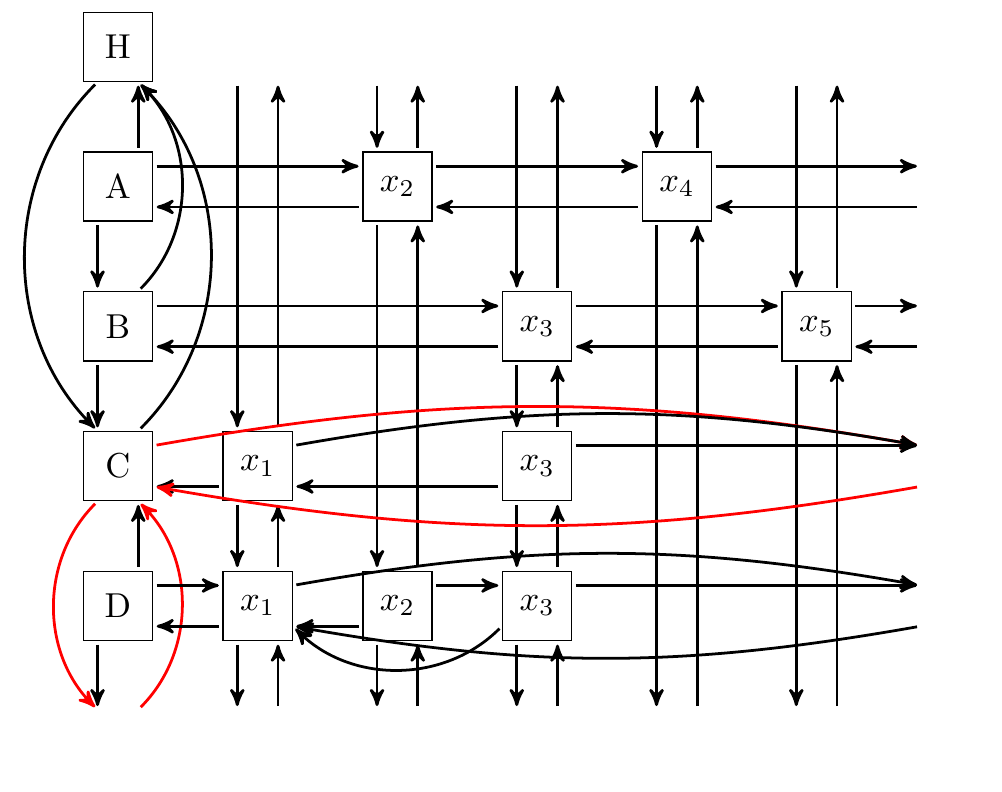}
        \caption{After covering row $D$.}
        \label{fig:dxl-coverAx4-CoverB-CoverD}
    \end{subfigure}
    \quad %add desired spacing between images, e. g. ~, \quad, \qquad, \hfill etc. 
    \begin{subfigure}[b]{0.45\textwidth}
        \includegraphics[width=\textwidth]{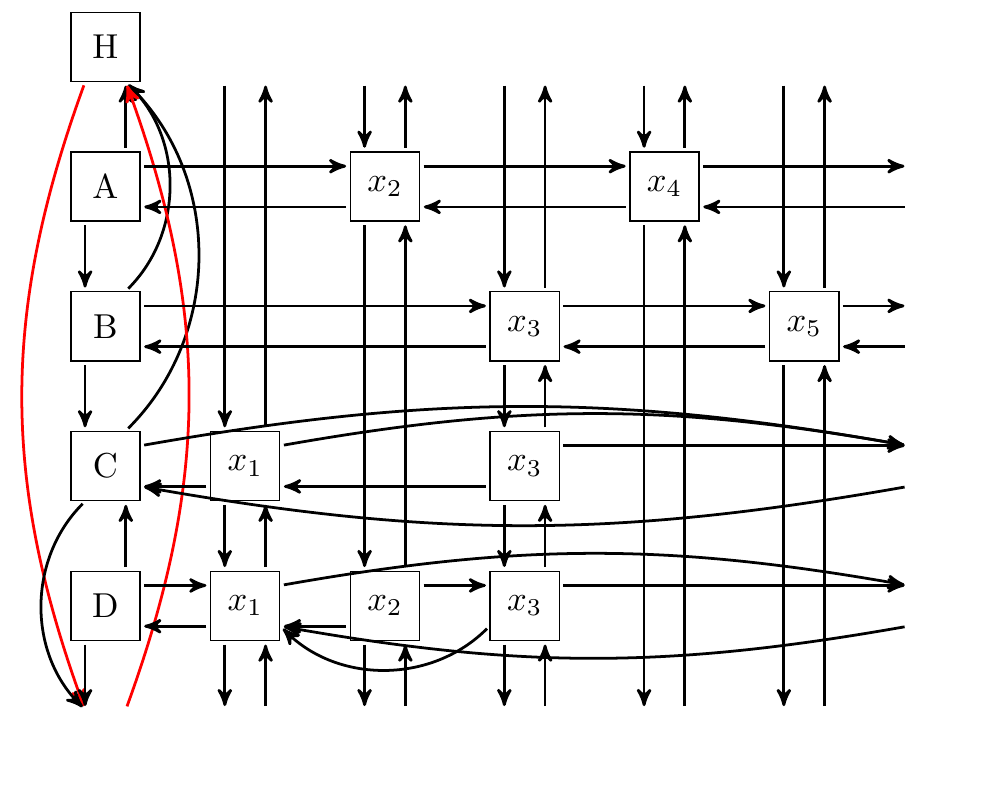}
        \caption{After covering row $C$. As $H$ points to itself, a feasible solution was found.}
        \label{fig:dxl-coverAx4-CoverB-CoverD-CoverC}
    \end{subfigure}

    \caption{The final state}\label{fig:dxl-final}
\end{figure}

There are no restrictions on how to pick the next row to process, other than it has to be one obtainable from the root's linked list. One could select a row randomly, or by selecting the row with the fewest (or largest) number of variables left. There are many ways a row could be given a priority. For example, a row's priority could be a function of how many of its variables appear in other rows.

Lastly, if these figures were placed in a movie, it would be clear why this algorithm is called dancing links! See \cite{dlx} for a full description of the \texttt{UNCOVER}  procedure. 

\section{How to use dancing links}
\label{sec:ch:sas:how-to-use-dacing-links}

There are a few ways dancing links can be applied in the solution process; for example, as a heuristic for finding a feasible point to a MILP problem. Even within this area, there are a few ways of doing this. Our approach will be to use dancing links to get a partial feasible solution and create a new MILP to build the fully feasible solution. 

Consider the MILP 

\begin{align*}
\min \quad & c_1^Tx  + c_2^Ty\\
Ax + By &\leq b\\
         Cy & = \ve 1 \\
         y_i \in \{0, 1\} & \\
         x_i \in \Z \text{ for } i \in I &
\end{align*}
where we explicitly write the set partition constraints in the form $Cy = \ve 1$, and $\ve 1$ is a vector of all ones. 

The strategy behind our heuristic is to apply dancing links to $Cy=\ve 1$. If dancing links shows that $Cy=\ve 1$ is infeasible, then the original MILP must be infeasible. Otherwise, we have that $Cy^0 = \ve 1$ and $y^0_i \in \{0, 1\}$. We can create a new MILP by plugging in the values for the $y$ variable and solve the smaller problem

\begin{align*}
\min \quad & c_1^Tx\\
Ax &\leq b - By^0\\
         x_i \in \Z \text{ for } i \in I. &
\end{align*}

This then produces a feasible point to the original problem. The idea is that by fixing the $y$ variable, the new problem is very easy to solve. If this is not the case, then creating a sub-MILP problem has no benefits. 

Extending the partial feasible solution $y^0$ to a feasible solution to the original problem can be done in many different ways. The above idea was to create a smaller MILP problem to find a $x^0$ solution to go with $y^0$. Extending the partial feasible solutions to fully feasible solutions are generally referred to as ``repair" heuristics in the literature, and so dancing links can be paired with many repair heuristics.

\section{Results}

One main contribution of the author's internship was implementing the above strategy as a heuristic plugin to the SAS/OR MILP solver. Table \ref{tab:sas-dl} illustrates the impact the heuristic had on a benchmark set. The  benchmark used contained $184$ problems, where each problem contained set partition constraints and the  SAS solver called the heuristic at least once on each of them. Of these $184$ problems, the sub-MILP heuristic was able to find feasible points to $35$ of these, so the heuristic was successful on $19\%$ of the problems. The feasible point that was found was the best feasible point in the solution process at the time to $13$ problems. Said differently, on $13$ or $7\%$ of the problems, the heuristic was responsible for producing new global primal bounds. 

\begin{table}[]
\centering
\caption{Impact of using dancing links in a sub-MILP heuristic}
\label{tab:sas-dl}
\begin{tabular}{@{}lll@{}}
\toprule
                          & Number of problems & Percent of problems \\ \midrule
size of benchmark & $184$                   &                     \\
new feasible points      & $35$                   & $19\%$                    \\
better feasible points   & $13$                   & $7\%$                    \\ \bottomrule
\end{tabular}
\end{table}

The dancing links data structure produces a decent MILP heuristic that takes advantage of the role set partitioning constraints play in many problems. A nice feature of the data structure is that it can be extended to work with more general constraints. 

It is easy to extend dancing links to also work with set packing constraints (e.g., $x_1 + x_2 \leq 1$, $x_i$ binary). For example, these rows could simply be added to the dancing links data structure. Every row could be given a priority based on their type, and when picking a new row to process, the priority could be taken into account. If set partition constraints correspond to a high priority and set packing constraints correspond to a low priority, then when a low priority row is processed, all the set partition constraints are satisfied and a feasible solution is obtained.

\chapter{Application in distance geometry for neuro-immune communication}
\label{ch:spleen} 

This chapter explains the mathematical tools used to study the neuro-immune interaction in the spleen during inflammation. See Section \ref{sec:bg:spleen} for a fuller description of the medical problem. There are two kinds of objects in the spleen we are interested in: $CD4^+ChAT^+$ T-cells (henceforth referred to as just T-cells) and nervous-system filaments. The main question was, how does the distribution of these two objects change as the spleen responds to inflammation? That is, are T-cells found closer to nerve cells during inflammation, or is there no relationship between the location of these two types of cells in the spleen? 

To research this question, my collaborators at the University of California, Davis, School of Veterinary Medicine have imaged the spleen of 18 mice, some of which suffer from inflammation. Advanced imaging technology is used to map the location of each cell type. However, the resulting image is extremely noisy. The software tool Imaris \cite{imaris-software} is used to construct the true location of the T-cells and nerve filaments. Despite being a powerful tool, using Imaris  to clean the imaging data is a labor- and time-intensive task. However, it is possible to write \matlab \cite{MATLAB:2015b} scripts that can interface with Imaris data sets. To help clean the data, we have written \matlab scripts that take Imaris data, clean it up by reducing the noise, and return it to Imaris. Section \ref{ch:spleen:sec:data} explains some of the methods we have developed to speed up the data cleaning steps.

After the data is cleaned, the central task was to study the distribution of T-cells and nerve cells. Our approach was to use cluster analysis to determine if there are differences during inflammation. Section \ref{ch:spleen:sec:cluster} is devoted to this process.

All code from this project can be found in the appendix. 

\section{Cleaning the data}
\label{ch:spleen:sec:data}

The initial data contains pixel information from the microscope. Each image is layered, creating a three dimensional array of pixel information called \emph{voxels}. The spleen is stained and the voxel values contain an intensity value of how much dye was detected by the microscope. Larger intensity values are associated with T-cells and nerve cells. Figure \ref{fig:spleen:rawdata} shows the raw voxel information. An immediate observation is that the imaging data is extremely noisy. Consider the raw nerve cell data. We can see a few regions of bright red lines or strings. These correspond to true nerve filaments. There is also a cloud of low intensity red points everywhere. These low intensity points mark other cells we are not interested in, and also contains imaging noise. The low intensity points must be filtered out. Likewise, the bright green areas are the T-cells we are interested in, while the low intensity cloud of points are noise representing other cell types. 

\begin{figure}
    \centering
    \begin{subfigure}[b]{0.3\textwidth}
        \includegraphics[width=\textwidth]{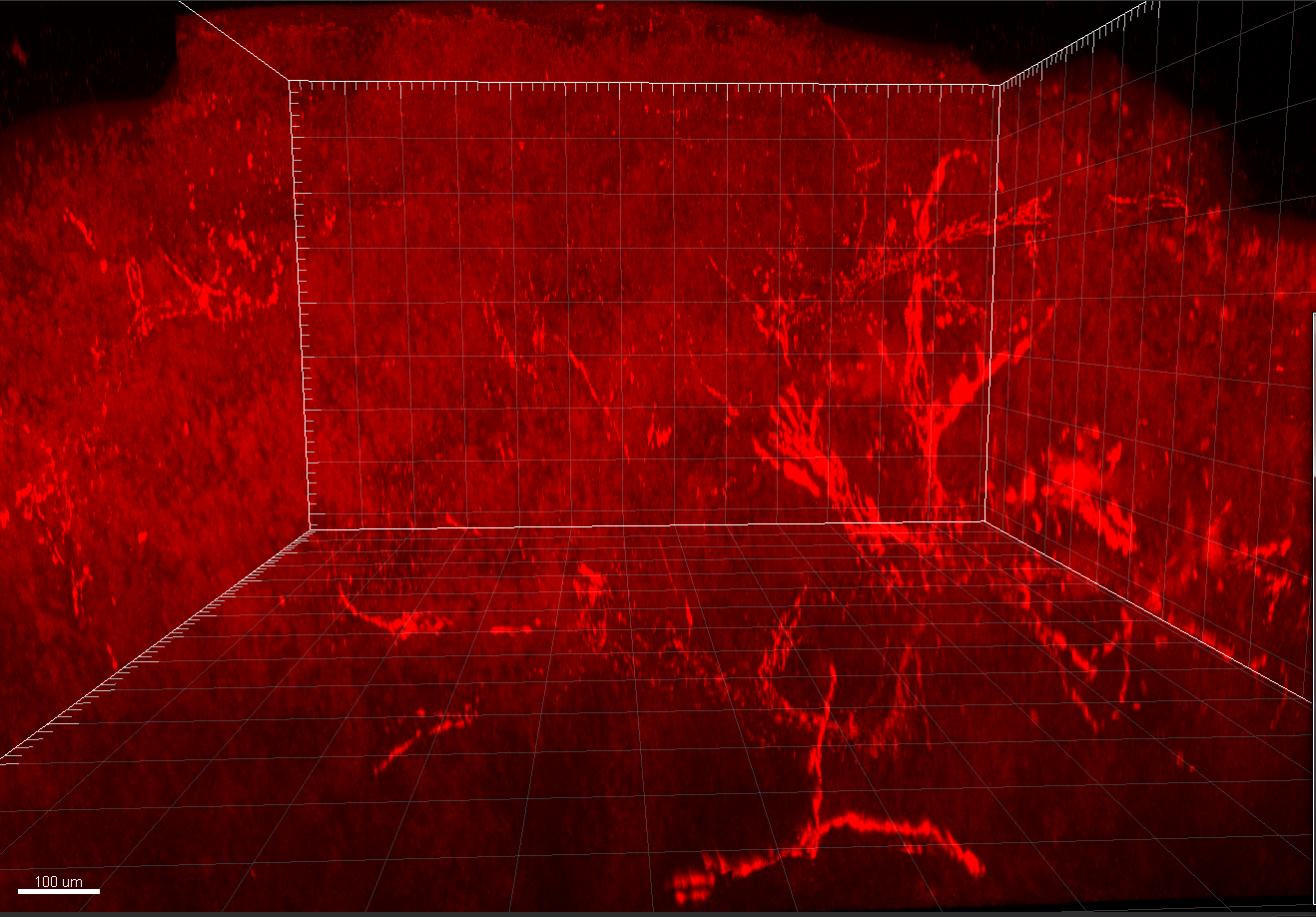}
        \caption{Raw nerve filament data}
        \label{fig:spleen:rawdata:a}
    \end{subfigure}
    ~ %add desired spacing between images, e. g. ~, \quad, \qquad, \hfill etc. 
    \begin{subfigure}[b]{0.3\textwidth}
        \includegraphics[width=\textwidth]{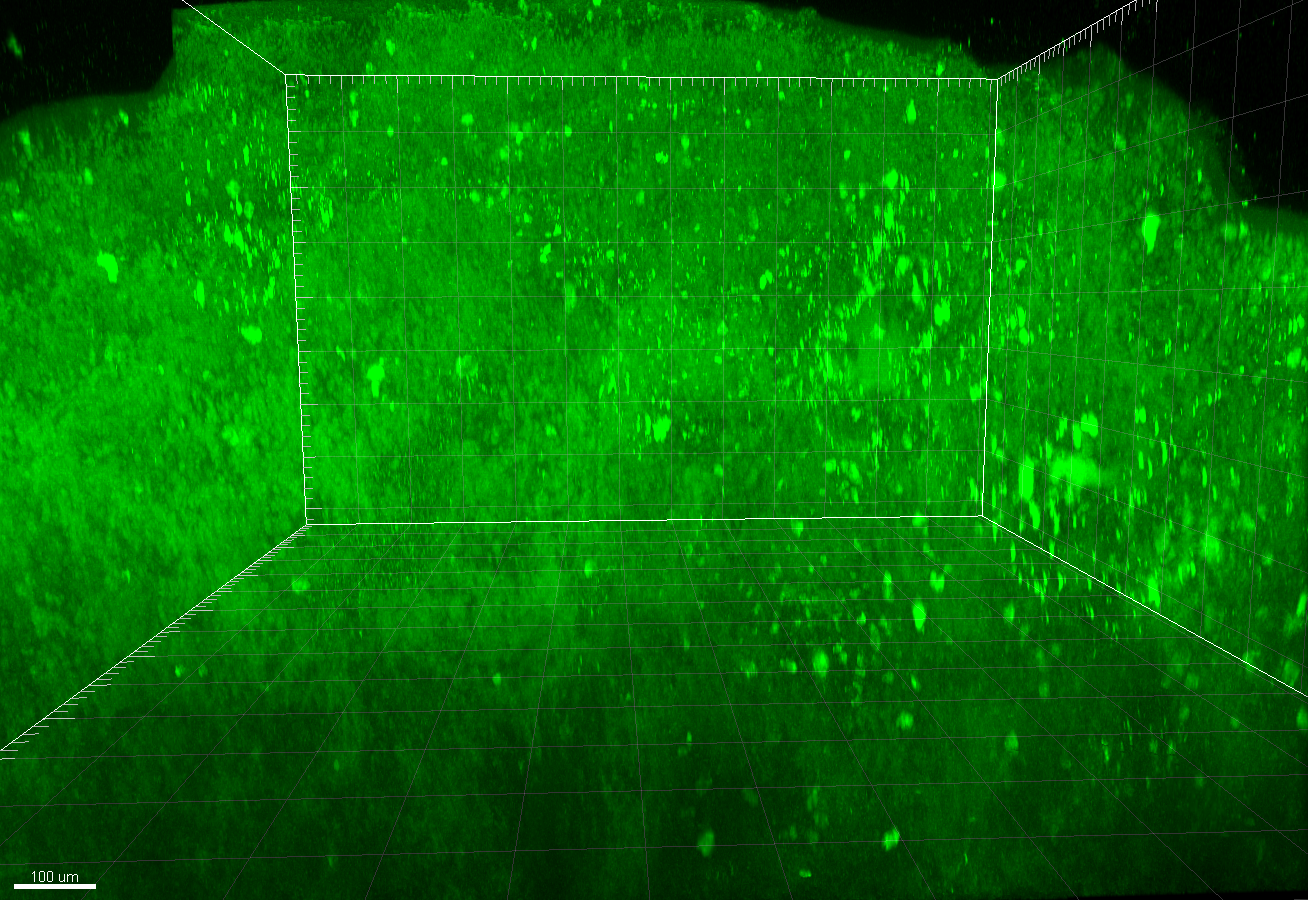}
        \caption{Raw T-cell data}
        \label{fig:spleen:rawdata:b}
    \end{subfigure}
    ~ %add desired spacing between images, e. g. ~, \quad, \qquad, \hfill etc. 
    \begin{subfigure}[b]{0.3\textwidth}
        \includegraphics[width=\textwidth]{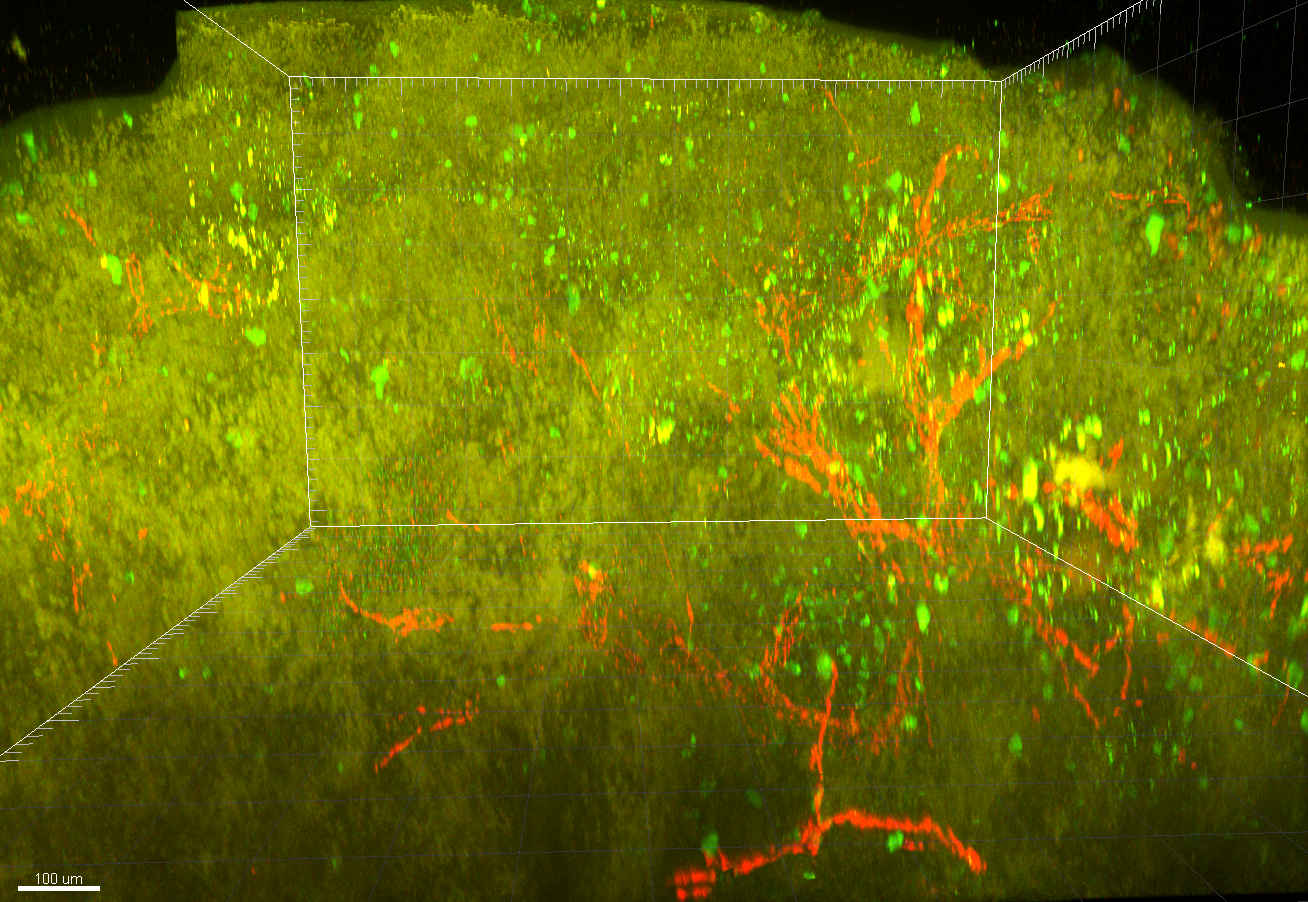}
        \caption{Overlay}
        \label{fig:spleen:rawdata:c}
    \end{subfigure}
    \caption{Raw microscope data}\label{fig:spleen:rawdata}
\end{figure}

In Figure \ref{fig:spleen:modelView}, we show how to use Imaris to remove the noise from the red nerve filaments, and how the filaments are modeled as a triangulated surface. That is, Figure \ref{fig:spleen:modelView:a} contains about 2000 different surfaces, where each surface is a two dimensional triangulated mesh. The T-cell data is simply modeled as a collection of spheres in Figure \ref{fig:spleen:modelView:b}, and we refer to these spheres as \emph{spots}. Going from Figure \ref{fig:spleen:rawdata} to Figure \ref{fig:spleen:modelView} is an extremely time-consuming process. In this section we describe the process of cleaning the data and how we applied mathematical tools to speed the process up. 

%MC avoid too much "in this section"

\begin{figure}
    \centering
    \begin{subfigure}[b]{0.4\textwidth}
        \includegraphics[width=\textwidth]{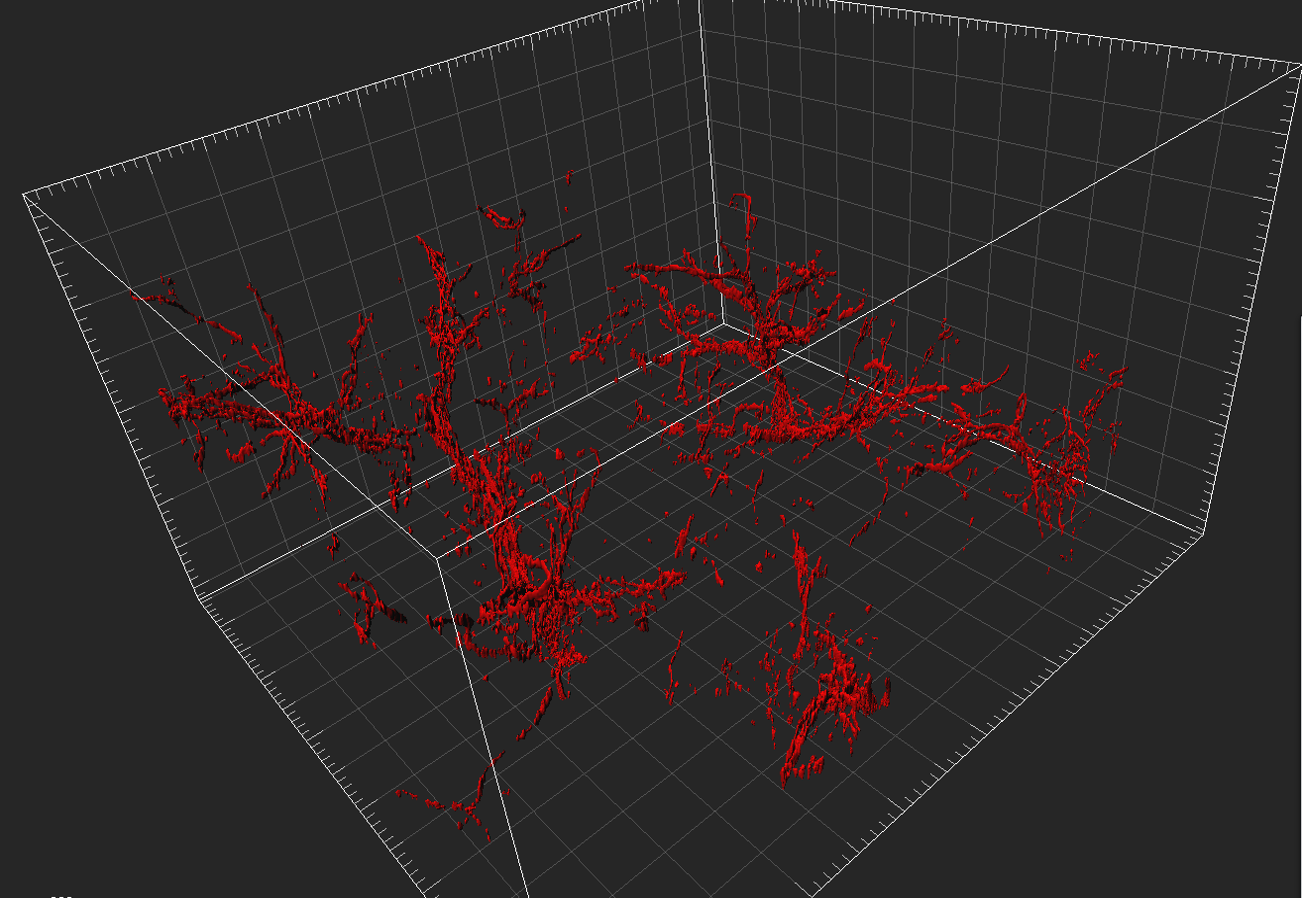}
        \caption{Filaments}
        \label{fig:spleen:modelView:a}
    \end{subfigure}
    ~ %add desired spacing between images, e. g. ~, \quad, \qquad, \hfill etc. 
    \begin{subfigure}[b]{0.4\textwidth}
        \includegraphics[width=\textwidth]{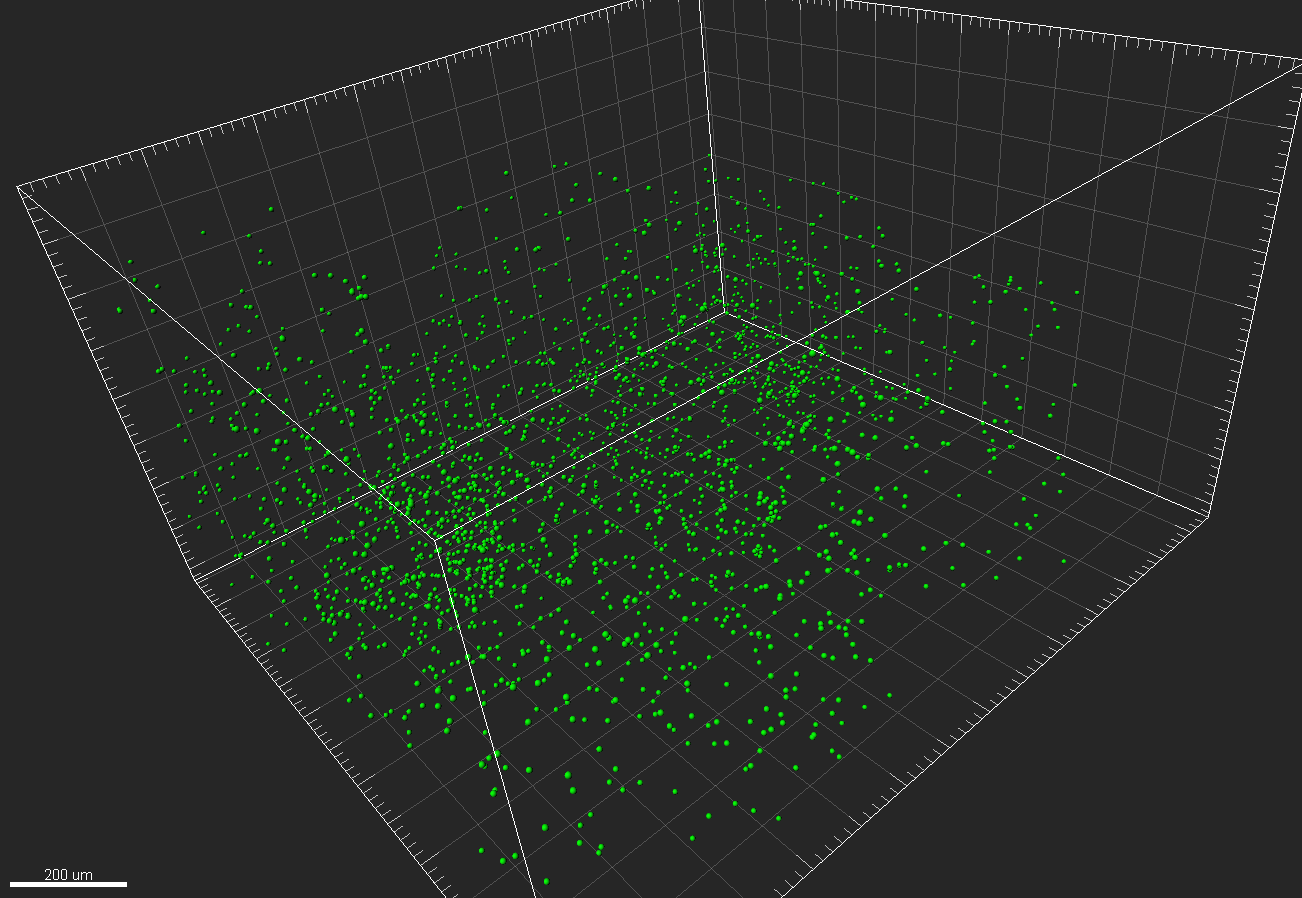}
        \caption{Spots}
        \label{fig:spleen:modelView:b}
    \end{subfigure}
    \caption{Clean data model}\label{fig:spleen:modelView}
\end{figure}

\subsection{Step 1: the boundary}
\label{ch:spleen:clean:1}

The boundary of the spleen produces imaging artifacts. Figure \ref{fig:spleen:noiseBoundary:a} shows a slice of the spleen in the original voxel data. There is a band of data near the boundary of the spleen where the voxel data looks fuzzy. This results in many false positives for T-cells and nerve filaments from the raw microscope data. If this band of noise is not removed from the data set, Imaris will later identify these areas as having unrealistically many T-cells and nerve filaments. Therefore, the first step in processing the image data is to remove the boundary from the raw voxel data set. Before joining the project, the researchers in the veterinary school were manually erasing the noisy boundary from each slice of the image. That is, if a spleen contained 2000 image slices, a human would manually edit 2000 images by erasing the boundary. This was a highly repetitive and time-consuming task.

\begin{figure}
    \centering
    \begin{subfigure}[b]{0.4\textwidth}
        \includegraphics[width=\textwidth]{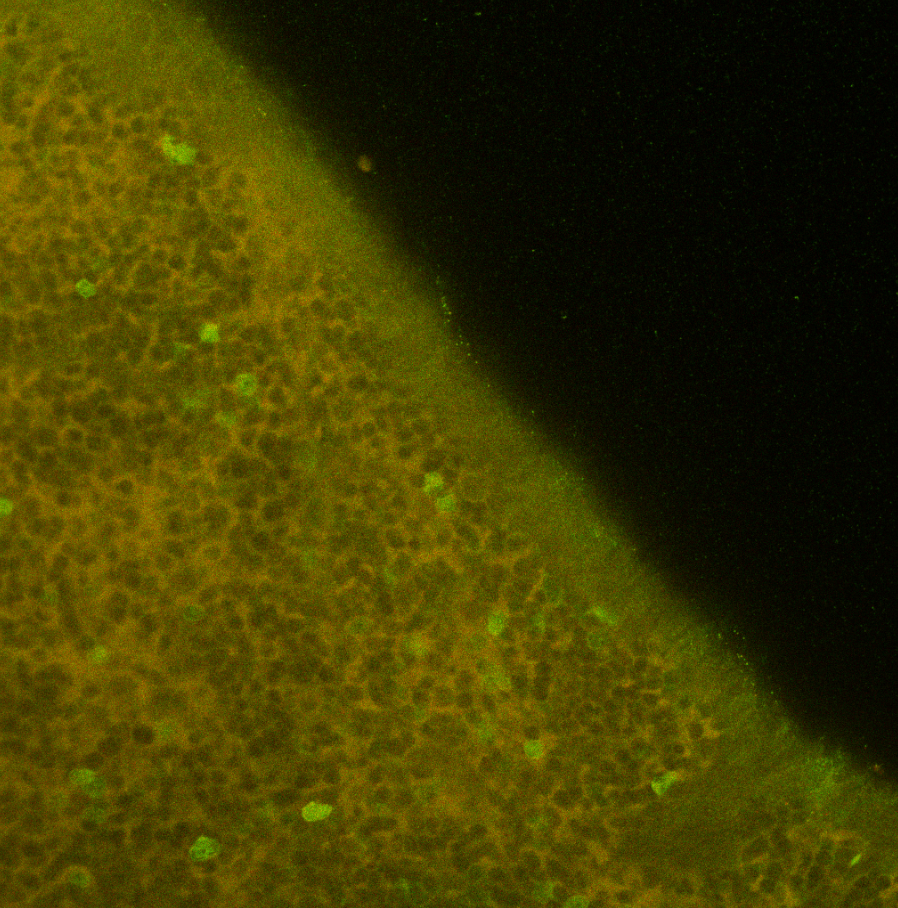}
        \caption{Spleen with boundary}
        \label{fig:spleen:noiseBoundary:a}
    \end{subfigure}
    ~ %add desired spacing between images, e. g. ~, \quad, \qquad, \hfill etc. 
    \begin{subfigure}[b]{0.4\textwidth}
        \includegraphics[width=\textwidth]{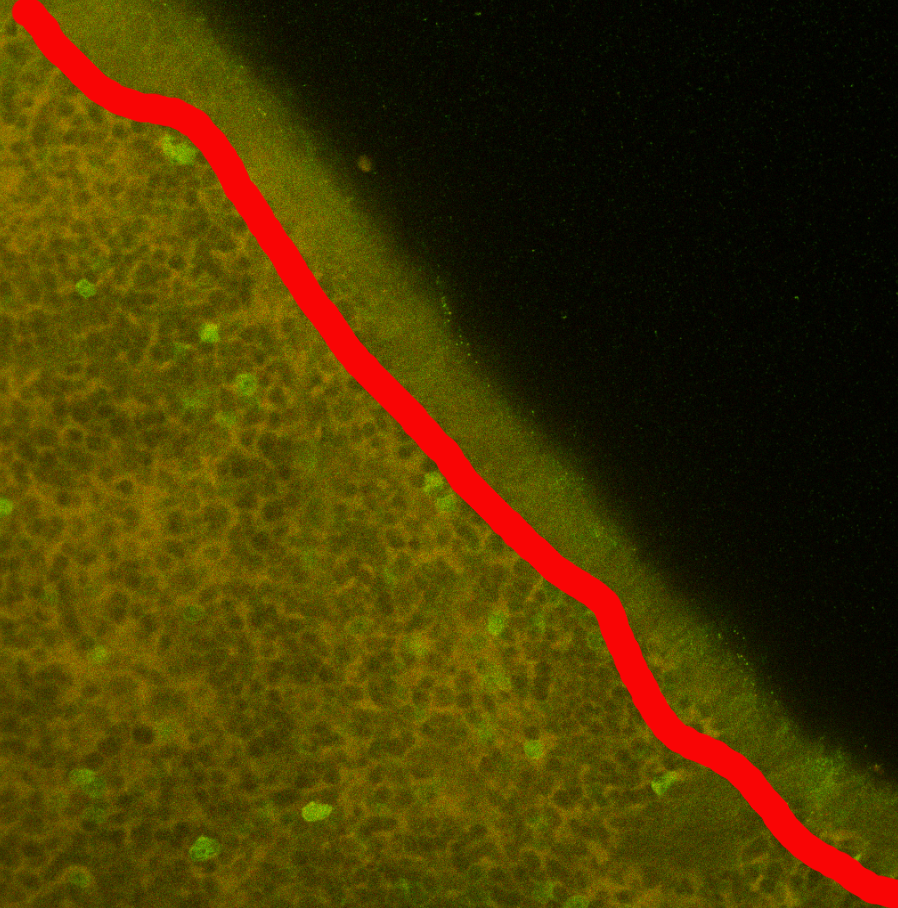}
        \caption{What we want to cut from the data}
        \label{fig:spleen:noiseBoundary:b}
    \end{subfigure}
    \caption{Boundary of the spleen}\label{fig:spleen:noiseBoundary}
\end{figure}

To do this automatically, our approach was to build a surface using Imaris that is a triangular mesh representing the boundary of the spleen. Imaris has the capability to filter a data set based on whether  it is inside of some surface. The boundary surface can be shifted inward by a small distance, and them Imaris can be used to filter data that now lies on the outside of the smaller surface. This  would remove the boundary of the spleen. Figure \ref{fig:spleen:triangularBoundary:a} is a partial view of the triangulated surface representing the boundary of the surface. Notice that this mesh has some flat sides representing the face where the spleen was cut by a scalpel. Also, the full surface is essentially a deformed sphere, without any holes.  Hence the problem was to take a polygonal mesh, and offset it by a small negative distance. In the computational geometry community, this problem is related to computing \emph{straight skeletons} \cite{aichholzer1996novel}, and many excellent software packages exist for this kind of computation such as the Computational Geometry
Algorithms Library \cite{cgal:c-sspo2-15b}. See also \cite{chen2005point, liu2011fast, liu2008vertex}. The difficulty in computing mesh offsets is that the topology of the mesh can change. For example, when shrinking a surface, two edges might intersect, requiring finding a new set of  verticals and triangles to describe the shifted mesh. Instead of proceeding further, we experimented with the naive approach of just shifting each vertex in the negative direction of its vertex normal vector. A vertex normal vector is the average of the normals at each facet the vertex is adjacent to. This simple approach is incorrect, as it does not work at sharp angles, see Figure \ref{fig:spleen:triangularBoundary:b}. However, we were able to easily remove the noisy spleen boundary in Imaris with this approach. This is a rare example when the \emph{correct} (and more sophisticated) method was not the best way to solve a problem! 

\begin{figure}
    \centering
    \begin{subfigure}[b]{0.4\textwidth}
        \includegraphics[width=\textwidth]{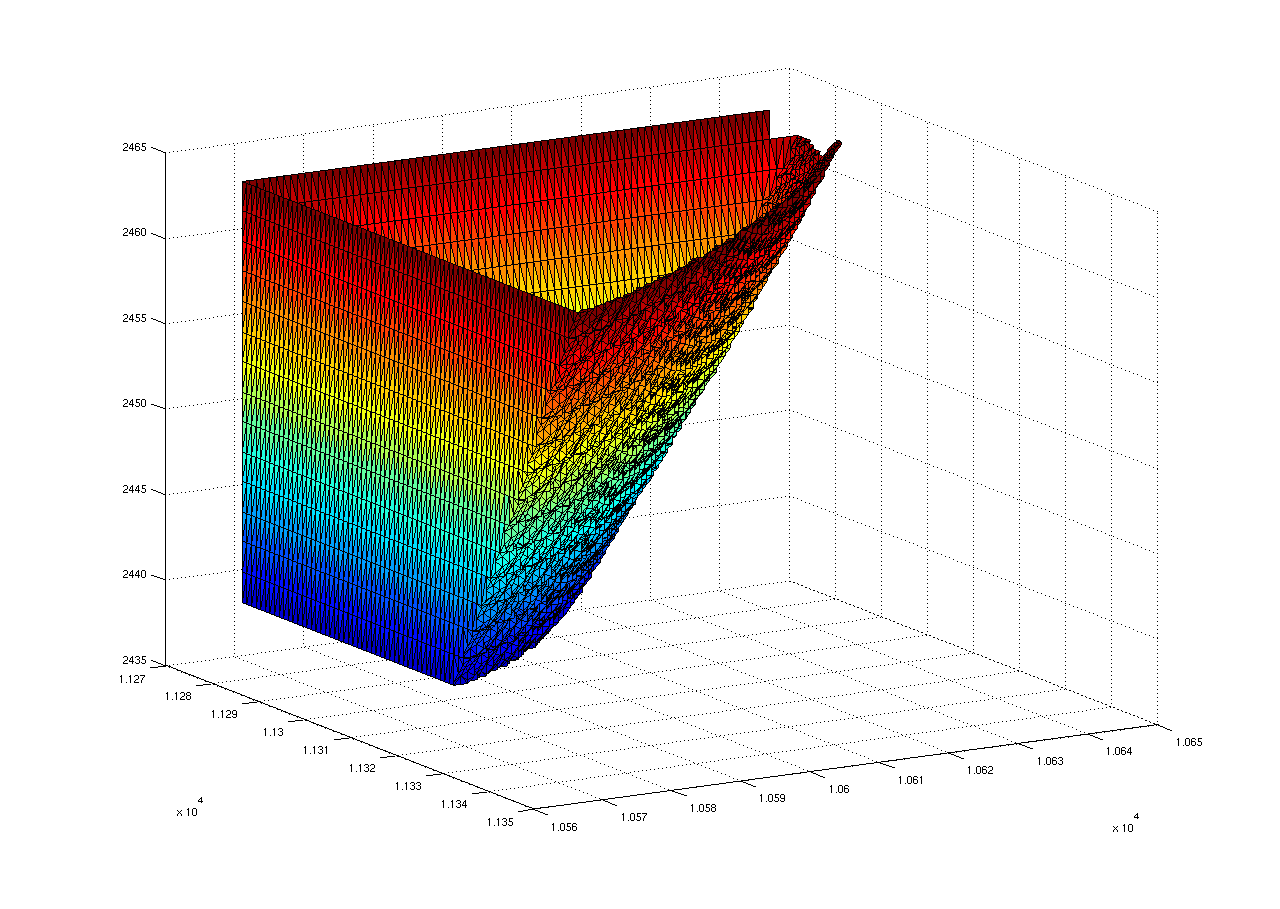}
        \caption{Partial view of triangulated boundary}
        \label{fig:spleen:triangularBoundary:a}
    \end{subfigure}
    ~ %add desired spacing between images, e. g. ~, \quad, \qquad, \hfill etc. 
    \begin{subfigure}[b]{0.4\textwidth}
        \includegraphics[width=\textwidth]{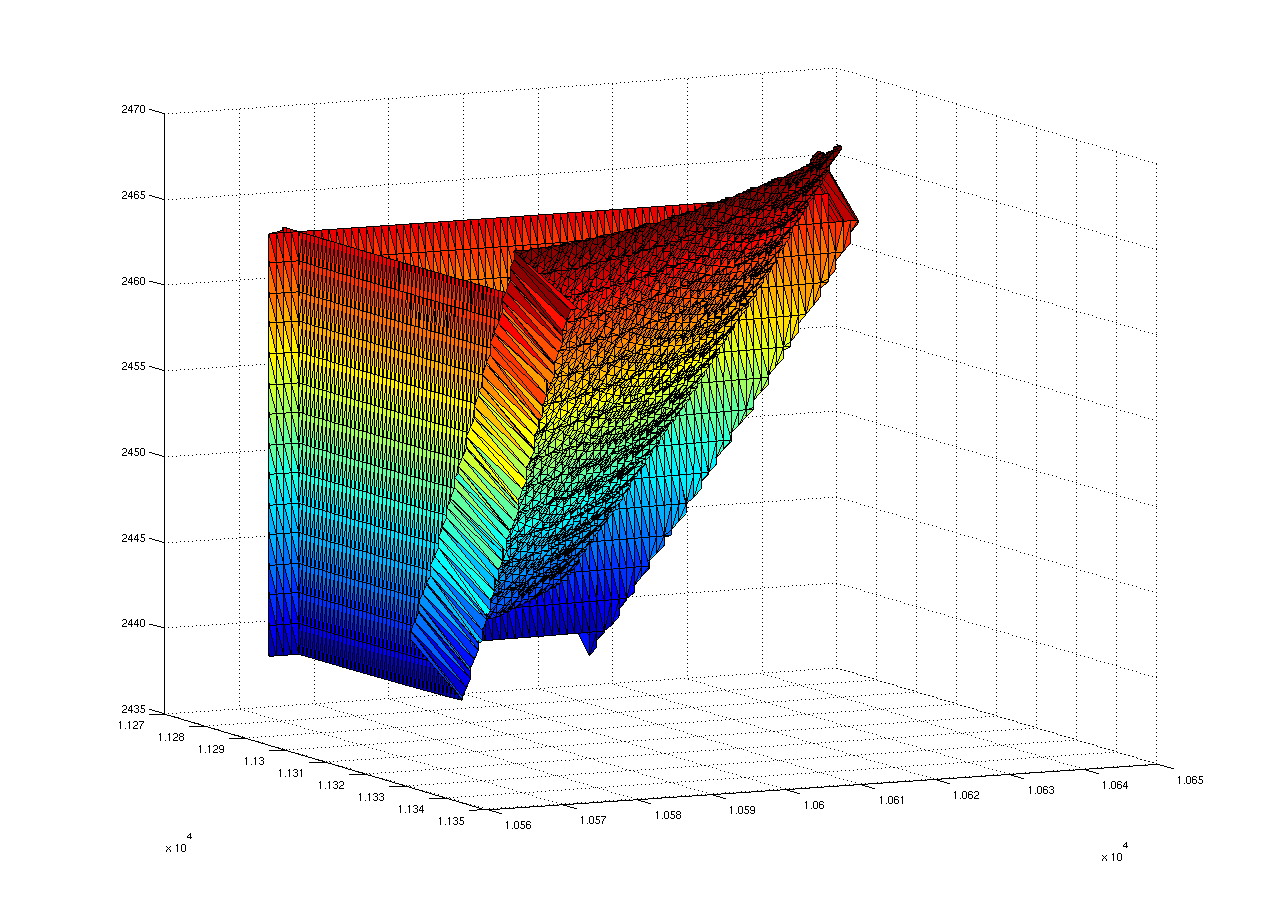}
        \caption{Naive shrinking of the boundary}
        \label{fig:spleen:triangularBoundary:b}
    \end{subfigure}
    \caption{Triangular mesh of spleen boundary}\label{fig:spleen:triangularBoundary}
\end{figure}

\subsection{Step 2: making Imaris spots from voxel data}
\label{ch:spleen:clean:2}
%MC "now" <-- watch your tenses, you switch a lot
%MC find all other isntances of MATLAB all caps

The task now is to go from the voxel data in Figure \ref{fig:spleen:rawdata:b} to Imaris spots in Figure \ref{fig:spleen:modelView:b}. 
Imaris constructs spots by looking at the voxel data and identifying a spot as a collection of voxels of high intensity. The threshold of when a voxel intensity is part of a spot is a user-defined constant in Imaris. This leads to a major problem: Imaris cannot be used to make spots for the entire spleen at once. This is because the average intensity values change throughout the spleen. The optimal thresholding value is a local value and changes throughout the spleen. If the Imaris constructs spots for the entire spleen at once, some regions will have too many spots, while other regions would appear not to contain T-cells. Looking at the raw data, we can see how Imaris is wrong in this case. Hence the spleen must be divided into smaller boxes, the spot-making function in Imaris must be used on each small box, and a unique threshold parameter must be assigned to each box. Speeding up this process is difficult. One thought was to write a \matlab extension that divides the spleen volume into smaller boxes, runs the Imaris spot making function on each box, and then sets the correct threshold value by also analyzing the raw data. After contacting Imaris, we discovered that each step is not supported. Because of these limitations, a human must divide up the spleen into smaller regions and apply the spot making function. This would not be a problem if we only needed to subdivide the spleen into a few boxes. But because of our dataset's size, we needed to subdivide the spleen into a few hundred boxes! Figure \ref{fig:spleen:makingSpots} shows the steps involved in using Imaris to make the spots in each subdivision. In total four menus must be clicked through before the spots are made, and for each menu, the user has to enter a command and click the next button. Figure \ref{fig:spleen:makingSpots:c}, or menu 2, is where the box boundary is modified, and Figure \ref{fig:spleen:makingSpots:d} is where the key spot parameter is modified. It is this parameter that must be set by the expert on each subdivision.

\begin{figure}
    \centering
    \begin{subfigure}[b]{0.24\textwidth}
        \includegraphics[width=\textwidth]{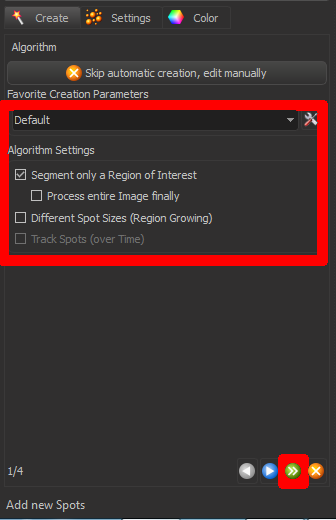}
        \caption{Menu 1}
        \label{fig:spleen:makingSpots:a}
    \end{subfigure}
    ~ %add desired spacing between images, e. g. ~, \quad, \qquad, \hfill etc. 
    \begin{subfigure}[b]{0.24\textwidth}
        \includegraphics[width=\textwidth]{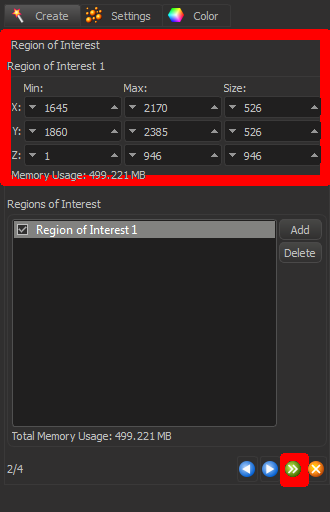}
        \caption{Menu 2}
        \label{fig:spleen:makingSpots:b}
    \end{subfigure}
    ~
    \begin{subfigure}[b]{0.24\textwidth}
        \includegraphics[width=\textwidth]{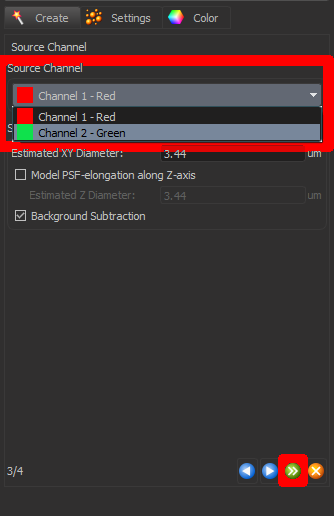}
        \caption{Menu 3}
        \label{fig:spleen:makingSpots:c}
    \end{subfigure}
    ~
    \begin{subfigure}[b]{0.24\textwidth}
        \includegraphics[width=\textwidth]{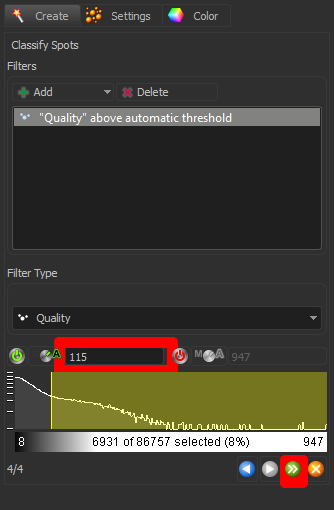}
        \caption{Menu 4}
        \label{fig:spleen:makingSpots:d}
    \end{subfigure}
            
    \caption{Sequence of menus for making spots}\label{fig:spleen:makingSpots}
\end{figure}

To reduce the necessary repetition, we sought to automatically divide the spleen into smaller boxes and run the Imaris spot making function on each box with the default value that Imaris produces on each box. The default value that Imaris uses on each small box is different for each box and is better than running the spot making function on the entire spleen at once, but a human might still need to modify the value. Hence our solution was to automatically build the spots on each box with the default value Imaris uses for each box. A human can then do a second pass through the set of boxes and modify the spot parameter on each box as necessary. 

To automatically make the spots, we wrote a script using AutoIt \cite{autoit-software}, which is a tool popular in IT administration for programming mouse and keyboard input. With this script, we automatically click through all the menus to build a few hundred boxes containing spots. Using this tool may not be an application of data science, but it is very important in reducing time spent cleaning the data. The AutoIt script allowed the researchers to better spend their time and improved their work flow. 

Dividing the spleen into smaller boxes and building spots in each box produces one problem.  The same spot can be identified twice. For example, let $B_1, B_2$ be boxes defined by 
\[B_1 = \{ x \in \R^3 \mid 0 \leq x_1 \leq 1, 0 \leq x_2 \leq 1, 0 \leq x_3 \leq 1\},\]\[B_2 = \{ x \in \R^3 \mid 1 \leq x_1 \leq 2, 0 \leq x_2 \leq 1, 0 \leq x_3 \leq 1\},  \]
and note that they have a common intersection. It is possible that a spot with center $(.9, 0, 0)$ and radius $0.2$ is produced in $B_1$. Notice that this spot produced by Imaris is allowed to extend past its box boundary! Moreover, Imaris could produce a spot with center $(1.1, 0, 0)$ and radius $0.2$ in $B_2$. This means the spots overlap, or that two different T-cells are in the same location, an impossibility. Therefore ``almost duplicate'' spots can be created. To remove these, we simply compute the nearest neighbor of each spot, and if two spots intercept, we keep just one of the spots. The nearest neighbor calculation is very similar to what is described in Section \ref{ch:spleen:knn}.

\subsection{Step 3: making Imaris filament surfaces from voxel data}
\label{ch:spleen:clean:3}

The process of going from voxel data in  \ref{fig:spleen:rawdata:a} to the Imaris surfaces in Figure \ref{fig:spleen:modelView:a} starts with applying the same techniques as for the spots. The spleen is subdivided, and the Imaris function that builds the surfaces from the voxel data is applied on each box. AutoIt can again be applied to speed things up. One difference is that duplicate surfaces are not produced by Imaris. 

But there is one key difference: sometimes Imaris will  identify noise as being true filaments. Sometimes we cannot modify Imaris parameters on a box that keep all the true filaments while deleting all the noise. Often there are small surfaces, in regions isolated from the other larger filaments, that Imaris identifies as filaments. The problem is illustrated in Figure \ref{fig:spleen:smallCleanedFilaments}. The purple area is a collection of surfaces that model the true nerve filaments. The isolated red surfaces are noise.  They are not clearly connected to a larger filament, nor do they seem to be connected by a line, which would also imply they form a filament. 

\begin{figure}
    \centering       \includegraphics[width=\textwidth]{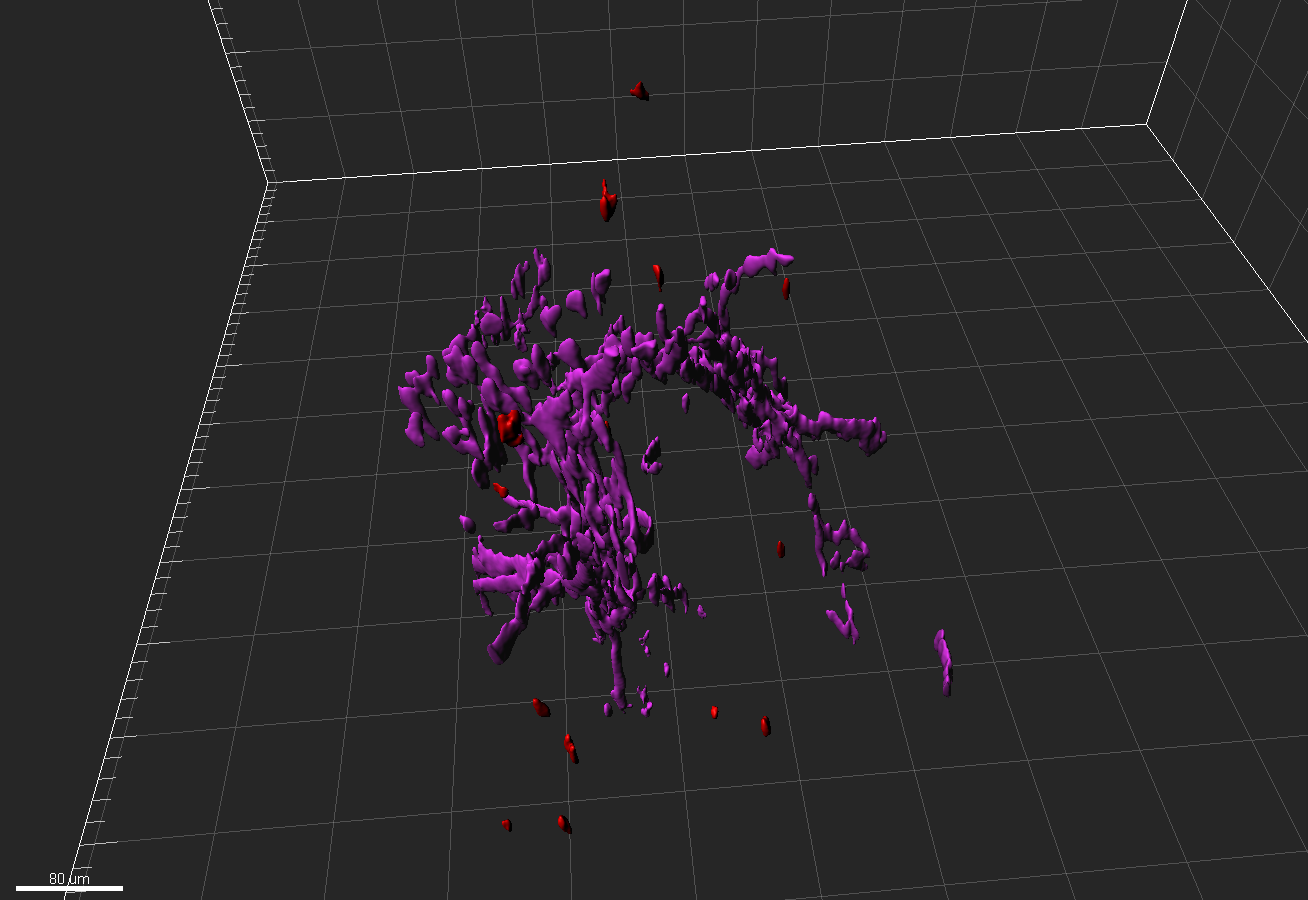}
   \caption{Red surfaces are noise, purple surfaces are true filament models }\label{fig:spleen:smallCleanedFilaments}
\end{figure}

Our strategy for removing the noise is to compute the closest neighbors for each surface. If a surface contains a large volume, or is near such a surface, we keep it. We also keep surfaces that form a large group. Figure \ref{fig:spleen:clustergoal}  illustrates in dimension two what kind of surfaces we want to identify. We keep both surfaces in the red C cluster because one surface is large, and the other is near to a large surface. We keep all the surfaces in the blue B cluster because their group size is large.  The green surfaces in cluster C are removed from everything, has a small total volume, and is a small group. Therefore the goal is to remove the surfaces in the C cluster.

\begin{figure}
    \centering       \includegraphics{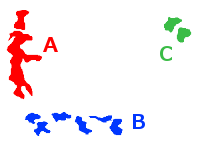}
   \caption{We remove cluster C}\label{fig:spleen:clustergoal}
\end{figure}

To implement this heuristic, we compute the distance between any two surfaces. Using this, we compute an agglomerative hierarchical clustering \cite{maimon2005data}. This creates a cluster tree where each surface starts off in its own cluster and forms the leaf nodes of the tree.  Pairs of clusters are merged as one moves up the hierarchy. The height at which two clusters are merged in the tree represents the minimum distance between the two clusters. We cut the tree at height $20 \mu$m, meaning that if two surfaces are further than $20 \mu$m, then they are not merged together. The value  $20 \mu$m was selected as it produced visually appropriate clusters. 
 
 For each resulting cluster of surfaces, we perform two heuristics on them. First, we compute the volume of each surface in a cluster. We then keep the largest clusters that account for $90\%$ of the total volume. In the example figure, this may mean that only cluster A would pass in Figure \ref{fig:spleen:clustergoal}. The clusters that fail this test  contain only small surfaces. We then apply the second heuristic on these small clusters, which is simply thresholding the number of surfaces in a cluster. In Figure \ref{fig:spleen:clustergoal} this means cluster B is kept while cluster C is removed. 
 
 Figure \ref{fig:spleen:bigCleanedFilaments} shows the result of this strategy on about 1600 surfaces, composing the filaments for the entire spleen. Originally, every surface is marked red. Our algorithm returns the surfaces to keep. The purple surfaces are the cleaned filament model, and the red surfaces are classified as noise and deleted.
 
\begin{figure}
    \centering       \includegraphics[width=\textwidth]{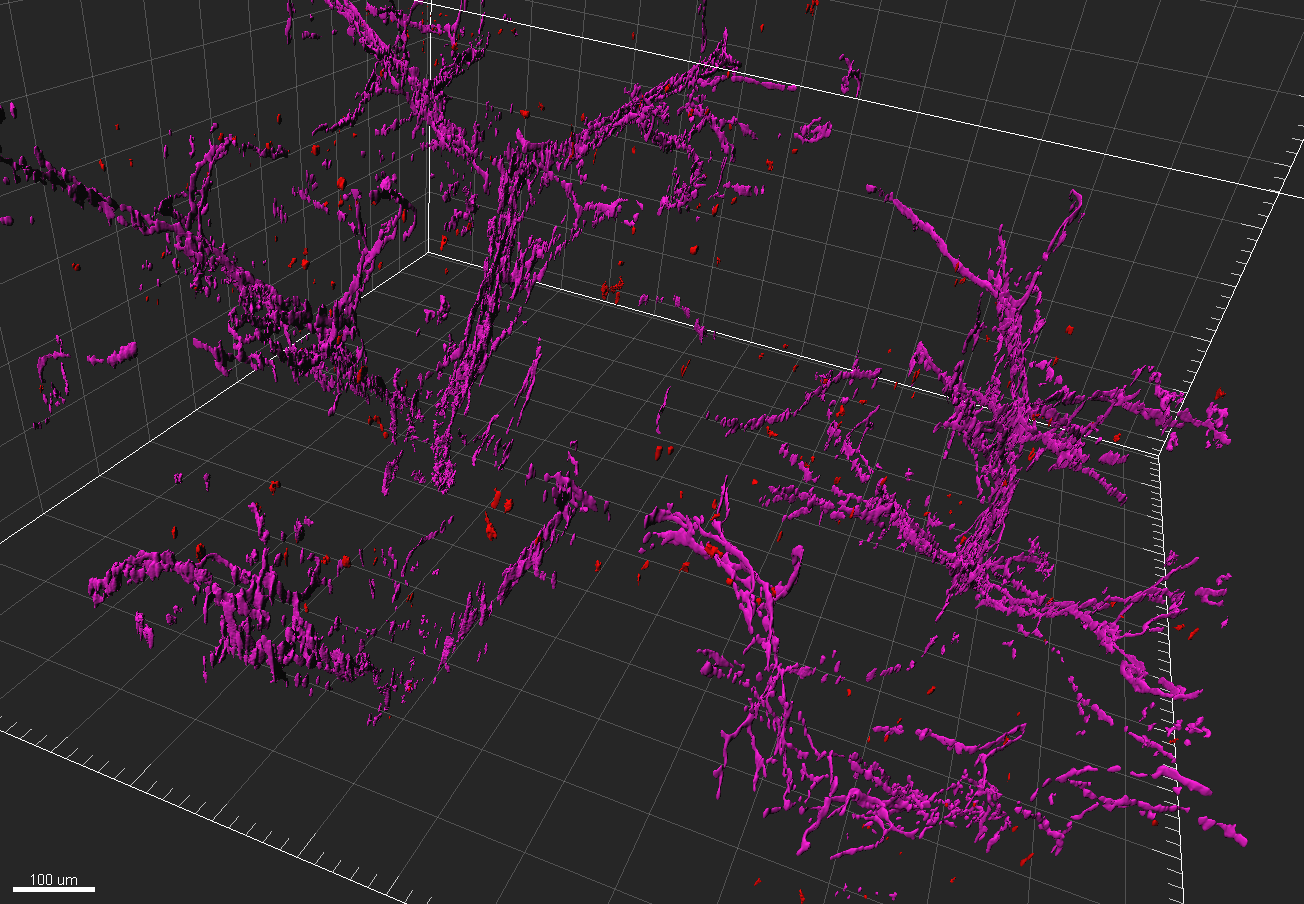}
   \caption{Red surfaces are noise, purple surfaces are true filament models }\label{fig:spleen:bigCleanedFilaments}
\end{figure}

\section{Cluster analysis}
\label{ch:spleen:sec:cluster}

The previous section described how we cleaned the data. Now we address the central research question: how does the distribution of spots and filaments change under inflammation? Figure \ref{fig:spleen:zoomedOveraly} shows a zoomed in overlay after cleaning the data in Figures \ref{fig:spleen:modelView:a} and \ref{fig:spleen:modelView:b}. We highlighted two regions in blue boxes. In the lower right box, we see that there are green spots really close to the red filaments. This suggest the T-cells and nerve cells are closely interacting, and potentially connecting. However, as the upper left box shows, many spots are not close to \emph{any} filament. To study the distribution of these two sets, we computed the distance from each spot to its closest point on a filament. For each spleen sample, we know if it came from a mouse with inflammation or not. Using the distribution of distance values between the spots and filaments, together with the inflammation label, we sought a pattern in the distance distributions. At the time of writing, work on this project was not completed. In the next sections, we describe the steps we \emph{will} perform in exploring how the distance distributions changes under  inflammation. 

\begin{figure}
    \centering       \includegraphics[width=\textwidth]{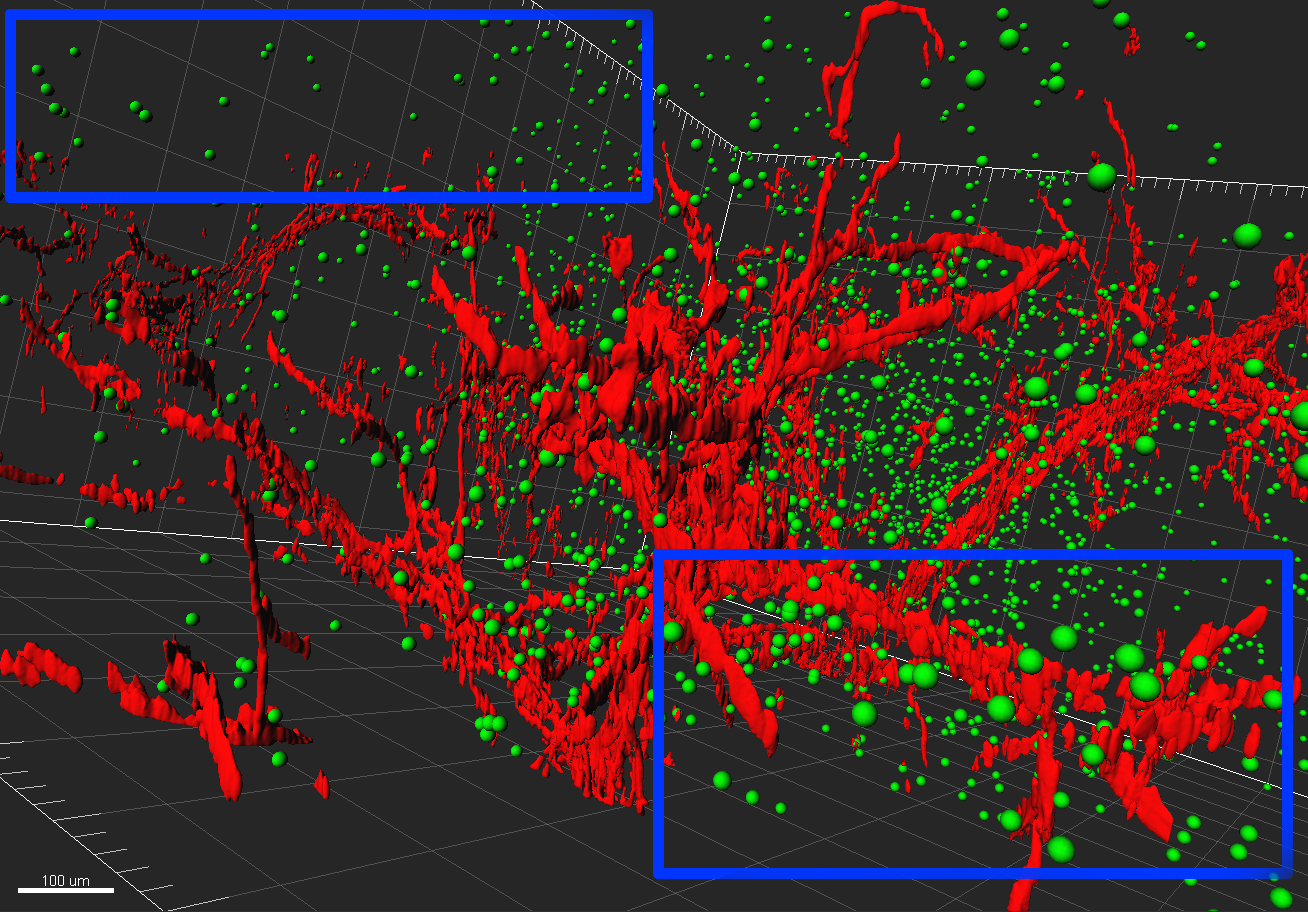}
   \caption{Zoomed in overlay of spots and filaments}\label{fig:spleen:zoomedOveraly}
\end{figure} 

\subsection{Step 1: computing nearest neighbour}
\label{ch:spleen:knn}

 Each surface in Imaris is a triangulated mesh containing vertices and triangles. Because the mesh is very fine, we approximate computing the nearest distance from a spot to a filament by instead computing for each spot, the closest distance from a spot's center to a red vertex. This is a classic nearest neighbor calculation on two point sets: for each green point, find its nearest neighbor among the red points. 

Before joining this project, the author's collaborators were using a \matlab extension from the vendor's website for computing these distance values. The algorithm employed took $O(rg)$ time where $r$ is the number of red vertices from the filaments, and $g$ is the number of green spots. The script was similar to Algorithm \ref{alg:ch:spleen:stupid}.  As $r$ is in the order of tens of millions, and $g$ is in the order of tens of thousands, this took many hours to compute. For a moderately sized spleen, this resulted in $6+$ hours of computational time to find these distances. Worse, sometimes things would crash and a day would be wasted. Again, the software tools are a shared resource, and so it could essentially take a few days to compute these distances! We immediately replaced the $O(rg)$ algorithm with a $O(g \log(r))$ average time algorithm using \matlab's nearest neighbor functions. The result was a $35x$ computing time decrease. The 6-hour computation was reduced to 10 minutes. Because this improvement was so large, we next describe the core  data structure behind \matlab's nearest neighbor search function. 

\begin{algorithm}                      
\caption{The initial algorithm from the Imaris website for computing the nearest distance between two point sets}
\label{alg:ch:spleen:stupid}
\begin{algorithmic}                    
\REQUIRE Set of spot centers $S \subset \R^3$, set of vertices $V \subset \R^3$
\ENSURE For each $s_i \in S$, the minimum distance $d_i \in \R$ to a vertex in $V$
\FORALL{ $s_i \in S$}
	\STATE $d_i \gets \infty$
	\FORALL{ $v \in V$}
		\STATE $d_i \gets \min(d_i, \norm{s_i - v})$
	\ENDFOR
\ENDFOR
\RETURN $d_i$ for $i=1,2, \dots, |S|$
\end{algorithmic}
\end{algorithm}

%\subsubsection{kd-trees} 
At the heart of \matlab's nearest neighbor function, \texttt{knnsearch}, is a kd-tree. A kd-tree is a binary tree where the leaf nodes are buckets containing points in $\R^k$ and each non-leaf node contains a coordinate index $i$ and a split value $v_i$. The left subtree of a non-leaf node contains points $x \in \R^k$ for which $x_i \leq v_i$ and the right subtree contains points where $x_i > v_i$. The non-leaf nodes implicitly define a splitting hyperplanes that are perpendicular to one axis. Figure \ref{fig:spleen:kdtreeEx} shows a kd-tree with the points  $(2,9,1)$, $(6,1,1)$, $(9,3,1)$, $(13,1,2)$, $(17,2,2)$, $(7,1,10)$, $(9,2,12)$, $(12,3,14)$, $(15,1,11)$, $(7,8,1)$, $(6,7,3)$, $(6,5,5)$, $(7,5,5)$, and $(7,6,5)$. Imagine inserting a new point, $(6, 5, 3)$ into the tree. The search process starts at the root. The root node says that points where the first coordinate is less than or equal to five are in the left child while points where the first coordinate is larger than five are in the right child. Because $6 > 5$, the right child is explored. The node $x_2:4$ tests the second coordinate of the point $(6, 5, 3)$. The right child is taken and we arrive at node $x_3:4$. The third coordinate is tested and the left child is taken. At this point, the node is simply a list of points, and the new point is appended to this list. Searching for a point is a similar process. Leaf nodes can be split if their size becomes too large. For example, if we impose a limit of four on the leaf nodes and the point $(9, 1, 10)$ is inserted, the right child of the node $x_3:9$ will have to be split. The test coordinates can be cycled, and the new node could just test the first coordinate. For a complete description of a kd-tree and how they are used for computing the nearest neighbor, see \cite{de2008computational, 1997Handbook-of-dis}.

\begin{figure}
    \centering       \includegraphics{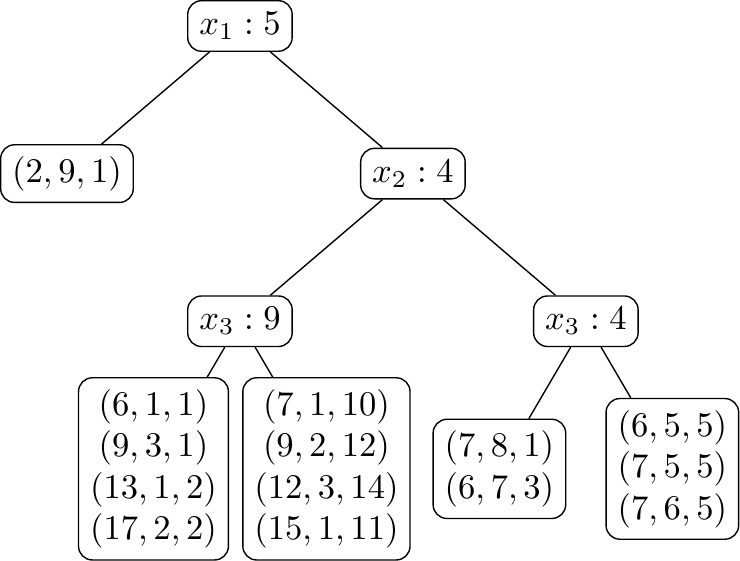}
   \caption{Example kd-tree}\label{fig:spleen:kdtreeEx}
\end{figure} 

\subsection{Step 2: dimensionality reduction}

Let $g$ be the number of spots in the data set, then the nearest neighbor calculation produces $g$ distance numbers. We compute a histogram of these distance numbers with $N$ bins. Computing the histograms of the nearest neighbor distance for each spleen allows us to study how the distribution of T-cells and nerve cells changes with inflammation. Figure \ref{fig:spleen:histo} shows one such histogram with $N=100$. The horizontal axis show the ranges of each bin, and the vertical axis is the frequency. Two observations are obvious. First, there is a large spike at the first bin, meaning there are many spots that are really close to a nerve filament. Second, most spots are not close to a nerve filament. This histogram is from a healthy spleen. At the time of writing, the data processing has not been finished for the other spleens, but it was hypothesized that the mass of the histogram would shift to the left for spleens with inflammation. 

\begin{figure}
    \centering       \includegraphics[width=0.80\textwidth]{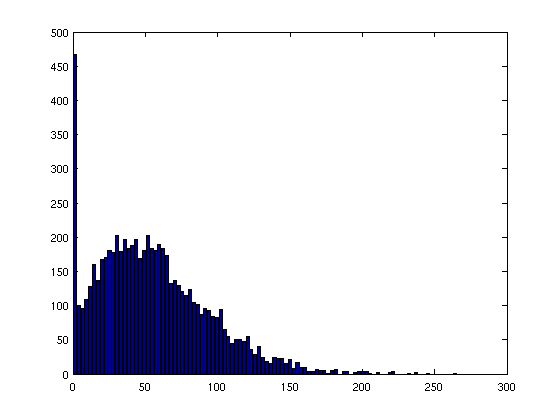}
   \caption{Histogram of nearest neighbor distances for a small spleen sample}\label{fig:spleen:histo}
\end{figure} 

Quantifying the difference between histograms with 100 bins is not easily done. Our approach was to encode the histogram into one vector $h \in \R^N$ where $h_i$ is the count in bin $i$. When we process all 18 spleens, we will have 18 histogram vectors $h^1, \dots, h^{18}$ in $\R^N$. Using the inflammation data, it would be trivial to cluster 18 points in a large dimensional space into two sets. To better analyze the 18 histograms, we first used dimensionality reduction techniques to map the $\R^N$ data into a low dimensional space such as $\R^2$. 

Principal component analysis (PCA) \cite{jolliffe2002pca} is a general technique for exploring the structure of data. It takes a set of points in $\R^N$ in the standard basis, and computes a new orthogonal basis. In the new basis, each coordinate is uncorrelated with the other coordinates. The first coordinate has the largest possible variance, and accounts for as much of the variability in the data as possible. The $i^{th}$ basis gives the direction that contains the largest variance in the data under the constraint that it is orthogonal to the preceding directions. PCA is often used to visualize a high-dimensional data set by first writing each point in the new basis, and then truncating the coordinates in each point. That is, for each data point, only the first few coordinates are kept. As an example, Figure \ref{fig:spleen:pcaEX:a} shows three different types of points in $\R^3$. It is difficult to see how these points are related. Figure \ref{fig:spleen:pcaEX:b} plots the same points after running PCA and keeping the leading two components. It is now easy to see that these three point types are clusterable. This illustrates the two main uses for PCA: visualizing high-dimensional data, and as a preprocessing step for a clustering algorithm. 

\begin{figure}
    \centering
    \begin{subfigure}[b]{0.5\textwidth}
        \includegraphics[width=\textwidth]{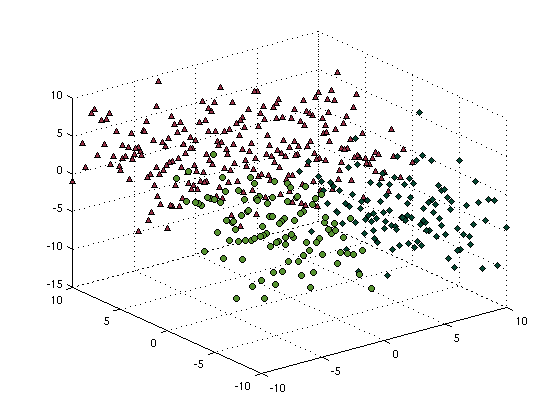}
        \caption{Raw data in $\R^3$}
        \label{fig:spleen:pcaEX:a}
    \end{subfigure}
    ~ %add desired spacing between images, e. g. ~, \quad, \qquad, \hfill etc. 
    \begin{subfigure}[b]{0.5\textwidth}
        \includegraphics[width=\textwidth]{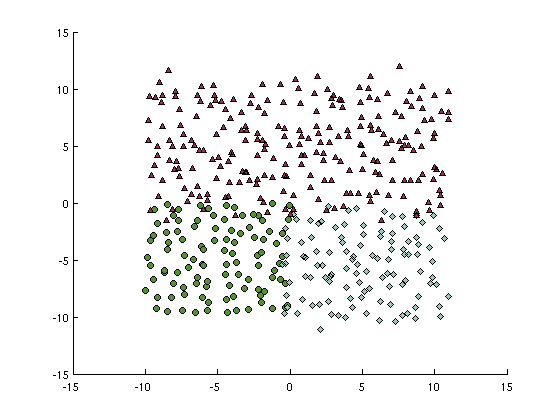}
        \caption{After mapping it into $\R^2$ with PCA}
        \label{fig:spleen:pcaEX:b}
    \end{subfigure}
            
    \caption{Example of using PCA for visualization}\label{fig:spleen:pcaEX}
\end{figure}

We now briefly mention one way to compute the PCA representation. Let $\hat H \in \R^{18 \times N}$ be a matrix containing the 18 spleen histogram vectors. PCA is a linear model, and each column of $\hat H$ should have zero mean. Let $\bar{h} = \frac{1}{18} \sum_{i=1}^{18} h^i$ be the average spleen histogram vectors, and define $H \in \R^{18 \times N}$ as the matrix formed by subtracting $\bar{h}$ from each row of $\hat H$. 

Singular value decomposition (SVD) on $H$ produces matrices $U \in \R^{18 \times 18}$, $\Sigma \in \R^{18 \times N}$ and $V \in \R^{N \times N}$ such that $H = U \Sigma V^{T}$, where $U$ and $V$ are orthogonal matrices, and $\Sigma$ is diagonal matrix with non-negative real numbers on the diagonal \cite{van2014parallel}. The histogram data in the new basis is given by the rows of $HV$. Let $1 \leq r \leq \min(18,N)$, and let $U_r$ be the first $r$ columns $U$, $V_r$ be the first $r$ columns of $V$, and $\Sigma_r$ be the upper left $r$ by $r$ submatrix. If $H$ has full rank, it can be shown that $U_r\Sigma_r V^{T}_r$ is the best rank $r$ approximation to $H$ under the Frobenius norm. The histogram data in the new basis after keeping the largest $r$ components is given by the rows of $HV_r$. Hence $V_r$ encodes a map from histogram vectors in $\R^N$ into $\R^r$. It was our belief that $r=2$ or $r=3$ would suffice for a nice mapping where the spleens labeled with inflammation are clearly grouped in the new space. 

Finally, there are many models for dimensionality reduction other than PCA. In particular, factor analysis \cite{comrey2013first}, projection pursuit \cite{huber1985projection}, independent component analysis \cite{hyvarinen2004independent}, neural networks \cite{hinton2006reducing}, and random projection methods\cite{kaski1998dimensionality, kohonen2000self} have been used for dimensionality reduction. See for example \cite{fodor2002survey, lee2007nonlinear}. These methods could also be explored. 

%After step 2, each spleen histogram has been mapped into a low dimensional space. It was our hope that the mapping will also identify the intrinsic properties of a spleen responding to inflammation. Meaning, in the low dimensional space, the spleens are well clustered together according to their inflammation label. 

%\section{Impact}

The author's impact in this collaboration was twofold. First, significant improvements to how the veterinary researchers clean and model their data was made. These improvements came about by recognizing classic problems from data science and geometry, and knowing how to implement more efficient algorithms to solve them. Second, we explored classic dimensionality reduction and clustering tools to automatically identify which spleens are connected to inflammation. Any result on how the distribution of T-cells and nervous cells is affected by inflammation can lead future researchers to better understand this interaction.

   %%%%%%%%%%%%%%%%%%%%%%%%%%%%%

   %all following chapters are part of the appendix.
   \appendix
\chapter{Computer source code}
\label{ch:AppendixALabel}
   
\section[Codes for integration and Ehrhart polynomials]{Codes for integrating a polynomial over a polytope and for computing the top coefficients of the Ehrhart polynomial}

The source code for many of the algorithms covered in Sections \ref{ch:Integration}  and \ref{ch:knapsack} are  online in the software package \latteintegrale \cite{latteintegrale}. All code is available under the GNU General Public
  License at 
\begin{center}
  \url{https://www.math.ucdavis.edu/~latte/}.
\end{center}  

The \latteintegrale 1.7.3 bundle contains all software dependencies to run \latte. Within the bundle, the folder ``latte-int-1.7.3'' contains the core \latte code. The algorithms from Section \ref{ch:Integration} can be found in two folders:
\begin{center}
 ``latte-int-1.7.3/code/latte/integration'' and ``latte-int-1.7.3/code/latte/valuation''.
\end{center} 
  Algorithms from Section \ref{ch:knapsack} are in the folder ``latte-int-1.7.3/code/latte/top-knapsack''. Refer to the manual for examples of integrating a polynomial over a polytope, and for computing Ehrhart polynomials. 

It must also be noted that \latte can also compute other things not covered in this thesis. These include counting the number of integer points in a polytope, and computing Ehrhart quasi-polynomials for integer and rational polytopes.

\section{Codes for cleaning and processing the spleen data}

This section contains the \matlab scripts used within Imaris \cite{imaris-software} and other source code from Section \ref{ch:spleen}. 

\lstset{language=Matlab,%
    %basicstyle=\color{red},
    breaklines=true,%
    morekeywords={matlab2tikz},
    keywordstyle=\color{blue},%
    morekeywords=[2]{1}, keywordstyle=[2]{\color{black}},
    identifierstyle=\color{black},%
    stringstyle=\color{red},
    commentstyle=\color{teal},%
    showstringspaces=false,%without this there will be a symbol in the places where there is a space
    numbers=left,%
    numberstyle={\tiny \color{black}},% size of the numbers
    numbersep=9pt, % this defines how far the numbers are from the text
    emph=[1]{for,end,break},emphstyle=[1]\color{red}, %some words to emphasise
    %emph=[2]{word1,word2}, emphstyle=[2]{style},    
    tabsize=4,
}

\subsection{Code for processing the boundary of the spleen}
This \matlab file is for processing the noisy boundary of the spleen from Section \ref{ch:spleen:clean:1} by shifting a triangular mesh inward. 

%\singlespace 
\begin{lstlisting}
%  Shrink a surface by its normals.
%
%
%  Installation:
%  - Copy this file into the XTensions folder
%  - You will find this function in the Image Processing menu
%
%    <CustomTools>
%      <Menu>
%       <Submenu name="Spots Functions">
%        <Item name="shrink" icon="Matlab" tooltip="shrink">
%          <Command>MatlabXT::XTshrink(%i)</Command>
%        </Item>
%       </Submenu>
%      </Menu>
%      <SurpassTab>
%        <SurpassComponent name="bpSpots">
%          <Item name="shrink" icon="Matlab" tooltip="shrink">
%            <Command>MatlabXT::XTshrink(%i)</Command>
%          </Item>
%        </SurpassComponent>
%        <SurpassComponent name="bpSurfaces">
%          <Item name="shrink" icon="Matlab" tooltip="shrink">
%            <Command>MatlabXT::XTshrink(%i)</Command>
%          </Item>
%        </SurpassComponent>
%      </SurpassTab>
%    </CustomTools>
%
%  Description:
%   Translate each vertex by its normal vector

function XTshrink(aImarisApplicationID)
	% connect to Imaris interface
	if ~isa(aImarisApplicationID, 'Imaris.IApplicationPrxHelper')
		javaaddpath ImarisLib.jar
		vImarisLib = ImarisLib;
		if ischar(aImarisApplicationID)
			aImarisApplicationID = round(str2double(aImarisApplicationID));
		end
		vImarisApplication = vImarisLib.GetApplication(aImarisApplicationID);
	else
		vImarisApplication = aImarisApplicationID;
	end

	% get the spots
	vSurfaces = vImarisApplication.GetFactory.ToSurfaces(vImarisApplication.GetSurpassSelection);

	if ~vImarisApplication.GetFactory.IsSurfaces(vSurfaces)
		msgbox('Please select a surfaces object!');
		return;
	end

	getData = true;
	while getData == true
		prompt = {'Distance to move each point:',};
		dlg_title = 'Distance';
		num_lines = 1;
		defaultans = {'10',};
		answer = inputdlg(prompt,dlg_title,num_lines,defaultans);
		scaleFactor = str2double(answer{1});
		if isnan(scaleFactor)
			disp('Incorrect input');
		else
			getData = false;
		end
	end

	vNumberOfSurfaces = 0;
	vParentGroup = vSurfaces.GetParent;
	vResult = vImarisApplication.GetFactory.CreateSurfaces;
	vEdges = [];
	vSelection = [];
	vSinglesCount = 0;
	isNewWrite = 1;
	vNumberOfChildren = vParentGroup.GetNumberOfChildren;

	for vChildIndex = 1:vNumberOfChildren
		vObjects = vParentGroup.GetChild(vChildIndex - 1);
		if vImarisApplication.GetFactory.IsSurfaces(vObjects)
			vSurface = vImarisApplication.GetFactory.ToSurfaces(vObjects);
			vSize = vSurface.GetNumberOfSurfaces;

			for vIndex = 1:vSize
				vTimeIndex = vSurface.GetTimeIndex(vIndex - 1);
				vVertices = vSurface.GetVertices(vIndex - 1);
				vTriangles = vSurface.GetTriangles(vIndex - 1);
				vNormals = vSurface.GetNormals(vIndex - 1);
				vVertices = updateVertices(vVertices, vNormals, scaleFactor);
				vResult.AddSurface(vVertices, vTriangles, vNormals, vTimeIndex);
			end

			vSurface.SetVisible(false);
			vNumberOfSurfaces = vNumberOfSurfaces + 1;
			vSinglesCount = vSinglesCount + vSize;
		end
	end

	% copy metadata from selected surfaces
	vRGBA = vSurfaces.GetColorRGBA;
	vResult.SetColorRGBA(vRGBA);
	vResult.SetName(sprintf('scaled surface %f', scaleFactor));
	vParentGroup.AddChild(vResult, -1);
end


function [newVertices] = updateVertices(vertices, normals, scaleFactor)
	%make sure the normals are of unit length
	normals = bsxfun(@times, normals, 1./sqrt(sum(normals.^2, 2)));
	newVertices = vertices - scaleFactor * normals;
end
\end{lstlisting}
%\doublespacing

\subsection{Code for processing the filament surfaces}

This \matlab file is for removing noisy filament surfaces from Section \ref{ch:spleen:clean:3}. 

%\singlespace 
\begin{lstlisting}
%  Remove noisy surfaces
%
%
%  Installation:
%  - Copy this file into the XTensions folder
%  - You will find this function in the Image Processing menu
%
% <CustomTools>
%  <Menu>
%   <Submenu name="Spots Functions">
%    <Item name="FilterFilaments" icon="Matlab" tooltip="FilterFilaments.">
% 	 <Command>MatlabXT::XTFilterFilaments(%i)</Command>
%    </Item>
%   </Submenu>
%  </Menu>
%  <SurpassTab>
%    <SurpassComponent name="bpSpots">
% 	 <Item name="FilterFilaments" icon="Matlab" tooltip="FilterFilaments.">
% 	   <Command>MatlabXT::XTFilterFilaments(%i)</Command>
% 	 </Item>
%    </SurpassComponent>
%    <SurpassComponent name="bpSurfaces">
% 	 <Item name="FilterFilaments" icon="Matlab" tooltip="FilterFilaments.">
% 	   <Command>MatlabXT::XTFilterFilaments(%i)</Command>
% 	 </Item>
%    </SurpassComponent>
%  </SurpassTab>
% </CustomTools>
%
%  Description:
%   Filters noise from filament surface data.

function XTFilterFilaments(aImarisApplicationID)
	% global parameters for easy editing
	global gp_SMALL_VOLUME_THRESHOLD;
	global gp_CLUSTER_CUTOFF;
	global gp_VOLUME_PERCENT_TEST;
	global gp_CLUSTER_CONNECTIVITY_TEST;

	% delete any surface with less than this 'volume'
	gp_SMALL_VOLUME_THRESHOLD = 300;

	% in micro meters.
	gp_CLUSTER_CUTOFF = 20;

	% larger percent = more clusters kept
	gp_VOLUME_PERCENT_TEST = 0.95;

	% number of surfaces in a cluster to keep.
	gp_CLUSTER_CONNECTIVITY_TEST = 3;

	% connect to Imaris interface
	if ~isa(aImarisApplicationID, 'Imaris.IApplicationPrxHelper')
		javaaddpath ImarisLib.jar
		vImarisLib = ImarisLib;
		if ischar(aImarisApplicationID)
			aImarisApplicationID = round(str2double(aImarisApplicationID));
		end
		vImarisApplication = vImarisLib.GetApplication(aImarisApplicationID);
	else
		vImarisApplication = aImarisApplicationID;
	end


	vSurface = vImarisApplication.GetFactory.ToSurfaces(vImarisApplication.GetSurpassSelection);
	if ~vImarisApplication.GetFactory.IsSurfaces(vSurface)
		msgbox('Please select a surfaces object!');
		return;
	end

	vSurface.SetVisible(false);
	vResult = vImarisApplication.GetFactory.CreateSurfaces;

	vSize = vSurface.GetNumberOfSurfaces;
	allVolume = zeros(vSize, 2);

	fprintf('Computing volume proxy for %d surfaces...', vSize);
	for vIndex = 1:vSize
		vVertices = vSurface.GetVertices(vIndex - 1);
		vTriangles = vSurface.GetTriangles(vIndex - 1);
		volume_proxy = 0.5 * size(vVertices, 1) + 0.5*size(vTriangles, 1);

		allVolume(vIndex, 1) = vIndex;
		allVolume(vIndex, 2) = volume_proxy;
	end
	fprintf('done.\n');

	%Filter the really small surfaces out right now.
	bigVolumeIndex = allVolume(:, 2) > gp_SMALL_VOLUME_THRESHOLD;
	deletedSurfaces = allVolume(~bigVolumeIndex, 1);
	allVolume = allVolume(bigVolumeIndex, :);

	%sort the rows of allVolume with the smallest volume first.
	[~, sortIndex] = sort(allVolume(:, 2)); 
	allVolume = allVolume(sortIndex, :);

	numSurfaces = size(allVolume, 1);

	% partition the surfaces into a 'keep' and 'delete' set.
	rng(654364);
	[surfIDKeep, surfIDDelete] = FilterSurfaces_(vSurface, allVolume);

	% insert the surfaces to keep into vResult
	fprintf('Removing surfaces %s...', sprintf('%d, ', [surfIDDelete, deletedSurfaces']));
	for vIndex = surfIDKeep
		vTimeIndex = vSurface.GetTimeIndex(vIndex - 1);
		vVertices = vSurface.GetVertices(vIndex - 1);
		vTriangles = vSurface.GetTriangles(vIndex - 1);
		vNormals = vSurface.GetNormals(vIndex - 1);
		vResult.AddSurface(vVertices, vTriangles, vNormals, vTimeIndex);
		pause(0.01); %try to prevent matlab from saying 'not responding'
	end
	fprintf('done.\n');

	vRGBA = vSurface.GetColorRGBA;
	vResult.SetColorRGBA(vRGBA);
	vResult.SetName(sprintf('Filtered Surface of %s', char(vSurface.GetName())));
	vParentGroup = vSurface.GetParent;
	vParentGroup.AddChild(vResult, -1);
end


function [surfIDKeep, surfIDDelete] = FilterSurfaces_(vSurface, allVolume)
	% Partitions the surface ids in the first column of allVolume into two
	% row vectors (surfIDKeep, surfIDDelete).
	% parameters
	%	vSurface: Imaris surface object, contains surfaces
	%	allVolume: n x 2 matrix. Colums are surface-id, volume.

	global gp_CLUSTER_CUTOFF;

	surfNumbers = allVolume(:, 1)';

	z = getLinkage_(vSurface, surfNumbers);
	%figure()
	%dendrogram(z, 0)
	%title('Dendrogram of input surfaces')


	% other interesting things that can be done with z
	% http://www.mathworks.com/help/stats/cophenet.html
	% http://www.mathworks.com/help/stats/inconsistent.html
	% http://www.mathworks.com/help/stats/cluster.html
	% http://www.mathworks.com/help/stats/silhouette.html
	%Tdist = cluster(z, 'cutoff', .5, 'criterion', 'inconsistent');
	fprintf('Clusering surfaces...');
	Tdist = cluster(z, 'cutoff', gp_CLUSTER_CUTOFF, 'criterion', 'distance');
	fprintf('done.\n');

	% cluster makes new cluster/surface id's.
	% surface surfNumbers(i) is now part of cluster Tdist(i)
	assert( length(surfNumbers) == length(Tdist));

	% Apply two test to keep surfaces.
	% 1) First keep, the largest clusters so that we keep x% of the total 
	% volume.
	[surfNKeep, surfNDelet] = filterClustersByVolume_(allVolume, surfNumbers, Tdist);
	% 2) Of the reject surfaces from test 1, keep thsoe that contain many 
	% surfaces
	[surfNDeletKeep, surfNDeletDelete] = filterClustersByConnectivity_(allVolume, surfNDelet, surfNumbers, Tdist);

	surfIDKeep = [surfNKeep, surfNDeletKeep];
	surfIDDelete = surfNDeletDelete;
end


function [z] = getLinkage_(vSurface, surfNumbers)
	%calls Matlab's linkage function.
	
	n = length(surfNumbers);
	D = zeros(n,n);

	for i = 2:n
		verticesSetI = vSurface.GetVertices(surfNumbers(i) - 1);
		verticesSetI = sampleVertices_(verticesSetI);
		for j = 1:(i-1)

			verticesSetJ = vSurface.GetVertices(surfNumbers(j) - 1);
			D(i,j) = min(pdist2(verticesSetI, verticesSetJ, 'euclidean', 'Smallest', 1));
		end
		fprintf('Finished distance for index %d/%d\n', i, n);
		pause(0.01); %try to prevent matlab from saying 'not responding'
	end

	Y = squareform(D);
	fprintf('Computing linkage...');
	z=linkage(Y, 'single'); 
	fprintf('done.\n');
end


function newVertices = sampleVertices_(oldVertices)
	%Sample surfaces if more than 500 points. 
	
	N = size(oldVertices, 1);
	if N <= 500
		newVertices = oldVertices;
		return;
	end

	numSample = round(max(500, 0.10 * N));
	numSample = min(numSample, 4000);
	newVertices = oldVertices(randsample(1:N,numSample),:);
end


function [surfNKeep, surfNDelete] = filterClustersByVolume_(volumes, oldClusterID, newClusterID)
	% Partitions oldClusterID into two sets, sufaces we will keep, and 
	% surfaces we can delete. This is done by keeping newClusterID's that 
	% account for the largest X% of the data.
	% Parameters:
	% volumes: n x 2 matrix. 1st column=old surface id, 2nd column=volume
	% oldClusterID, newClusterID: each are length n arrays.
	%   oldClusterID(i) is mapped into newClusterID(i)

	assert( length(oldClusterID) == length(newClusterID));

	% First, find the volume the new clusters contain by adding the volumes
	% of the corresponding origional surface ids
	uniqueClusters = unique(newClusterID(:));

	% 1st column=newClusterID, 2nd column=total volume of surfaces in this
	% new cluster
	newVolumes = zeros(length(uniqueClusters), 2);

	for i = 1:length(uniqueClusters)
		volumesToAdd = oldClusterID(newClusterID == uniqueClusters(i));
		idv = ismember(volumes(:, 1), volumesToAdd);
		newVolumes(i, 1) = uniqueClusters(i);
		newVolumes(i, 2) = sum(volumes(idv, 2));

	end

	%clusters2Keep is a list of newClusterID's to keep.

	%select surfaces with p% of the volume
	% larger percent = more clusters kept
	global gp_VOLUME_PERCENT_TEST;
	p = gp_VOLUME_PERCENT_TEST;
	sortedVolumes = sort(newVolumes(:, 2));
	cumulativeSortedSum = cumsum(sortedVolumes);
	cumulativeSortedSumPercent = cumulativeSortedSum/cumulativeSortedSum(end);
	
	%find the first index where cumulativeSortedSumPercent < (1-p)
	[~, cutOffIndex] = min(cumulativeSortedSumPercent <  (1-p)); 
	clusters2Keep = newVolumes(newVolumes(:, 2) >= sortedVolumes(cutOffIndex), 1);

	% clusters2Keep is a list of newClusterID's to keep, so map these ids
	% into oldClusterID's.

	surfNKeep = [];
	surfNDelete =[];
	for i =1:length(oldClusterID)
		if any( clusters2Keep == newClusterID(i))
			surfNKeep = [surfNKeep, oldClusterID(i)];
		else
			surfNDelete = [surfNDelete, oldClusterID(i)];
		end
	end
end

function [surfNKeep, surfNDelete] = filterClustersByConnectivity_(volumes, surfN, oldClusterID, newClusterID)
	% Partitions surfN  into two lits of surfaces to keep or delete
	% Parameters
	%	surfN: list of surface id's to partition. Id's are in the old 
	%     numbering. 
	%   oldClusterID(i) maps to newClusterID(i)

	global gp_CLUSTER_CONNECTIVITY_TEST;

	assert( length(oldClusterID) == length(newClusterID));
	surfNDelete = [];
	maybe = [];

	%delete any small-cluster surfaces
	for i = 1:length(surfN)
		oldIndex = find(oldClusterID == surfN(i));
		newID = newClusterID(oldIndex);
		numClusters = sum(newClusterID == newID);

		if numClusters >= gp_CLUSTER_CONNECTIVITY_TEST
			maybe = [maybe, surfN(i)];
		else
			surfNDelete = [surfNDelete, surfN(i)];
		end
	end

	%todo: add other heristics, like checking if the surface is part of a
	%larger "line" of clusters forming a filament.
	surfNKeep = maybe;
end
\end{lstlisting}
%\doublespacing

\subsection{Code for computing nearest neighbors}

This \matlab file is used in Section \ref{ch:spleen:knn}

%\singlespace 
\begin{lstlisting}
% Read in two files (spots and vertices), and write a file of nearest 
% neighbor distances


vertexFileName = 'path/RedPoints.csv';
spotFileName = 'path/GreenPoints.csv';
distanceFileName = 'output_distance_file.csv';


bucketSize = 10000;


vertices = csvread(vertexFileName);
spots = csvread(spotFileName);


fprintf('Building nearest neighbor tree...');
tic;
knn = createns(vertices, 'BucketSize', bucketSize);
t = toc;
fprintf('done. Time was %g seconds\n', t); 

fprintf('Searching over spots...');
tic;
%no need to record the vertex id of the nearest neighbor. 
[~, nearestDistance] = knnsearch(knn, spots); 
t = toc;
fprintf('done. Time was %g seconds\n', t); 

fprintf('Saving distance values...');
tic;
csvwrite(distanceFileName, nearestDistance);
t = toc;
fprintf('done. Time was %g seconds\n', t); 
\end{lstlisting}
%\doublespacing

   \backmatter
   
   \singlespace 
   \bibliographystyle{plain}
   \bibliography{biblio}

\begin{thebibliography}{100}

\bibitem{4ti2}
4ti2 team.
\newblock 4ti2---a software package for algebraic, geometric and combinatorial
  problems on linear spaces.
\newblock {A}vailable from URL \url{www.4ti2.de}.

\bibitem{aardallenstra}
Karen Aardal and Arjen~K. Lenstra.
\newblock Hard equality constrained integer knapsacks.
\newblock {\em Math. Oper. Res.}, 29(3):724--738, 2004.

\bibitem{agnarsson}
Geir Agnarsson.
\newblock On the {S}ylvester denumerants for general restricted partitions.
\newblock In {\em Proceedings of the {T}hirty-third {S}outheastern
  {I}nternational {C}onference on {C}ombinatorics, {G}raph {T}heory and
  {C}omputing ({B}oca {R}aton, {FL}, 2002)}, volume 154, pages 49--60, 2002.

\bibitem{agnati2006volume}
Luigi Agnati, Giuseppina Leo, Alessio Zanardi, Susanna Genedani, Alicia Rivera,
  Kjell Fuxe, and Diego Guidolin.
\newblock Volume transmission and wiring transmission from cellular to
  molecular networks: history and perspectives.
\newblock {\em Acta Physiologica}, 187(1-2):329--344, 2006.

\bibitem{aichholzer1996novel}
Oswin Aichholzer, Franz Aurenhammer, David Alberts, and Bernd G{\"a}rtner.
\newblock {\em A novel type of skeleton for polygons}.
\newblock Springer, 1996.

\bibitem{alexanderhirschowitz}
James Alexander and Andr\'{e} Hirschowitz.
\newblock Polynomial interpolation in several variables.
\newblock {\em J. Algebraic Geom.}, 4:201--222, 1995.

\bibitem{andersson2012new}
Ulf Andersson and Kevin Tracey.
\newblock A new approach to rheumatoid arthritis: treating inflammation with
  computerized nerve stimulation.
\newblock In {\em Cerebrum: the Dana forum on brain science}, volume 2012. Dana
  Foundation, 2012.

\bibitem{andrewsbook}
George Andrews.
\newblock {\em The theory of partitions}.
\newblock Cambridge Mathematical Library. Cambridge University Press,
  Cambridge, 1998.
\newblock Reprint of the 1976 original.

\bibitem{anjos2011handbook}
Miguel Anjos and Jean Lasserre.
\newblock {\em Handbook on Semidefinite, Conic and Polynomial Optimization}.
\newblock International Series in Operations Research \& Management Science.
  Springer US, 2011.

\bibitem{avis2000revised}
David Avis.
\newblock A revised implementation of the reverse search vertex enumeration
  algorithm.
\newblock In {\em Polytopes—combinatorics and computation}, pages 177--198.
  Springer, 2000.

\bibitem{Balas76}
Egon Balas and Manfred Padberg.
\newblock Set partitioning: A survey.
\newblock {\em SIAM Review}, 18(4):710--760, 1976.

\bibitem{baldoni-et-al:denumerant-full-paper}
Velleda Baldoni, Nicole Berline, Jes{\'u}s De~Loera, Brandon Dutra, Matthias
  K{\"o}ppe, and Mich{\`e}le Vergne.
\newblock Coefficients of {S}ylvester's denumerant.
\newblock {\em INTEGERS, Electronic Journal of Combinatorial Number Theory},
  15(A11):1--32, 2015.

\bibitem{baldoni-berline-deloera-koeppe-vergne:integration}
Velleda Baldoni, Nicole Berline, Jes{\'u}s De~Loera, Matthias K{\"o}ppe, and
  Mich{\`e}le Vergne.
\newblock How to integrate a polynomial over a simplex.
\newblock {\em Mathematics of Computation}, 80(273):297--325, 2011.

\bibitem{so-called-paper-1}
Velleda Baldoni, Nicole Berline, Jes{\'u}s De~Loera, Matthias K{\"o}ppe, and
  Mich{\`e}le Vergne.
\newblock Computation of the highest coefficients of weighted {E}hrhart
  quasi-polynomials of rational polyhedra.
\newblock {\em Foundations of Computational Mathematics}, 12:435--469, 2012.

\bibitem{Baldoni2011WeightedEhrhart}
Velleda Baldoni, Nicole Berline, Jes{\'u}s~A. De~Loera, Matthias K{\"o}ppe, and
  Mich{\`e}le Vergne.
\newblock Computation of the highest coefficients of weighted {E}hrhart
  quasi-polynomials of rational polyhedra.
\newblock {\em Foundations of Computational Mathematics}, 12(4):435--469, 2011.

\bibitem{so-called-paper-2}
Velleda Baldoni, Nicole Berline, Matthias K{\"o}ppe, and Mich{\`e}le Vergne.
\newblock Intermediate sums on polyhedra: Computation and real {E}hrhart
  theory.
\newblock {\em Mathematika}, 59(1):1--22, September 2013.

\bibitem{bar}
Alexander Barvinok.
\newblock Polynomial time algorithm for counting integral points in polyhedra
  when the dimension is fixed.
\newblock {\em Mathematics of Operations Research}, 19:769--779, 1994.

\bibitem{barvinok-2006-ehrhart-quasipolynomial}
Alexander Barvinok.
\newblock Computing the {E}hrhart quasi-polynomial of a rational simplex.
\newblock {\em Math. Comp.}, 75(255):1449--1466, 2006.

\bibitem{barvinokzurichbook}
Alexander Barvinok.
\newblock {\em Integer Points in Polyhedra}.
\newblock Z\"urich Lectures in Advanced Mathematics. European Mathematical
  Society (EMS), Z\"urich, Switzerland, 2008.

\bibitem{BarviPom}
Alexander Barvinok and James Pommersheim.
\newblock An algorithmic theory of lattice points in polyhedra.
\newblock In Louis~J. Billera, Anders Bj\"orner, Curtis Greene, Rodica~E.
  Simion, and Richard~P. Stanley, editors, {\em New Perspectives in Algebraic
  Combinatorics}, volume~38 of {\em Math. Sci. Res. Inst. Publ.}, pages
  91--147. Cambridge Univ. Press, Cambridge, 1999.

\bibitem{barvinokwood}
Alexander Barvinok and Kevin Woods.
\newblock Short rational generating functions for lattice point problems.
\newblock {\em J. Amer. Math. Soc.}, 16(4):957--979 (electronic), 2003.

\bibitem{beckgesselkomatsu}
Matthias Beck, Ira Gessel, and Takao Komatsu.
\newblock The polynomial part of a restricted partition function related to the
  {F}robenius problem.
\newblock {\em Electron. J. Combin.}, 8(1):Note 7, 5 pp. (electronic), 2001.

\bibitem{beck-haase-sottile:theorema}
Matthias Beck, Christian Haase, and Frank Sottile.
\newblock (formulas of brion, lawrence, and varchenko on rational generating
  functions for cones).
\newblock {\em The Mathematical Intelligencer}, 31(1):9--17, 2009.

\bibitem{beckrobins}
Matthias Beck and Sinai Robins.
\newblock {\em Computing the continuous discretely: integer-point enumeration
  in polyhedra}.
\newblock Undergraduate Texts in Mathematics. Springer, 2007.

\bibitem{beck-sottile:irrational}
Matthias Beck and Frank Sottile.
\newblock Irrational proofs for three theorems of {S}tanley.
\newblock {\em European Journal of Combinatorics}, 28(1):403--409, 2007.

\bibitem{Behrends2015}
S{\"o}nke {Behrends}, Ruth {H{\"u}bner}, and Anita {Sch{\"o}bel}.
\newblock {Norm Bounds and Underestimators for Unconstrained Polynomial Integer
  Minimization}.
\newblock {\em ArXiv e-prints}, February 2015.

\bibitem{bellET}
Eric Bell.
\newblock Interpolated denumerants and {L}ambert series.
\newblock {\em Amer. J. Math.}, 65:382--386, 1943.

\bibitem{Bellare:1995}
Mihir Bellare and Phillip Rogaway.
\newblock The complexity of approximating a nonlinear program.
\newblock {\em Math. Program.}, 69(3):429--441, September 1995.

\bibitem{berthold2006primal}
Timo Berthold.
\newblock Primal heuristics for mixed integer programs.
\newblock Master's thesis, Technische Universität Berlin, 2006.

\bibitem{imaris-software}
Bitplane.
\newblock Imaris 8.1.
\newblock Badenerstrasse 682, CH-8048 Zurich, Switzerland.
\newblock \url{http://www.bitplane.com/imaris/imaris}.

\bibitem{brambillaottaviani}
Maria~Chiara Brambilla and Giorgio Ottaviani.
\newblock On the {A}lexander--{H}irschowitz theorem.
\newblock e-print arXiv:math.AG/0701409v2, 2007.

\bibitem{brianchon1837}
C.J. Brianchon.
\newblock Th{\'e}or{\`e}me nouveau sur les poly{\`e}dres.
\newblock {\em {\'E}cole Polytechnique}, 15:317--319, 1837.

\bibitem{Brion88}
Michel Brion.
\newblock Points entiers dans les poly{\`e}dres convexes.
\newblock {\em Ann. Sci. {\'E}cole Norm. Sup.}, 21(4):653--663, 1988.

\bibitem{Brion1997residue}
Michel Brion and Mich{\`e}le Vergne.
\newblock Residue formulae, vector partition functions and lattice points in
  rational polytopes.
\newblock {\em J. Amer. Math. Soc.}, 10(4):797--833, 1997.

\bibitem{bronstein2005symbolic}
Manuel Bronstein.
\newblock {\em Symbolic Integration I: Transcendental Functions}.
\newblock Algorithms and Combinatorics. Springer, 2005.

\bibitem{Claudia2014}
Christoph Buchheim and Claudia D'Ambrosio.
\newblock Box-constrained mixed-integer polynomial optimization using separable
  underestimators.
\newblock In Jon Lee and Jens Vygen, editors, {\em Integer Programming and
  Combinatorial Optimization}, volume 8494 of {\em Lecture Notes in Computer
  Science}, pages 198--209. Springer International Publishing, 2014.

\bibitem{bueler-enge-fukuda-2000:exact-volume}
Benno B{\"u}eler, Andreas Enge, and Komei Fukuda.
\newblock Exact volume computation for polytopes: A practical study.
\newblock In Gil Kalai and G{\"u}nter~M. Ziegler, editors, {\em Polytopes --
  Combinatorics and Computation}, volume~29 of {\em DMV-Seminars}, Basel, 2000.
  Birkh{\"a}user Verlag.

\bibitem{cgal:c-sspo2-15b}
Fernando Cacciola.
\newblock {2D} straight skeleton and polygon offsetting.
\newblock In {\em {CGAL} User and Reference Manual}. {CGAL Editorial Board},
  {4.7} edition, 2015.

\bibitem{Carlini20125}
Enrico Carlini, Maria Catalisano, and Anthony Geramita.
\newblock The solution to the {W}aring problem for monomials and the sum of
  coprime monomials.
\newblock {\em Journal of Algebra}, 370(0):5 -- 14, 2012.

\bibitem{Castle20091285}
Mari Castle, Victoria Powers, and Bruce Reznick.
\newblock A quantitative {P}\'{o}lya's {T}heorem with zeros.
\newblock {\em Journal of Symbolic Computation}, 44(9):1285 -- 1290, 2009.
\newblock Effective Methods in Algebraic Geometry.

\bibitem{chen2005point}
Yong Chen, Hongqing Wang, David~W Rosen, and Jarek Rossignac.
\newblock A point-based offsetting method of polygonal meshes.
\newblock {\em ASME Journal of Computing and Information Science in
  Engineering}, 2005.

\bibitem{chung2013clarity}
Kwanghun Chung and Karl Deisseroth.
\newblock Clarity for mapping the nervous system.
\newblock {\em Nature methods}, 10(6):508--513, 2013.

\bibitem{chung2013structural}
Kwanghun Chung, Jenelle Wallace, Sung-Yon Kim, Sandhiya Kalyanasundaram, Aaron
  Andalman, Thomas Davidson, Julie Mirzabekov, Kelly Zalocusky, Joanna Mattis,
  Aleksandra Denisin, et~al.
\newblock Structural and molecular interrogation of intact biological systems.
\newblock {\em Nature}, 497(7449):332--337, 2013.

\bibitem{comrey2013first}
Andrew Comrey and Howard Lee.
\newblock {\em A first course in factor analysis}.
\newblock Psychology Press, 2013.

\bibitem{comtetbook}
Louis Comtet.
\newblock {\em Advanced combinatorics}.
\newblock D. Reidel Publishing Co., Dordrecht, enlarged edition, 1974.
\newblock The art of finite and infinite expansions.

\bibitem{ConwayBookOfNumbers}
John Conway and Richard Cormen.
\newblock {\em The Book of Numbers}.
\newblock Springer Science \& Business Media, 1996.

\bibitem{CUBPACK}
Ronald Cools and Ann Haegemans.
\newblock Algorithm 824: {CUBPACK}: a package for automatic cubature; framework
  description.
\newblock {\em ACM Trans. Math. Software}, 29(3):287--296, 2003.

\bibitem{cormen2009introduction}
Thomas Cormen, Charles Leiserson, Ronald Rivest, and Clifford Stein.
\newblock {\em Introduction to algorithms third edition}.
\newblock The MIT Press, 2009.

\bibitem{de2008computational}
Mark de~Berg.
\newblock {\em Computational Geometry: Algorithms and Applications}.
\newblock Springer, 2008.

\bibitem{deKlerk2015}
Etienne {de Klerk}, Jean {Lasserre}, Monique {Laurent}, and Zhao {Sun}.
\newblock {Bound-constrained polynomial optimization using only elementary
  calculations}.
\newblock {\em ArXiv e-prints}, July 2015.

\bibitem{klerk2010}
Etienne de~Klerk and Monique Laurent.
\newblock Error bounds for some semidefinite programming approaches to
  polynomial minimization on the hypercube.
\newblock {\em SIAM Journal on Optimization}, 20(6):3104--3120, 2010.

\bibitem{deKlerk2006210}
Etienne de~Klerk, Monique Laurent, and Pablo Parrilo.
\newblock A {PTAS} for the minimization of polynomials of fixed degree over the
  simplex.
\newblock {\em Theoretical Computer Science}, 361(2–3):210 -- 225, 2006.
\newblock Approximation and Online Algorithms.

\bibitem{brandon-handelman-paper}
Jes{\'u}s De~Loera, Brandon Dutra, and Matthias K\"oppe.
\newblock Approximating the maximum of a polynomial over a polytope using
  {H}andelman's decomposition and continuous generating functions, 2016.
\newblock Draft Paper.

\bibitem{latteintegrale}
Jes{\'u}s De~Loera, Brandon Dutra, Matthias K{\"o}ppe, Stanislav Moreinis,
  Gregory Pinto, and Jianqiu Wu.
\newblock A users guide for latte integrale v1.5.
\newblock Available from URL {\url{http://www.math.ucdavis.edu/~latte/}}, 2011.

\bibitem{deloera:software-exact-integration-polynomials}
Jes{\'u}s De~Loera, Brandon Dutra, Matthias K\"{o}ppe, Stanislav Moreinis,
  Gregory Pinto, and Jianqiu Wu.
\newblock Software for exact integration of polynomials over polyhedra.
\newblock {\em ACM Commun. Comput. Algebra}, 45(3/4):169--172, January 2012.

\bibitem{extratables}
Jes{\'u}s De~Loera, Brandon Dutra, Matthias K{\"o}ppe, Stanislav Moreinis,
  Gregory Pinto, and Jianqiu Wu.
\newblock Software for exact integration of polynomials over polyhedra: online
  supplement.
\newblock Available from URL
  {\url{http://www.math.ucdavis.edu/~latte/theory/SoftwareExactIntegrationPolynomialsPolyhedraOnlineSupplement.pdf}},
  2012.

\bibitem{latte-1.2}
Jes{\'u}s De~Loera, David Haws, Raymond Hemmecke, Peter Huggins, Jeremiah
  Tauzer, and Ruriko Yoshida.
\newblock {LattE}, version 1.2.
\newblock Available from URL {\url{ http://www.math.ucdavis.edu/~latte/}},
  2005.

\bibitem{de2013algebraic}
Jes{\'u}s De~Loera, Raymond Hemmecke, and Matthias K{\"o}ppe.
\newblock {\em Algebraic and Geometric Ideas in the Theory of Discrete
  Optimization}.
\newblock MOS-SIAM Series on Optimization. Society for Industrial and Applied
  Mathematics, 2013.

\bibitem{deloera-hemmecke-koeppe-weismantel:intpoly-fixeddim}
Jes{\'u}s De~Loera, Raymond Hemmecke, Matthias K{\"o}ppe, and Robert
  Weismantel.
\newblock Integer polynomial optimization in fixed dimension.
\newblock {\em Mathematics of Operations Research}, 31(1):147--153, 2006.

\bibitem{deloera-hemmecke-koeppe-weismantel:mixedintpoly-fixeddim-fullpaper}
Jes{\'u}s De~Loera, Raymond Hemmecke, Matthias K{\"o}ppe, and Robert
  Weismantel.
\newblock {FPTAS} for optimizing polynomials over the mixed-integer points of
  polytopes in fixed dimension.
\newblock {\em Mathematical Programming, Series~A}, 118:273--290, 2008.

\bibitem{latte1}
Jes{\'u}s De~Loera, Raymond Hemmecke, Jeremiah Tauzer, and Ruriko Yoshida.
\newblock Effective lattice point counting in rational convex polytopes.
\newblock {\em Journal of Symbolic Computation}, 38(4):1273--1302, 2004.

\bibitem{DRStriangbook}
Jes{\'u}s De~Loera, Joerg Rambau, and Francisco Santos.
\newblock {\em Triangulations: Structures for Algorithms and Applications},
  volume~25 of {\em Algorithms and Computation in Mathematics}.
\newblock Springer, 1st edition, 2010.

\bibitem{DyerKannan97}
Martin Dyer and Ravi Kannan.
\newblock On {B}arvinok's algorithm for counting lattice points in fixed
  dimension.
\newblock {\em Mathematics of Operations Research}, 22:545--549, 1997.

\bibitem{ehrhart1962geometrie}
Eugene Ehrhart.
\newblock Geometrie diophantienne-sur les polyedres rationnels homothetiques an
  dimensions.
\newblock {\em Comptes rendus de l'Acad{\'e}mie des sciences}, 254(4):616,
  1962.

\bibitem{ehrhartbook}
Eugene Ehrhart.
\newblock {\em Polyn\^omes arithm\'etiques et m\'ethode des poly\`edres en
  combinatoire}.
\newblock Birkh\"auser Verlag, Basel, 1977.
\newblock International Series of Numerical Mathematics, Vol. 35.

\bibitem{wagonetal}
David Einstein, Daniel Lichtblau, Adam Strzebonski, and Stan Wagon.
\newblock Frobenius numbers by lattice point enumeration.
\newblock {\em Integers}, 7:A15, 63, 2007.

\bibitem{elad2010sparse}
Michael Elad.
\newblock {\em Sparse and Redundant Representations: From Theory to
  Applications in Signal and Image Processing}.
\newblock Springer New York, 2010.

\bibitem{famm2013drug}
Kristoffer Famm, Brian Litt, Kevin Tracey, Edward Boyden, and Moncef Slaoui.
\newblock Drug discovery: a jump-start for electroceuticals.
\newblock {\em Nature}, 496(7444):159--161, 2013.

\bibitem{felten1987noradrenergic}
David Felten, Suzanne Felten, Denise Bellinger, Sonia Carlson, Kurt Ackerman,
  Kelley Madden, John Olschowki, and Shmuel Livnat.
\newblock Noradrenergic sympathetic neural interactions with the immune system:
  structure and function.
\newblock {\em Immunological reviews}, 100(1):225--260, 1987.

\bibitem{fischettiMILP}
Matteo Fischetti and Andrea Lodi.
\newblock {\em Heuristics in Mixed Integer Programming}.
\newblock John Wiley Sons, Inc., 2010.

\bibitem{fodor2002survey}
Imola Fodor.
\newblock A survey of dimension reduction techniques, 2002.

\bibitem{cddlib-094a}
Komei Fukuda.
\newblock {\texttt{cddlib}}, version 094a.
\newblock Available from URL
  \url{http://www.cs.mcgill.ca/~fukuda/soft/cdd_home/cdd.html}, 2005.

\bibitem{Garfinkel69}
Robert Garfinkel and George Nemhauser.
\newblock The set-partitioning problem: Set covering with equality constraints.
\newblock {\em Operations Research}, 17(5):848--856, 1969.

\bibitem{garfinkel1972integer}
Robert Garfinkel and George Nemhauser.
\newblock {\em Integer programming}, volume~4.
\newblock Wiley New York, 1972.

\bibitem{trip-burst-tries-gastineaau-2006}
Micka{\"e}l Gastineau and Jacques Laskar.
\newblock Development of {TRIP}: Fast sparse multivariate polynomial
  multiplication using burst tries.
\newblock In Vassil Alexandrov, Geert van Albada, Peter Sloot, and Jack
  Dongarra, editors, {\em Computational Science -- ICCS 2006}, volume 3992 of
  {\em Lecture Notes in Computer Science}, pages 446--453. Springer Berlin /
  Heidelberg, 2006.

\bibitem{polymake-software}
Ewgenij Gawrilow and Michael Joswig.
\newblock polymake: a framework for analyzing convex polytopes.
\newblock In Gil Kalai and G\"unter~M. Ziegler, editors, {\em Polytopes ---
  Combinatorics and Computation}, pages 43--74. Birkh\"auser, 2000.

\bibitem{gmp-4.1.4}
{GMP}, version 4.1.4, the {GNU} multiple precision arithmetic library.
\newblock Available from URL \url{http://www.swox.com/gmp/}, 2004.

\bibitem{1997Handbook-of-dis}
{Jacob } Goodman and {Joseph} O'{R}ourke, editors.
\newblock {\em Handbook of discrete and computational geometry}.
\newblock CRC Press, Inc., Boca Raton, FL, USA, 1997.

\bibitem{gram1871}
J{\o}rgen Gram.
\newblock Om rumvinklerne i et polyeder.
\newblock {\em Tidsskrift for Math}, 3(4):161--163, 1871.

\bibitem{grossmanReview}
Ignacio Grossmann.
\newblock Review of nonlinear mixed-integer and disjunctive programming
  techniques.
\newblock {\em Optimization and Engineering}, 3(3):227--252, 2002.

\bibitem{milpApplicationXpress}
Christelle Gu{\'e}ret, Christian Prins, Marc Sevaux, and Susanne Heipcke.
\newblock {\em Applications of Optimization with Xpress-MP}.
\newblock Dash Optimization Limited, 2002.

\bibitem{handelman1988}
David Handelman.
\newblock Representing polynomials by positive linear functions on compact
  convex polyhedra.
\newblock {\em Pacific J. Math.}, 132(1):35--62, 1988.

\bibitem{Hastad99someoptimal}
Johan H{\r{a}}stad.
\newblock Some optimal inapproximability results.
\newblock In {\em Journal of the ACM}, pages 1--10, 1999.

\bibitem{hemmeckeetal}
Raymond Hemmecke, Akimichi Takemura, and Ruriko Yoshida.
\newblock Computing holes in semi-groups and its application to transportation
  problems.
\newblock {\em Contributions to Discrete Mathematics}, 4:81--91, 2009.

\bibitem{Henrici}
Peter Henrici.
\newblock {\em Applied and computational complex analysis. {V}ol. 1}.
\newblock Wiley Classics Library. John Wiley \& Sons Inc., New York, 1988.
\newblock Power series---integration---conformal mapping---location of zeros,
  Reprint of the 1974 original, A Wiley-Interscience Publication.

\bibitem{hinton2006reducing}
Geoffrey Hinton and Ruslan Salakhutdinov.
\newblock Reducing the dimensionality of data with neural networks.
\newblock {\em Science}, 313(5786):504--507, 2006.

\bibitem{huber1985projection}
Peter Huber.
\newblock Projection pursuit.
\newblock {\em The annals of Statistics}, pages 435--475, 1985.

\bibitem{hyvarinen2004independent}
Aapo Hyv{\"a}rinen, Juha Karhunen, and Erkki Oja.
\newblock {\em Independent component analysis}, volume~46.
\newblock John Wiley \& Sons, 2004.

\bibitem{jolliffe2002pca}
Ian Jolliffe.
\newblock {\em {P}rincipal {C}omponent {A}nalysis}.
\newblock Springer, 2002.

\bibitem{KanekoBernoulli}
Masanobu Kaneko.
\newblock The {A}kiyama--{T}anigawa algorithm for {B}ernoulli numbers.
\newblock {\em Journal of Integer Sequences}, 3(00.2.9):1--7, 2000.

\bibitem{kannanfrobenius}
Ravi Kannan.
\newblock Lattice translates of a polytope and the {F}robenius problem.
\newblock {\em Combinatorica}, 12(2):161--177, 1992.

\bibitem{kaski1998dimensionality}
Samuel Kaski.
\newblock Dimensionality reduction by random mapping: Fast similarity
  computation for clustering.
\newblock In {\em Neural Networks Proceedings, 1998. IEEE World Congress on
  Computational Intelligence. The 1998 IEEE International Joint Conference on},
  volume~1, pages 413--418. IEEE, 1998.

\bibitem{kellereretalbook}
Hans Kellerer, Ulrich Pferschy, and David Pisinger.
\newblock {\em Knapsack problems}.
\newblock Springer-Verlag, Berlin, 2004.

\bibitem{dlx}
Donald Knuth.
\newblock Dancing links.
\newblock In Jim Davies, editor, {\em Millennial perspectives in computer
  science : proceedings of the 1999 Oxford-Microsoft Symposium in honour of
  Professor Sir Antony Hoare}, pages 187--214. Palgrave, Houndmills,
  Basingstoke, Hampshire, 2000.

\bibitem{miplib2010}
Thorsten Koch, Tobias Achterberg, Erling Andersen, Oliver Bastert, Timo
  Berthold, Robert Bixby, Emilie Danna, Gerald Gamrath, Ambros Gleixner, Stefan
  Heinz, Andrea Lodi, Hans Mittelmann, Ted Ralphs, Domenico Salvagnin, Daniel
  Steffy, and Kati Wolter.
\newblock {MIPLIB} 2010.
\newblock {\em Mathematical Programming Computation}, 3(2):103--163, 2011.

\bibitem{kohonen2000self}
Teuvo Kohonen, Samuel Kaski, Krista Lagus, Jarkko Saloj{\"a}rvi, Jukka Honkela,
  Vesa Paatero, and Antti Saarela.
\newblock Self organization of a massive document collection.
\newblock {\em Neural Networks, IEEE Transactions on}, 11(3):574--585, 2000.

\bibitem{koeppe:irrational-barvinok}
Matthias K{\"o}ppe.
\newblock A primal {B}arvinok algorithm based on irrational decompositions.
\newblock {\em SIAM Journal on Discrete Mathematics}, 21(1):220--236, 2007.

\bibitem{latte-macchiato}
Matthias K{\"o}ppe.
\newblock {LattE macchiato}, version 1.2-mk-0.9.3, an improved version of {De
  Loera} et al.'s {LattE} program for counting integer points in polyhedra with
  variants of {Barvinok}'s algorithm.
\newblock Available from URL
  {\url{http://www.math.ucdavis.edu/~mkoeppe/latte/}}, 2008.

\bibitem{koeppeComplexity2012}
Matthias K{\"o}ppe.
\newblock On the complexity of nonlinear mixed-integer optimization.
\newblock In Jon Lee and Sven Leyffer, editors, {\em Mixed Integer Nonlinear
  Programming}, volume 154 of {\em The IMA Volumes in Mathematics and its
  Applications}, pages 533--557. Springer New York, 2012.

\bibitem{koeppe-verdoolaege:parametric}
Matthias K\"oppe and Sven Verdoolaege.
\newblock Computing parametric rational generating functions with a primal
  {B}arvinok algorithm.
\newblock {\em The Electronic Journal of Combinatorics}, 15:1--19, 2008.
\newblock \#R16.

\bibitem{Krivine1964}
Jean-Louis Krivine.
\newblock Quelques propri\'et\'es des pr\'eordres dans les anneaux commutatifs
  unitaires.
\newblock {\em Comptes Rendus de l'Acad\'emie des Sciences de Paris},
  258:3417--3418, 1964.

\bibitem{lasserre-volume1983}
Jean Lasserre.
\newblock An analytical expression and an algorithm for the volume of a convex
  polyhedron in $\mathbb{R}^n$.
\newblock {\em Journal of Optimization Theory and Applications},
  39(3):363--377, 1983.

\bibitem{Lasserre2000929}
Jean Lasserre.
\newblock Optimisation globale et théorie des moments.
\newblock {\em Comptes Rendus de l'Académie des Sciences - Series I -
  Mathematics}, 331(11):929 -- 934, 2000.

\bibitem{Lasserre01globaloptimization}
Jean Lasserre.
\newblock Global optimization with polynomials and the problem of moments.
\newblock {\em SIAM Journal on Optimization}, 11:796--817, 2001.

\bibitem{lasserre2002semidefinite}
Jean Lasserre.
\newblock Semidefinite programming vs. lp relaxations for polynomial
  programming.
\newblock {\em Mathematics of operations research}, 27(2):347--360, 2002.

\bibitem{lasserre2009momentsBook}
Jean Lasserre.
\newblock {\em Moments, Positive Polynomials and Their Applications}.
\newblock Imperial College Press Optimization. Imperial College Press, 2009.

\bibitem{lasserre2011NewLook}
Jean Lasserre.
\newblock A new look at nonnegativity on closed sets and polynomial
  optimization.
\newblock {\em SIAM Journal on Optimization}, 21(3):864--885, 2011.

\bibitem{Thanh2013}
Jean Lasserre and Tung~Phan Thanh.
\newblock Convex underestimators of polynomials.
\newblock {\em Journal of Global Optimization}, 56(1):1--25, 2013.

\bibitem{laurentsurvey}
Monique Laurent.
\newblock Sums of squares, moment matrices and optimization over polynomials.
\newblock In {\em Emerging applications of algebraic geometry}, pages 157--270.
  Springer, 2009.

\bibitem{monique2014}
Monique Laurent and Zhao Sun.
\newblock Handelman’s hierarchy for the maximum stable set problem.
\newblock {\em Journal of Global Optimization}, 60(3):393--423, 2014.

\bibitem{lawrence91-2}
Jim Lawrence.
\newblock Rational-function-valued valuations on polyhedra.
\newblock In {\em Discrete and computational geometry (New Brunswick, NJ,
  1989/1990)}, volume~6 of {\em DIMACS Ser. Discrete Math. Theoret. Comput.
  Sci.}, pages 199--208. Amer. Math. Soc., Providence, RI, 1991.

\bibitem{leeRegularTriangulations}
Carl Lee.
\newblock Regular triangulations of convex polytopes.
\newblock In P.~Gritzmann and B.~Sturmfels, editors, {\em Applied Geometry and
  Discrete Mathematics: The Victor}, pages 443--456. American Mathematical
  Society, Providence, RI,, 1991.

\bibitem{lee2007nonlinear}
John Lee and Michel Verleysen.
\newblock {\em Nonlinear dimensionality reduction}.
\newblock Springer Science \& Business Media, 2007.

\bibitem{linke:rational-ehrhart}
Eva Linke.
\newblock Rational {E}hrhart quasi-polynomials.
\newblock {\em Journal of Combinatorial Theory, Series A}, 118(7):1966--1978,
  2011.

\bibitem{lisonek}
Petr Lison{\v{e}}k.
\newblock Denumerants and their approximations.
\newblock {\em J. Combin. Math. Combin. Comput.}, 18:225--232, 1995.

\bibitem{liu2011fast}
Shengjun Liu and Charlie~CL Wang.
\newblock Fast intersection-free offset surface generation from freeform models
  with triangular meshes.
\newblock {\em Automation Science and Engineering, IEEE Transactions on},
  8(2):347--360, 2011.

\bibitem{liu2008vertex}
Yang Liu and Wenping Wang.
\newblock On vertex offsets of polyhedral surfaces.
\newblock {\em Proc. of Advances in Architectural Geometry}, pages 61--64,
  2008.

\bibitem{autoit-software}
AutoIt~Consulting Ltd.
\newblock Autoit v3.3.14.2.
\newblock https://www.autoitscript.com/site/.

\bibitem{maimon2005data}
Oded Maimon and Lior Rokach.
\newblock {\em Data mining and knowledge discovery handbook}, volume~2.
\newblock Springer, 2005.

\bibitem{marshall2008positive}
Murray Marshall.
\newblock {\em Positive Polynomials and Sums of Squares}.
\newblock Mathematical surveys and monographs. American Mathematical Society,
  2008.

\bibitem{marsten1981exact}
Roy Marsten and Fred Shepardson.
\newblock Exact solution of crew scheduling problems using the set partitioning
  model: recent successful applications.
\newblock {\em Networks}, 11(2):165--177, 1981.

\bibitem{martellotothbook}
Silvano Martello and Paolo Toth.
\newblock {\em Knapsack problems}.
\newblock Wiley-Interscience Series in Discrete Mathematics and Optimization.
  John Wiley \& Sons Ltd., Chichester, 1990.
\newblock Algorithms and computer implementations.

\bibitem{MATLAB:2015b}
MATLAB.
\newblock {\em version 8.6.0 (R2015b)}.
\newblock The MathWorks Inc., Natick, Massachusetts, 2015.

\bibitem{metz2005takes}
Christine Metz and Kevin Tracey.
\newblock It takes nerve to dampen inflammation.
\newblock {\em Nature immunology}, 6(8):756--757, 2005.

\bibitem{motzkin-straus}
Theodore Motzkin and Ernst Straus.
\newblock Maxima for graphs and a new proof of a theorem of {T}ur\'an.
\newblock {\em Canadian Journal of Mathematics}, 17:533--540, 1965.

\bibitem{muller2014cell}
Arnaud Muller, Victory Joseph, Paul Slesinger, and David Kleinfeld.
\newblock Cell-based reporters reveal in vivo dynamics of dopamine and
  norepinephrine release in murine cortex.
\newblock {\em nAture methods}, 11(12):1245--1252, 2014.

\bibitem{nesterov2000sqFunctinalSystems}
Yurii Nesterov.
\newblock Squared functional systems and optimization problems.
\newblock In Hans Frenk, Kees Roos, Tam\'{a}s Terlaky, and Shuzhong Zhang,
  editors, {\em High Performance Optimization}, volume~33 of {\em Applied
  Optimization}, pages 405--440. Springer US, 2000.

\bibitem{nguyen2010vivo}
Quoc-Thang Nguyen, Lee Schroeder, Marco Mank, Arnaud Muller, Palmer Taylor,
  Oliver Griesbeck, and David Kleinfeld.
\newblock An in vivo biosensor for neurotransmitter release and in situ
  receptor activity.
\newblock {\em Nature neuroscience}, 13(1):127--132, 2010.

\bibitem{parrilo2003SDP}
Pablo Parrilo.
\newblock Semidefinite programming relaxations for semialgebraic problems.
\newblock {\em Mathematical Programming}, 96(2):293--320, 2003.

\bibitem{ParriloSturmfels2003}
Pablo Parrilo and Bernd Sturmfels.
\newblock Minimizing polynomial functions.
\newblock In Saugata Basu and Laureano Gonzalez-Vega, editors, {\em Algorithmic
  and Quantitative Real Algebraic Geometry}, volume~60 of {\em {DIMACS} Series
  in Discrete Mathematics and Theoretical Computer Science}, pages 83--99. AMS,
  2003.

\bibitem{pavlov2012vagus}
Valentin Pavlov and Kevin Tracey.
\newblock The vagus nerve and the inflammatory reflex—linking immunity and
  metabolism.
\newblock {\em Nature Reviews Endocrinology}, 8(12):743--754, 2012.

\bibitem{Pfeifle2003}
Julian Pfeifle and J{\"o}rg Rambau.
\newblock Computing triangulations using oriented matroids.
\newblock In Michael Joswig and Nobuki Takayama, editors, {\em Algebra,
  Geometry, and Software Systems}, pages 49--75. Springer, 2003.

\bibitem{pukhlikov1992riemann}
Aleksandr Pukhlikov and Askold Khovanskii.
\newblock The {R}iemann-{R}och theorem for integrals and sums of
  quasipolynomials on virtual polytopes.
\newblock {\em Algebra i analiz}, 4(4):188--216, 1992.

\bibitem{topcom-software}
Joerg Rambau.
\newblock Topcom.
\newblock {A}vailable from {URL}
  \url{http://www.rambau.wm.uni-bayreuth.de/TOPCOM/}.

\bibitem{rambau2002topcom}
Joerg Rambau.
\newblock {\em TOPCOM: Triangulations of Point Configurations and Oriented
  Matroids}.
\newblock Konrad-Zuse-Zentrum f{\"u}r Informationstechnik Berlin. ZIB, 2002.

\bibitem{ramirezalfonsinbook}
Jorge Ram{\'{\i}}rez~Alfons{\'{\i}}n.
\newblock {\em The {D}iophantine {F}robenius problem}, volume~30 of {\em Oxford
  Lecture Series in Mathematics and its Applications}.
\newblock Oxford University Press, Oxford, 2005.

\bibitem{riordanbook}
John Riordan.
\newblock {\em An introduction to combinatorial analysis}.
\newblock Dover Publications Inc., Mineola, NY, 2002.
\newblock Reprint of the 1958 original [Wiley, New York; MR0096594 (20
  \#3077)].

\bibitem{rosas2011acetylcholine}
Mauricio Rosas-Ballina, Peder Olofsson, Mahendar Ochani, Sergio
  Vald{\'e}s-Ferrer, Yaakov Levine, Colin Reardon, Michael Tusche, Valentin
  Pavlov, Ulf Andersson, Sangeeta Chavan, et~al.
\newblock Acetylcholine-synthesizing t cells relay neural signals in a vagus
  nerve circuit.
\newblock {\em Science}, 334(6052):98--101, 2011.

\bibitem{sage}
Sage.
\newblock {\em {S}age {M}athematics {S}oftware ({V}ersion 6.10)}, 2016.
\newblock {A}vailable from {URL} \url{http://www.sagemath.org}.

\bibitem{sankaranarayanan2013lyapunov}
Sriram Sankaranarayanan, Xin Chen, and Erika Abrah{\'a}m.
\newblock Lyapunov function synthesis using handelman representations.
\newblock In {\em The 9th IFAC Symposium on Nonlinear Control Systems}, pages
  576--581, 2013.

\bibitem{sarter2009phasic}
Martin Sarter, Vinay Parikh, and Matthew Howe.
\newblock Phasic acetylcholine release and the volume transmission hypothesis:
  time to move on.
\newblock {\em Nature Reviews Neuroscience}, 10(5):383--390, 2009.

\bibitem{Schrijver2003Combinatorial-O}
Alexandeer Schrijver.
\newblock {\em Combinatorial Optimization: Polyhedra and Efficiency}.
\newblock Springer, 2003.

\bibitem{schrijver}
Alexander Schrijver.
\newblock {\em Theory of linear and integer programming}.
\newblock John Wiley \& Sons, Inc., New York, NY, USA, 1986.

\bibitem{ntl-5.4}
Victor Shoup.
\newblock {NTL}, a library for doing number theory.
\newblock Available from URL \url{http://www.shoup.net/ntl/}, 2005.

\bibitem{sillszeilberger}
Andrew Sills and Doron Zeilberger.
\newblock Formul\ae \, for the number of partitions of $n$ into at most $m$
  parts (using the quasi-polynomial ansatz).
\newblock {\em Adv. in Appl. Math.}, 48:640--645, 2012.

\bibitem{stanley}
Richard~P. Stanley.
\newblock {\em Enumerative Combinatorics}, volume~I.
\newblock Cambridge, 1997.

\bibitem{tawarmalani2002convexification}
Mohit Tawarmalani. and Nikolaos Sahinidis.
\newblock {\em Convexification and Global Optimization in Continuous and
  Mixed-Integer Nonlinear Programming: Theory, Algorithms, Software, and
  Applications}.
\newblock Nonconvex Optimization and Its Applications. Springer, 2002.

\bibitem{baron}
Mohit Tawarmalani and Nikolaos Sahinidis.
\newblock {A polyhedral branch-and-cut approach to global optimization}.
\newblock {\em Mathematical Programming}, 103:225--249, 2005.

\bibitem{tournier2003neuro}
Jean-Nicolas Tournier and Anne Hellmann.
\newblock Neuro-immune connections: evidence for a neuro-immunological synapse.
\newblock {\em Trends in immunology}, 24(3):114--115, 2003.

\bibitem{tsai2009phasic}
Hsing-Chen Tsai, Feng Zhang, Antoine Adamantidis, Garret Stuber, Antonello
  Bonci, Luis De~Lecea, and Karl Deisseroth.
\newblock Phasic firing in dopaminergic neurons is sufficient for behavioral
  conditioning.
\newblock {\em Science}, 324(5930):1080--1084, 2009.

\bibitem{van2014parallel}
Henk van~der Vorst and Paul Van~Dooren.
\newblock {\em Parallel algorithms for numerical linear algebra}, volume~1.
\newblock Elsevier, 2014.

\bibitem{vemuganti1999applications}
Rao Vemuganti.
\newblock Applications of set covering, set packing and set partitioning
  models: A survey.
\newblock In {\em Handbook of combinatorial optimization}, pages 573--746.
  Springer, 1999.

\bibitem{barvinok-noversion}
Sven Verdoolaege.
\newblock {\texttt{barvinok}}.
\newblock Available from URL \url{http://freshmeat.net/projects/barvinok/},
  2007.

\bibitem{wapner2005pea}
Leonard Wapner.
\newblock {\em The Pea and the Sun: A Mathematical Paradox}.
\newblock Ak Peters Series. Taylor \& Francis, 2005.

\bibitem{GuoceXin2004}
Guoce Xin.
\newblock A fast algorithm for {MacMahon}'s partition analysis.
\newblock {\em Electr. J. Comb.}, 11(1), 2004.

\bibitem{guocexin}
Guoce Xin.
\newblock A {Euclid} style algorithm for {MacMahon} partition analysis.
\newblock e-print http://arxiv.org/abs/1208.6074, 2012.

\bibitem{zieglerpolybook}
G\"{u}nter Ziegler.
\newblock {\em Lectures on Polytopes}.
\newblock Number 152 in Graduate texts in Mathematics. Springer, New York,
  1995.

\end{thebibliography}

\end{document}